\documentclass[10pt,a4paper,final,reqno]{amsart}
\usepackage{amsmath}
\usepackage{paralist}
\usepackage[english]{babel}
\usepackage{amsfonts}
\usepackage{amscd}
\usepackage{amssymb}
\usepackage{amsmath}
\usepackage{amstext}
\usepackage{latexsym}
\usepackage{enumitem}
\usepackage{scalerel,stackengine}
\usepackage{xspace}
\usepackage[all]{xy}
\usepackage[me]{optional}
\usepackage[colorlinks=true]{hyperref}

\newcommand{\CPSTRING}{\textsc{\tiny{CP}}}
\newcommand{\CPS}{{\CPSTRING}}
\newcommand{\CPHS}{\ensuremath{{\CPSTRING,h}}}
\newcommand{\CPDS}{\ensuremath{{\CPSTRING,\cdot}}}
\newcommand{\CPZS}{\ensuremath{{\CPSTRING,0}}}
\newcommand{\CPTS}{\ensuremath{{\CPSTRING,t}}}
\newcommand{\MWSTRING}{\textsc{\tiny{MW}}}
\newcommand{\MWS}{\ensuremath{{\MWSTRING}}}

\newcommand{\FDMSCP}{FDMS\xspace}
\newcommand{\FDMSCPs}{FDMSs\xspace}
\newcommand{\FDMSCPExp}{\ensuremath{\text{FDMS}_{TQ}}\xspace}
\newcommand{\FDMSCPsExp}{\ensuremath{\text{FDMSs}_{TQ}}\xspace}
\newcommand{\FDMSMW}{\ensuremath{\text{FDMS}_{Q\times Q}}\xspace}
\newcommand{\FDMSMWs}{\ensuremath{\text{FDMSs}_{Q\times Q}}\xspace}

\newcommand{\discretizationCPExp}{\ensuremath{\text{discretization}_{TQ}}\xspace}
\newcommand{\discretizationsCPExp}{\ensuremath{\text{discretizations}_{TQ}}\xspace}
\newcommand{\discretizationMW}{\ensuremath{\text{discretization}_{Q\times Q}}\xspace}
\newcommand{\discretizationsMW}{\ensuremath{\text{discretizations}_{Q\times Q}}\xspace}
\newcommand{\FMS}{FMS\xspace}
\newcommand{\FMSs}{FMSs\xspace}
\newcommand{\FVF}[1]{\ensuremath{\check{#1}}}
\newcommand{\gF}{\ensuremath{\alpha}}
\newcommand{\gAF}{\ensuremath{\alpha}}
\newcommand{\dAFop}{\ensuremath{\mu}}
\newcommand{\dAFtp}{\ensuremath{\sigma}}
\newcommand{\dAFtpst}{\ensuremath{\sigma^{*}}}
\newcommand{\dAFnp}{\ensuremath{\sigma^{\mathcal{C}_N}}}
\newcommand{\dAFfhop}{\ensuremath{\mu^{\mathcal{V}_h}}}
\newcommand{\dAFfhtp}{\ensuremath{\sigma^{\mathcal{V}_h}}}
\newcommand{\mwAFop}{\ensuremath{\mu_d}}
\newcommand{\mwAFtp}{\ensuremath{\sigma_d}}
\newcommand{\mwAFnp}{\ensuremath{\sigma_d}}
\newcommand{\mwAFDop}{\ensuremath{\mu_{\MWS,\cdot}}}
\newcommand{\mwAFDtp}{\ensuremath{\sigma_{\MWS,\cdot}}}

\newcommand{\cAFop}{\ensuremath{\nu_{L,f}}}
\newcommand{\dBTp}{\ensuremath{\alpha^+}}
\newcommand{\dBTm}{\ensuremath{\alpha^-}}
\newcommand{\grade}{\ensuremath{\mathfrak{h}}}
\newcommand{\discMW}{\ensuremath{\Delta_{Q\times Q}}}
\newcommand{\discMWE}{\ensuremath{\Delta_{Q\times Q}^E}}
\newcommand{\discCP}{\ensuremath{\Delta}}
\newcommand{\discCPE}{\ensuremath{\discCP^E}}
\newcommand{\discCPexp}{\ensuremath{\Delta_{TQ}}}
\newcommand{\discCPEexp}{\ensuremath{\discCP_{TQ}^E}}
\newcommand{\aFDLS}{\ensuremath{\Sigma}}
\newcommand{\aFDLSh}{\ensuremath{\Sigma_h}}
\newcommand{\aFDLSE}{\ensuremath{\Sigma^E}}

\newcommand{\aFDMS}{\ensuremath{\Sigma_{Q\times Q}}}
\newcommand{\aFDMSh}{\ensuremath{\Sigma_{Q\times Q,h}}}
\newcommand{\aFDMSE}{\ensuremath{\Sigma_{Q\times Q}^E}}
\newcommand{\aFDMSEh}{\ensuremath{\Sigma_{Q\times Q,h}^E}}
\newcommand{\cupdot}{\mathbin{\mathaccent\cdot\cup}}

\newcommand{\jdef}[1]{\index{#1}\emph{#1}}
\newcommand{\ti}[1]{\widetilde{#1}}
\newcommand{\baseRing}[1]{\ensuremath{\mathbb{#1}}}
\newcommand{\N}{\baseRing{N}}
\newcommand{\Z}{\baseRing{Z}}
\newcommand{\R}{\baseRing{R}}
\newcommand{\F}{\baseRing{F}}
\newcommand{\DD}{\baseRing{D}}
\newcommand{\CP}{\baseRing{P}}
\newcommand{\NZ}{\ensuremath{\N\cup\{0\}}}
\newcommand{\del}{\ensuremath{\partial}}
\newcommand{\stext}[1]{\ensuremath{\quad\text{#1}\quad}}
\newcommand{\transpose}[1]{{#1}^t}
\providecommand{\abs}[1]{\ensuremath{\left\lvert{#1}\right\rvert}}
\providecommand{\norm}[1]{\ensuremath{\left\lVert{#1}\right\rVert}}
\newcommand{\pd}[2]{{\frac{\partial {#1}}{\partial {#2}}}}
\newcommand{\svec}[2]{\ensuremath{\bigl( \begin{smallmatrix} 
      {#1}\\{#2} \end{smallmatrix} \bigr)}}
\newcommand{\SM}{\ensuremath{\smallsetminus}}
\newcommand{\VF}{\mathfrak{X}}
\newcommand{\conj}{\overline}
\newcommand{\imp}{\ensuremath{\Rightarrow}\xspace}

\stackMath
\newcommand\widecheck[1]{%
\savestack{\tmpbox}{\stretchto{%
  \scaleto{%
    \scalerel*[\widthof{\ensuremath{#1}}]{\kern-.6pt\bigwedge\kern-.6pt}%
    {\rule[-\textheight/2]{1ex}{\textheight}}
  }{\textheight}%
}{0.5ex}}%
\stackon[1pt]{#1}{\scalebox{-1}{\tmpbox}}%
}
\DeclareMathOperator{\cpres}{res}
\DeclareMathOperator{\Gr}{Gr}
\DeclareMathOperator{\GrB}{\mathcal{G}}
\DeclareMathOperator{\im}{Im}
\DeclareMathOperator{\vspan}{Span}

\theoremstyle{plain}
\newtheorem{theorem}{Theorem}[section]
\newtheorem{corollary}[theorem]{Corollary}
\newtheorem{prop}[theorem]{Proposition}
\newtheorem{proposition}[theorem]{Proposition}
\newtheorem{lemma}[theorem]{Lemma}
\theoremstyle{definition}
\newtheorem{definition}[theorem]{Definition}
\newtheorem{remark}[theorem]{Remark}
\newtheorem{example}[theorem]{Example}

\numberwithin{equation}{section}

\begin{document}

\title[Forced discrete systems]{Error analysis of forced \\discrete
  mechanical systems}

\author{Javier Fern\'andez}

\address{Instituto Balseiro, Universidad Nacional de Cuyo --
  C.N.E.A.\\Av. Bustillo 9500, San Carlos de Bariloche, R8402AGP,
  Rep\'ublica Argentina}

\email{jfernand@ib.edu.ar}

\author{Sebasti\'an Graiff Zurita}

\address{Graduate School of Mathematics, Kyushu University\\
  744 Motooka, Nishi-ku, Fukuoka, 819- 0395, Japan}

\email{s-graiff@kyudai.jp}

\author{Sergio Grillo}

\address{Instituto Balseiro, Universidad Nacional de Cuyo --
  C.N.E.A. and CONICET\\ Av. Bustillo 9500, San Carlos de
    Bariloche, R8402AGP, Rep\'ublica Argentina}

\email{sergiog@cab.cnea.gov.ar}



\thanks{This research was partially supported by grants from the
  Universidad Nacional de Cuyo (grants 06/C567 and 06/C574) and
  CONICET}

\begin{abstract}
  The purpose of this paper is to perform an error analysis of the
  variational integrators of mechanical systems subject to external
  forcing. Essentially, we prove that when a discretization of contact
  order $r$ of the Lagrangian and force are used, the integrator has
  the same contact order. Our analysis is performed first for discrete
  forced mechanical systems defined over $TQ$, where we study the
  existence of flows, the construction and properties of discrete
  exact systems and the contact order of the flows (variational
  integrators) in terms of the contact order of the original
  systems. Then we use those results to derive the corresponding
  analysis for the analogous forced systems defined over $Q\times Q$.
\end{abstract}


\bibliographystyle{amsalpha}

\maketitle



\section{Introduction}
\label{sec:introduction}

The study of mechanical systems subject to external forcing is
interesting, among other reasons, because this type of dynamical
system appears very naturally in many real world situations, where
control forces, dissipation or friction are usually present. Even
though the general theory of these systems has been studied for
centuries, a practical problem still remains: the concrete solution of
the corresponding equations of motion can be too involved or
practically impossible to find analytically. Hence, a valuable tool is
the numerical solution of those equations.

There are many approaches to discretizing ordinary differential
equations and, subsequently, finding appropriate numerical
solutions. In the special case of equations that arise as equations of
motion of mechanical systems, a successful technique of numerical
integration is that of variational integration, where rather than
discretizing those equations, the continuous variational principle
that leads to the equations of motion ---the Lagrange--D'Alembert
principle--- is discretized. The solutions of this discrete
variational principle are the numerical approximations of the
solutions of the equation of motion of the system. This mechanism has
been explained and analyzed, for example, by J. Marsden and M. West
in~\cite{ar:marsden_west-discrete_mechanics_and_variational_integrators}.

In the case of Lagrangian systems with external forces the basic idea
consists of associating to the given (continuous) Lagrangian and force
discrete analogues and a discrete-time dynamics based on a variational
principle that mimics the Lagrange--D'Alembert principle. Then, the
question is how well the discrete trajectories that solve the discrete
variational problem match the solutions of the original continuous
problems, in terms of how well the discrete Lagrangian and discrete
forces match their continuous counterparts. An answer to this question
is proposed
in~\cite{ar:marsden_west-discrete_mechanics_and_variational_integrators}
(Section 3.2.1) where most of the analysis is referred to the unforced
case. Unfortunately, it was shown by G. Patrick and C. Cuell
in~\cite{ar:patrick_cuell-error_analysis_of_variational_integrators_of_unconstrained_lagrangian_systems}
that this last analysis was
flawed. Still,~\cite{ar:patrick_cuell-error_analysis_of_variational_integrators_of_unconstrained_lagrangian_systems}
proves for (unforced) Lagrangian systems that the contact order
between the discrete and the continuous trajectories is, at least,
that of the discrete and continuous Lagrangians. The purpose of this
paper is to apply the approach
of~\cite{ar:patrick_cuell-error_analysis_of_variational_integrators_of_unconstrained_lagrangian_systems}
to the error analysis of the variational integrators of forced
mechanical systems. To do so, we also try to explain and clarify the
content of some results that appear
in~\cite{ar:patrick_cuell-error_analysis_of_variational_integrators_of_unconstrained_lagrangian_systems},~\cite{ar:cuell_patrick-skew_critical_problems}
and~\cite{ar:cuell_patrick-geometric_discrete_analogues_of_tangent_bundles_and_constrained_lagrangian_systems}.

Discrete mechanical systems, forced or not, have usually been
described as dynamical systems over a ``discrete phase space''
$Q\times Q$ and, also, over $T^*Q$ (using a ``discrete Legendre
transform''). An alternative point of view is developed
in~\cite{ar:cuell_patrick-geometric_discrete_analogues_of_tangent_bundles_and_constrained_lagrangian_systems}
and~\cite{ar:patrick_cuell-error_analysis_of_variational_integrators_of_unconstrained_lagrangian_systems}
where such systems are defined over $TQ$, calling them \jdef{discrete
  mechanical systems}. In all cases, the dynamics is defined in terms
of a variational principle for the space of (discrete) paths. In
general, a process of \jdef{discretization} of a (continuous) system
assigns a family of discrete systems depending on a ``deformation
parameter'' $h$ to a given continuous one, for instance, by choosing
an $h$-dependent family of local diffeomorphism between $TQ$ (the
``velocity phase space'' of the continuous system) and $Q\times Q$.
Cuell \& Patrick introduce a refined notion of discretization
in~\cite[Definitions 3.1 and
3.3]{ar:patrick_cuell-error_analysis_of_variational_integrators_of_unconstrained_lagrangian_systems}
with the characteristic that, for each fixed $h$, the discretization
produces a discrete mechanical system, whose properties can be derived
from those of the original (continuous) system. For example, they are
able to show that all discretizations have local flows and,
consequently, that the discrete mechanical systems obtained from them
have trajectories. In addition, they prove that for regular
(continuous) Lagrangian systems, there are discretizations, called
\jdef{exact}, so that the trajectories of the discrete mechanical
systems obtained from them are the same as the continuous trajectories
of the (continuous) systems, evaluated at discrete time
($t=0,h,2h,\ldots$); unfortunately, such exact systems are not
practically constructible except for a few trivial cases, otherwise
the trajectories of the continuous system would be easily
available. Also, after introducing a notion of \jdef{contact order}
between maps, they are able to prove that if a family of discrete
mechanical systems has contact order $r$ with the exact system, then
the flow of the discrete system has contact order $r$ with the
original continuous system.

As we said above, the goal of our paper is to apply the approach of
Cuell \& Patrick to \emph{forced} mechanical systems, that we review
in Section~\ref{sec:forced_mechanical_systems}. In order to do so, we
introduce in Section~\ref{sec:discretizations_a_la_cuell_and_patrick}
the \jdef{forced discrete mechanical systems} and a notion of
discretization of (continuous) forced mechanical systems that,
for fixed $h$, produce forced discrete mechanical
systems. Using~\cite[Proposition
1]{ar:cuell_patrick-geometric_discrete_analogues_of_tangent_bundles_and_constrained_lagrangian_systems}
and the methods of~\cite[Theorem
3.7]{ar:patrick_cuell-error_analysis_of_variational_integrators_of_unconstrained_lagrangian_systems}
we prove in Theorem~\ref{th:flow_of_discretizations_of_FMS} that
discretizations of regular forced mechanical systems always have
smooth flows, even when $h=0$. Along the way we discuss a point
in~\cite[Proposition
1]{ar:cuell_patrick-geometric_discrete_analogues_of_tangent_bundles_and_constrained_lagrangian_systems}
that we find obscure and provide an alternative proof for part of that
result. Then, given a regular forced mechanical system we introduce
and analyze the properties of some discrete forced mechanical systems
that we call \jdef{exact}; as before, the name comes from the fact,
proved in Theorem~\ref{thm:flow_of_exact_discrete_systems}, that their
(discrete) trajectories coincide with the trajectories of the
continuous system from whom they are constructed, at discrete times.
The next important result that we discuss in
Section~\ref{sec:error_analysis} is
Theorem~\ref{th:contact_order_discrete_FMS}, which proves that if a
family of forced discrete mechanical systems that is obtained from a
discretization $\discCP$ of a regular forced mechanical system has
contact order $r$ with the exact system associated to the same
continuous system, then the flows of the discretization $\discCP$ and
the exact discretization also have contact order $r$; hence the
variational integrator derived from the discretization has contact
order $r$ with the exact solution of the equation of motion the
original continuous system.

There is a more well known type of discrete system used to model
forced mechanical systems as discrete dynamical systems over
$Q\times Q$; see for instance, the presentation in Part III
of~\cite{ar:marsden_west-discrete_mechanics_and_variational_integrators}. Still,
the analytic study of these systems given in that reference is rather
sketchy. In particular, the construction of exact discrete systems is
not completely justified and the error analysis is suggested as
similar to the one performed in the unforced case (which was found to
be flawed). In
Section~\ref{sec:forced_discrete_mechanical_systems_on_QxQ} we use our
results on forced discrete mechanical systems to derive the existence
of forced exact systems on $Q\times Q$ for a given regular forced
mechanical system (see
Example~\ref{ex:exact_discretization_MW-simple_bd} and
Theorem~\ref{th:flow_of_exact_discrete_systems-FDMS}). We are also
able to transfer the error analysis developed for systems defined on
$TQ$ to the systems on $Q\times Q$
(Theorem~\ref{th:order_of_flow_MW_vs_order_discMW}). It is worth
mentioning that another approach to this problem has been considered
by D. Mart\'in de Diego and R. Sato Mart\'in de Almagro
in~\cite{ar:deDiego_deAlmagro-variational_order_for_forced_lagrangian_systems},
where they associate a non-forced mechanical system to the forced one
and, then, apply the results of Cuell \& Patrick to the former.

A key tool used by Cuell \& Patrick and us to obtain results for
systems over $TQ$ is the study of the asymptotic behavior of the
system when $h\rightarrow 0$ which, after a kind of blow-up, is very
good (yielding, for instance, smooth flows even at $h=0$). The same
asymptotic techniques don't seem to apply in the case of systems over
$Q\times Q$, which do not seem to extend to $h=0$. This difference is
what justifies the ``extra work'' of having to introduce dynamical
systems over $TQ$ in order to study systems over $Q\times Q$. It seems
to us that the discrete systems over $TQ$ allow for easier study of
matters related to the evolution of the systems, separating them from
the ``inversion problem'' of assigning a velocity to a pair of points,
as is evident in the treatment of the discrete exact systems.

The paper includes a few Appendices that review some basic aspects of
Grassmannian manifolds and Grassmann bundles as well as some (lengthy)
local computations used in the proof of
Theorem~\ref{th:contact_order_discrete_FMS}.

\emph{Notation:} Throughout the paper many spaces are cartesian
products. In general we denote the projections by
$p_k:\prod_{j=1}^N X_j\rightarrow X_k$ with the obvious
adaptations. For any space $X$, its \jdef{diagonal} is
$\Delta_X:=\{(x,x')\in X\times X:x=x'\} \subset X\times X$.


\section{Preliminaries}
\label{sec:preliminaries}

In this section we review three topics: first, in
Section~\ref{sec:forced_mechanical_systems}, the notion of forced
mechanical system using the variational formalism is reviewed. Then,
in Section~\ref{sec:contact_order}, we touch on the notion of contact
order between maps, that will be very important for the rest of the
paper; here we follow Section 3
of~\cite{ar:cuell_patrick-skew_critical_problems}, albeit with a
slight difference, noted in
Remark~\ref{rem:no_same_contact_order_defs}. Last, in
Section~\ref{sec:critical_problems}, we survey the notion of critical
problem that will allow us to characterize, later, trajectories of
some dynamical systems as critical points; here we mostly follow
Sections 1 and 2 of~\cite{ar:cuell_patrick-skew_critical_problems}.


\subsection{Forced mechanical systems}
\label{sec:forced_mechanical_systems}

Let $\tau_Q:TQ\rightarrow Q$ be the tangent bundle of the (finite
dimensional) differentiable manifold $Q$; the vectors in
$\ker(T \tau_Q)\subset TTQ$ are said to be \jdef{vertical}. A $1$-form
on $TQ$ is said to be \jdef{horizontal} if it vanishes on vertical
vectors.

In what follows we consider \jdef{paths} in a manifold $Q$, that is,
smooth maps $q:[a,b]\rightarrow Q$, and also \jdef{infinitesimal
  variations} over $q$: smooth maps $\delta q:[a,b]\rightarrow TQ$
such that $\delta q(t)\in T_{q(t)}Q$. If $q$ is a path and $\delta q$
is an infinitesimal variation over $q$, we consider their velocities
$q':[a,b]\rightarrow TQ$ and $(\delta q)':[a,b]\rightarrow TTQ$, that
are infinitesimal variations over $q$ and $\delta q$,
respectively. Last, we recall the \jdef{canonical involution}
$\kappa:TTQ\rightarrow TTQ$ defined in local coordinates as follows:
let $(U,\phi)$ be a local coordinate chart of $Q$, then $(TTU,TT\phi)$
is a local coordinate chart of $TTQ$ and, if
$\check{\kappa}:=TT\phi\circ \kappa\circ (TT\phi)^{-1}$ is the local
expression of $\kappa$ with respect to that chart, then
$\check{\kappa}(q,\dot{q},\delta q, \delta \dot{q}) = (q,\delta
q,\dot{q},\delta \dot{q})$. It is well known that $\kappa$ is a well
defined smooth involution and that
$T(\tau_Q) \circ \kappa = \tau_{TQ}$ (see
\cite{ar:cendra_ferraro_grillo-lagrangian_reduction_of_generalized_NHS}
or~\cite{ar:tulczyjew_urbanski-a_slow_and_careful_legendre_transformation_for_singular_lagrangians}).

\begin{definition}
  A \jdef{forced mechanical system} (\FMS) is a triple $(Q,L,f)$ where
  $Q$ is a finite dimensional differentiable manifold, the
  \jdef{configuration space}, $L:TQ\rightarrow \R$ is a smooth
  function, the \jdef{Lagrangian}, and $f$ is a horizontal $1$-form on
  $TQ$, the \jdef{force}. A forced mechanical system determines a
  dynamical system whose trajectories are the curves
  $q:[a,b]\rightarrow Q$ such that for all fixed endpoint variations
  $\delta q(t)$ over $q(t)$,
  \begin{equation}\label{eq:FMS-variational_pple}
    \int_a^b (dL+f)(q'(t))(\kappa((\delta q)'(t))) dt =0.
  \end{equation}
\end{definition}

\begin{remark}\label{rem:force_field}
  It is customary to give the force of a mechanical system as a fiber
  preserving smooth map $\FVF{f}:TQ\rightarrow T^*Q$ that we call
  the \jdef{force field}. In this case, we define the $1$-form
  $f\in \mathcal{A}^1(TQ)$ by
  \begin{equation*}
    f(v)(\delta v) := \FVF{f}(v)(T_{v}\tau_Q(\delta v)) \stext{ for }
    \delta v \in T_vTQ.
  \end{equation*}
  By construction, $f$ is a horizontal $1$-form and, in fact, there is
  a correspondence between horizontal $1$-forms on $TQ$ and
  fiber-preserving smooth maps from $TQ\rightarrow T^*Q$ (Prop.
  7.8.2
  in~\cite{bo:marsden_ratiu-introduction_to_mechanics_and_symmetry_2e})
  so, in what follows, we will describe the force alternatively in
  either way. In terms of the force field $\FVF{f}$, the trajectories
  of the system are the curves $q:[a,b]\rightarrow Q$ such that for
  all fixed endpoint variations $\delta q(t)$ over $q(t)$,
  \begin{equation}\label{eq:FMS-variational_pple-force_as_maps}
    \int_a^b dL(q'(t))(\kappa((\delta q)'(t))) dt + 
    \int_a^b \FVF{f}(q'(t))(\delta q(t)) dt =0
  \end{equation}
\end{remark}

\begin{remark}\label{rem:def_of_fiber_derivative}
  We recall the notion of \jdef{fiber derivative}. Let
  $\phi_E:E\rightarrow M$ and $\phi_F:F\rightarrow M$ be vector
  bundles and $f:E\rightarrow F$ be a fiber-preserving smooth map over
  $id_M$. As for each $m\in M$ the fiber $E_m$ is a vector space, for
  each $v\in E_m$ the mapping
  $w\mapsto w^v:=\frac{d}{dt}\big|_{t=0}(v+t w)$ is an isomorphism
  from $E_m$ onto $T_vE_m$, with a parallel construction for $F$.  For
  each $m\in M$, let $f_m:=f|_{E_m}:E_m\rightarrow F_m$. The
  \jdef{fiber derivative} of $f$ is defined as
  $\F f:E\rightarrow \hom(E,F)$ by
  \begin{equation*}
    \F f(e)(e'):=T_e(f|_{\phi_E(e)})((e')^e) = \frac{d}{dt}\bigg|_{t=0} f(e+t e')
  \end{equation*}
  (where we identify $T_{f(e)}F_{\phi_E(e)} \simeq F_{\phi_E(e)}$). It
  can be seen that $\F f$ is, indeed, a smooth morphism of vector
  bundles over $M$. For example, when $L:TQ\rightarrow\R$, we have
  $\F L:TQ\rightarrow \hom(TQ,\R) \simeq T^*Q$ that is computed by
  \begin{equation*}
    \F L(v_q)(w_q) := \frac{d}{dt}\bigg|_{t=0}L(v_q+t w_q)
  \end{equation*}
  and is known as the \jdef{Legendre transform} of $L$.  We can also
  consider
  $\F^2 L:=\F(\F L):TQ\rightarrow \hom(TQ,T^*Q) \simeq T^*Q\otimes
  T^*Q \simeq Bil(TQ,\R)$ (where $Bil(TQ,\R)$ are the $\R$-valued
  bilinear forms on $TQ$).
\end{remark}

\begin{lemma}\label{le:variation_of_lagrangian}
  Let $L:TQ\rightarrow\R$ be a smooth function. Then, for any curve
  $q:[a,b]\rightarrow Q$ and smooth variation
  $\delta q:[a,b]\rightarrow TQ$ over $q$ we have that
  \begin{equation}\label{eq:variation_of_lagrangian}
    \int_a^bdL(q'(t))(\kappa((\delta q)'(t))) dt =
    \int_a^b D_{EL}L(q''(t))(\delta q(t)) dt +
    \F L(q'(t))(\delta q(t)) \big|_a^b,
  \end{equation}
  where $D_{EL}:\ddot{Q}\rightarrow T^*Q$ is the \jdef{Euler--Lagrange
    operator} (see p. 26
  of~\cite{bo:cendra_marsden_ratiu-lagrangian_reduction_by_stages}).
\end{lemma}

If $(Q,L,f)$ is a \FMS, from~\eqref{eq:variation_of_lagrangian} we see
that, with the same notation used in
Lemma~\ref{le:variation_of_lagrangian},
\begin{equation*}
  \begin{split}
    \int_a^bdL(q'(t))&(\kappa((\delta q)'(t)))dt + \int_a^b
    \FVF{f}(q'(t))(\delta q(t)) dt\\=&\int_a^b D_{EL}L(q''(t))(\delta
    q(t)) dt + \int_a^b \FVF{f}(q'(t))(\delta q(t)) dt+ \F
    L(q'(t))(\delta q(t)) \big|_a^b \\=& \int_a^b
    \big(D_{EL}L(q''(t))+\FVF{f}\big)(q'(t))(\delta q(t)) dt + \F
    L(q'(t))(\delta q(t)) \big|_a^b.
  \end{split}
\end{equation*}
From this last expression we deduce the following result.

\begin{prop}\label{prop:equation_of_motion-FMS}
  Let $(Q,L,f)$ be a \FMS. Then, the path $q:[a,b]\rightarrow Q$ is a
  trajectory of the system if and only if
  \begin{equation}
    \label{eq:equation_of_motion-FMS}
    D_{EL}L(q''(t))+\FVF{f}(q'(t)) =0 \stext{ for all } t\in [a,b].
  \end{equation}
  Also, when $q:[a,b]\rightarrow Q$ is a trajectory of the system and
  $\delta q:[a,b]\rightarrow TQ$ is a smooth variation over $q$,
  \begin{equation}
    \label{eq:variation_of_lagrangian-FMS}
    \int_a^bdL(q'(t))(\kappa((\delta q)'(t))) dt + \int_a^b
    \FVF{f}(q'(t))(\delta q(t)) dt = \F
    L(q'(t))(\delta q(t)) \big|_a^b.
  \end{equation}
\end{prop}

\begin{definition}\label{def:regularity-GFMS}
  A \FMS $(Q,L,f)$ is said to be \jdef{regular} if the Legendre
  transform $\F L:TQ\rightarrow T^*Q$ is a local diffeomorphism or,
  equivalently, if $\F^2 L(v)$ is a nondegenerate bilinear form at
  each $v\in TQ$ (see Remark~\ref{rem:def_of_fiber_derivative} and
  Proposition 3.5.10 in~\cite {bo:AM-mechanics}).
\end{definition}

\begin{remark}\label{rem:legendre_transform_from_1-form}
  In general, given $\gF\in \mathcal{A}^1(TQ)$, we can define
  $\ti{\gF}:TQ\rightarrow T^*Q$ by
  $\ti{\gF}(v_q)(w_q) := \gF(v_q)(w_q^{v_q})$ for all $w_q\in T_qQ$,
  where
  $w_q^{v_q} := \frac{d}{dt}\big|_{t=0} (v_q+t w_q)\in
  T_{v_q}T_qQ\subset T_{v_q}TQ$ is a vertical vector.  From the
  definition $\ti{\gF}$ is smooth and preserves the fibers under the
  canonical projections on $Q$. In particular, for a \FMS $(Q,L,f)$,
  the form $\cAFop:=dL+f \in \mathcal{A}^1(TQ)$ produces the map
  $\ti{\cAFop}$ such that
  $\ti{\cAFop}(v_q)(w_q) = \cAFop(v_q)(w_q^{v_q}) =
  (dL+f)(v_q)(w_q^{v_q}) = dL(v_q)(w_q^{v_q}) = \F L(v_q)(w_q)$, the
  Legendre transform of $L$. Thus, $(Q,L,f)$ is regular if and only if
  $\ti{\cAFop}$ is a local diffeomorphism.
\end{remark}

\begin{example}\label{ex:particle_with_friction-def}
  A simple example of a \FMS is that of a (unit mass) particle, whose
  movement is subject to a friction force, proportional to its
  velocity. In this case, $Q:=\R$, $L(q,v):=\frac{1}{2} v^2$ and
  $f(q,v):=-\alpha v\, dq$, for a constant $\alpha>0$. The Legendre
  transform is given by $\F L(q,v) = v\, dq$, that is a diffeomorphism
  between $TQ$ and $T^*Q$, hence the system is regular.
\end{example}

Given a \FMS $(Q,L,f)$, we can define its \jdef{energy function}
$E_{L}:TQ\rightarrow\R$ as
\begin{equation*}
  E_{L}(v):=\F L(v)(v)-L(v)
\end{equation*}
and the closed $2$-form
$\omega_{L}:= \F L^*\omega\in\mathcal{A}^{2} (TQ)$, where $\omega$ is
the canonical symplectic form of the cotangent bundle $T^*Q$. For
regular systems, following similar steps as for the unforced case (see
\cite{bo:AM-mechanics}, page 215), the next result easily follows.

\begin{theorem}
  If $(Q,L,f)$ is a regular \FMS, then $\omega_{L}$ is a symplectic
  form on $TQ$ and the integral curves of the vector field
  $X_{L,f}\in\VF(TQ)$, defined by the equation
  \begin{equation*}
    i_{X_{L,f}}\omega_{L}=dE_{L}+f,
  \end{equation*}
  are exactly the trajectories of the system.
\end{theorem}

Then, according to the general theory of ordinary differential
equations (see for instance~\cite{bo:Boothby-differential_geometry},
page 127), the trajectories of $(Q,L,f)$ are given by a function
$F^{X_{L,f}}:U^{X_{L,f}}\rightarrow TQ$, the \jdef{flow} of the vector
field $X_{L,f}$, with $U^{X_{L,f}}$ an open subset of $\R\times TQ$
containing $\{0\}\times TQ$. More precisely, for every $v\in TQ$ there
exists only one trajectory $\gamma:(-t,t)\rightarrow TQ$ of $(Q,L,f)$
such that $\gamma(0)=v$, and it is given by the formula
$\gamma(t) = F^{X_{L,f}}(t,v)$.

\begin{example}\label{ex:particle_with_friction-flow}
  It is easy to check that the trajectories of the system introduced
  in Example~\ref{ex:particle_with_friction-def} are of the form
  \begin{equation*}
    q(t) = q_0 + \frac{1}{\alpha} v_0 (1-e^{-\alpha t}),
  \end{equation*}
  where $q_0=q(0)$ and $v_0 = q'(0)$. Thus, the corresponding
  (globally defined) flow is
  \begin{equation*}
    F^{X_{L,f}}(t,q,v) = \big(q + \frac{1}{\alpha} v (1-e^{-\alpha t}\big),
    v e^{-\alpha t}).
  \end{equation*}
\end{example}


\subsection{Contact order}
\label{sec:contact_order}

In this section we mostly review the notion of contact order for maps
between manifolds, as introduced by Cuell \& Patrick in Section 2
of~\cite{ar:cuell_patrick-skew_critical_problems}. There is, though, a
small difference between our definitions and theirs, that we explain
in Remark~\ref{rem:no_same_contact_order_defs}.

\begin{definition}
  Let $M$ be a differentiable manifold and $\grade_M:M\rightarrow \R$
  be a smooth function for which $0\in\R$ is a regular value. Then,
  the pair $(M,\grade_M)$ is called a \jdef{manifold with a
    grade}\footnote{In fact, in the original definition (Definition
    4.1 in~\cite{ar:cuell_patrick-skew_critical_problems}) such a pair
    $(M,\grade_M)$ is called just a \jdef{manifold}, but we don't find this
    name too appropriate.}. A \jdef{smooth map between manifolds with
    a grade} $f:(M,\grade_M)\rightarrow (N,\grade_N)$ is a smooth map
  $f:M\rightarrow N$ such that $\grade_N\circ f = \grade_M$.
\end{definition}

\begin{definition}\label{def:contact_order}
  Let $(M,\grade_M)$ be a manifold with a grade, $N$ be a manifold and
  $f_1,f_2:M\rightarrow N$ be smooth functions such that
  $f_1(m)=f_2(m)$ for all $m\in \grade_M^{-1}(0)$. We say that
  $f_2 = f_1 +\mathcal{O}(\grade_M^r)$ for $r\in\N$ if, for all
  $m_0\in \grade_M^{-1}(0)$, there is an open subset $U\subset M$
  containing $m_0$, a chart $(W,\psi)$ of $N$ such that
  $f_j(U)\subset W$ for $j=1,2$ and a continuous function
  $(\delta f)_\psi:U\rightarrow \R^{\dim(N)}$, such that
  \begin{equation}\label{eq:contact_order_condition-def}
    \psi(f_2(m))-\psi(f_1(m)) = \grade_M(m)^{r} (\delta f)_\psi(m)
  \end{equation}
  for all $m\in U$. If $f_2 = f_1 +\mathcal{O}(\grade_M^r)$, it is also
  said that $f_1$ and $f_2$ have \jdef{order $\grade_M^{r-1}$ contact}. In
  what follows and in order to simplify the notation we may write
  $f_2 = f_1 +\mathcal{O}(r)$ instead of
  $f_2 = f_1 +\mathcal{O}(\grade_M^r)$, when the function $\grade_M$ can be
  deduced from the context.
\end{definition}

\begin{remark}\label{rem:no_same_contact_order_defs}
  Definition 5 in~\cite{ar:cuell_patrick-skew_critical_problems} only
  requires that $(\delta f)_\psi$ be continuous at $m_0$ instead of in
  the open neighborhood $U$, but we prefer the stronger condition.
\end{remark}

\begin{lemma}\label{le:order_of_maps_in_terms_of_local_expression}
  Let $(M,\grade_M)$ be a manifold with a grade and
  $f_1,f_2:M\rightarrow N$ be smooth maps. Then,
  $f_2=f_1+\mathcal{O}(\grade_M^r)$ if and only if for each
  $m_0\in \grade_M^{-1}(0)$ there are coordinate charts $(V,\phi)$ and
  $(W,\psi)$ of $M$ and $N$ such that
  \begin{enumerate}
  \item \label{it:order_of_maps_in_terms_of_local_expression-contained}
    $m_0\in V$ and $f_a(V)\subset W$ for $a=1,2$,
  \item \label{it:order_of_maps_in_terms_of_local_expression-difference}
    and there is an open subset $U'\subset \phi(V)$ containing
    $\phi(m_0)$ and a continuous function
    $\delta\check{f}:U'\rightarrow \R^{\dim(N)}$ such that
    $\check{f}_2(u')-\check{f}_1(u') = \check{\grade}(u')^r
    \delta\check{f}(u')$ for all $u'\in U'$, where
    $\check{\grade} := \grade_M \circ \phi^{-1}$ and
    $\check{f}_a:=\psi\circ f_a \circ \phi^{-1}$ ($a=1,2$).
  \end{enumerate}
\end{lemma}

\begin{proof}
  It follows readily by unraveling the definitions.
\end{proof}

\begin{remark}\label{rem:local_description_of_delta_f}
  Notice that the ``local description'' of the contact order condition
  provided by
  Lemma~\ref{le:order_of_maps_in_terms_of_local_expression} is valid
  not just at any $m_0\in \grade_M^{-1}(0)$, but also for all
  $m\in \grade_M^{-1}(0)$ in an open neighborhood of $m_0$.
\end{remark}

\begin{lemma}\label{le:residual_well_defined}
  Let $(M,\grade_M)$ be a manifold with a grade, $N$ be a manifold and
  $f_1,f_2:M\rightarrow N$ be smooth maps such that
  $f_2 = f_1 +\mathcal{O}(\grade_M^r)$ for some $r\in\N$. Then,
  \begin{enumerate}
  \item the property that $f_2 = f_1 +\mathcal{O}(\grade_M^r)$ for some
    $r\in\N$ as in Definition~\ref{def:contact_order}, is independent
    of the coordinate chart and open subsets used.
  \item Furthermore, for $m_0\in \grade_M^{-1}(0)$,
    $(\delta f)_\psi(m_0)$ transforms as (the coordinates of) a
    tangent vector of $N$ under coordinate change.
  \end{enumerate}
\end{lemma}

\begin{proof}
  See the arguments between Definitions 5 and 6
  in~\cite{ar:cuell_patrick-skew_critical_problems}.
\end{proof}

\begin{definition}
  In the context of Lemma~\ref{le:residual_well_defined}, the tangent
  vector in $T_{f(m_0)}N$ whose coordinates with respect to the
  coordinate chart $(W,\psi)$ are $(\delta f)_\psi(m_0)$ will be
  denoted by $\cpres^r(f_2,f_1)(m_0)\in T_{f(m_0)}N$ and called the
  \jdef{$r$-residual} of $f_2$ with respect to $f_1$. In other words,
  \begin{equation}\label{eq:residual_vs_delta_f}
    \cpres^r(f_2,f_1)(m_0) = T_{\psi(f_2(m_0))}\psi^{-1}((\delta f)_\psi(m_0)).
  \end{equation}
\end{definition}

\begin{remark}\label{rem:residuals_from_local_description}
  If $f_2 = f_1 +\mathcal{O}(\grade_M^r)$ for some $r\in\N$, using the
  local description by $\check{f}_1$ and $\check{f}_2$ satisfying
  $\check{f}_2(u') - \check{f}_1(u') = \check{\grade}(u')^r \delta
  \check{f}(u')$ as in
  Lemma~\ref{le:order_of_maps_in_terms_of_local_expression}, and taking
  into account Remark~\ref{rem:local_description_of_delta_f}, we see
  that for $m\in \grade_M^{-1}(0)$:
  \begin{equation*}
    \cpres^r(f_2,f_1)(m) = T_{\psi(f_2(m))}\psi^{-1}((\delta f)_\psi(m))
    = T_{\psi(f_2(m))}\psi^{-1}(\delta \check{f}(\phi(m))).
  \end{equation*}
\end{remark}

\begin{lemma}\label{le:order_of_sections_of_VB}
  Let $(M,\grade_M)$ be a manifold with a grade,
  $\phi:\mathcal{V}\rightarrow M$ be a vector bundle and
  $f_1,f_2\in
  \Gamma(M,\mathcal{V})$\footnote{$\Gamma(M,{\mathcal V})$ is the
    $\R$-vector space of $C^\infty$ global sections of
    ${\mathcal V}$.}. Then, the following assertions are equivalent.
  \begin{enumerate}
  \item \label{it:order_of_sections_of_VB-def}
    $f_2 = f_1 + \mathcal{O}(\grade_M^r)$.
  \item \label{it:order_of_sections_of_VB-local} For each
    $m_0\in \grade_M^{-1}(0)$ there is an open neighborhood $U$ of
    $m_0$ in $M$ and a continuous section
    $\delta f_U:U\rightarrow \mathcal{V}|_U$ such that
    \begin{equation}\label{eq:order_of_sections_of_VB-local}
      f_2(m)-f_1(m) = \grade_M(m)^r \delta f_U(m) \stext{ for all } m\in U.
    \end{equation}
  \item \label{it:order_of_sections_of_VB-global} There is a
    continuous section $\delta f:M\rightarrow \mathcal{V}$ such that
    \begin{equation}\label{eq:order_of_sections_of_VB-global}
      f_2(m)-f_1(m) = \grade_M(m)^r \delta f(m) \stext{ for all } m\in M.
    \end{equation}
  \end{enumerate}
  If condition~\eqref{eq:order_of_sections_of_VB-local} is satisfied,
  we have
  \begin{equation*}
    \cpres^r(f_2,f_1)(m) =  \frac{d}{dt}\big|_{t=0}(f_2(m)+t \delta f_U(m))
    \stext{ for all } m\in \grade_M^{-1}(0)\cap U,
  \end{equation*}
  while, if condition~\eqref{eq:order_of_sections_of_VB-global} is
  satisfied, we have
  \begin{equation*}
    \cpres^r(f_2,f_1)(m) = \frac{d}{dt}\big|_{t=0}(f_2(m)+t \delta f(m))
    \stext{ for all } m\in \grade_M^{-1}(0).
  \end{equation*}
\end{lemma}

\begin{proof}
  The equivalence of points~\ref{it:order_of_sections_of_VB-def}
  and~\ref{it:order_of_sections_of_VB-local} is checked using
  coordinates for $\mathcal{V}$ derived from the local trivialization
  $\Phi_U$ of $\mathcal{V}$ and the fact that the order is independent
  of the coordinate charts
  used. That~\ref{it:order_of_sections_of_VB-global} \imp
  \ref{it:order_of_sections_of_VB-local} is trivial,
  while~\ref{it:order_of_sections_of_VB-local} \imp
  \ref{it:order_of_sections_of_VB-global} can be obtained using
  partitions of unity. Both formulas for the residual are obtained
  from~\eqref{eq:residual_vs_delta_f} using coordinate charts derived
  from $\Phi_U$.
\end{proof}

Before stating our next result, we recall a parameterized version of
Taylor's Theorem.

\begin{theorem}\label{th:taylor-r_minus_1}
  Let $U\subset\R^n$ be an open set and $a>0$; define
  $\ti{U}:=U\times (-a,a)$. If $f\in C^k(\ti{U},\R^m)$ with $k\geq 1$,
  for each $x_0\in U$ and $r\in\NZ$ such that $r\leq k$ there are an
  open subset $U_{x_0}\subset U$ containing $x_0$ and a real number
  $a_{x_0} \in (0,a)$, such that
  \begin{equation}\label{eq:taylor-r_minus_1}
    f(x,h) = \sum_{j=0}^{r-1} \frac{h^j}{j!} (D_2^j f)(x,0) + 
    \frac{h^r}{r!} R_r(x,h) 
    \stext{ for all } (x,h)\in U_{x_0}\times (-a_{x_0}, a_{x_0}),
  \end{equation}
  where 
  \begin{equation}\label{eq:taylor-r_minus_1-error_term}
    R_r(x,h) =
    \begin{cases}
      \frac{r}{h^r} \int_0^h (h-t)^{r-1} D_2^r f(x,t) dt \stext{ if } h\neq 0,\\
      D_2^r f(x,0) \stext{ otherwise.}
    \end{cases}
  \end{equation}
  Furthermore, $R_r\in C^{k-r}(U_{x_0}\times (-a_{x_0}, a_{x_0}),\R^m)$.
\end{theorem}

\begin{proof}
  It suffices to consider the case $m=1$ since the case of vector
  valued functions can be treated in each component independently. The
  proof in this case is given
  in~\cite{ar:whitney-differentiability_of_the_remainder_term_in_taylors_formula}
  (in the case $n=1$ and can be extended to the general case as
  mentioned there).
\end{proof}

\begin{lemma}\label{le:0_residual_implies_order+1}
  Let $f_1,f_2\in C^k(M,N)$ with $k\geq 1$ be such that
  $f_2 = f_1 +\mathcal{O}(\grade_M^r)$ for some $r\in\NZ$ with
  $r\leq k$. Then,
  \begin{enumerate}
  \item \label{it:0_residual_implies_order+1-smooth} the continuous
    function $(\delta f)_\psi$ that appears
    in~\eqref{eq:contact_order_condition-def} is in
    $C^{k-r}(U,\R^{\dim(N)})$, and
  \item \label{it:0_residual_implies_order+1-residual} if $r< k$ and
    $\cpres^r(f_1,f_2)(m)=0$ for all $m\in \grade_M^{-1}(0)$, then,
    $f_2 = f_1 +\mathcal{O}(\grade_M^{r+1})$.
  \end{enumerate}
\end{lemma}

\begin{proof}
  Fix $m_0\in \grade_M^{-1}(0)$ and let $U\subset M$ and
  $(\delta f)_\psi$ be as in Definition~\ref{def:contact_order}. If
  $m_1\in U$ with $\grade_M(m_1)\neq 0$, as $f_1$ and $f_2$ are in
  $C^k(M)$ from~\eqref{eq:contact_order_condition-def} we see that
  $(\delta f)_\psi$ is in $C^k$ in some open neighborhood of $m_1$. If
  $\grade_M(m_1) = 0$, by
  Lemma~\ref{le:order_of_maps_in_terms_of_local_expression}, there are
  coordinate charts $(V,\phi)$ and $(W,\psi)$ of $M$ and $N$
  satisfying all conditions stated in that result and $m_1\in
  V$. Furthermore, as $0$ is a regular value of $\grade_M$, we see
  that there is an open neighborhood of $m_1$ where $\grade_M$ is a
  submersion. Hence, it is easy to see that, by an additional change
  of coordinates in $M$ one can assume that $\check{\grade}(x,h) = h$
  for all $(x,h)\in \phi(V)\subset \R^{\dim(M)-1}\times\R$. The local
  description of $f_2 = f_1 +\mathcal{O}(\grade_M^r)$ given by
  Lemma~\ref{le:order_of_maps_in_terms_of_local_expression} says that
  $\check{f}_2(x,h)-\check{f}_1(x,h) = h^r \delta\check{f}(x,h)$ for
  all $(x,h)\in U'$ where
  $\delta\check{f} = (\delta f)_\psi\circ \phi^{-1}$. Applying
  Theorem~\ref{th:taylor-r_minus_1} to each $\check{f}_a$ we have
  \begin{equation*}
    \check{f}_a(x,h) =
    \sum_{j=0}^{r-1} \frac{h^j}{j!} (D_2^j \check{f}_a)(x,0) + 
    \frac{h^r}{r!} R_{r,a}(x,h) 
  \end{equation*}
  for all $(x,h)$ in an open subset $V'\subset \phi(V)$ containing
  $\phi(m_1)$, and where $R_{r,a}\in C^{k-r}(V',\R^{\dim(N)})$. Hence
  \begin{equation*}
    h^r \delta\check{f}(x,h) =
    \sum_{j=0}^{r-1} \frac{h^j}{j!} (D_2^j \check{f}_2(x,0) -
    D_2^j \check{f}_1(x,0)) + 
    \frac{h^r}{r!} (R_{r,2}(x,h)-R_{r,1}(x,h))
  \end{equation*}
  for all $(x,h)\in U'\cap V'$. Then, as both $\delta\check{f}$ and
  $R_{r,2}-R_{r,1}$ are continuous functions over $U'\cap V'$ it is
  easy to conclude that, for all $(x,0)\in U'\cap V'$,
  \begin{equation}\label{eq:local_delta_check_in_terms_of_remainder_terms-der}
    D_2^j \check{f}_2(x,0) = D_2^j \check{f}_1(x,0)
    \stext{ for } 0\leq j\leq r-1
  \end{equation}
  and, consequently,
  \begin{equation}\label{eq:local_delta_check_in_terms_of_remainder_terms-rem}
    \delta \check{f}(x,h) = \frac{1}{r!}
    (R_{r,2}(x,h)-R_{r,1}(x,h)).
  \end{equation}
  As $R_{r,a}\in C^{k-r}(V',\R^{\dim(N)})$, it follows
  from~\eqref{eq:local_delta_check_in_terms_of_remainder_terms-rem}
  that $\delta \check{f} \in C^{k-r}(U'\cap V',\R^{\dim(N)})$. Also,
  as $\delta\check{f}\circ \phi = (\delta f)_\psi$, we conclude that
  $(\delta f)_\psi$ is $C^{k-r}$ in an open set
  $\ti{U}:=\phi^{-1}(V'\cap U')$ containing $m_1$. As $m_1\in U$ is
  arbitrary, we conclude that
  $(\delta f)_\psi\in C^{k-r}(U,\R^{\dim(N)})$, thus proving
  part~\ref{it:0_residual_implies_order+1-smooth} of the Lemma.

  Assume now that $r<k$ and that $\cpres^r(f_1,f_2)(m)=0$ for all
  $m\in \grade_M^{-1}(0)$. We specialize the previous local
  construction to $m_1=m_0$. Taking
  Remark~\ref{rem:residuals_from_local_description} into account, we
  have that
  \begin{equation*}
    0 = \cpres^r(f_2,f_1)(\phi^{-1}(x,0)) =
    T_{\psi(f_2(\phi^{-1}(x,0)))}\psi^{-1}(\delta \check{f}(x,0))
  \end{equation*}
  for all $(x,0)\in U'$, so that $\delta \check{f}(x,0)=0$ in that
  region. We conclude
  from~\eqref{eq:local_delta_check_in_terms_of_remainder_terms-rem}
  that $R_{r,2}(x,0)-R_{r,1}(x,0)=0$ for all $(x,0)\in U'\cap V'$;
  equivalently, by~\eqref{eq:taylor-r_minus_1-error_term}, we have
  that $D_2^r\check{f}_2(x,0) = D_2^r\check{f}_1(x,0)$, for all
  $(x,0)\in U'\cap V'$. Then, a new application of
  Theorem~\ref{th:taylor-r_minus_1}
  and~\eqref{eq:local_delta_check_in_terms_of_remainder_terms-der}
  lead to
  \begin{equation*}
    \begin{split}
      \check{f}_2(x,h) - \check{f}_1(x,h) =& \sum_{j=0}^{r}
      \frac{h^j}{j!} (\underbrace{D_2^j \check{f}_2(x,0)-D_2^j
        \check{f}_2(x,0)}_{=0}) \\&+ \frac{h^{r+1}}{(r+1)!}
      (R_{(r+1),2}(x,h)-R_{(r+1),1}(x,h))
    \end{split}
  \end{equation*}
  for all $(x,h)\in U'\cap V'$. As
  $R_{(r+1),2}(x,h)-R_{(r+1),1}(x,h) \in C^{k-(r+1)}(U'\cap
  V',\R^{\dim(N)})$ with $k-(r+1)\geq 0$, the last expression together
  with Lemma~\ref{le:order_of_maps_in_terms_of_local_expression}
  proves that $f_2 = f_1 +\mathcal{O}(\grade_M^{r+1})$, proving
  part~\ref{it:0_residual_implies_order+1-residual} of the Lemma.
\end{proof}

\begin{lemma}\label{le:residuals_of_maps_of_manifolds_with_a_grade}
  For $j=1,2$, let $f_j:(M,\grade_M)\rightarrow (N,\grade_N)$ be smooth maps
  between manifolds with a grade such that
  $f_2 = f_1 + \mathcal{O}(\grade_M^r)$. Then
  $\cpres^r(f_2,f_1)(m)\in T_{f_2(m)} \grade_N^{-1}(0)$ for all
  $m\in \grade_M^{-1}(0)$.
\end{lemma}

\begin{proof}
  By hypothesis, $\grade_N\circ f_j=\grade_M$, so that
  $f_j(\grade_M^{-1}(0))\subset \grade_N^{-1}(0)$. For any
  $m\in \grade_M^{-1}(0)$ and $n:=f_1(m)$ we apply Proposition 3
  in~\cite{ar:cuell_patrick-skew_critical_problems} with $g_j:=\grade_N$
  (for $j=1,2$) and obtain
  \begin{equation*}
    \cpres^r(\underbrace{\grade_N\circ f_2}_{=\grade_M},
    \underbrace{\grade_N\circ f_1}_{=\grade_M})(m) =
    \dot{f}_2(m)^r \cpres^r(\grade_N,\grade_N)(n)
    + T_n\grade_N(\cpres^r(f_2,f_1)(m)).
  \end{equation*}
  Observe that, in general, if $g_1=g_2:M\rightarrow N$, then
  $g_2=g_1+\mathcal{O}(\grade_M^s)$ for all $s\in\NZ$ and, so,
  $\cpres^s(g_2,g_1)(m)=0$ for all $m\in \grade_M^{-1}(0)$ and
  $s\in\NZ$. Then, the previous expression reduces to
  \begin{equation*}
    0 = T_n\grade_N(\cpres^r(f_2,f_1)(m)),
  \end{equation*}
  so that
  $\cpres^r(f_2,f_1)(m) \in \ker(T_n\grade_N) = T_n\grade_N^{-1}(0)$.
\end{proof}

Next, we quote a useful result
from~\cite{ar:cuell_patrick-skew_critical_problems}.

\begin{proposition}\label{prop:prop_1_in_CP07}
  Let $(M,\grade_M)$ be a manifold with a grade, $N$ be a manifold,
  and $\pi:E\rightarrow N$ a vector bundle. Suppose that
  $f:M\rightarrow E$ is $C^k$ with $k\geq 1$ and that $f(m)\in 0(E)$
  whenever $\grade_M(m)=0$. Then, for all $m\in \grade_M^{-1}(0)$
  there is a unique $e(m)\in E_{\pi(f(m))}$ such that
  \begin{equation*}
    Vert\  df(m)(v_m) = (d\grade_M(m)(v_m)) e(m) \stext{ for all } v_m\in T_mM.
  \end{equation*}
  Moreover, the function $\hat{f}:M\rightarrow E$ defined by
  \begin{equation*}
    \hat{f}(m) :=
    \begin{cases}
      \frac{f(m)}{\grade_M(m)}, \text{ if } \grade_M(m)\neq 0,\\
      e(m), \text{ if } \grade_M(m)=0
    \end{cases}
  \end{equation*}
  is $C^{k-1}$.
\end{proposition}

\begin{proof}
  This is Proposition 1
  in~\cite{ar:cuell_patrick-skew_critical_problems}.
\end{proof}

We close this section with a characterization of contact order in
terms of local bounds, as it appears, for instance, in Section 2.3
of~\cite{ar:marsden_west-discrete_mechanics_and_variational_integrators},
and, also, in terms of matching of derivatives.

\begin{lemma}\label{le:contact_order_with_local_bounds_and_derivatives}
  Let $(M,\grade_M)$ be a manifold with a grade and
  $f_1,f_2:M\rightarrow N$ be smooth maps. Then, the following
  assertions are equivalent.
  \begin{enumerate}
  \item \label{it:contact_order_with_local_bounds_and_derivatives-order}
    $f_2=f_1+\mathcal{O}(\grade_M^r)$.
  \item \label{it:contact_order_with_local_bounds_and_derivatives-bound}
    For each $m_0\in \grade_M^{-1}(0)$ there are coordinate charts
    $(V,\phi)$ and $(W,\psi)$ of $M$ and $N$, an open subset
    $U\subset \R^{\dim(M)-1}$, and constants $a>0$, $C\geq 0$ such
    that
    \begin{enumerate}
    \item \label{it:contact_order_with_local_bounds_and_derivatives-bound-contained}
      $m_0\in V$ and $f_j(V)\subset W$ for $j=1,2$,
    \item \label{it:contact_order_with_local_bounds_and_derivatives-bound-bound}
      $U\times (-a,a) \subset \phi(V)$; $\phi(m_0)=(u_0,0)$ with
      $u_0\in U$; $\grade_M(\phi^{-1}(u,h)) = h$; and where the local
      expressions $\check{f}_1$ and $\check{f}_2$ of $f_1$ and $f_2$
      satisfy
      \begin{equation}\label{eq:contact_order_with_local_bounds_and_derivatives-bound}
        \norm{\check{f}_2(u,h)-\check{f}_1(u,h)}\leq C h^r
        \stext{ for all } u\in U \text{ and } \abs{h}<a,
      \end{equation}
      where $\norm{\cdot}$ is the euclidean norm in $\R^{\dim(N)}$.
    \end{enumerate}
  \item \label{it:contact_order_with_local_bounds_and_derivatives-derivative}
    For each $m_0\in \grade_M^{-1}(0)$ there are coordinate charts
    $(V,\phi)$ and $(W,\psi)$ of $M$ and $N$, an open subset
    $U'\subset \R^{\dim(M)-1}$, and a constant $a'>0$ such that
    \begin{enumerate}
    \item \label{it:contact_order_with_local_bounds_and_derivatives-bound-contained}
      $m_0\in V$ and $f_j(V)\subset W$ for $j=1,2$,
    \item \label{it:contact_order_with_local_bounds_and_derivatives-bound-bound}
      $U'\times (-a',a') \subset \phi(V)$; $\phi(m_0)=(u_0,0)$ with
      $u_0\in U'$; $\grade_M(\phi^{-1}(u,h)) = h$; and where the local
      expressions $\check{f}_1$ and $\check{f}_2$ of $f_1$ and $f_2$
      satisfy
      \begin{equation}\label{eq:contact_order_with_local_bounds_and_derivatives-derivatives}
        D_2^j\check{f}_2(u,0)=D_2^j\check{f}_1(u,0)
        \stext{ for all } j=0,\ldots,r-1, \stext{ and } u \in U'.
      \end{equation}
    \end{enumerate}
  \end{enumerate}
\end{lemma}

\begin{proof}
  It is a simple application of Theorem~\ref{th:taylor-r_minus_1} and
  the localization provided by
  Lemma~\ref{le:order_of_maps_in_terms_of_local_expression}.
\end{proof}


\subsection{Critical problems}
\label{sec:critical_problems}

Just as the dynamics of \FMSs is defined in terms of a critical
problem, the dynamics of their discrete counterparts will be defined in
the same terms. It is convenient to use the geometric approach to such
problems provided by the notion of \jdef{skew critical problems},
considered in~\cite{ar:cuell_patrick-skew_critical_problems}. Next we
review some of the basic ideas.

\begin{definition}\label{def:critical_problem_and_point}
  A \jdef{critical problem} is a triple $(\gAF,g,\mathcal{D})$ where
  $g:M\rightarrow N$ is a smooth submersion, $\gAF$ is a smooth $1$-form
  on $M$ and $\mathcal{D}$ is a smooth distribution on $M$. In this paper
  the most common case is that $\mathcal{D}:=\ker(Tg)$; in this case
  we refer to the critical problem as $(\gAF,g)$.  A point $m\in M$ is
  a \jdef{critical point} of the critical problem
  $(\gAF,g,\mathcal{D})$ at $n\in N$ if
  \begin{equation*}
    \begin{cases}
      \gAF(m)(v)=0 \stext{ for all } v\in \mathcal{D}_{m},\\
      g(m)=n.
    \end{cases}
  \end{equation*}
\end{definition}

\begin{remark}
  Definition 1 of~\cite{ar:cuell_patrick-skew_critical_problems} is
  more general than our
  Definition~\ref{def:critical_problem_and_point} in two aspects: it
  allows for $C^k$ problems and, also, an infinite-dimensional
  context. As we don't need that level of generality in this paper, we
  stay with our simpler version. Also, we drop the ``skew'' modifier in
  the names for simplicity.
\end{remark}

\begin{remark}\label{rem:critical_condition_for_D=Tg_crit_problems}
  For critical problems $(\gAF,g)$ over $M$, that is, the distribution
  is $\mathcal{D}=\ker(Tg)$, for each $n\in N$, we can define the
  embedded submanifold $\mathcal{Z}_n:=g^{-1}(\{n\})\subset M$ and the
  inclusion map $i_{\mathcal{Z}_n}:\mathcal{Z}_n\rightarrow M$. Then
  $m\in M$ is a critical point of $(\gAF,g)$ at $n$ if and only if
  $m\in \mathcal{Z}_n$ and $i_{\mathcal{Z}_n}^*(\gAF)(m)=0$.
\end{remark}

\begin{definition}\label{def:hessian}
  Let $m\in M$ be a critical point of the problem
  $(\gAF,g,\mathcal{D})$. The \jdef{Hessian} of $(\gAF,g,\mathcal{D})$
  at $m$ is the bilinear form
  $d_{\mathcal{D},g}\gAF(m):\ker(T_{m}g)\times
  \mathcal{D}_{m}\rightarrow \R$ defined by
  \begin{equation}\label{eq:skew_hessian-def}
    d_{\mathcal{D},g}\gAF(m)(\delta u, \delta w) = 
    \big(d(\gAF(\delta W))\big)(m)(\delta u)
  \end{equation}
  where $\delta W$ is a (local) vector field on $M$ with values in
  $\mathcal{D}$ such that $\delta W(m) = \delta w$. The critical point
  $m$ is \jdef{nondegenerate} if $d_{\mathcal{D},g}\gAF(m)$ is
  nondegenerate in the sense that
  $d_{\mathcal{D},g}\gAF(m)(\delta u, \delta w) = 0$ for all
  $\delta w$ implies that $\delta u=0$.
\end{definition}

By Remark 1 in~\cite{ar:cuell_patrick-skew_critical_problems},
$d_{\mathcal{D},g}\gAF(m)$ is well defined (that is, it is independent
of the extension used in the construction).

\begin{lemma}\label{le:critical_problems_and_diffeomorphisms}
  For $j=1,2$, let $(\gAF_j,g_j,\mathcal{D}_j)$ be two critical
  problems for $g_j:M_j\rightarrow N_j$. If $F:M_1\rightarrow M_2$ is
  a diffeomorphism and $f:g_1(M_1)\rightarrow g_2(M_2)$ is a bijection
  so that $g_2\circ F = f\circ g_1$,
  $\mathcal{D}_2 = TF(\mathcal{D}_1)$ and
  $\gAF_1(v)(\delta v)=F^*(\gAF_2)(v)(\delta v)$ for all $v\in M_1$
  and $\delta v \in (\mathcal{D}_1)_{v}$, then there is a bijection
  between the sets of critical points of both problems. In particular,
  $m_1\in M_1$ is a critical problem of the first problem over
  $n_1\in N_1$ if and only if $F(m_1)$ is a critical point of the
  second problem over $f(n_1)$. When $f$ is a diffeomorphism and
  $\mathcal{D}_j := \ker(T g_j)$ for $j=1,2$, the hypothesis
  $\mathcal{D}_2 = TF(\mathcal{D}_1)$ is automatically satisfied.
\end{lemma}

\begin{proof}
  If $m_1\in M_1$ is a critical problem of the first problem over
  $n_1\in N_1$ we have that $g_1(m_1)=n_1$ and
  $\gAF_1(m_1)(\delta m)=0$ for all
  $\delta m\in (\mathcal{D}_1)_{m_1}$. But then,
  $g_2(F(m_1)) = f(g_1(m_1)) = f(n_1)$ and, as
  $T_{m_1}F|_{(\mathcal{D}_1)_{m_1}}:(\mathcal{D}_1)_{m_1}\rightarrow
  (\mathcal{D}_2)_{F(m_1)}$ is an isomorphism, for any
  $\delta m' \in (\mathcal{D}_2)_{F(m_1)}$ we have
  \begin{equation*}
    \gAF_2(F(m_1))(\delta m') = \gAF_2(F(m_1))(T_{m_1}F(\delta m)) =
    F^*(\gAF_2)(m_1)(\delta m) = \gAF_1(m_1)(\delta m) =0.
  \end{equation*}
  Hence, $F(m_1)$ is a critical point of the second problem over
  $f(n_1)$. As the process can be applied to $F^{-1}$ and $f^{-1}$, we
  see that it defines a bijection between the corresponding sets of
  critical points.
\end{proof}

When the data of two critical problems are related as in the
hypotheses of Lemma~\ref{le:critical_problems_and_diffeomorphisms}, we
say that the problems are \jdef{equivalent}.

A typical problem is that of finding the critical points of a problem
for different ``boundary values'' $n\in N$. A technique that we will
use later starts by finding critical points for some special boundary
values where the computations are easier and, then, apply the
following result (Theorem 2
from~\cite{ar:cuell_patrick-skew_critical_problems}) to extend the
construction to other cases.

\begin{theorem}\label{th:thm_2_in_cp07}
  Let $(\gAF,\mathcal{D},g)$ be a critical problem, where
  $g:M\rightarrow N$ and assume that
  \begin{enumerate}[label=(\alph*)]
  \item $M_0\subset M$ and $N_0\subset N$ are closed submanifolds and
    $\gamma_0:N_0\rightarrow M_0$ is a diffeomorphism and
  \item for all $n\in N_0$, $\gamma_0(n)$ is a nondegenerate critical
    point of $(\gAF,\mathcal{D},g)$ at $n$.
  \end{enumerate}
  Then, there are open neighborhoods $U$ of $M_0$ and $V$ of $N_0$ and
  a smooth extension $\gamma:V\rightarrow U$ of $\gamma_0$ such that
  \begin{enumerate}
  \item for all $n\in V$, $\gamma(n)$ is a critical point of
    $(\gAF,\mathcal{D},g)$ at $n$ and
  \item $\gamma(n)$ is the unique critical point of
    $(\gAF,\mathcal{D},g)$ in $U$.
  \end{enumerate}
\end{theorem}

The extensions $\gamma$ whose existence is proved in
Theorem~\ref{th:thm_2_in_cp07} will be called \jdef{critical point
  functions}.

Later on, we will consider the behavior of families of critical
problems and their contact order. When $(M,\grade_M)$ and
$(N,\grade_N)$ are manifolds with a grade and, for $j=1,2$,
$(\gAF_j,g_j,\mathcal{D}_j)$ are critical problems over $M$, the
meaning of $\gAF_2=\gAF_1+\mathcal{O}(\grade_M^r)$ and
$g_2=g_1+\mathcal{O}(\grade_M^r)$ is known as both
$g_j:M\rightarrow N$ and $\gAF_j:M\rightarrow T^*M$ are maps between
manifolds. In order to give a meaning to
$\mathcal{D}_2=\mathcal{D}_1+\mathcal{O}(\grade_M^r)$ we observe that
a (regular) distribution $\mathcal{D}$ of rank $r$ over $M$ assigns to
each $m\in M$ a subspace $\mathcal{D}_m\in\Gr(T_mM,r)$, the
Grassmannian of $r$ dimensional subspaces of $T_mM$. So, if
$\GrB(TM,r)$ is the Grassmann bundle over $M$ (the fiber at each
$m\in M$ is $\Gr(T_mM,r)$), we have the function
$i_\mathcal{D}:M\rightarrow \GrB(TM,r)$ defined by
$i_\mathcal{D}(m):= \mathcal{D}_m$; it turns out that $i_\mathcal{D}$
is a smooth map (see Appendices~\ref{sec:things_about_grassmannians}
and~\ref{sec:local_computations_in_the_grassmann_bundle} for more on
Grassmannians). Then, we say that
$\mathcal{D}_2=\mathcal{D}_1+\mathcal{O}(\grade_M^r)$ when
$i_{\mathcal{D}_2}=i_{\mathcal{D}_1}+\mathcal{O}(\grade_M^r)$. Eventually,
when $\mathcal{D}_2=\mathcal{D}_1+\mathcal{O}(\grade_M^r)$ we define
$\cpres^r(\mathcal{D}_2,\mathcal{D}_1)(m) :=
\cpres^r(i_{\mathcal{D}_2},i_{\mathcal{D}_1})(m)$ for all
$m\in \grade_M^{-1}(0)$. The following important result (Theorem 3
in~\cite{ar:cuell_patrick-skew_critical_problems}) relates the order
of contact of the critical point functions $\gamma_j$ corresponding to
the two critical problems $(\gAF_j,g_j,\mathcal{D}_j)$ over $M$ to the
order of contact of the corresponding data of the problems.

\begin{theorem}\label{th:thm_3_in_cp07}
  Let $(M,\grade_M)$ and $(N,\grade_N)$ be manifolds with a grade and,
  for $j=1,2$, $(\gAF_j,g_j,\mathcal{D}_j)$ critical problems over
  $M$. Let $M_0\subset \grade_M^{-1}(0)$ and
  $N_0\subset \grade_N^{-1}(0)$ be closed submanifolds and $\gamma_0$
  as in Theorem~\ref{th:thm_2_in_cp07}. Assume that
  $\gAF_2=\gAF_1+\mathcal{O}(\grade_M^r)$,
  $g_2=g_1+\mathcal{O}(\grade_M^r)$ and that
  $\mathcal{D}_2=\mathcal{D}_1+\mathcal{O}(\grade_M^r)$. Then the
  corresponding critical point maps (constructed by
  Theorem~\ref{th:thm_2_in_cp07}) satisfy
  $\gamma_2=\gamma_1+\mathcal{O}(\grade_N^r)$.
\end{theorem}


\subsection{Remarks on Cartesian products}
\label{sec:remarks_on_cartesian_products}

In this section we briefly review some constructions that are
available when considering smooth maps between manifolds that are
Cartesian products.

When $X:=X_1\times X_2$, we know that
$TX\simeq p_1^*TX_1 \oplus p_2^*TX_2$ and each $p_j^*TX_j$ can be
identified with a subbundle of $TX$:
\begin{equation*}
  p_1^*TX_1 \simeq \ker(Tp_2) \stext{ and } p_2^*TX_2 \simeq \ker(Tp_1),
\end{equation*}
so that $TX = \ker(Tp_2)\oplus \ker(Tp_1)$. Similarly,
$T^*X\simeq p_1^*T^*X_1 \oplus p_2^*T^*X_2$ where we can identify the
pullback bundles with annihilator subbundles:
\begin{equation*}
  p_1^*T^*X_1 \simeq (\ker(Tp_1))^\circ  \stext{ and }
  p_2^*T^*X_2 \simeq (\ker(Tp_2))^\circ,
\end{equation*}
so that $T^*X = (\ker(Tp_1))^\circ \oplus (\ker(Tp_2))^\circ$.

If $\gF\in\mathcal{A}^1(X)$, for each $x\in X$ we have that
$\gF(x) = \gF_1(x) + \gF_2(x)$ with
$\gF_j(x) \in (\ker(T_xp_j))^\circ$. We define $d_j\gF:=\gF_j$ and,
so, $d_j\gF \in \Gamma(X,(\ker(Tp_j))^\circ)$.

If $Y:=Y_1\times Y_2$ and $F:X\rightarrow Y$ we define
$F^j:= p_j\circ F$ and, then, for $(x_1,x_2)\in X$ decompose
\begin{equation}\label{eq:cartesian_product_TF_decomp}
  T_{(x_1,x_2)}F := \left(
    \begin{array}{cc}
      T^1_{(x_1,x_2)}F^1 & T^2_{(x_1,x_2)}F^1\\
      T^1_{(x_1,x_2)}F^2 & T^2_{(x_1,x_2)}F^2
    \end{array}
  \right),
\end{equation}
according to the decompositions of $TX$ and $TY$.  We will also use
the notation
\begin{equation*}
    T^1_{(x_1,x_2)}F := \left(
    \begin{array}{c}
      T^1_{(x_1,x_2)}F^1\\
      T^1_{(x_1,x_2)}F^2
    \end{array}
  \right) \stext{ and }
  T^2_{(x_1,x_2)}F := \left(
    \begin{array}{c}
      T^2_{(x_1,x_2)}F^1\\
      T^2_{(x_1,x_2)}F^2
    \end{array}
  \right)
\end{equation*}

If $\gF\in \mathcal{A}^1(Y)$ we have that
$F^*(\gF)\in \mathcal{A}^1(X)$ is given by
$F^*(\gF)(x)(\delta x) := \gF(F(x))(T_xF(\delta x))$ for all $x\in X$
and $\delta x \in T_xX$. Taking into account the corresponding
decompositions we can write $\gF = \gF_1+\gF_2$ with
$\gF_j \in (\ker(Tp_j))^\circ$, $\delta x = \delta x_1 +\delta x_2$
with $\delta x_j\in (\ker(Tp_{3-j}))$ and, so,
\begin{equation*}
  \begin{split}
    F^*(\gF)(x)(\delta x) =& (\gF_1(x)+\gF_2(x))(T_{(x_1,x_2)}F(\delta
    x_1+\delta x_2)) \\=& \gF_1(x)(T_{(x_1,x_2)}F^1(\delta x_1+\delta
    x_2)) + \gF_2(x)(T_{(x_1,x_2)}F^2(\delta x_1+\delta x_2)) \\=&
    \gF_1(x)(T^1_{(x_1,x_2)}F^1(\delta x_1)) +
    \gF_1(x)(T^2_{(x_1,x_2)}F^1(\delta x_2)) \\&+
    \gF_2(x)(T^1_{(x_1,x_2)}F^2(\delta x_1)) +
    \gF_2(x)(T^2_{(x_1,x_2)}F^2(\delta x_2))
  \end{split}
\end{equation*}
Notice that if $\gF = \gF_2 \in \Gamma(Y,(\ker(Tp_2))^\circ)$ and
$\delta x = \delta x_1 \in \ker(Tp_2)$, we have
\begin{equation*}
  \begin{split}
    F^*(\gF)(x)(\delta x) = \gF_2(x)(T^1_{(x_1,x_2)}F^2(\delta x_1))
  \end{split}
\end{equation*}
that need not vanish, so that $F^*(\gF)$ may not be in
$\Gamma(X,(\ker(Tp_2))^\circ)$. We can modify the pullback operation
so that this new operation maps $\Gamma(Y,(\ker(Tp_2))^\circ)$ into
$\Gamma(X,(\ker(Tp_2))^\circ)$. Let $i_2:\ker(Tp_1)\rightarrow TX$ be
the inclusion map and
$i_2^*:T^*X\rightarrow \ker(Tp_1)^*\simeq (\ker(Tp_2))^\circ$ be the
dual map; both maps are vector bundle maps over $id_X$. Then we define
\begin{equation}\label{eq:*_2-def}
  F^{*_2}:\Gamma(Y,(\ker(Tp_2))^\circ)\rightarrow \Gamma(X,(\ker(Tp_2))^\circ)
  \stext{ by } F^{*_2}(\gF) := i_2^*\circ F^*(\gF).
\end{equation}
In simple terms, $F^{*_2}(\gF)$ is $F^*(\gF)$ restricted to
$\ker(Tp_1)$ and vanishing on $\ker(Tp_2)$.

When $F:X\rightarrow Y$ is a diffeomorphism, there is the smooth map
$T^*F:T^*Y\rightarrow T^*X$ that, in a similar way
to~\eqref{eq:cartesian_product_TF_decomp}, can be written according to
the decompositions of $T^*X$ and $T^*Y$ as
\begin{equation*}
  T^*F = \left(
    \begin{array}{cc}
      (T^1)^*F^1 & (T^1)^* F^2\\
      (T^2)^*F^1 & (T^2)^* F^2
    \end{array}
  \right)
\end{equation*}
and, also, as
\begin{equation*}
  (T^1)^* F = \left(
    \begin{array}{cc}
      (T^1)^*F^1 & (T^1)^* F^2 
    \end{array}
  \right) \stext{ and }
    (T^2)^* F = \left(
    \begin{array}{cc}
      (T^2)^*F^1 & (T^2)^* F^2 
    \end{array}
  \right).
\end{equation*}

For notational convenience we may write the ``partial derivative
indexes'' as sub-indexes instead of super-indexes, as in $T_2F$
instead of $T^2F$ or $T_2^*F$ instead of $(T^2)^*F$.

Of course, all the previous discussion can be adapted to spaces that
are the Cartesian product of any finite number of factors.


\section{Discretizations (a la Cuell \& Patrick)}
\label{sec:discretizations_a_la_cuell_and_patrick}

In this section we discuss a type of discretization of mechanical
systems introduced by Cuell \& Patrick that allows for a careful study
of the asymptotic behavior of a family of discrete mechanical systems
when the discretization parameter $h$ approaches $0$. We first
introduce in Section~\ref{sec:discretizations_of_TQ} a notion of
discretization of $TQ$ that is purely geometrical and, in
Section~\ref{sec:discretizations_of_a_forced_mechanical_system},
consider discretizations of mechanical systems with forces. Then,
in Section~\ref{sec:flows_of_discretizations_of_FMS}, we prove that
discretizations of regular continuous systems have flows. Finally, in
Section~\ref{sec:trajectories_of_exact_forced_discrete_mechanical_systems},
we study the properties of certain discretizations of regular forced
mechanical systems that we call \jdef{exact}.


\subsection{Discretizations of $TQ$}
\label{sec:discretizations_of_TQ}

In broad terms, discretizations in Discrete Mechanics have been taken
to be diffeomorphisms from (an open neighborhood of the image of the
zero section $Z_Q$ of) $TQ$ into (an open neighborhood of the diagonal
$\Delta_Q$ in) $Q\times Q$. Cuell \& Patrick introduce
in~\cite{ar:patrick_cuell-error_analysis_of_variational_integrators_of_unconstrained_lagrangian_systems}
and~\cite{ar:cuell_patrick-geometric_discrete_analogues_of_tangent_bundles_and_constrained_lagrangian_systems}
a more refined notion as follows.

\begin{definition}\label{def:discretization_of_TQ}
  Let $Q$ be a manifold. A \jdef{smooth discretization} of $TQ$ is a
  triple $(\psi,\dBTp,\dBTm)$ where $\psi:U\rightarrow Q$ is a
  function from an open subset $U\subset\R^2\times TQ$ and, for some
  $a>0$, $\dBTp,\dBTm:(-a,a)\rightarrow \R$  are such that
  \begin{enumerate}
  \item $\{0\}\times\{0\}\times TQ \subset U$ and
    $\psi\in C^\infty(U,Q)$,
  \item $h\, \dBTp(h)\geq 0$ and $h\,\dBTm(h)\leq 0$ for all
    $h\in (-a,a)$,
  \item $\dBTp$ and $\dBTm$ are smooth and $\dBTp(h)-\dBTm(h)=h$ for
    all $h\in (-a,a)$,
  \item $\psi(h,0,v_q)=q$ and $\pd{\psi}{t}(h,t,v_q)|_{t=0}=v_q$ for all $h$
    and $v_q\in T_qQ$ with $(h,0,v_q)\in U$.
  \end{enumerate}
\end{definition}

In fact, Definition 1
in~\cite{ar:cuell_patrick-geometric_discrete_analogues_of_tangent_bundles_and_constrained_lagrangian_systems}
is more general as it allows for nonsmooth objects, which we don't
need in this paper. Also, we use $(-a,a)$ as the domain for
$\dBTp,\dBTm$ instead of $[0,a)$ in order to avoid working with
manifolds with boundary.

\begin{remark}\label{rem:bounds_for_boundary_times}
  It is easy to check that
  \begin{equation*}
    \dBTp(h)
    \begin{cases}
      \leq h \text{ if } h\geq 0,\\
      \geq h \text{ if } h\leq 0,
    \end{cases}
    \stext{ and }
    \dBTm(h)
    \begin{cases}
      \geq -h \text{ if } h\geq 0,\\
      \leq -h \text{ if } h\leq 0.
    \end{cases}
  \end{equation*}
\end{remark}

\begin{definition}\label{def:smooth_discretization}
  If $(\psi,\dBTp,\dBTm)$ is a smooth discretization of $TQ$,
  let $A^+,A^-: (-a,a)\times TQ\rightarrow \R\times\R\times TQ$ be
  defined by
  \begin{equation*}
    A^+(h,v) := (h,\dBTp(h),v) \stext{ and }
    A^-(h,v) := (h,\dBTm(h),v).
  \end{equation*}
  Then we have the open subsets $U^+:=(A^+)^{-1}(U)$ and
  $U^-:=(A^-)^{-1}(U)$, where $U$ is the domain of $\psi$.  The
  \jdef{boundary maps}
  \begin{equation*}
    \del^+:U^+\rightarrow Q \stext{ and } \del^-:U^-\rightarrow Q
  \end{equation*}
  are defined by
  \begin{equation*}
    \begin{split}
      \del^+(h,v):=& \psi(A^+(h,v)) = \psi(h,\dBTp(h),v) \stext{ and }\\
      \del^-(h,v):=& \psi(A^-(h,v)) = \psi(h,\dBTm(h),v).
    \end{split}
  \end{equation*}
  Both $\del^+,\del^-$ are smooth maps; for each fixed $h$, we also
  consider $\del^+_h$ as the restriction of $\del^+$ to
  $(\{h\}\times TQ)\cap U^+$ and, analogously, $\del^-_h$. Two
  associated operators
  $\del^\pm, \del^\mp:U^+\cap U^-\rightarrow Q\times Q$ are defined by
  \begin{equation*}
    \del^{\pm}(h,v) := (\del^+(h,v), \del^-(h,v)) \stext{ and }
    \del^{\mp}(h,v) := (\del^-(h,v), \del^+(h,v)).
  \end{equation*}
  We will also use
  $\ti{\del}^{\mp}, \ti{\del}^{\mp}: U^+\cap U^-\rightarrow \R\times
  Q\times Q$ defined by
  \begin{equation*}
    \ti{\del}^{\pm}(h,v):=(h,\del^\pm(h,v))\footnote{In~\cite{ar:cuell_patrick-skew_critical_problems}
      and~\cite{ar:patrick_cuell-error_analysis_of_variational_integrators_of_unconstrained_lagrangian_systems}
      the map $\ti{\del}^{\pm}$ is denoted by $\Psi$.} \stext{ and }
    \ti{\del}^{\mp}(h,v):=(h,\del^\mp(h,v)).
  \end{equation*}
  
\end{definition}

\begin{remark}\label{rem:motivation_for_discretization}
  The main motivation for the notion of smooth discretization of $TQ$
  is that one such object $(\psi,\dBTp,\dBTm)$ assigns, to each
  $v_q\in T_qQ$ and $h\in [0,a)$ (subject to domain restrictions), a
  smooth curve $t\mapsto \psi(h,t,v_q)\in Q$; this curve passes
  through $q$ with velocity $v_q$ when $t=0$. The values
  $\del^-_h(v_q)$ and $ \del^+_h(v_q)$ correspond to the
  ``boundaries'' of such a curve, for $t\in
  [\dBTm(h),\dBTp(h)]$. Hence, a choice of discretization $\psi$
  assigns to each $v_q\in T_qQ$ a family (parametrized by $h$) of
  curves passing through $v_q$ and also, using $\del^{\mp}_h$, a
  (family parametrized by $h$ of) pair of points in $Q$. This last
  assignment of pairs of points to tangent vectors is what intuitively
  corresponds to a discretization.
\end{remark}

\begin{example}\label{ex:linear_discretization}
  Let $U:=\R^2\times T\R^n$, $\psi:U\rightarrow \R^n$ be defined by
  $\psi(h,t,(q,v)):= q+t v$, $\dBTm(h):=0$ and $\dBTp(h):=h$ for
  all $h\in\R$. Then $(\psi,\dBTp,\dBTm)$ is a smooth
  discretization of $T\R^n$. The corresponding boundary maps are:
  \begin{equation*}
    \del^-_h(q,v) = q \stext{ and } \del^+_h(q,v) = q+h v.
  \end{equation*}
  Notice that, for all $h\neq 0$,
  \begin{equation*}
    \del^\mp_h(q,v) = (q,q+h v)
  \end{equation*}
  is a diffeomorphism
  $\del^\mp_h:T\R^n\rightarrow \R^n\times \R^n$ with
  \begin{equation*}
    (\del^\mp_h)^{-1}(q_0,q_1) = (q_0,\frac{q_1-q_0}{h}).
  \end{equation*}
  The same construction is interesting for other choices of $\dBTp$
  and $\dBTm$ as, for instance, $\dBTm(h):=-\frac{h}{2}$ and
  $\dBTp(h):=\frac{h}{2}$.
\end{example}

\begin{example}\label{ex:exact_discretization_forced}
  Let $(Q,L,f)$ be a regular \FMS. Let $X_{L,f}\in\VF(TQ)$ be the
  corresponding Lagrangian vector field and
  $U^{X_{L,f}}\subset \R\times TQ$ be the open subset containing
  $\{0\}\times TQ$ where the flow
  $F^{X_{L,f}}:U^{X_{L,f}}\rightarrow TQ$ is defined (see
  Section~\ref{sec:forced_mechanical_systems}). Observe that
  $(h,v)\in U^{X_{L,f}}$ if and only if $(-h,v)\in
  U^{X_{L,f}}$. Define the open set
  \begin{equation*}
    \begin{split}
      U^E:=&\{(h,t,v)\in\R\times\R\times TQ:(h,v)\in U^{X_{L,f}} \text{
        and } (t,v)\in U^{X_{L,f}}\} \\=& p_{13}^{-1}(U^{X_{L,f}})\cap
      p_{23}^{-1}(U^{X_{L,f}})
    \end{split}
  \end{equation*}
  and the function
  \begin{equation*}
    \psi^E:U^E\rightarrow Q \stext{ by } \psi^E(h,t,v) :=
    \tau_Q(F^{X_{L,f}}_t(v)).
  \end{equation*}
  If $\dBTp,\dBTm$ are functions as in
  Definition~\ref{def:discretization_of_TQ}, then
  $(\psi^E,\dBTp,\dBTm)$ is a smooth discretization of $TQ$ known as
  an \jdef{exact discretization} of $TQ$ arising from $(Q,L,f)$. It is
  easy to check using Remark~\ref{rem:bounds_for_boundary_times} that
  $(U^E)^+ = U^{X_{L,f}} = (U^E)^-$; hence, the exact boundary maps
  are $\del^{E+},\del^{E-}:U^{X_{L,f}}\rightarrow Q$.
\end{example}

\begin{example}\label{ex:particle_with_friction-exact_discretization}
  Now we specialize the construction of
  Example~\ref{ex:exact_discretization_forced} for the \FMS considered
  in Example~\ref{ex:particle_with_friction-def}. As the flow is
  globally defined (Example~\ref{ex:particle_with_friction-flow}) ,
  $U^{X_{L,f}} = \R\times TQ$ and, then, $U^E=\R\times\R\times
  TQ$. The exact discretization is
  \begin{equation*}
    \psi^E:\R\times\R\times TQ \stext{ for }
    \psi^E(h,t,q,v):= q+\frac{1}{\alpha} v (1-e^{-\alpha t}).
  \end{equation*}
  Any pair of maps $\alpha^+$ and $\alpha^-$ satisfying the conditions
  of Definition~\ref{def:discretization_of_TQ} could be used with
  $\psi^E$. In what follows we choose $\dBTp(h):=h$ and $\dBTm(h):=0$
  for all $h$. Therefore we have the boundary maps
  \begin{gather*}
    \del^{E+},\del^{E-}:\R\times\R\times TQ \rightarrow Q \stext{ with
    }\\ \del^{E+}(h,q,v) = q+\frac{1}{\alpha} v (1-e^{-\alpha h})
    \stext{ and } \del^{E-}(h,q,v) = q.
  \end{gather*}
\end{example}

A discretization $\psi$ seems to introduce, actually, a family of
``discrete'' versions of $TQ$ (in the sense of
Remark~\ref{rem:motivation_for_discretization}), parametrized by
$h$. Cuell \& Patrick introduce the following notion to capture the
structure obtained by fixing a value of $h$ in a given
discretization $\psi$.

\begin{definition}
  Let $Q$ be a manifold. A \jdef{discrete tangent bundle} of $Q$ is a
  triple $(\mathcal{V},\del^+,\del^-)$, where $\mathcal{V}$ is
  a manifold such that $\dim(\mathcal{V}) =2\dim(Q)$ and
  $\del^+,\del^-:\mathcal{V}\rightarrow Q$ satisfy:
  \begin{enumerate}
  \item $\del^+$ and $\del^-$ are submersions such that
    $\ker(T\del^+) \cap \ker(T\del^-) = \{0\}$ and
  \item for all $q\in Q$, if we let
    $\mathcal{V}^+_q:=(\del^+)^{-1}(q)$ and
    $\mathcal{V}^-_q:=(\del^-)^{-1}(q)$, then
    $\mathcal{V}^+_q\cap\mathcal{V}^-_q$ consists of a single point,
    denoted by $0_q$.
  \end{enumerate}
\end{definition}

\begin{remark}\label{rem:discrete_tangent_bundle_borderPM_is_diffeo}
  It follows from the definition that if $(\mathcal{V},\del^+,\del^-)$
  is a discrete tangent bundle of $Q$ the map
  $\del^\pm : \mathcal{V}\rightarrow Q\times Q$ defined by
  $\del^\pm(v):=(\del^+(v),\del^-(v))$ is a local diffeomorphism. On
  the other hand, also from the definition, if $Z_Q:=\{0_q:q\in Q\}$,
  we have that $\del^\pm|_{Z_Q}:Z_Q\rightarrow \Delta_Q$ is a
  bijection. Then, by Theorem 1
  in~\cite{ar:cuell_patrick-skew_critical_problems}, there are open
  neighborhoods of $Z_Q$ in $\mathcal{V}$ and of
  $\Delta_Q\subset Q\times Q$ so that $\del^\pm$ is a diffeomorphism
  between these open neighborhoods. The same is true for
  $\del^\mp(v):=(\del^-(v),\del^+(v))$.
\end{remark}

\begin{example}
  If $Q$ is a manifold, then $(Q\times Q,p_2,p_1)$ is a discrete
  tangent bundle of $Q$ (where $p_a$ is the projection from the
  product space onto the corresponding component).
\end{example}

The following important result essentially states that, for each fixed
$h$ (in a range), a discretization of $TQ$ defines a discrete tangent
bundle of $Q$.

\begin{theorem}\label{th:discretizations_yield_discrete_tangent_bundles}
  Let $(\psi,\dBTp,\dBTm)$ be a discretization of $TQ$ and let
  $Q_0\subset Q$ be a relatively compact\footnote{A subset
    $X\subset Y$ of a topological space is \jdef{relatively compact}
    if its closure $\conj{X}\subset Y$ is compact.} open
  submanifold. Then, there exists $a>0$ such that, for each
  $h\in (0,a)$, there is an open subset $\mathcal{V}_h\subset TQ$
  satisfying
  \begin{enumerate}
  \item the triple $(\mathcal{V}_h,\del^+_h,\del^-_h)$ is a
    discrete tangent bundle of $Q_0$ and
  \item $\del^{\pm}_h$ is a diffeomorphism from $\mathcal{V}_h$ onto
    an open neighborhood $\mathcal{W}_h$ of $\Delta_{Q_0}$ in
    $Q_0\times Q_0$.
  \end{enumerate}
  Moreover, given $v_q\in T_qQ_0$, the constant $a>0$ can be chosen so
  that $v_q\in\mathcal{V}_h$ for all $h\in (0,a)$.
\end{theorem}

The proof of
Theorem~\ref{th:discretizations_yield_discrete_tangent_bundles} is
essentially that of Proposition 1
in~\cite{ar:cuell_patrick-geometric_discrete_analogues_of_tangent_bundles_and_constrained_lagrangian_systems}. Still,
below we sketch that proof so that we can clarify a point that, we
think, is unclear in the original argument (see
Remark~\ref{rem:problem_with_CP09_prop1}).

\begin{lemma}\label{le:discretizations_yield_discrete_tangent_bundles-U_and_W}
  Under the hypotheses of
  Theorem~\ref{th:discretizations_yield_discrete_tangent_bundles},
  there are subsets $\mathcal{U}\subset \R\times TQ$ and
  $\mathcal{W}\subset \R\times Q\times Q$ such that
  $\{0\}\times TQ\subset \mathcal{U}$, $\mathcal{U}$ is contained in
  the domain of $\ti{\del^\pm}$,
  $\{0\}\times \Delta_Q\subset \mathcal{W}$,
  $\ti{\del^\pm}(\mathcal{U}) = \mathcal{W}$ and $\ti{\del}^\pm$ is a
  diffeomorphism between the open subsets
  $\mathcal{U}\SM (\{0\}\times TQ)$ and
  $\mathcal{W}\SM (\{0\}\times Q\times Q)$. Furthermore,
  $\mathcal{U}\subset \R\times TQ$ is an open subset. The same result
  is valid with $\ti{\del}^\mp$ instead of $\ti{\del}^\pm$.
\end{lemma}

The argument given
in~\cite{ar:cuell_patrick-geometric_discrete_analogues_of_tangent_bundles_and_constrained_lagrangian_systems}
uses the notion of tubular neighborhood that we review next. A
\jdef{tubular map} of an embedded submanifold $X\subset Y$ is a
diffeomorphism $\zeta: W^N\rightarrow W^Y$, where $W^N$ is an open
neighborhood of the (image of the) zero section $0_X$ of the normal
bundle $N$ associated to $X\subset Y$ and $W^Y$ is an open
neighborhood of $X$ in $Y$ called a \jdef{tubular neighborhood} of
$X$. In addition, if $i_X:X\rightarrow Y$ is the inclusion, the
relation $\zeta \circ 0_X = i_X$ holds. That such tubular maps and
neighborhoods exist is the content of the Tubular Neighborhood Theorem
(see \S 6.2
of~\cite{bo:cannasDaSilva-lectures_on_symplectic_geometry}).

\begin{proof}[Proof of
  Lemma~\ref{le:discretizations_yield_discrete_tangent_bundles-U_and_W}]
  The space $E:=\{(v,-v):v\in TQ\}$ is a vector bundle over the
  diagonal $\Delta_Q\subset Q\times Q$ via the map
  $(v,-v)\mapsto (\tau_Q(v),\tau_Q(v))$, where
  $\tau_Q:TQ\rightarrow Q$ is the tangent bundle. It is easy to check
  that $E$ is isomorphic, as a vector bundle over $\Delta_Q$, to the
  normal bundle associated to the embedding
  $\Delta_Q\subset Q\times Q$. So, as discussed above, there is a
  tubular map $\zeta:W^E\rightarrow W^{Q\times Q}$. By imposing
  $\zeta$ a convenient normalization and using
  Proposition~\ref{prop:prop_1_in_CP07} it is proved
  in~\cite{ar:cuell_patrick-geometric_discrete_analogues_of_tangent_bundles_and_constrained_lagrangian_systems}
  that the map
  $\varphi:(\del^{\pm})^{-1}(W^{Q\times Q})\rightarrow \R\times E$
  defined by
  \begin{equation*}
    \varphi(h,v):=
    \begin{cases}
      \left(h,\frac{1}{h}\zeta^{-1}(\del^{\pm}_h(v))\right),
      \stext{ if } h\neq 0,\\
      \left(0,(\frac{v}{2},-\frac{v}{2})\right), \stext{ if } h=0.
    \end{cases}
  \end{equation*}
  is smooth. As the derivative of $\varphi$ is nonsingular at each
  $(0,v)$, $\varphi$ is a local diffeomorphism at each one of these
  points; also, as $\varphi$ is a diffeomorphism from $\{0\}\times TQ$
  onto $\{0\}\times E$, we apply Theorem 1
  in~\cite{ar:cuell_patrick-skew_critical_problems} to conclude that
  $\varphi$ is a diffeomorphism from an open neighborhood
  $U\subset\R\times TQ$ of $\{0\}\times TQ$ onto an open neighborhood
  $V\subset \R\times E$ of $\{0\}\times E$. Notice that, in
  particular, $\varphi$ is defined in
  $(\del^{\pm})^{-1}(W^{Q\times Q})$ that is contained in the domain
  of $\del^\pm$ (that coincides with that of $\ti{\del^\pm}$); as
  $U\subset (\del^{\pm})^{-1}(W^{Q\times Q})$, the same thing happens
  to $U$.

  $\ti{W^E}:=\{(h,(v,-v))\in\R\times E:(hv,-hv)\in W^E\}$ is an open
  neighborhood of $\{0\}\times E$ in $\R\times E$. The map
  $\ti{\varphi}:\ti{W^E}\rightarrow \R\times Q\times Q$ defined by
  $\ti{\varphi}(h,(v,-v)):=(h,\zeta(h v, -h v))$ is smooth. In
  addition, since $(h,(v,-v))\mapsto (h,\zeta(v,-v))$ and, when
  $h\neq 0$, $(h,(v,-v)) \mapsto (h,(hv,-hv))$ are diffeomorphisms, we
  have that $\ti{\varphi}|_{\ti{W^E}\SM(\{0\}\times E)}$ is a
  diffeomorphism onto an open subset
  $\ti{W^{Q\times Q}}\subset \R\times Q\times Q$.
  
  We have that $\varphi^{-1}(\ti{W^E})$ is an open neighborhood of
  $\{0\}\times TQ$ in $(\del^{\pm})^{-1}(W^{Q\times Q})$. Also, for
  $(h,v)\in \varphi^{-1}(\ti{W^E})$,
  \begin{equation*}
    \ti{\varphi}(\varphi(h,v))=(h,\del^\pm_h(v)) =\ti{\del}^\pm(h,v)
  \end{equation*}

  Notice that $\ti{W^E}\cap V$ is an open neighborhood of
  $\{0\}\times E$ and $\varphi^{-1}(\ti{W^E}\cap V)$ is an open
  neighborhood of $\{0\}\times TQ$ contained in $U$. We have that
  $\varphi^{-1}(\{0\}\times E) = \{(h,v)\in
  (\del^{\pm})^{-1}(W^{Q\times Q}) : \varphi(h,v)\in \{0\}\times E\} =
  \{0\}\times TQ$ and $\varphi(\{0\}\times TQ) = \{0\}\times E$; thus,
  $\varphi$ establishes a bijection between $\{0\}\times TQ$ and
  $\{0\}\times E$.

  Let $\mathcal{U}:=\varphi^{-1}(\ti{W^E}\cap V)$; we have that
  $\mathcal{U}$ is an open neighborhood of $\{0\}\times TQ$ contained
  in $U\subset\R\times TQ$ (hence $\mathcal{U}$ is contained in the
  domain of $\ti{\del^\pm}$). Also, denoting the union of two
  \emph{disjoint} subsets $A$ and $B$ by $A\cupdot B$,
  \begin{equation*}
    \begin{split}
      \mathcal{U} =& \varphi^{-1}(\ti{W^E}\cap V) =
      \varphi^{-1}(((\ti{W^E}\cap V)\SM(\{0\}\times E)) \cupdot
      (\{0\}\times E)) \\=& \varphi^{-1}((\ti{W^E}\cap
      V)\SM(\{0\}\times E)) \cupdot \varphi^{-1}(\{0\}\times E) \\=&
      \underbrace{\varphi^{-1}((\ti{W^E}\cap V)\SM(\{0\}\times
        E))}_{\subset U} \cupdot (\{0\}\times TQ).
    \end{split}
  \end{equation*}
  As $\varphi|_U:U\rightarrow V$ is a diffeomorphism, it induces a
  diffeomorphism between the open subsets
  $\varphi^{-1}((\ti{W^E}\cap V)\SM(\{0\}\times E)))$ and
  $(\ti{W^E}\cap V)\SM(\{0\}\times E)$. Similarly, as
  $\ti{\varphi}|_{\ti{W^E}\SM (\{0\}\times E)}:\ti{W^E}\SM (\{0\}\times
  E)\rightarrow \ti{W^{Q\times Q}}$ is a diffeomorphism, it induces a
  diffeomorphism between the open subsets
  $(\ti{W^E}\cap V)\SM(\{0\}\times E)$ and
  $\ti{\varphi}((\ti{W^E}\cap V)\SM(\{0\}\times E))$.

  Letting
  $\mathcal{W}:=\ti{\varphi}((\ti{W^E}\cap V)\SM(\{0\}\times E))
  \cupdot (\{0\}\times \Delta_Q) \subset \R\times Q\times Q$, we have
  that
  $\mathcal{W} \SM (\{0\}\times \Delta_Q) = \ti{\varphi}((\ti{W^E}\cap
  V)\SM(\{0\}\times E))$ that is open and diffeomorphic to
  $\varphi^{-1}((\ti{W^E}\cap V)\SM(\{0\}\times E))) = \mathcal{U}
  \SM(\{0\}\times TQ)$ under $\ti{\del^\pm}$.

  Finally, we have
  \begin{equation*}
    \begin{split}
      \ti{\del^{\pm}}(\mathcal{U}) =&
      \ti{\del^{\pm}}((\mathcal{U}\SM(\{0\}\times TQ)) \cupdot
      (\{0\}\times TQ)) \\=& \ti{\del^{\pm}}((\mathcal{U}\SM(\{0\}\times
      TQ))) \cupdot \ti{\del^{\pm}}(\{0\}\times TQ) \\=& (\mathcal{W} \SM
      (\{0\}\times \Delta_Q)) \cupdot (\{0\}\times \Delta_Q) = \mathcal{W},
    \end{split}
  \end{equation*}
  thus proving that the sets $\mathcal{U}$ and $\mathcal{W}$
  constructed above have the required properties.
\end{proof}

\begin{lemma}
  \label{le:discretizations_yield_discrete_tangent_bundles-W_and_diagonal}
  In the context of
  Lemma~\ref{le:discretizations_yield_discrete_tangent_bundles-U_and_W},
  given a relatively compact open subset $Q_0\subset Q$, there is
  $a>0$ such that $(-a,a)\times \Delta_{Q_0} \subset \mathcal{W}$.
\end{lemma}

\begin{proof}
  Given $q\in Q$ and $h\in\R$, as $(h,(0_q,-0_q))\in \ti{W^E}$ (see
  the proof of
  Lemma~\ref{le:discretizations_yield_discrete_tangent_bundles-U_and_W})
  and the tubular map $\zeta$ satisfies
  $\zeta\circ\, 0_{\Delta_Q} = i_{\Delta_Q}$, we have that
  $\ti{\varphi}(h,(0_q,-0_q)) = (h,\zeta(h 0_q, -h 0_q)) =
  (h,\zeta(0_q,-0_q)) = (h,(q,q))$, hence $(h,(q,q))$ is in the
  image of $\ti{\varphi}$.

  On the other hand, $\varphi$ is a diffeomorphism from an open subset
  $U\subset \R\times TQ$ containing $\{0\}\times TQ$ onto an open
  subset $V\subset\R\times E$ containing $\{0\}\times E$ (see
  constructions in the proof of
  Lemma~\ref{le:discretizations_yield_discrete_tangent_bundles-U_and_W}),
  and we have $(0,(0_q,-0_q))\in V$, that is open, then there are
  $a_q>0$, $(V_q,\phi_q)$ a coordinate chart of $Q$ and
  $\ti{V}_q\subset \R^{\dim(Q)}$ open such that $q\in V_q$,
  $0\in \ti{V}_q$ and
  $(-a_q,a_q)\times T\phi_q^{-1}(\phi_q(V_q)\times \ti{V}_q)\subset
  V$. In particular, for $\abs{h}<a_q$ and $q'\in V_q$, we have that
  $(h,(0_{q'},-0_{q'}))\in V=\varphi(U)$; since we also have that
  $(h,(0_{q'},-0_{q'}))\in \ti{W^E}$, we conclude that
  $(h,(0_{q'},-0_{q'}))\in \ti{W^E}\cap V =\varphi(\mathcal{U})$, so
  that
  $(h,(q',q')) = \ti{\varphi}(h,(0_{q'},-0_{q'})) \in
  \ti{\varphi}(\varphi(\mathcal{U})) = \mathcal{W}$. The family
  $\{V_q:q\in \conj{Q_0}\}$ is an open cover of $\conj{Q_0}$, that is
  compact, so that there are finitely many $V_{q_1},\ldots, V_{q_N}$
  that still cover $\conj{Q_0}$; define
  $a:=\min\{a_{q_1},\ldots,a_{q_N}\}>0$. If $h\in\R$ with $\abs{h}<a$
  and $q\in Q_0$, we have that $q\in V_{q_k}$ for some $k=1,\ldots, N$
  and $\abs{h}<a\leq a_{q_k}$. Then, as we noticed above,
  $(h,(q,q)) \in \mathcal{W}$, proving that
  $(-a,a)\times \Delta_{Q_0}\subset \mathcal{W}$.
\end{proof}

\begin{proof}[Proof of Theorem~\ref{th:discretizations_yield_discrete_tangent_bundles}]
  Let $\mathcal{U}\subset\R\times TQ$ be the open neighborhood of
  $\{0\}\times TQ$ constructed in
  Lemma~\ref{le:discretizations_yield_discrete_tangent_bundles-U_and_W}
  and $a>0$ as in the statement of
  Lemma~\ref{le:discretizations_yield_discrete_tangent_bundles-W_and_diagonal}. As
  $Q_0\subset Q$ is open,
  $(\ti{\del^\pm})^{-1}(\R\times Q_0\times Q_0)$ is open in
  $\R\times TQ$ and contains $\{0\}\times TQ_0$; we have that
  $\mathcal{U}\cap (\ti{\del^\pm})^{-1}(\R\times Q_0\times Q_0)$ is an
  open neighborhood of $\{0\}\times TQ_0$ in $\R\times TQ$. Let $h$ be
  such that $0\neq \abs{h}<a$; as $p_1:\R\times TQ\rightarrow \R$ is a
  surjective submersion, $S_h:=p_1^{-1}(\{h\})$ is an embedded
  submanifold of $\R\times TQ$. We have that
  $S_h\cap \mathcal{U} \cap (\ti{\del^\pm})^{-1}(\R\times Q_0\times
  Q_0)$ is an open subset of $S_h$. Also, since
  $p_2:\R\times TQ\rightarrow TQ$ restricted to $S_h$ is a
  diffeomorphism with $TQ$,
  \begin{equation*}
    \mathcal{V}_h:= p_2(S_h\cap \mathcal{U} \cap
    (\ti{\del^\pm})^{-1}(\R\times Q_0\times Q_0)) \subset TQ
  \end{equation*}
  is an open subset. As $\ti{\del^\pm}$ is a diffeomorphism from
  $\mathcal{U}\SM(\{0\}\times TQ)$ onto
  $\mathcal{W}\SM(\{0\}\times Q\times Q)$, it follows that
  $\del^\pm_h$ is a diffeomorphism from $\mathcal{V}_h$ onto
  \begin{equation*}
    \mathcal{W}_h:= \del^\pm_h(\mathcal{V}_h) =
    p_{23}(p_1^{-1}(\{h\}) \cap \mathcal{W} \cap p_{23}^{-1}(Q_0\times Q_0)).
  \end{equation*}
  It is easy to conclude from the previous arguments that
  $(\mathcal{V}_h,\del^+_h,\del^-_h)$ is a discrete tangent bundle
  over $Q_0$. By
  Lemma~\ref{le:discretizations_yield_discrete_tangent_bundles-W_and_diagonal},
  $\{h\}\times \Delta_Q\subset \mathcal{W}$, so that
  $\{h\}\times \Delta_{Q_0}\subset p_1^{-1}(\{h\}) \cap \mathcal{W}
  \cap p_{23}^{-1}(Q_0\times Q_0)$ and, then,
  $\Delta_{Q_0}\subset \mathcal{W}_h$.

  Last, for a fixed $v_q\in TQ_0$ we have that
  $(0,v_q)\in \{0\}\times TQ\subset \mathcal{U}$ and, as $\mathcal{U}$
  is open in $\R\times TQ$, there is $a'>0$ such that if $\abs{h}<a'$,
  then $(h,v_q)\in \mathcal{U}$. Also, as
  $\ti{\del^\pm}(0,v_q)=(0,q,q)\in \R\times Q_0\times Q_0$ that is
  open in $\R\times Q\times Q$ we have that
  $(\ti{\del^\pm})^{-1}(\R\times Q_0\times Q_0)$ is an open
  neighborhood of $(0,v_q)$ in $\R\times TQ$ and, so, there is $a''>0$
  such that if $\abs{h}<a''$, then
  $\ti{\del^\pm}(h,v_q)\in \R\times Q_0\times Q_0$. Thus, if
  $0\neq \abs{h} <\min\{a,a',a''\}$, we have that
  $v_q\in \mathcal{V}_h$.
\end{proof}

\begin{remark}\label{rem:problem_with_CP09_prop1}
  In the proof of
  Theorem~\ref{th:discretizations_yield_discrete_tangent_bundles}
  given in~\cite[Proposition
  1]{ar:cuell_patrick-geometric_discrete_analogues_of_tangent_bundles_and_constrained_lagrangian_systems},
  after establishing what we called
  Lemma~\ref{le:discretizations_yield_discrete_tangent_bundles-U_and_W},
  they claim (but don't prove) at the top of page 981 that there is a
  set $\mathcal{W}$ that, in addition to satisfying the conditions
  stated in
  Lemma~\ref{le:discretizations_yield_discrete_tangent_bundles-U_and_W},
  is open. It appears that this last condition is used in the
  remaining part of their proof. Unfortunately, there are cases where
  no such \emph{open} $\mathcal{W}$ exists (see
  Example~\ref{ex:W_is_a_wedge}). This is the reason for basing our
  proof of
  Theorem~\ref{th:discretizations_yield_discrete_tangent_bundles} on
  Lemma~\ref{le:discretizations_yield_discrete_tangent_bundles-W_and_diagonal},
  that does not require $\mathcal{W}$ to be open.
\end{remark}

\begin{example}\label{ex:W_is_a_wedge}
  Here we introduce a very concrete example of a discretization and,
  next, discuss the possible subsets $\mathcal{U}$ and $\mathcal{W}$
  that arise from
  Lemma~\ref{le:discretizations_yield_discrete_tangent_bundles-U_and_W}
  in this case.  Consider the setup of
  Example~\ref{ex:linear_discretization} with $n=1$ and
  $U:=\{(h,t,(q,v)):h,t,q,v\in\R \text{ with }
  \sqrt{\abs{h}}\abs{v}<1\}$. We have that $U\subset \R^2\times T\R$
  is open and $\{0\}\times\{0\}\times T\R\subset U$.  Then,
  $\ti{\del^{\pm}}:\{(h,(q,v))\in\R^3:\sqrt{\abs{h}}\abs{v}<1\}\rightarrow
  \R\times \R^2$ is $\ti{\del^\pm}(h,(q,v)) := (h,\del^\pm_h(q,v))$
  with $\del^{\pm}_h(q,v):=(q+hv,q)$. For a fixed $h\neq 0$ we have
  that
  $\del^{\pm}_h(\R\times \{v\in\R:\abs{v}<\frac{1}{\sqrt{\abs{h}}}\})
  = \{(q^+,q^-)\in\R^2:\abs{q^+-q^-}<\sqrt{\abs{h}}\}$, while, if
  $h=0$,
  $\del^{\pm}_h(\R\times\R) = \Delta_\R =
  \{(q^+,q^-)\in\R^2:q^+=q^-\}$. Thus, $\ti{\del^\pm}$ maps
  $\{(h,(q,v))\in\R^3:\sqrt{\abs{h}}\abs{v}<1\}$ onto
  $(\{0\}\times \Delta_{\R}) \cupdot \{(h,(q^+,q^-)):h\neq 0 \text{ and }
  \abs{q^+-q^-}<\sqrt{\abs{h}} \}$, being a diffeomorphism for
  $h\neq 0$.
  
  Say that $\mathcal{U}\subset \R\times T\R$ and
  $\mathcal{W}\subset \R\times (\R\times\R)$ are subsets related as in
  the statement of
  Lemma~\ref{le:discretizations_yield_discrete_tangent_bundles-U_and_W}
  applied to the discretization of the previous paragraph and assume
  that $\mathcal{W}$ is open. In particular, as $\mathcal{U}$ is
  contained in the domain of $\ti{\del^\pm}$, we have
  $\mathcal{U}\subset
  \{(h,(q,v))\in\R^3:\sqrt{\abs{h}}\abs{v}<1\}$. Also, we know that
  $\{0\}\times \Delta_\R\subset \mathcal{W}$ and
  $\mathcal{W}=\ti{\del^\pm}(\mathcal{U}) \subset (\{0\}\times
  \Delta_{\R}) \cupdot \{(h,(q^+,q^-)):h\neq 0 \text{ and }
  \abs{q^+-q^-}<\sqrt{\abs{h}} \}$. Hence,
  $\mathcal{W}\SM(\{0\}\times \R^2) \subset \{(h,(q^+,q^-)):h\neq 0
  \text{ and } \abs{q^+-q^-}<\sqrt{\abs{h}} \}$. Then, for any
  $q\in \R$, as
  $(0,(q,q))\in \{0\}\times \Delta_\R\subset \mathcal{W}$ that is
  open, there exists $\epsilon\in (0,\frac{1}{2})$ such that
  $(-\epsilon,\epsilon)\times (q-\epsilon,q+\epsilon)\times
  (q-\epsilon,q+\epsilon)\subset \mathcal{W}$. Thus
  \begin{equation*}
    \begin{split}
      ((-\epsilon,\epsilon)\SM\{0\})\times
      (q-\epsilon,q+\epsilon)\times& (q-\epsilon,q+\epsilon) \subset\\&
      \{(h,(q^+,q^-)):h\neq 0 \text{ and }
      \abs{q^+-q^-}<\sqrt{\abs{h}} \}.
    \end{split}
  \end{equation*}
  The point $(\epsilon^2,(q-\frac{\epsilon}{2},q+\frac{\epsilon}{2}))$
  is in the set on the left side but, as
  $\abs{(q-\frac{\epsilon}{2}) - (q+\frac{\epsilon}{2})} = \epsilon =
  \sqrt{\epsilon^2}$, that point is not in the set on the right
  side. This shows that, in this example, no set $\mathcal{W}$
  compatible with
  Lemma~\ref{le:discretizations_yield_discrete_tangent_bundles-U_and_W}
  can be open.
\end{example}


\subsection{Discretizations of a forced mechanical system}
\label{sec:discretizations_of_a_forced_mechanical_system}

In this section we expand the notion of discretization of $TQ$ given
in section~\ref{sec:discretizations_of_TQ} to include the dynamical
parts of a forced mechanical system. This notion coincides with the
notion of discretization of a mechanical system given
in~\cite{ar:patrick_cuell-error_analysis_of_variational_integrators_of_unconstrained_lagrangian_systems},
when the force vanishes.

\begin{definition}\label{def:discretization_CP}
  Let $(Q,L,f)$ be a \FMS. A \jdef{discretization} of $(Q,L,f)$ is a
  quadruple $(Q,\psi,L_\CPDS,f_\CPDS)$\footnote{The sub-index $\CPS$
    that appears, for instance, in $L_\CPS$ comes from Cuell and
    Patrick who formulated a discrete mechanics based on $TQ$ in their
    work.} such that $\psi$ is a discretization of $TQ$,
  $L_\CPDS:\R\times TQ\rightarrow \R$ is smooth and satisfies
  $L_\CPHS = h L + \mathcal{O}(h^2)$ and, for
  $p_2:\R\times TQ\rightarrow TQ$ the projection onto the second
  variable, $f_\CPDS$ is a section of
  $p_2^*(T^*TQ)\rightarrow \R\times TQ$\footnote{$f_\CPDS$ is just an
    element of $\mathcal{A}^1(\R\times TQ)$ that satisfies
    $f_\CPHS(v)(\delta h) = 0$ for all
    $\delta h\in T_h\R\subset T_h R\oplus T_v TQ \simeq
    T_{(h,v)}(\R\times TQ)$.}  that satisfies
  $f_\CPHS = h f + \mathcal{O}(h^2)$. In this context, we consider
  $\R\times TQ$ with the grading $\grade_{\R\times TQ}(h,v):=h$.
\end{definition}

\begin{example}\label{ex:linear_discretization_with_FMS}
  Let $(\R^n,L,f)$ be a \FMS on $\R^n$. Define $L_\CPHS(v):=h L(v)$ and 
  $f_\CPHS(v):=h f(v)$, and let $(\psi,\dBTp,\dBTm)$ be the smooth
  discretization of $\R^n$ introduced in
  Example~\ref{ex:linear_discretization}. Then
  $(\R^n,\psi,L_\CPDS,f_\CPDS)$ is a discretization of
  $(\R^n,L,f)$.
\end{example}

\begin{remark}\label{rem:local_Ld_or_f}
  In some applications the discrete Lagrangian or the discrete force
  of a discretization of a \FMS may not be defined on $\R\times TQ$
  but, rather, on some open neighborhood $V_{L,f}\subset \R\times TQ$
  of $\{0\}\times TQ$ (see, for instance,
  Example~\ref{ex:exact_discretization_forced-system} below). In this case,
  $U^+$ and $U^-$, the domains of $\del^+$ and $\del^-$, are replaced
  by $U^+\cap V_{L,f}$ and $U^-\cap V_{L,f}$.
\end{remark}

\begin{example}\label{ex:exact_discretization_forced-system}
  Let $(Q,L,f)$ be a regular \FMS. Consider an exact discretization
  $(\psi^E,\dBTp,\dBTm)$ of $TQ$ as introduced in
  Example~\ref{ex:exact_discretization_forced} associated to
  $(Q,L,f)$. Define $L^E_\CPDS:U^{X_{L,f}}\rightarrow \R$ and
  $f^E_\CPDS:U^{X_{L,f}} \rightarrow T^*TQ$ as
  \begin{gather*}
    L^E_\CPHS(v) := \int_{\dBTm(h)}^{\dBTp(h)} (L\circ F^{X_{L,f}}_t)(v) dt,\\
    f^E_\CPHS(v)(\delta v) := \int_{\dBTm(h)}^{\dBTp(h)}
    f(F^{X_{L,f}}_t(v))(T_vF^{X_{L,f}}_t(\delta v)) dt.
  \end{gather*}
  Then, $(Q,\psi^E,L^E_\CPDS,f^E_\CPDS)$ is a discretization
  of $(Q,L,f)$ known as an \jdef{exact discretization}. Indeed, using
  a local description
  (Lemma~\ref{le:order_of_maps_in_terms_of_local_expression}),
  Theorem~\ref{th:taylor-r_minus_1} and the facts that
  $L^E_\CPZS(v) = 0$ and
  \begin{equation*}
    \begin{split}
      \pd{}{h}\bigg|_{h=0} L^E_\CPHS(v) =& (L\circ
      F^{X_{L,f}}_{\dBTp(0)})(v)(\dBTp)'(0) - (L\circ
      F^{X_{L,f}}_{\dBTm(0)})(v)(\dBTm)'(0) \\=& L(v)
      ((\dBTp)'(0)-(\dBTm)'(0)) = L(v),
    \end{split}
  \end{equation*}
  we have that $L^E_\CPHS(v) = h L(v) +\mathcal{O}(h^2)$. Similarly,
  $f^E_\CPZS(v)(\delta v)=0$ and
  \begin{equation*}
    \begin{split}
      \pd{}{h}\bigg|_{h=0}f^E_\CPHS(v)(\delta v) =&
      f(F^{X_{L,f}}_{\dBTp(0)}v)(T_vF^{X_{L,f}}_{\dBTp(0)}(\delta
      v))(\dBTp)'(0) \\& -
      f(F^{X_{L,f}}_{\dBTm(0)}v)(T_vF^{X_{L,f}}_{\dBTm(0)}(\delta
      v))(\dBTm)'(0) \\=& f(v)(\delta v)
      ((\dBTp)'(0)-(\dBTm)'(0)) = f(v)(\delta v),
    \end{split}
  \end{equation*}
  lead to
  $f^E_\CPHS(v)(\delta v) = h f(v)(\delta v) + \mathcal{O}(h^2)$.
\end{example}

\begin{remark}
  If $(Q,L,f)$ is a regular \FMS and $\check{f}:TQ\rightarrow T^*Q$ is
  the corresponding force field (see Remark~\ref{rem:force_field}),
  the discrete exact force can be expressed as
  \begin{equation*}
    f^E_\CPHS(v)(\delta v) := \int_{\dBTm(h)}^{\dBTp(h)}
    \check{f}(F^{X_{L,f}}_t(v))(T_v(\tau_Q\circ F^{X_{L,f}}_t)(\delta v)) dt.
  \end{equation*}
\end{remark}

\begin{example}\label{ex:particle_with_friction-exact_discretization_system}
  Now we specialize the construction given in
  Example~\ref{ex:exact_discretization_forced-system} of an exact
  discretization of a \FMS to the case of the system introduced in
  Example~\ref{ex:particle_with_friction-def}. We use the exact
  discretization $\psi^E$ and the maps $\dBTp,\dBTm$ from
  Example~\ref{ex:particle_with_friction-exact_discretization}. Some
  straightforward calculations lead to
  \begin{equation*}
    \begin{split}
      L^E_\CPHS(q,v) =& \frac{1}{4\alpha} v (1-e^{-2\alpha h}),\\
      f^E_\CPHS(q,v) =& -v \big( (1-e^{-\alpha h}) dq +
      \frac{1}{\alpha}\big( (1-e^{-\alpha h}) - \frac{1}{2}
      (1-e^{-2\alpha h})\big) dv\big).
    \end{split}
  \end{equation*}
\end{example}

Just as in the case of discretizations of $TQ$, a discretization of a
\FMS consists of a family of objects, parametrized by $h$. In what
follows we describe a type of dynamical system on $TQ$ that models the
data of a discretization of a \FMS at a fixed value of $h$.

\begin{definition}\label{def:forced_discrete_mechanicalCP_system}
  A \jdef{forced discrete mechanical system} (\FDMSCP) consists of a
  quadruple $(Q,(\mathcal{V},\del^+,\del^-),L_\CPS,f_\CPS)$, where $Q$
  is a differentiable manifold, $(\mathcal{V},\del^+,\del^-)$ is a
  discrete tangent bundle of $Q$, $L_\CPS:\mathcal{V}\rightarrow \R$
  is a smooth function and $f_\CPS$ is a smooth $1$-form over
  $\mathcal{V}$\footnote{More explicitly, the discrete force $f_\CPS$
    is a smooth section of the cotangent bundle
    $T^*\mathcal{V}\rightarrow \mathcal{V}$.}. For simplicity, we
  denote a \FDMSCP by $(Q,\mathcal{V},L_\CPS,f_\CPS)$.
\end{definition}

Below we introduce the notion of trajectory of a \FDMSCP, making it a
discrete-time dynamical system. We define the \jdef{space of first
  order discrete paths} on $\mathcal{V}$ of length $N$ as
\begin{equation}\label{eq:space_of_N-paths_CP-def}
  \mathcal{C}_N:=\{v_\cdot\in \mathcal{V}^N:\del^+(v_k) =
    \del^-(v_{k+1}) \text{ for } k=1,\ldots, N-1\}.
\end{equation}
By definition, $\del^+$, $\del^-$ are submersions and $\del^\pm$ is a
local diffeomorphism (see
Remark~\ref{rem:discrete_tangent_bundle_borderPM_is_diffeo}); it is
easy to see that, then, $\mathcal{C}_N\subset \mathcal{V}^N$ is an
embedded submanifold.  An \jdef{infinitesimal variation} over the path
$v_\cdot\in \mathcal{C}_N$ is an element
$\delta v_\cdot \in T_{v_\cdot}\mathcal{C}_N$.  Notice that the
tangency condition is equivalent to
$\delta v_k \in T_{v_k}\mathcal{V}$ for $k=1,\ldots, N$ and
$(T_{v_k}\del^+)(\delta v_k) = (T_{v_{k+1}}\del^-)(\delta v_{k+1})$
for $k=1,\ldots,N-1$. An infinitesimal variation $\delta v_\cdot$ over
$v_\cdot\in \mathcal{C}_N$ is said to have \jdef{fixed endpoints} if
$(T_{v_1}\del^-)(\delta v_1) =0$ and $(T_{v_N}\del^+)(\delta v_N) =0$;
notice that it is not true that a fixed endpoint infinitesimal
variation must satisfy $\delta v_1=0$ and $\delta v_N=0$.

\begin{definition}
  \label{def:forced_discrete_mechanicalCP_systems-variational_pple}
  Let $(Q,\mathcal{V},L_\CPS,f_\CPS)$ be an \FDMSCP. A discrete path
  $v_\cdot\in \mathcal{C}_N$ is a \jdef{trajectory} of the system if
  it satisfies the critical condition
  \begin{equation*}
    d(\sum_{k=1}^NL_\CPS\circ p_k)(v_\cdot)(\delta v_\cdot) + 
    \sum_{k=1}^N f_\CPS(v_k)(\delta v_k) = 0
  \end{equation*}
  for all infinitesimal variations $\delta v_\cdot$ over $v_\cdot$
  with fixed endpoints.
\end{definition}

\begin{remark}
  An \FDMSCP with vanishing force is exactly the same thing as a
  discrete Lagrangian system considered by Cuell \& Patrick in Section
  3
  of~\cite{ar:cuell_patrick-geometric_discrete_analogues_of_tangent_bundles_and_constrained_lagrangian_systems}.
\end{remark}

\begin{remark}\label{rem:forced_variational_pple_as_skew_critical}
  The trajectories of a \FDMSCP $(Q,\mathcal{V},L_\CPS,f_\CPS)$ can be
  seen as critical points of a critical problem. Indeed, with the
  notation used in Section~\ref{sec:critical_problems}, let
  $g:M\rightarrow N$ be defined by
  \begin{equation*}
    g_{\mathcal{C}_N}:\mathcal{C}_N\rightarrow Q\times Q \stext{ where }
    g_{\mathcal{C}_N}(v_\cdot) := (\del^-(v_1),\del^+(v_N)).
  \end{equation*}
  Also, define the distribution $\mathcal{D}$ on $\mathcal{C}_N$ such
  that, for $v_\cdot\in \mathcal{C}_N$, 
  \begin{equation*}
    \begin{split}
      \mathcal{D}_{v_\cdot} := \ker(Tg_{\mathcal{C}_N}(v_\cdot)) =&\{
      \delta v_\cdot \in T_{v_\cdot}\mathcal{C}_N :
      (T_{v_1}\del^-)(\delta v_1) = 0 \text{ and }
      (T_{v_N}\del^+)(\delta v_N) = 0\} \\=&
      T_{v_\cdot}(g_{\mathcal{C}_N}^{-1}(g_{\mathcal{C}_N}(v_\cdot))).
    \end{split}
  \end{equation*}
  Last, we define $\dAFop\in\mathcal{A}^1(\mathcal{V})$ and
  $\dAFnp\in \mathcal{A}^1(\mathcal{C}_N)$ by
  \begin{equation*}
    \dAFop := dL_\CPS + f_\CPS \stext{ and }
    \dAFnp := i_{\mathcal{C}_N,\mathcal{V}^N}^*
    \bigg(\underbrace{\sum_{k=1}^N p_k^*(\dAFop)}_{\in \mathcal{A}^1(\mathcal{V}^N)} 
    \bigg),
  \end{equation*}
  where
  $i_{\mathcal{C}_N,\mathcal{V}^N}:\mathcal{C}_N\rightarrow
  \mathcal{V}^N$ is the inclusion map and
  $p_k:\mathcal{V}^N \rightarrow \mathcal{V}$ is the projection onto
  the $k$-th factor.

  Then $(\dAFnp,g_{\mathcal{C}_N},\mathcal{D})$ is a critical problem
  and its critical points are the length $N$ trajectories of the
  \FDMSCP $(Q,\mathcal{V},L_\CPS,f_\CPS)$. As a consequence, Lemma 1
  in~\cite{ar:cuell_patrick-skew_critical_problems} can be used to
  prove, in certain cases, the existence of trajectories of
  $(Q,\mathcal{V},L_\CPS,f_\CPS)$.

  In the context of
  Remark~\ref{rem:critical_condition_for_D=Tg_crit_problems}, the
  criticality condition at $(q_0,q_1)\in Q\times Q$ for points
  $v_\cdot\in \mathcal{C}_N$ is that
  $v_\cdot\in \mathcal{Z}_{q_0,q_1}$ and that
  $i_{\mathcal{Z}_{q_0,q_1}}^*(\dAFnp)(v_\cdot)=0$, for
  $\mathcal{Z}_{q_0,q_1}:=\{v_\cdot\in \mathcal{C}_N: \del^-(v_1)=q_0
  \text{ and } \del^+(v_N)=q_1\}$.
\end{remark}


\subsection{Flows of discretizations of \FMSs}
\label{sec:flows_of_discretizations_of_FMS}

It can be seen that an arbitrary \FDMSCP (even with vanishing forces)
does not need to have trajectories. In this section we want to prove
that, under certain conditions, the \FDMSCPs arising from the
discretization of a regular \FMS for a fixed $h$ do have trajectories,
at least, for values of $h$ sufficiently small.

Before discussing the main results, we introduce some of the basic
objects that are required.

Let $(Q,(\psi,\del^+,\del^-),L_\CPDS,f_\CPDS)$ be a
discretization of the regular \FMS $(Q,L,f)$. Then, as $\psi$ is a
discretization of $TQ$, we have that
$\del^+_0(v_q) = q = \del^-_0(v_q)$, so there is an open neighborhood
$A\subset \R\times TQ$ of $\{0\}\times TQ$ on which $\del^+_h$ and
$\del^-_h$ are submersions. Define
\begin{equation}\label{eq:submanifold_of_matching_arrows_after_blow_up}
  \begin{gathered}
    \mathcal{C}:=\{(h,v,\ti{v}):(h,v),(h,\ti{v})\in A \text{ and }
    \del^+_h(v) = \del^-_h(\ti{v})\} \subset \R\times TQ\times TQ,\\
    \mathcal{C}^*:=\{(h,v,\ti{v})\in \mathcal{C}:h\neq 0\} \stext{ and
    } \mathcal{C}_0 := (p_1|_\mathcal{C})^{-1}(0).
  \end{gathered}
\end{equation}
We also define
\begin{equation}\label{eq:g_cp-def}
  g:\mathcal{C}\rightarrow \R\times Q\times Q \stext{ by }
  g(h,v,\ti{v}):=(h,\del_h^-(v),\del_h^+(\ti{v})).
\end{equation}
Last, we introduce $\dAFop_{\cdot}$ as a section of
$p_2^*(T^*TQ)\rightarrow \R\times TQ$ by
\begin{equation*}
  \dAFop_{\cdot} := d_2 L_\CPDS + f_\CPDS
\end{equation*}
where $d_2$ was defined in
Section~\ref{sec:remarks_on_cartesian_products} and then
$\dAFtpst\in\mathcal{A}^1(\mathcal{C}^*)$ by
\begin{equation}\label{eq:alpha_cp-def}
  \dAFtpst :=
  i_{\mathcal{C}^*,\R\times TQ\times TQ}^*(p_{12}^*(\dAFop) + 
  p_{13}^*(\dAFop)).
\end{equation}

\begin{lemma}\label{le:properties_of_C}
  With the definitions as above and $h_0\neq 0$ fixed, the following
  assertions are true.
  \begin{enumerate}
  \item \label{it:properties_of_C-submanifold}
    $\mathcal{C}\subset \R\times TQ\times TQ$ is an embedded
    submanifold.
  \item
    \label{it:properties_of_C-C_h_submanifold}
    $\mathcal{C}_{h_0} := \{(h_0,v,\ti{v})\in \mathcal{C}\}\subset
    \mathcal{C}$ is an embedded submanifold.
  \item \label{it:properties_of_C-regular_values} Let
    $g_{h_0}:\mathcal{C}_{h_0}\rightarrow \R\times Q\times Q$ be
    defined by $g_{h_0} := g\circ i_{\mathcal{C}_{h_0}}$. Then,
    $g|_{\mathcal{C}^*}:\mathcal{C}^*\rightarrow \R\times Q\times Q$
    and $g_{h_0}:\mathcal{C}_{h_0}\rightarrow \{h_0\}\times Q\times Q$
    are submersions.
  \end{enumerate}
\end{lemma}

\begin{proof}
  Let
  $F^\pm(h,v,\ti{v}):=(h,\del^+_h(v),\del^-_h(\ti{v})) \in
  \R\times Q\times Q$,
  that is defined and smooth in an open subset
  $U\subset \R\times TQ\times TQ$ containing $\{0\}\times TQ\times TQ$.
  Being $p_{12},p_{13}:\R\times TQ\times TQ\rightarrow \R\times TQ$
  continuous, we have that
  $U':=U\cap p_{12}^{-1}(A) \cap p_{13}^{-1}(A)\subset \R\times
  TQ\times TQ$
  is open; notice that $U'\supset \{0\}\times TQ\times TQ$. Using the
  Cartesian product structure of the respective spaces, for
  $(h,v,\ti{v}) \in U$ we can write
  \begin{equation*}
    T_{(h,v,\ti{v})}F^\pm = \left(
      \begin{array}{ccc}
        1 & 0 & 0\\
        T^1_{(h,v)}\del^+_h(v) & T^2_{(h,v)}\del^+_h(v) & 0\\
        T^1_{(h,\ti{v})}\del^-_h(\ti{v}) & 0 & T^2_{(h,\ti{v})} \del^-_h(\ti{v})
      \end{array}
    \right),
  \end{equation*}
  where, $T_{(h,v)}^jF$ was defined in
  Section~\ref{sec:remarks_on_cartesian_products}. It is easy to check
  that, for $(h,v,\ti{v}) \in U'$, $T_{(h,v,\ti{v})}F^\pm$ is onto, so
  that $F^\pm|_{U'}$ is a submersion into $\R\times Q\times Q$. Hence,
  as $\R\times \Delta_Q \subset \R\times Q\times Q$ is an embedded
  submanifold, $\mathcal{C} = (F^\pm|_{U'})^{-1}(\R\times \Delta_Q)$
  is an embedded submanifold of $U'$ (hence of
  $\R\times TQ\times TQ$), proving
  point~\eqref{it:properties_of_C-submanifold}.  In addition, for
  $(h,v,\ti{v})\in \mathcal{C}$,
  \begin{equation}\label{eq:tangent_to_C}
    \begin{split}
      T_{(h,v,\ti{v})}\mathcal{C} =&
      (T_{(h,v,\ti{v})}F^\pm)^{-1}(T_{F^\pm(h,v,\ti{v})}(\R\times
      \Delta_Q)) \\=& \{(\delta h, \delta v, \delta \ti{v}) \in
      T_{(h,v,\ti{v})}(\R\times TQ\times TQ) :\\& \phantom{\{(\delta
        h, \delta v, \delta \ti{v})} T^1_{(h,v)}\del^+_h(v)(\delta h)
      + T^2_{(h,v)}\del^+_h(v)(\delta v) \\& \phantom{\{(\delta h,
        \delta v, \delta \ti{v})} =
      T^1_{(h,\ti{v})}\del^-_h(\ti{v})(\delta h) + T^2_{(h,\ti{v})}
      \del^-_h(\ti{v})(\delta \ti{v})\}
    \end{split}
  \end{equation}
    
  Now, consider $p_1|_\mathcal{C}$, where
  $p_1:\R\times TQ\times TQ\rightarrow \R$ is the projection. For
  $(h,v,\ti{v})\in \mathcal{C}$, using the characterization of
  $T_{(h,v,\ti{v})}\mathcal{C}$ provided by~\eqref{eq:tangent_to_C}
  and recalling that $T^2_{(h,v)}\del^+_h(v)$ and
  $T^2_{(h,\ti{v})}\del^-_h(\ti{v})$ are onto, it is easy to see
  that $T_{(h,v,\ti{v})}(p_1|_\mathcal{C})$ is onto, so that
  $p_1|_\mathcal{C}$ is a submersion and
  $\mathcal{C}_{h_0} = (p_1|_\mathcal{C})^{-1}(\{h_0\})$ is an
  embedded submanifold of $\mathcal{C}$. Hence,
  point~\eqref{it:properties_of_C-C_h_submanifold} is valid.

  Let $\ti{g}:U\rightarrow \R\times Q\times Q$ be defined by
  $\ti{g}(h,v,\ti{v}) := (h, \del^-_h(v),\del^+_h(\ti{v}))$;
  in addition, let $g:=\ti{g}|_{\mathcal{C}}$.  Clearly $\ti{g}$ and
  $g$ are smooth. We want to prove that $g$ is a submersion if
  $h\neq 0$.  Using the Cartesian product structure of
  $\R\times TQ\times TQ$ and $\R\times Q\times Q$, we see that, for
  $(h,v,\ti{v})\in U$,
  \begin{equation*}
    T_{(h,v,\ti{v})}\ti{g} = \left(
      \begin{array}{ccc}
        1 & 0 & 0\\
        T^1_{(h,v)}\del^-_h(v) & T^2_{(h,v)}\del^-_h(v) & 0\\
        T^1_{(h,\ti{v})}\del^+_h(\ti{v}) & 0 & T^2_{(h,\ti{v})} \del^+_h(\ti{v})
      \end{array}
    \right)
  \end{equation*}
  and, if $(h,v,\ti{v})\in \mathcal{C}$,
  $T_{(h,v,\ti{v})}g =
  (T_{(h,v,\ti{v})}\ti{g})|_{T_{(h,v,\ti{v})}\mathcal{C}}$.  For
  $(h,v,\ti{v})\in \mathcal{C}$, let $q_0:=\del^-_h(v)$ and
  $q_2:=\del^+_h(\ti{v})$. Then, for any
  $(\delta h, \delta q_0, \delta q_2)\in T_{(h,q_0,q_2)}(\R\times
  Q\times Q)$, choose $\delta v\in T_vTQ$ so that
  $T^2_{(h,v)}\del^-_h(v)(\delta v) = \delta q_0 -
  T^1_{(h,v)}\del^-_h(v)(\delta h)\in T_{q_0}Q$, which is possible
  because $\del^-_h(v)$ is a submersion for $(h,v)\in A$ (which
  holds automatically for $(h,v,\ti{v})\in \mathcal{C}$). Next, we
  define $q_1:= \del^+_h(v)$ and
  $\delta q_1 := T^1_{(h,v)}\del^+_h(v)(\delta h) +
  T^2_{(h,v)}\del^+_h(v)(\delta v) -
  T^1_{(h,\ti{v})}\del^-_h(\ti{v})(\delta h) \in T_{q_1}Q$.  As
  $\del^\mp_h$ is a submersion\footnote{That $\del^\mp_h$ is a
    submersion at every $(h,\ti{v})$ for $h\neq 0$ is part of
    Theorem~\ref{th:discretizations_yield_discrete_tangent_bundles}.}
  at $(h,\ti{v})$ as long as $h\neq 0$, there are
  $\delta \ti{v} \in T_{\ti{v}} TQ$ such that
  \begin{equation*}
    (T^2_{(h,\ti{v})}\del^-_h(\ti{v})(\delta \ti{v}),
    T^2_{(h,\ti{v})}\del^+_h(\ti{v})(\delta \ti{v})) = 
    T^2_{(h,\ti{v})}\del^\mp_h(\ti{v})(\delta \ti{v}) = (\delta q_1, \delta
    q_2 - T^1_{(h,\ti{v})}\del^+_h(\ti{v})(\delta h)).
  \end{equation*}
  It is immediate to check that
  $(\delta h, \delta v, \delta \ti{v}) \in
  T_{(h,v,\ti{v})}\mathcal{C}$
  and that
  $T_{(h,v,\ti{v})}g(\delta h, \delta v, \delta \ti{v}) = (\delta h,
  \delta q_0, \delta q_1)$.
  This proves that $g|_{\{h\neq 0\}}$ is a submersion, so that for any
  $(h,q_0,q_2)\in \R\times Q\times Q$ with $h\neq 0$,
  $g^{-1}(h,q_0,q_2)\subset \mathcal{C}$ is a submanifold.  The same
  computation as above, with $\delta h=0$ shows that $g_{h_0}$ is also
  a submersion, concluding the proof of
  point~\eqref{it:properties_of_C-regular_values}.
\end{proof}

The main result of this section is the following extension of Theorem
3.7
in~\cite{ar:patrick_cuell-error_analysis_of_variational_integrators_of_unconstrained_lagrangian_systems}
to forced systems.

\begin{theorem}\label{th:flow_of_discretizations_of_FMS}
  Let $(Q,\psi,L_\CPDS,f_\CPDS)$ be a discretization of the
  regular \FMS $(Q,L,f)$. Then, there are open neighborhoods
  $W\subset \R\times TQ$ of $\{0\}\times TQ$ and
  $U\subset \R\times TQ\times TQ$ of $\{0\}\times \Delta_{TQ}$ such
  that for all $(h,v)\in W$ with $h\neq 0$ there is a unique
  $\ti{v}\in TQ$ such that $(h,v,\ti{v})\in U$ is a critical point of
  the problem
  $(\dAFtpst,g|_{\mathcal{C}^*},\ker(Tg|_{\mathcal{C}^*}))$. Moreover,
  $U$ and $W$ can be chosen so that $F:W\rightarrow TQ$ defined by
  \begin{equation*}
    F(h,v):=
    \begin{cases}
      \ti{v} \stext{ if } h\neq 0,\\
      v \stext{ if } h=0
    \end{cases}
  \end{equation*}
  is smooth (in $h$ and $v$).
\end{theorem}

Much of the proof is identical to the original proof of Theorem 3.7
in~\cite{ar:patrick_cuell-error_analysis_of_variational_integrators_of_unconstrained_lagrangian_systems},
and we go over it below in a number of small steps, so that we can
make the proper adaptations to the presence of a discrete force.

\begin{remark}\label{rem:discrete_flow_existence-for_local_data}
  As noticed in Remark~\ref{rem:local_Ld_or_f}, sometimes, the
  discrete Lagrangians or discrete forces are only defined in an open
  neighborhood $V_{L,f}\subset \R\times TQ$ of $\{0\}\times TQ$. After
  shrinking the $U^+$ and $U^-$, the domains of $\del^+$ and $\del^-$,
  as indicated in Remark~\ref{rem:local_Ld_or_f}, the open set
  $A\subset \R\times TQ$ given above satisfies
  $A\subset U^+\cap U^-\subset V_{L,f}$. Consequently,
  $\mathcal{C}\subset p_{12}^{-1}(V_{L,f})\cap
  p_{13}^{-1}(V_{L,f})$. Also, the open subsets $U$ and $W$ that
  appear in Theorem~\ref{th:flow_of_discretizations_of_FMS} can be
  chosen so that $W\subset V_{L,f}$ and
  $U\subset p_{12}^{-1}(V_{L,f})\cap p_{13}^{-1}(V_{L,f})$. Indeed, in
  the proof of Theorem~\ref{th:flow_of_discretizations_of_FMS} given
  below, $U$ arises as the open neighborhood
  $\hat{U}\subset \R\times TQ\times TQ$ of $\{0\}\times \Delta_{TQ}$
  such that $\hat{\gamma}(h,-z_q,z_q)$ is the unique critical point at
  $(h,-z_q,z_q)$ of
  $(\hat{\dAFtp},\ker(T\hat{\varphi}),\hat{\varphi})$ in $\hat{U}$. As
  all the critical points of that critical problem are contained in
  $W^\mathcal{C} \subset \mathcal{C}\subset p_{12}^{-1}(V_{L,f})\cap
  p_{13}^{-1}(V_{L,f})$, we can replace $\hat{U}$ by
  $\hat{U} \cap (p_{12}^{-1}(V_{L,f})\cap p_{13}^{-1}(V_{L,f}))$ that
  has the same properties as $\hat{U}$ and is contained in
  $p_{12}^{-1}(V_{L,f})\cap p_{13}^{-1}(V_{L,f})$. Also, $W$ arises as
  the domain of $F$ which can be cut with $V_{L,f}$ and, still, retain
  all the properties, in addition to being contained in $V_{L,f}$.
\end{remark}

\begin{remark}\label{rem:FDLS_from_discCP_as_crit_probs}
  Let $\discCP = (Q,\psi,L_\CPDS,f_\CPDS)$ be a
  discretization of the regular \FMS $(Q,L,f)$. Fix $Q_0\subset Q$, a
  relatively compact open submanifold. By
  Theorem~\ref{th:discretizations_yield_discrete_tangent_bundles},
  there is $a>0$ such that for any $h$ with $0\neq \abs{h}<a$,
  $(\mathcal{V}_h,\del^+_h,\del^-_h)$ is a discrete tangent bundle of
  $Q_0$, where $\mathcal{V}_h\subset TQ$ is open. Then,
  \begin{equation*}
    \aFDLS_h =
    (Q_0,\mathcal{V}_h,L_\CPHS|_{\mathcal{V}_h},f_\CPHS|_{\mathcal{V}_h})
  \end{equation*}
  is a \FDMSCP. As mentioned in
  Remark~\ref{rem:forced_variational_pple_as_skew_critical}, the
  trajectories of $\aFDLS_h$ can be described as the critical points
  of $(\dAFfhtp,g^{\mathcal{V}_h},\ker(Tg^{\mathcal{V}_h}))$ where
  $\mathcal{C}^{\mathcal{V}_h} := \{(v,\ti{v})\in \mathcal{V}_h\times
  \mathcal{V}_h:\del^+_h(v)=\del^-_h(\ti{v})\}$,
  $g^{\mathcal{V}_h}:\mathcal{C}^{\mathcal{V}_h} \rightarrow Q_0\times
  Q_0$ is
  $g^{\mathcal{V}_h}(v,\ti{v}) := (\del^-_h(v),\del^+_h(\ti{v}))$,
  $\dAFfhop := d L_\CPHS|_{\mathcal{V}_h} + f_\CPHS|_{\mathcal{V}_h}
  \in \mathcal{A}^1(\mathcal{V}_h)$ and
  $\dAFfhtp := p_1^*(\dAFfhop) + p_2^*(\dAFfhop) \in
  \mathcal{A}^1(\mathcal{C}^{\mathcal{V}_h})$.
\end{remark}

\begin{lemma}\label{le:trajectories_and_critical_points-h_not_0}
  With the notation as above, for $0\neq \abs{h} <a$, an element
  $(h,v,\ti{v})\in \{h\}\times
  \mathcal{C}^{\mathcal{V}_h}\subset\mathcal{C}^*$ defines a
  trajectory $(v,\ti{v})$ of the \FDMSCP $\aFDLS_h$ if and only if it is
  a critical point of the critical problem
  $(\dAFtpst, \ker(T(g|_{\mathcal{C}^*})), g|_{\mathcal{C}^*}))$ on
  $\mathcal{C}^*$.
\end{lemma}

\begin{proof}
  Under the current hypotheses, the equivalence between $(v,\ti{v})$
  being a critical point of
  $(\dAFfhtp,g^{\mathcal{V}_h},\ker(Tg^{\mathcal{V}_h}))$ over
  $\mathcal{C}^{\mathcal{V}_h}$ and $(h,v,\ti{v})$ being a critical
  point of
  $(\dAFtpst, g|_{\mathcal{C}^*},\ker(T(g|_{\mathcal{C}^*})))$ over
  $\mathcal{C}^*$ is fairly clear once one notices that
  $(\delta h, \delta v, \delta \ti{v}) \in \ker(T g|_{\mathcal{C}^*})$
  is equivalent to $\delta h = 0$ and
  $(\delta v, \delta \ti{v}) \in \ker(Tg^{\mathcal{V}_h})$.
\end{proof}

The next step in the original proof, the construction of the map
$\hat{\varphi}$, is purely geometric, that is, independent of
Lagrangians or forces, so that it remains valid in its original
form. From the proof of
Lemma~\ref{le:discretizations_yield_discrete_tangent_bundles-U_and_W}
we recall the vector bundle $E$, the open sets $W^E$ and
$W^{Q\times Q}$, and the tubular map $\zeta:W^E\rightarrow W^{Q\times Q}$.

As $g:\mathcal{C}\rightarrow \R\times Q\times Q$ is continuous and
$g(\mathcal{C}_0)\subset \{0\}\times \Delta_Q$, we have that
$W^\mathcal{C}:= g^{-1}(\R\times W^{Q\times Q})\subset \mathcal{C}$ is
open and contains $\mathcal{C}_0$. Then, as in the original proof, the
map $\hat{\varphi}:W^\mathcal{C}\rightarrow \R\times E$ defined by
\begin{equation*}
  \hat{\varphi}(h,v,\ti{v}) := 
  \begin{cases}
    (h,\frac{1}{h}\zeta^{-1}(\del^-_h(v),\del^+_h(\ti{v})))
    \stext{ if } h\neq 0,\\
    (0,-\frac{1}{2}(v+\ti{v}),\frac{1}{2}(v+\ti{v})) \stext{ if } h=0,
  \end{cases}
\end{equation*}
is
smooth\footnote{In~\cite{ar:patrick_cuell-error_analysis_of_variational_integrators_of_unconstrained_lagrangian_systems}
  the function $\widehat{\varphi}$ is defined interchanging the second
  and third variable. Our choice is made to stay closer to the
  standard use of $(q_0,q_1)$ being a pair where $q_0$ ``is before''
  $q_1$.} (in $h$, $v$ and $\ti{v}$).

The previous definitions give rise to the following commutative
diagram.
\begin{equation*}
  \xymatrix{
    {W^\mathcal{C}} \ar[r]^{\hat{\varphi}} & {\R\times E}\\
    {W^\mathcal{C} \cap \mathcal{C}^*} \ar[r]_{g|_{\mathcal{C}^*}} 
    \ar[u]^{i_{\mathcal{C}^*}} & 
    {\R\times W^{Q\times Q}} \ar[u]_{(h,\frac{1}{h}\zeta^{-1})}
  }
\end{equation*}

We define the maps
\begin{equation*}
  \begin{split}
    \widehat{L_\CPS}(h,v) :=&
    \begin{cases}
      \frac{1}{h}L_\CPHS(v) \stext{ if } h\neq 0,\\
      L(v) \stext{ if } h=0
    \end{cases}
    \stext{ and }\\  \widehat{f_\CPS}(h,v):=&
    \begin{cases}
      \frac{1}{h} f_\CPHS(v) \stext{ if } h\neq 0,\\
      f(v) \stext{ if } h=0.
    \end{cases}
  \end{split}
\end{equation*}
Recalling that $L_\CPHS = h L + \mathcal{O}(h^2)$ and
$f_\CPHS = h f + \mathcal{O}(h^2)$,
Proposition~\ref{prop:prop_1_in_CP07} proves that $\widehat{L_\CPS}$ and
$\widehat{f_\CPS}$ are smooth (in $h$ and $v_q$).

Let $\hat{\dAFop}\in\mathcal{A}^1(\R\times TQ)$ and
$\hat{\dAFtp}\in \mathcal{A}^1(W^\mathcal{C})$ be defined by
\begin{equation}\label{eq:mu_hat_and_sigma_hat-def}
  \hat{\dAFop} := d_2 \widehat{L_\CPS} + \widehat{f_\CPS} \stext{ and } 
  \hat{\dAFtp} := 
  i_{W^\mathcal{C},\R\times TQ\times TQ}^*(p_{12}^*(\hat{\dAFop}) + 
  p_{13}^*(\hat{\dAFop})).
\end{equation}
Comparison of this last expression with~\eqref{eq:alpha_cp-def} shows
that, for $(h,v,\ti{v}) \in W^\mathcal{C}\cap \mathcal{C}^*$,
\begin{equation}\label{eq:comparison_sigma_and_hatsigma}
  h\ i_{W^\mathcal{C}\cap \mathcal{C}^*,W^\mathcal{C}}^*(\hat{\dAFtp})(h,v,\ti{v}) 
  = i_{W^\mathcal{C}\cap \mathcal{C}^*,\mathcal{C}^*}^*(\dAFtpst)(h,v,\ti{v}).
\end{equation}

\begin{lemma}\label{le:comparison_crit_problems}
  \begin{enumerate}
  \item \label{it:comparison_crit_problems-is_critical_problem}
    $(\hat{\dAFtp}, \hat{\varphi},\ker(T\hat{\varphi}))$ is a
    critical problem on $W^\mathcal{C}$ .
  \item The critical points of the problem
    $(\hat{\dAFtp}, \hat{\varphi}, \ker(T \hat{\varphi}))$ on
    $W^\mathcal{C}$ at $(h,-z_q,z_q)\in \R^*\times E$ coincide with
    those of the problem
    $(\dAFtpst, g|_{\mathcal{C}^*}, \ker(T(g|_{\mathcal{C}^*})))$ at
    $(h,q_0,q_2)$ (with $(q_0,q_2):=\zeta(h(-z_q,z_q))$) on
    $W^\mathcal{C}\cap \mathcal{C}^*$.
  \end{enumerate}
\end{lemma}

\begin{proof}
  Notice that $p_1\circ \hat{\varphi} = p_1|_{W^\mathcal{C}}$ with
  $W^\mathcal{C}\subset \mathcal{C}$ open and, as
  $p_1|_{\mathcal{C}} : \mathcal{C}\rightarrow \R$ is a submersion
  (see the proof of Lemma~\ref{le:properties_of_C}), we conclude that
  $p_1\circ \hat{\varphi}$ is a submersion. Next we consider
  $p_2\circ (\hat{\varphi}|_{\mathcal{C}_h}):\mathcal{C}_h\rightarrow
  E$; there are two cases: $h=0$ and $h\neq 0$. If $h=0$, we have that
  $\mathcal{C}_0 = \{0\}\times (TQ\oplus TQ)$ (Whitney sum) is a
  vector bundle over $Q$ and
  $\hat{\varphi}|_{\mathcal{C}_0}:\mathcal{C}_0\rightarrow \{0\}\times
  E$ is given by
  $\hat{\varphi}(0,v,\ti{v}) =
  (0,-\frac{1}{2}(v+\ti{v}),\frac{1}{2}(v+\ti{v}))$, that is a vector
  bundle morphism over the identity that, fiberwise, is onto, so it is
  onto globally. Consequently, $\hat{\varphi}|_{\mathcal{C}_0}$ is a
  submersion. If $h\neq 0$,
  $(p_2\circ \hat{\varphi})(h,v,\ti{v}) =
  \frac{1}{h}\zeta^{-1}(\del^-_h(\ti{v}),\del^+_h(v))$, which
  is a submersion because
  $(v,\ti{v})\mapsto (\del^-_h(v),\del^+_h(\ti{v}))$ is a
  submersion (property of discretizations) and $\zeta$ is a
  diffeomorphism.
  
  Now we use the observations from the previous paragraph to check
  that $\hat{\varphi}:W^\mathcal{C}\rightarrow \R\times E$ is a
  submersion. Let $(h,v,\ti{v})\in W^\mathcal{C}$ and
  $(\delta h, \delta e)\in T_{\hat{\varphi}(h,v,\ti{v})}(\R\times
  E)$. As $p_1\circ \hat{\varphi}$ is a submersion, there is
  $(\delta h,\delta v_1,\delta \ti{v}_1) \in
  T_{(h,v,\ti{v})}W^\mathcal{C}$ such that
  $T_{(h,v,\ti{v})}(p_1\circ \hat{\varphi})(\delta h,\delta v_1,\delta
  \ti{v}_1) = \delta h$. Let
  $(\delta h,\delta e_1) := T_{(h,v,\ti{v})}\hat{\varphi}(\delta
  h,\delta v_1,\delta \ti{v}_1) \in
  T_{\hat{\varphi}(h,v,\ti{v})}(\R\times E)$. Then, as
  $p_2\circ (\hat{\varphi}|_{\mathcal{C}_h}):\mathcal{C}_h\rightarrow
  E$ is a submersion, there is
  $(0,\delta v_2, \delta \ti{v}_2) \in T_{(h,v,\ti{v})} \mathcal{C}_h$
  such that
  $T_{(h,v,\ti{v})}(p_2\circ \hat{\varphi})(0,\delta v_2, \delta
  \ti{v}_2) = \delta e -\delta e_1 \in
  T_{p_2(\hat{\varphi}(h,v,\ti{v}))}E$. All together we have that
  $(\delta h,\delta v_1 +\delta v_2, \delta \ti{v}_1+\delta \ti{v}_2)
  \in T_{(h,v,\ti{v})} W^{\mathcal{C}}$ and
  \begin{equation*}
    T_{(h,v,\ti{v})}\hat{\varphi}(\delta h,\delta v_1 +\delta v_2,
    \delta \ti{v}_1+\delta \ti{v}_2) =
    (\delta h, \delta e_1 ) + (0,\delta e - \delta e_1) = (\delta h, \delta e)
  \end{equation*}
  proving that $\hat{\varphi}$ is a submersion.
  
  By definition, $\hat{\dAFop}\in \mathcal{A}^1(\R\times TQ)$ and, by
  the previous analysis, $\hat{\varphi}$ is a submersion, so that
  $\ker(T \hat{\varphi})$ is a distribution. Hence,
  point~\eqref{it:comparison_crit_problems-is_critical_problem} holds.

  $(h,v,\ti{v})\in \mathcal{C}^*$ is a critical point of
  $(\hat{\dAFtp}, \hat{\varphi}, \ker(T \hat{\varphi}))$ at
  $(h,-z_q,z_q)\in \R^*\times E$ iff
  \begin{equation}\label{eq:criticality_hat}
    \begin{gathered}
      \hat{\varphi}(h,v,\ti{v}) = (h,-z_q,z_q) \stext{ and } \\
      \hat{\dAFtp}(h,v,\ti{v})(\delta h, \delta v, \delta \ti{v}) = 0
      \stext{ for all } (\delta h, \delta v, \delta \ti{v})\in \ker(T
      \hat{\varphi})
    \end{gathered}
  \end{equation}
  The condition $\hat{\varphi}(h,v,\ti{v}) = (h,-z_q,z_q)$ is
  equivalent to $g|_{\mathcal{C}^*}(h,v,\ti{v}) = (h,q_0,q_2)$ for
  $(q_0,q_2):=\zeta(h(-z_q,z_q))$. The condition
  $(\delta h, \delta v, \delta \ti{v})\in \ker(T \hat{\varphi})$ is
  equivalent to $\delta h = 0$ (because
  $p_1\circ\hat{\varphi} = p_1|_{\mathcal{C}^*}$), and
  $\delta \ti{v}\in \ker(T^2_{(h,\ti{v})}\del^+_h(\ti{v}))$ and
  $\delta v\in \ker(T^2_{(h,v)}\del^-_h(v))$ (because $\zeta^{-1}$
  is a diffeomorphism). Finally, taking into
  account~\eqref{eq:comparison_sigma_and_hatsigma},
  $\hat{\dAFtp}(h,v,\ti{v})(\delta h, \delta v, \delta \ti{v}) = 0$
  for all
  $(\delta h, \delta v, \delta \ti{v})\in \ker(T \hat{\varphi})$ is
  equivalent to
  $\dAFtpst(h,v,\ti{v})(\delta h, \delta v, \delta \ti{v}) = 0$ for
  all
  $(\delta h, \delta v, \delta \ti{v})\in
  \ker(T(g|_{\mathcal{C}^*}))$.  Hence,~\eqref{eq:criticality_hat} is
  equivalent to
  \begin{gather*}
    g|_{\mathcal{C}^*}(h,v,\ti{v}) = (h,q_0,q_2) \stext{ and }\\
    \dAFtpst(h,v,\ti{v})(\delta h, \delta v, \delta \ti{v})=0 \stext{
      for all } (\delta h, \delta v, \delta \ti{v})\in \ker(T
    g|_{\mathcal{C}^*}).
  \end{gather*}
  This last expression is equivalent to $(h,v,\ti{v})$ being a
  critical point of
  $(\dAFtpst, g|_{\mathcal{C}^*}, \ker(T(g|_{\mathcal{C}^*})))$ at
  $(h,q_0,q_2)$ (for $(q_0,q_2):=\zeta(h(-z_q,z_q))$).
\end{proof}

The key point is that even though the original critical problem
$(\dAFtpst, g|_{\mathcal{C}^*}, \ker(T(g|_{\mathcal{C}^*})))$ is not
well behaved as $h\rightarrow 0$, the (equivalent for $h\neq 0$)
critical problem $(\hat{\dAFtp}, \hat{\varphi},\ker(T \hat{\varphi}))$
extends seamlessly over $h=0$.  We analyze this last problem in what
follows.

\begin{lemma}\label{le:critical_points_with_h_0}
  For any $z_q\in T_qQ$ in a neighborhood of $0_q$,
  $(0,z_q,z_q)\in\mathcal{C}_0$ is a critical point of the problem
  $(\hat{\dAFtp}, \hat{\varphi},\ker(T \hat{\varphi}))$ at
  $(0,-z_q,z_q)$ on $W^\mathcal{C}$.
\end{lemma}

\begin{proof}
  Recall that $\del^-_0 = \del^+_0 = \tau_Q$ so that condition
  $(0,v,\ti{v})\in \mathcal{C}_0$ means that
  $\tau_Q(v)=\tau_Q(\ti{v})$, while
  $\hat{\varphi}(0,v,\ti{v}) = (0,-z_q,z_q)$ means that
  $-\frac{1}{2}(v+\ti{v}) = -z_q$; as a consequence,
  $v,\ti{v},z_q\in T_qQ$.  Hence,
  $\hat{\varphi}^{-1}(0,-z_q,z_q) = \{(0,v,\ti{v})\in \{0\}\times
  T_qQ\times T_qQ : v+\ti{v}=2z_q\}$
  and, consequently, for
  $(0,v,\ti{v}) \in \hat{\varphi}^{-1}(0,-z_q,z_q)$,
  \begin{equation*}
    \begin{split}
      \ker(T_{(0,v,\ti{v})}\hat{\varphi}) =& T_{(0,v,\ti{v})}
      (\hat{\varphi}^{-1}(0,-z_q,z_q)) \\=&
      \left\{(0,\frac{d}{dt}\big|_{t=0}(v+t
        w),\frac{d}{dt}\big|_{t=0}(\ti{v}-t w)): w\in T_qQ\right\}
      \subset T_{(0,v,\ti{v})}\mathcal{C}.
    \end{split}
  \end{equation*}
  Then, the critical condition~\eqref{eq:criticality_hat} (with $h=0$)
  for $(0,v,\ti{v})$ satisfying $v+\ti{v}=2z_q$ becomes
  \begin{equation*}
    \big(\ti{\hat{\dAFop}}(0,v) - \ti{\hat{\dAFop}}(0,\ti{v})\big)(w) = 0
    \stext{ for all } w\in T_qQ
  \end{equation*}
  or, recalling Remark~\ref{rem:legendre_transform_from_1-form},
  \begin{equation*}
    \F L(v)-\F L(\ti{v}) = 0.
  \end{equation*}
  We see that $v=\ti{v}=z_q$ is a solution of this last critical
  condition, so that $(0,z_q,z_q)\in \mathcal{C}_0$ is a critical
  point of the problem
  $(\hat{\dAFtp}, \hat{\varphi},\ker(T \hat{\varphi}))$ at
  $(0,-z_q,z_q)$
\end{proof}

\begin{lemma}\label{le:nondegeneray_of_critical_points}
  In the context of the proof of
  Theorem~\ref{th:flow_of_discretizations_of_FMS}, all critical
  points $(0,v,\ti{v}) = (0,z_q,z_q)$ are nondegenerate.
\end{lemma}

\begin{proof}
  Consider a critical point $(0,z_q,z_q)$ of the critical problem
  $(\hat{\dAFtp},\hat{\varphi},\ker(T\hat{\varphi}))$ at
  $(0,-z_q,z_q) \in\R\times E$. Let
  $\mathcal{Z}:=(\hat{\varphi})^{-1}(0,-z_q,z_q)\subset \mathcal{C}$,
  that is a submanifold because $\hat{\varphi}$ is a submersion
  (point~\eqref{it:comparison_crit_problems-is_critical_problem} of
  Lemma~\ref{le:comparison_crit_problems}). Then, the skew Hessian
  $d_{\ker(T\hat{\varphi}),\hat{\varphi}}\hat{\dAFtp}$ defined
  in~\eqref{eq:skew_hessian-def} becomes, for a critical point
  $(0,z_q,z_q)$ and
  $(0,\delta u,\delta \ti{u}), (0,\delta w,\delta \ti{w}) \in
  T_{(0,z_q,z_q)} \mathcal{Z}$,
  \begin{equation*}
    \begin{split}
      \conj{\mathcal{H}}(\hat{\dAFtp})(0,z_q,z_q))((0,\delta u,\delta
      \ti{u}), (0,\delta w,\delta \ti{w})) :=
      d_{\mathcal{Z}}(i_{\widehat{(\delta w,\delta
          \ti{w})}}(i_{\mathcal{Z},W^\mathcal{C}}^*(\hat{\dAFtp})))
      (0,z_q,z_q) (\delta u,\delta \ti{u}),
    \end{split}
  \end{equation*}
  where $\widehat{(\delta w,\delta \ti{w})} \in\VF(\mathcal{Z})$ is an
  extension of $(\delta w,\delta \ti{w})$. This is so, essentially,
  because all operations are natural with respect to the restriction
  to a submanifold (that is, everything can be restricted to
  $\mathcal{Z}$ and, then, extended, contracted or differentiated) and
  the ``input vectors''
  $(\delta u,\delta \ti{u}), (\delta w,\delta \ti{w}) \in
  T_{(0,z_q,z_q)} \mathcal{Z}$.

  Next notice that, by the computations performed in the proof of
  Lemma~\ref{le:critical_points_with_h_0},
  $\mathcal{Z} = \{(0,v,\ti{v})\in \{0\}\times T_qQ\times T_qQ :
  v+\ti{v}=2z_q\}$.  As $\theta:T_qQ\rightarrow \mathcal{Z}$ defined
  by $\theta(v):=(0,v,2z_q-v)$ is a diffeomorphism (an affine
  isomorphism, in fact), the bilinear form ($2$-tensor)
  $\conj{\mathcal{H}}(\hat{\dAFtp})$ on $T_{(0,z_q,z_q)}\mathcal{Z}$
  defines a bilinear form $\theta^*(\conj{\mathcal{H}}(\hat{\dAFtp}))$
  on $T_{z_q}T_qQ$, whose nondegeneracy is equivalent to that of
  $\conj{\mathcal{H}}(\hat{\dAFtp})$.

  Next we study the nondegeneracy of the bilinear form
  $\theta^*(\conj{\mathcal{H}}(\hat{\dAFtp}))$ on $T_{z_q}T_qQ$. As
  all elements of $T_{z_q}T_qQ$ are of the form $u^{z_q}$ and
  $w^{z_q}$ for $u,w\in T_qQ$, we have
  \begin{equation*}
    \begin{split}
      \theta^*(\conj{\mathcal{H}}(\hat{\dAFtp}))(z_q)(u^{z_q},w^{z_q})
      =& \conj{\mathcal{H}}(\hat{\dAFtp})(0,z_q,z_q)
      ((0,(u^{z_q},-(u^{z_q}))), (0,(w^{z_q},-(w^{z_q})))\\=&
      d_{\mathcal{Z}}(i_{\widehat{(w^{z_q},-(w^{z_q}))}}(\hat{\dAFtp}))
      (0,z_q,z_q)(0,(u^{z_q},-(u^{z_q})) \\=&
      d_{\mathcal{Z}}(i_{\widehat{(w^{z_q},-(w^{z_q}))}}(
      i_{\mathcal{Z},W^\mathcal{C}}^*(p_{12}^*(\hat{\dAFop}))))
      (0,z_q,z_q)(0,(u^{z_q},-(u^{z_q})) +\\&
      d_{\mathcal{Z}}(i_{\widehat{(w^{z_q},-(w^{z_q}))}}(
      i_{\mathcal{Z},W^\mathcal{C}}^*(p_{13}^*(\hat{\dAFop}))))
      (0,z_q,z_q)(0,(u^{z_q},-(u^{z_q})) \\=&
      d_{T_qQ}(i_{\widehat{w^{z_q}}}(\hat{\dAFop}))(z_q)(u^{z_q}) +
      d_{T_qQ}(i_{\widehat{-(w^{z_q}})}(\hat{\dAFop}))(z_q)(-u^{z_q})
      \\=&
      2d_{T_qQ}(i_{\widehat{w^{z_q}}}(\hat{\dAFop}))(z_q)(u^{z_q}).
    \end{split}
  \end{equation*}
  As the computation of the skew Hessian is independent of the field
  $\widehat{w^{z_q}}$ chosen to extend $w^{z_q}$, we use the constant
  extension. Then, by the relationship between $\ti{\cAFop}$ and
  $\F L$ mentioned in
  Remark~\ref{rem:legendre_transform_from_1-form},
  \begin{equation*}
    \begin{split}
      2d_{T_qQ}(i_{\widehat{w^{z_q}}}(\hat{\dAFop}))(z_q)(u^{z_q}) =&
      2d_{T_qQ}(\hat{\dAFop}(w^{z_q}))(z_q)(u^{z_q}) =
      2d_{T_qQ}(\ti{\hat{\dAFop}}(\cdot)(w))(z_q)(u^{z_q}) \\=& 2\F
      (\ti{\hat{\dAFop}}(\cdot)(w))(z_q)(u) = 2 \F (\F
      L(\cdot)(w))(z_q)(u) \\=& 2\F^2 L(z_q)(u,w).
    \end{split}
  \end{equation*}
  Then, as noticed in Remark~\ref{rem:def_of_fiber_derivative}, the
  regularity of $L$ is equivalent to the nondegeneracy of $\F^2 L$
  which, in our current context is equivalent to that of
  $\theta^*(\conj{\mathcal{H}}(\hat{\dAFtp}))$, and so, that of
  $\conj{\mathcal{H}}(\hat{\dAFtp})$ at the critical points
  $(0,z_q,z_q)$, concluding the proof of the result.
\end{proof}

\begin{proof}[Proof of Theorem~\ref{th:flow_of_discretizations_of_FMS}]
  Each critical point $(0,v,\ti{v}) = (0,z_q,z_q)$ at $(0,-z_q,z_q)$
  is nondegenerate, in the sense of Definition 3
  in~\cite{ar:cuell_patrick-skew_critical_problems} by
  Lemma~\ref{le:nondegeneray_of_critical_points}. So,
  $\hat{\gamma}_0(0,-z_q,z_q):=(0,z_q,z_q)$ is a smooth map from
  $N_0:=\{0\}\times E$ to $M_0:=\{0\}\times \Delta_{TQ}$.

  The rest of the proof is identical to the original proof. By
  Theorem~\ref{th:thm_2_in_cp07} there are open neighborhoods
  $\hat{U}\subset \R\times TQ\times TQ$ of $\{0\}\times \Delta_{TQ}$
  and $\hat{V}\subset \R\times E$ of $\{0\}\times E$ and a smooth map
  $\hat{\gamma}:\hat{V}\rightarrow \hat{U}$ extending $\hat{\gamma}_0$
  such that $\hat{\gamma}(h,-z_q,z_q)$ is the unique critical point at
  $(h,-z_q,z_q)$ of
  $(\hat{\dAFtp},\hat{\varphi},\ker(T\hat{\varphi}))$ in $\hat{U}$.

  Since, at $h=0$, we have
  $\hat{\gamma}(0,-z_q,z_q) = (0,z_q,z_q)\in \{0\}\times \Delta_{TQ}$,
  the image of $\hat{\gamma}(0,\cdot,\cdot)$ is the graph of the
  identity map $id_{TQ}:TQ\rightarrow TQ$. Hence, for small $h$, there
  is a smooth function $F(h,\cdot)$ whose graph coincides with the
  image of $\hat{\gamma}(h,\cdot,\cdot)$ (see Proposition 5
  in~\cite{ar:cuell_patrick-skew_critical_problems}). In particular,
  for $h$ small enough and $v$ in the domain of $F$, $F(h,v)=\ti{v}$
  such that $(h,v,\ti{v})\in\hat{U}$ is a critical point of
  $(\hat{\dAFtp},\hat{\varphi},\ker(T\hat{\varphi}))$ and, when
  $h\neq 0$, by Lemma~\ref{le:comparison_crit_problems}, a critical
  point of the problem
  $(\dAFtpst,g|_{\mathcal{C}^*},\ker(Tg|_{\mathcal{C}^*}))$.
\end{proof}

\begin{prop}\label{prop:FDLS_from_discretization-general_disc_TQ}
  Let $(Q,L,f)$ be a regular \FMS, $(\psi,\dBTp,\dBTm)$ a
  discretization of $TQ$ and $Q_0\subset Q$ a relatively compact open
  subset. Let $a>0$ and $\mathcal{U}$ be the ones produced by
  Theorem~\ref{th:discretizations_yield_discrete_tangent_bundles} and
  Lemma~\ref{le:discretizations_yield_discrete_tangent_bundles-U_and_W}
  applied to the discretization $\psi$ and the given $Q_0$.  Let
  $\discCP = (Q,\psi,L_\CPDS,f_\CPDS)$ be a discretization
  of $(Q,L,f)$ with $L_\CPDS$ and $f_\CPDS$ defined over
  $\mathcal{U}$. Then, for each $h$ where $0\neq \abs{h}<a$,
  $(\mathcal{V}_h, \del^{-}_h|_{\mathcal{V}_h},
  \del^{+}_h|_{\mathcal{V}_h})$ with
  \begin{equation}\label{eq:FDLS_from_discretization-general_disc_TQ-V_h}
    \mathcal{V}_h := p_2(p_1^{-1}(\{h\})\cap \mathcal{U}\cap
    (\del^{\mp})^{-1}(Q_0\times Q_0))
  \end{equation}
  is a discrete tangent bundle of $Q_0$ and
  $\aFDLS_h = (Q_0,\mathcal{V}_h, L_\CPHS|_{\mathcal{V}_h},
  f_\CPHS|_{\mathcal{V}_h})$ is a \FDMSCP.
\end{prop}

\begin{proof}
  The result about
  $(\mathcal{V}_h, \del^{-}_h|_{\mathcal{V}_h},
  \del^{+}_h|_{\mathcal{V}_h})$ follows by applying
  Theorem~\ref{th:discretizations_yield_discrete_tangent_bundles} (see
  the proof for the explicit form of $\mathcal{V}_h$) to the
  discretization $\psi$. The rest follows, essentially from the
  definitions and the fact that $\{h\}\times \mathcal{V}_h$ is a
  submanifold of $\mathcal{U}$.
\end{proof}

\begin{definition}\label{def:flow_discretization_CP}
  Let $\discCP = (Q,\psi,L_\CPDS,f_\CPDS)$ be a
  discretization of the regular \FMS $(Q,L,f)$. A \jdef{flow} of
  $\discCP$ is a smooth function $F:U\rightarrow TQ$ where
  $U\subset \R\times TQ$ is an open neighborhood of $\{0\}\times TQ$
  such that, for each $(h,v)\in U$, $(h,v,F(h,v))\in \mathcal{C}$ is a
  critical point of $\discCP$ and $F(0,v)=v$ for all $v\in TQ$. An
  \jdef{extended flow} is smooth a map
  $\ti{F}:U\rightarrow \R\times TQ$ of the form
  $\ti{F}(h,v) = (h,F(h,v))$, where $F$ is a flow.
\end{definition}

\begin{remark}
  By Theorem~\ref{th:flow_of_discretizations_of_FMS}, all
  discretizations of a regular \FMS $(Q,L,f)$ have flows.
\end{remark}


\subsection{Trajectories of exact forced discrete mechanical systems}
\label{sec:trajectories_of_exact_forced_discrete_mechanical_systems}

Let $(Q,L,f)$ be a regular \FMS. In this section we consider the
trajectories of the \FDMSCP obtained using an exact discretization
given in Example~\ref{ex:exact_discretization_forced-system}.

Given $v\in TQ$ and $h>0$, we write $q_v(t):=\tau_Q(F^{X_{L,f}}_t(v))$
for $t\in [\dBTm(h),\dBTp(h)]$. Similarly, given
$\delta v \in T_v TQ$, we define
$\delta q_v(t):=T_v(\tau_Q\circ F^{X_{L,f}}_t)(\delta v) \in
T_{q_v(t)}Q$ for $t\in [\dBTm(h),\dBTp(h)]$; clearly,
$\delta q_v(t)$ is an infinitesimal variation over $q_v(t)$. If we
write $\delta v = \frac{d}{d\epsilon}\big|_{\epsilon=0}v(\epsilon)$,
we have
\begin{equation*}
  \begin{split}
    (\delta q_v)'(t) =& \frac{d}{dt} \delta q_v(t) = \frac{d}{dt}
    T_v(\tau_Q\circ F^{X_{L,f}}_t)(\delta v) = \frac{d}{dt}
    T_v(\tau_Q\circ
    F^{X_{L,f}}_t)(\frac{d}{d\epsilon}\big|_{\epsilon=0}v(\epsilon))
    \\=& \frac{d}{dt} (\frac{d}{d\epsilon}\big|_{\epsilon=0}
    \tau_Q\circ F^{X_{L,f}}_t(v(\epsilon))) =\kappa\bigg(
    \frac{d}{d\epsilon}\big|_{\epsilon=0}(\frac{d}{dt} \tau_Q\circ
    F^{X_{L,f}}_t(v(\epsilon)))\bigg) \\=& \kappa\bigg(
    \frac{d}{d\epsilon}\big|_{\epsilon=0}( F^{X_{L,f}}_t(v(\epsilon)))\bigg)
    = \kappa\big((T_{v(0)}F^{X_{L,f}}_t)
    (\frac{d}{d\epsilon}\big|_{\epsilon=0}v(\epsilon))\big) \\=& \kappa\big(
    (T_{v}F^{X_{L,f}}_t)(\delta v)\big),
  \end{split}
\end{equation*}
where the sixth equality is due to $X_{L,f}$ being a second order
differential equation.

Next, for $(h,v) \in U^{X_{L,f}}$ and $\delta v\in T_vTQ$ we plug the
trajectory $q_v(t)$ and the infinitesimal variation $\delta q_v(t)$
into the left hand side of~\eqref{eq:variation_of_lagrangian-FMS} to
obtain
\begin{equation*}
  \begin{split}
    \int_{\dBTm(h)}^{\dBTp(h)}&dL(q_v'(t))(\kappa((\delta q_v)'(t)))dt +
    \int_{\dBTm(h)}^{\dBTp(h)} \FVF{f}(q_v'(t))(\delta q_v(t))
    dt \\=&
    \int_{\dBTm(h)}^{\dBTp(h)}dL(F^{X_{L,f}}_t(v))((T_{v}F^{X_{L,f}}_t)(\delta
    v))dt \\& + \int_{\dBTm(h)}^{\dBTp(h)}
    \FVF{f}(F^{X_{L,f}}_t(v))(T_v(\tau_Q\circ F^{X_{L,f}}_t)(\delta
    v)) dt \\=& \int_{\dBTm(h)}^{\dBTp(h)} d(L\circ
    F^{X_{L,f}}_t)(v)(\delta v) dt + f^E_\CPDS(v)(\delta v) = d_2
    L^E_\CPDS(v)(\delta v) + f^E_\CPDS(v)(\delta v).
  \end{split}
\end{equation*}
The same substitutions in the right hand side
of~\eqref{eq:variation_of_lagrangian-FMS} lead to
\begin{equation*}
  \begin{split}
    \F L(q_v'(t))(\delta q_v(t)) \big|_{\dBTm(h)}^{\dBTp(h)} =&
    \F L(F^{X_{L,f}}_t(v))(T_v(\tau_Q\circ F^{X_{L,f}}_t)(\delta
    v))\big|_{\dBTm(h)}^{\dBTp(h)} \\=& \F
    L(F^{X_{L,f}}_{\dBTp(h)}(v))(T_v(\tau_Q\circ
    F^{X_{L,f}}_{\dBTp(h)})(\delta v)) \\&- \F
    L(F^{X_{L,f}}_{\dBTm(h)}(v))(T_v(\tau_Q\circ
    F^{X_{L,f}}_{\dBTm(h)})(\delta v)) \\=& \F
    L(F^{X_{L,f}}_{\dBTp(h)}(v))(T_v\del^{E+}_h(\delta v)) - \F
    L(F^{X_{L,f}}_{\dBTm(h)}(v))(T_v\del^{E-}_h(\delta v))
  \end{split}
\end{equation*}

From~\eqref{eq:variation_of_lagrangian-FMS} we conclude that, for any
$(h,v)\in U^{X_{L,f}}$ and $\delta v\in T_vTQ$,
\begin{equation}
  \label{eq:eq:variation_of_lagrangian-FMS-variation_pushed}
  \begin{split}
    dL^E_\CPHS(v)(\delta v) + f^E_\CPHS(v)(\delta v) =& \F
    L(F^{X_{L,f}}_{\dBTp(h)}(v))(T_v\del^{E+}_h(\delta v)) \\&- \F
    L(F^{X_{L,f}}_{\dBTm(h)}(v))(T_v\del^{E-}_h(\delta v)).
  \end{split}
\end{equation}

In order to state the next two results, we let $(Q,L,f)$ be a regular
\FMS, $Q_0\subset Q$ be a relatively compact open subset and $a>0$ the
constant obtained when
Proposition~\ref{prop:FDLS_from_discretization-general_disc_TQ} is
applied to an exact discretization
$\discCPE = (Q,\psi^E,L_\CPDS^E,f_\CPDS^E)$ of $(Q,L,f)$, so
that for any $h$ with $0\neq \abs{h}<a$,
$\aFDLSE_h = (Q_0, \mathcal{V}_h, L_\CPHS^E|_{\mathcal{V}_h},
f_\CPHS^E|_{\mathcal{V}_h})$ is a \FDMSCP over $Q_0$. Also, define
$\mathcal{U}^E_0:= \mathcal{U}^E \cap (\del^{E\mp})^{-1}(Q_0\times
Q_0) \subset\R\times TQ$ where $\mathcal{U}^E\subset U^{X_{L,f}}$ is
the open subset produced when
Lemma~\ref{le:discretizations_yield_discrete_tangent_bundles-U_and_W}
is applied to the exact discretization $\psi^E$ of $TQ$.

\begin{lemma}\label{le:exact_discretization-continuous_imp_discrete_traj}
  For every $(h,v)\in \mathcal{U}^E_0$ with $0\neq \abs{h}<a$ such
  that $(h,F^{X_{L,f}}_h(v))\in \mathcal{U}^E_0$, the pair
  $(v,F^{X_{L,f}}_h(v))$ is a trajectory of $\aFDLSE_h$.
\end{lemma}

\begin{proof}
  Let $\ti{v}:=F^{X_{L,f}}_h(v)$. As, by hypotheses,
  $(h,v),(h,\ti{v})\in \mathcal{U}^E_0$,
  using~\eqref{eq:FDLS_from_discretization-general_disc_TQ-V_h}, we
  have that $v,\ti{v}\in \mathcal{V}_h$.  As $\dBTp(h)-\dBTm(h)=h$ we
  have that
  \begin{equation*}
    \begin{split}
      \del^{E+}_h(v) =& \tau_Q(F^{X_{L,f}}_{\dBTp(h)}(v)) =
      \tau_Q(F^{X_{L,f}}_{\dBTm(h)+h}(v)) =
      \tau_Q(F^{X_{L,f}}_{\dBTm(h)}(F^{X_{L,f}}_h(v))) \\=&
      \del^{E-}_h(F^{X_{L,f}}_h(v)) = \del^{E-}_h(\ti{v}).
    \end{split}
  \end{equation*}
  Thus, using the notation introduced in
  Remark~\ref{le:trajectories_and_critical_points-h_not_0},
  $(v,\ti{v})\in \mathcal{C}^{\mathcal{V}_h}$. Also, for
  $g^{\mathcal{V}_h}(w,\ti{w}) :=
  (\del^{E-}_h(w),\del^{E+}_h(\ti{w}))$ we have that
  $(\delta v, \delta \ti{v}) \in
  \ker(T_{(v,\ti{v})}g^{\mathcal{V}_h})$ if and only if
  \begin{equation*}
    T_v\del^{E-}_h(\delta v) = 0,
    \quad T_v\del^{E+}_h(\delta v) =
    T_{\ti{v}}\del^{E-}_h(\delta\ti{v}),\quad
    T_{\ti{v}}\del^{E+}_h(\delta \ti{v}) = 0.
  \end{equation*}
  Then, using the previous equations as well
  as~\eqref{eq:eq:variation_of_lagrangian-FMS-variation_pushed}
  (recalling that
  $\mathcal{U}^E_0\subset\mathcal{U}^E\subset U^{X_{L,f}}$) we have
  \begin{equation*}
    \begin{split}
      \dAFfhtp(v,\ti{v})(\delta v,& \delta \ti{v}) =
      dL^E_\CPHS(v)(\delta v) + f^E_\CPHS(v)(\delta v) +
      dL^E_\CPHS(\ti{v})(\delta \ti{v}) + f^E_\CPHS(\ti{v})(\delta
      \ti{v}) \\=& \F
      L(F^{X_{L,f}}_{\dBTp(h)}(v))(T_v\del^{E+}_h(\delta v)) - \F
      L(F^{X_{L,f}}_{\dBTm(h)}(v))(\underbrace{T_v\del^{E-}_h(\delta
        v)}_{=0}) \\&+ \F
      L(F^{X_{L,f}}_{\dBTp(h)}(\ti{v}))(\underbrace{T_v\del^{E+}_h(\delta
        \ti{v})}_{=0}) - \F
      L(F^{X_{L,f}}_{\dBTm(h)}(\ti{v}))(T_v\del^{E-}_h(\delta \ti{v}))
      \\=& \F L(F^{X_{L,f}}_{\dBTp(h)}(v))(T_v\del^{E+}_h(\delta v))
      -\F L(F^{X_{L,f}}_{\dBTm(h)}(\ti{v}))(T_v\del^{E+}_h(\delta v))
      \\=& \F L(F^{X_{L,f}}_{\dBTp(h)}(v))(T_v\del^{E+}_h(\delta v))
      -\F
      L(F^{X_{L,f}}_{\dBTm(h)}(F^{X_{L,f}}_h(v)))(T_v\del^{E+}_h(\delta
      v)) \\=& \F L(F^{X_{L,f}}_{\dBTp(h)}(v))(T_v\del^{E+}_h(\delta
      v)) -\F L(F^{X_{L,f}}_{\dBTm(h)+h}(v))(T_v\del^{E+}_h(\delta
      v)) \\=& \F L(F^{X_{L,f}}_{\dBTp(h)}(v))(T_v\del^{E+}_h(\delta
      v)) -\F L(F^{X_{L,f}}_{\dBTp(h)}(v))(T_v\del^{E+}_h(\delta v))
      = 0.
    \end{split}
  \end{equation*}
  Thus, we have that $(v,\ti{v})$ is a critical point of the critical
  problem $(\dAFfhtp,g^{\mathcal{V}_h},\ker(Tg^{\mathcal{V}_h})$;
  hence, as observed in
  Remark~\ref{rem:FDLS_from_discCP_as_crit_probs} $(v,\ti{v})$ is a
  trajectory of $\aFDLSE_h$.
\end{proof}

Still in the context set up before
Lemma~\ref{le:exact_discretization-continuous_imp_discrete_traj},
we have the following result.

\begin{theorem}\label{thm:flow_of_exact_discrete_systems}
  There is an open neighborhood $\ti{U}\subset \R\times TQ\times TQ$
  of $\{0\}\times \Delta_{T{Q_0}}$ such that, for all
  $(h,v,\ti{v})\in \ti{U}$ with $0\neq \abs{h}<a$, the following
  statements are equivalent.
  \begin{enumerate}
  \item \label{it:flow_of_exact_discrete_systems-continuous}
    $\ti{v} = F^{X_{L,f}}_h(v)$.
  \item \label{it:flow_of_exact_discrete_systems-discrete}
    $(v,\ti{v})$ is a trajectory of $\aFDLSE_h$.
  \end{enumerate}
\end{theorem}

\begin{proof}
  As $L^E_\CPDS$ and $f^E_\CPDS$ are defined over
  $U^{X_{L,f}}$ let $W$ and $U$ be the open subsets obtained when
  Theorem~\ref{th:flow_of_discretizations_of_FMS} ---taking into
  account Remark~\ref{rem:discrete_flow_existence-for_local_data} with
  $V_{L,f}:=U^{X_{L,f}}$--- is applied to $\discCPE$. By construction,
  $W\subset U^{X_{L,f}}$ and
  $U\subset p_{12}^{-1}(U^{X_{L,f}})\cap p_{13}^{-1}(U^{X_{L,f}})$.
  Let $p_{12},p_{13}:\R\times TQ\times TQ\rightarrow \R\times TQ$ be
  the projection mappings and define the smooth functions
  $F_3:U^{X_{L,f}}\rightarrow \R\times TQ\times TQ$ and
  $F_2:U^{X_{L,f}}\rightarrow \R\times TQ$ by
  $F_3(h,v):=(h,v,F^{X_{L,f}}_h(v))$ and $F_2:=p_{13}\circ F_3$. Let
  $\mathcal{U}^E_0 := \mathcal{U}^E\cap (\del^{E\mp})^{-1}(Q_0\times
  Q_0) \subset U^{X_{L,f}}$ that is an open neighborhood of
  $\{0\}\times TQ_0$ and define
  $W_0 := W\cap \mathcal{U}^E_0\cap F_2^{-1}(\mathcal{U}^E_0)$.  Last,
  define the open subset of $\R\times TQ\times TQ$:
  \begin{equation*}
    \ti{U}:=U\cap p_{12}^{-1}(W_0) \cap p_{12}^{-1}(F_3^{-1}(U)).
  \end{equation*}
  Notice that $\ti{U}$ is an open neighborhood of
  $\{0\}\times \Delta_{T{Q_0}}$. Indeed,
  $\{0\}\times \Delta_{TQ}\subset U$ and, for any $v_0\in TQ_0$,
  $(0,v_0) \in W_0$ (because $\{0\}\times TQ\subset W$ and
  $\{0\}\times TQ_0\subset \mathcal{U}^E_0$), so that
  $(0,v_0,v_0)\in p_{12}^{-1}(W_0)$; also, as
  $(0,v,v)\in U$ for any $v\in TQ$, we have that $(0,v)\in F_3^{-1}(U)$
  and, then, $\{0\}\times TQ\subset F_3^{-1}(U)$.

  \eqref{it:flow_of_exact_discrete_systems-continuous} \imp
  \eqref{it:flow_of_exact_discrete_systems-discrete} As
  $(h,v,\ti{v})\in \ti{U}\subset U$ with $0\neq \abs{h}<a$ and
  $\ti{v}=F^{X_{L,f}}_h(v)$ we have that
  $(h,v)\in W_0\subset\mathcal{U}^E_0$ and
  $(h,F^{X_{L,f}}_h(v))=F_2(h,v)\in \mathcal{U}^E_0$. Then, by
  Lemma~\ref{le:exact_discretization-continuous_imp_discrete_traj},
  $(v,F^{X_{L,f}}_h(v))$ is a trajectory of the \FDMSCP $\aFDLSE_h$.

  \eqref{it:flow_of_exact_discrete_systems-discrete} \imp
  \eqref{it:flow_of_exact_discrete_systems-continuous} Let
  $(h,v,\ti{v})\in \ti{U}\subset U$ so that $(v,\ti{v})$ is a
  trajectory of $\aFDLSE_h$. Then
  $(h,v)\in W_0 \subset \mathcal{U}^E_0 \cap
  F_2^{-1}(\mathcal{U}^E_0)$ and
  $(h,v,F^{X_{L,f}}_h(v)) = F_3(h,v) \in U$. Then, both
  $(h,v),(h,F^{X_{L,f}}_h(v))=F_2(h,v) \in \mathcal{U}^E_0$ so that,
  by Lemma~\ref{le:exact_discretization-continuous_imp_discrete_traj},
  $(v,F^{X_{L,f}}_h(v))$ is a trajectory of $\aFDLSE_h$. As
  $(v,\ti{v})$ and $(v,F^{X_{L,f}}_h(v))$ are trajectories of
  $\aFDLSE_h$ both $(h,v,\ti{v})$ and $(h,v,F^{X_{L,f}}_h(v))$ are
  critical points of the critical problem associated to the exact
  discretization $\discCPE$ and, as both are contained in $U$, by the
  uniqueness of those critical points proved by
  Theorem~\ref{th:flow_of_discretizations_of_FMS} we conclude that
  $\ti{v} = F^{X_{L,f}}_h(v)$.
\end{proof}

\begin{remark}
  By Theorem~\ref{thm:flow_of_exact_discrete_systems} the critical
  points $(h,v,\ti{v})\in \ti{U}$ are exactly the triples of the form
  $(h,v,F^{X_{L,f}}_h(v))$. Thus, more informally, the flow of a \FMS
  sampled at discrete times coincides with the discrete trajectories
  of any of its exact forced discrete mechanical systems introduced in
  Example~\ref{ex:exact_discretization_forced-system}. This explains
  the name ``exact system'' for such \FDMSCP.
\end{remark}


\section{Error analysis}
\label{sec:error_analysis}

Discretizations of a \FMS are, as introduced in
Section~\ref{sec:discretizations_a_la_cuell_and_patrick} families of
\FDMSCP ``parametrized by $h$''. In this section we want to study how
the similarity of different discretizations (as $h\rightarrow 0$)
translates into the similarity of the corresponding discrete
trajectories. The proof of main result,
Theorem~\ref{th:contact_order_discrete_FMS}, is done in several parts
and involves the asymptotic behavior of the action forms in
Section~\ref{sec:analysis_of_sigma_bar}, the restriction maps in
Section~\ref{sec:analysis_of_phi_bar}, the constraints (associated to
the restriction maps) in Section~\ref{sec:analysis_of_D_bar} and the
critical-point map in Section~\ref{sec:analysis_of_bar_gamma}.

\begin{definition}\label{def:contact_order_of_discretizations}
  Let $(Q,L,f)$ be a \FMS. Two discretizations
  $(Q,\psi^j,L_\CPHS^j,f_\CPHS^j)$ ($j=1,2$) of $(Q,L,f)$ are said
  to have \jdef{order $r$ contact} if
  $\psi^2(h,t,v) = \psi^1(h,t,v) + \mathcal{O}(\grade_{\R^2\times
    TQ}^{r+1})$,
  $L_\CPHS^2(v) = L_\CPHS^1(v) + \mathcal{O}(\grade_{\R\times
    TQ}^{r+1})$ and
  $f_\CPHS^2(v) = f_\CPHS^1(v) + \mathcal{O}(\grade_{\R\times
    TQ}^{r+1})$ where
  \begin{equation}
    \label{eq:gades_RxTQ_and_R^2xTQ}
    \grade_{\R^2\times TQ}(h,t,v):=t \stext{ and } \grade_{\R\times TQ}(h,v):=h.
  \end{equation}
  A single discretization $(Q,\psi,L_\CPDS,f_\CPDS)$ is said
  to have \jdef{order $r$ contact} if it has order $r$ contact with
  the exact discretization of $(Q,L,f)$ introduced in
  Example~\ref{ex:exact_discretization_forced-system} (where $\dBTp$
  and $\dBTm$ are the ones coming with $\psi$).
\end{definition}


\begin{remark}\label{rem:contact_order_of_discretization-alternatives}
  The contact order conditions that appear in
  Definition~\ref{def:contact_order_of_discretizations} can be
  translated into alternative formulations using
  Lemma~\ref{le:contact_order_with_local_bounds_and_derivatives}. The
  order condition for the $\psi^j$ is equivalent to requiring that,
  for any point $(h_0,0,v_0)$ in the domain of the discretizations
  $\psi^j$, there are local coordinates and a constant $C_{h_0,v_0}>0$
  ---depending on the point $(h_0,v_0)$--- where the local expressions
  $\check{\psi^j}$ satisfy
  \begin{equation*}
    \norm{\check{\psi^2}(\check{h},\check{t},\check{v}) -
      \check{\psi^1}(\check{h},\check{t},\check{v})}\leq C_{h_0,v_0}
    \check{t}^{r+1}
  \end{equation*}
  for all $(\check{h},\check{t},\check{v})$ in the image of the
  coordinate domain. Another way of expressing the same conditions is
  that the local expressions have the same $\check{t}$ derivatives up
  to order $r$:
  \begin{equation*}
    D_2^j\check{\psi^2}(\check{h},0,\check{v}) =
    D_2^j\check{\psi^1}(\check{h},0,\check{v}), \stext{ for all } j=0,\ldots,r
  \end{equation*}
  and all $(\check{h},0,\check{v})$ in the image of the coordinate
  chart. In the same way, the order conditions for $L_\CPHS$ and
  $f_\CPHS$ are equivalent to the their coordinate expressions
  satisfying
  $\abs{\widecheck{L_\CPHS^2}(\check{v})
    -\widecheck{L_\CPHS^1}(\check{v})} \leq C^L_{v_0} h^{r+1}$ and
  $\norm{\widecheck{f_\CPHS^2}(\check{v})
    -\widecheck{f_\CPHS^1}(\check{v})} \leq C^f_{v_0} h^{r+1}$ for
  all $(h,\check{v})$ in the image of the coordinate domains ---where
  $C^L_{v_0}, C^f_{v_0}>0$ are constants depending on $v_0$--- or, in
  terms of derivatives, that
  $D_h^j|_{h=0}\widecheck{L_\CPHS^2}(\check{v}) =
  D_h^j|_{h=0}\widecheck{L_\CPHS^1}(\check{v})$ and
  $D_h^j|_{h=0}\widecheck{f_\CPHS^2}(\check{v}) =
  D_h^j|_{h=0}\widecheck{f_\CPHS^1}(\check{v})$, for all
  $(0,\check{v})$ in the image of the coordinate domain and
  $j=0,\ldots,r$.
\end{remark}

\begin{example}
  For any $r\in\N$, a discretization of order $r$ of the \FMS
  introduced in Example~\ref{ex:particle_with_friction-def} can be
  constructed as follows. Let
  $S_k(x) := -\sum_{j=1}^k \frac{1}{j!} x^j$ (the Taylor polynomial of
  order $k$ of $1-e^x$, centered at $x_0=0$). Then, using the
  expressions for the exact discretization system obtained in
  Examples~\ref{ex:particle_with_friction-exact_discretization}
  and~\ref{ex:particle_with_friction-exact_discretization_system}, we
  define
  \begin{gather*}
    \psi(h,t,q,v):=q+\frac{1}{\alpha} v S_r(-\alpha t), \quad
    L_\CPHS(q,v):=\frac{1}{4\alpha} v S_r(-2\alpha h), \stext{ and }\\
    f_\CPHS(q,v):=-v \big(S_r(-\alpha h) dq + \frac{1}{\alpha}\big(
    S_r(-\alpha h) - \frac{1}{2} S_r(-2\alpha h)\big) dv\big).
  \end{gather*}
  It is immediate that the discretization so defined has order $r$
  contact. Indeed, this is the case because the first $r$ derivatives
  of the $\psi$, $L_\CPHS$ and $f_\CPHS$ are the same as those of
  the exact system (see
  Remark~\ref{rem:contact_order_of_discretization-alternatives} and
  Lemma~\ref{le:contact_order_with_local_bounds_and_derivatives}).
\end{example}


The main result is the following theorem, that will be proved along
the section.

\begin{theorem}\label{th:contact_order_discrete_FMS}
  Let $\discCP^j=(Q,\psi^j,L_\CPDS^j,f_\CPDS^j)$ be
  discretizations of the regular \FMS $(Q,L,f)$ for $j=1,2$ and assume
  that they have order $r$ contact. Then, the corresponding discrete
  flows have order $r$ contact, that is,
  $F_h^2=F_h^1+\mathcal{O}(\grade_{\R\times TQ}^{r+1})$.
\end{theorem}

An immediate and important consequence is the following.

\begin{corollary}\label{cor:contact_of_psi_vs_contact_exact_flow-new}
  Let $\discCP:=(Q,\psi,L_\CPDS,f_\CPDS)$ be a
  discretization of order $r$ of the regular \FMS $(Q,L,f)$ and
  $\conj{v}\in TQ$. If $F^\discCP$ is the discrete flow of $\discCP$
  given by Theorem~\ref{th:flow_of_discretizations_of_FMS} and
  $F^{X_{L,f}}$ is the (continuous) flow of $(Q,L,f)$ defined near
  $\conj{v}$, then $F^{\discCP}(h,\conj{v})$ and
  $F^{X_{L,f}}_h(\conj{v})$ have order $r$ contact, that is,
  $F^{\discCP}(h,\conj{v}) = F^{X_{L,f}}_h(\conj{v}) +
  \mathcal{O}(\grade_{\R\times TQ}^{r+1})$.
\end{corollary}

\begin{remark}\label{rem:contact_order_condition_of_flows_in_terms_of_bounds}
  The contact order condition for the discrete flow $F^\Delta$
  guaranteed by
  Corollary~\ref{cor:contact_of_psi_vs_contact_exact_flow-new} is
  equivalent to stating that, for any point $(0,v_0)$ in the common
  domain of $F^\Delta$ and $F^{X_{L,f}}$ there are local coordinates
  and a constant $C_{v_0}>0$ ---depending of $v_0$--- where the local
  expressions of $F^\Delta$ and $F^{X_{L,f}}$ satisfy
  \begin{equation*}
    \norm{\widecheck{F^\Delta}(\check{h},\check{v}) -
      \widecheck{F^{X_{L,f}}}(\check{h},\check{v})} \leq C_{v_0} \check{h}^{r+1}
  \end{equation*}
  for all $(\check{h},\check{v})$ in the image of the coordinate
  domain. This condition is, also, equivalent to the fact that
  \begin{equation*}
    D_1^j\widecheck{F^\Delta}(0,\check{v}) =
    D_1^j\widecheck{F^{X_{L,f}}}(0,\check{v}), \stext{ for all } j=0,\ldots,r
  \end{equation*}
  and all $(0,\check{v})$ in the image of the coordinate chart (see
  Lemma~\ref{le:contact_order_with_local_bounds_and_derivatives}).
\end{remark}

\begin{proof}[Proof of
  Corollary~\ref{cor:contact_of_psi_vs_contact_exact_flow-new}]
  Let $\discCPE:=(Q,\psi^E,L^E_\CPDS,f^E_\CPDS)$ be the
  exact discretization of $(Q,L,f)$ corresponding to the functions
  $\dBTp$ and $\dBTm$ coming from the discretization $\psi$, defined
  in Example~\ref{ex:exact_discretization_forced-system}. As $\discCP$
  has order $r$, the contact order between $\discCP$ and $\discCPE$ is
  $r$. Then, by Theorem~\ref{th:contact_order_discrete_FMS}, the
  contact order between the corresponding flows is $r$. Thus,
  $F^{\discCP}(h,\conj{v}) = F^{\discCPE}(h,\conj{v}) +
  \mathcal{O}(\grade_{\R\times TQ}^{r+1})$ and, since by
  Theorem~\ref{thm:flow_of_exact_discrete_systems}
  $F^{\discCPE}(h,\conj{v}) = F^{X_{L,f}}_h(\conj{v})$, the statement
  follows.
\end{proof}

The proof of Theorem~\ref{th:contact_order_discrete_FMS} will parallel
that of Theorem 4.7
in~\cite{ar:patrick_cuell-error_analysis_of_variational_integrators_of_unconstrained_lagrangian_systems},
but in the more general context of forced systems. The plan is as
follows: Theorem~\ref{th:flow_of_discretizations_of_FMS} constructs
the flow $F$ of a discretization $(Q,\psi,L_\CPDS,f_\CPDS)$
of $(Q,L,f)$ by finding critical points of the problem
$(\hat{\dAFtp}, \hat{\varphi}, \ker(T \hat{\varphi}))$
(Lemma~\ref{le:comparison_crit_problems}) over given ``boundary data''
in $\R\times E$ when $h=0$ first and, then, constructing a function
$\hat{\gamma}$ that assigns the corresponding critical points to the
given boundary data, even for $h\neq 0$. Eventually, $F$ is defined in
terms of $\hat{\gamma}$. In the presence of two discretizations
$\psi^1$ and $\psi^2$ as in the statement of
Theorem~\ref{th:contact_order_discrete_FMS} it seems natural to use
Theorem~\ref{th:thm_3_in_cp07} that allows for the comparison of the
contact order of the functions $\hat{\gamma^j}$ ---the critical
points--- corresponding to two critical problems that have a known
contact order. Unfortunately, one hypothesis of the Theorem is that
both discretizations $\psi^j$ be defined on the same manifold, which
doesn't happen in our setting, where the manifolds $\mathcal{C}^1$ and
$\mathcal{C}^2$ need not coincide.

The path used at this point in the proof of Theorem 4.7
in~\cite{ar:patrick_cuell-error_analysis_of_variational_integrators_of_unconstrained_lagrangian_systems}
is to construct a new pair of critical problems
$(\conj{\dAFtp_j}, \conj{\varphi^j}, \ker(T\conj{\varphi^j}))$ over a
submanifold $\conj{\mathcal{C}}\subset \R\times TQ\times TQ$ (the same
$\conj{\mathcal{C}}$ for both problems) whose critical points
correspond to those of the problems
$(\widehat{\dAFtp_j},\hat{\varphi}_j, \ker(T \hat{\varphi}_j))$. Then,
Theorem~\ref{th:thm_3_in_cp07} can be used in this ``bar context'' to
prove that the contact order between the obtained critical points is,
at least, the one between the data of the critical problems. The
general result then follows readily.

In order to construct the bar context mentioned above, Cuell \&
Patrick consider $\R\times TQ\times TQ$ with the grade
$\grade_{\R\times TQ\times TQ} :=p_1$ and introduce functions
$\Theta^j:\R\times TQ\times TQ\rightarrow \R\times TQ\times TQ$ such
that
$\Theta^2=\Theta^1+\mathcal{O}(\grade_{\R\times TQ\times TQ}^{r+1})$
and, for $j=1,2$, $T_{(0,v,\ti{v})}\Theta^j$ are nonsingular for all
$v,\ti{v}\in TQ$,
$\Theta^j|_{\{0\}\times TQ\times TQ} = id_{\{0\}\times TQ\times TQ}$,
and
\begin{equation*}
  (\tau_Q\circ p_2\circ \Theta^j)(h,v,\ti{v}) = \del^{+,j}_h(v) \stext{ and }
  (\tau_Q\circ p_3\circ \Theta^j)(h,v,\ti{v}) = \del^{-,j}_h(\ti{v}).
\end{equation*}
Such functions $\Theta^j$ can be constructed, for instance, by putting
a Riemannian metric on $Q$ and letting $\CP_{q_2,q_1}$ be the
corresponding parallel transport along the geodesics starting at $q_1$
and ending at $q_2$ (for sufficiently close $q_1,q_2$). Then,
\begin{equation*}
  \Theta^j(h,v,\ti{v}) :=(h,\CP_{\del^{+,j}_h(v),\tau_Q(v)}(v),
  \CP_{\del^{-,j}_h(\ti{v}),\tau_Q(\ti{v})}(\ti{v}));
\end{equation*}
Patrick \& Cuell claim that $\Theta^j$ have the required
properties. We note that, in addition, the maps $\Theta^j$ constructed
are of ``product type'', in the sense that
$\Theta^j(h,v,\ti{v}) = (h, \Theta^{+,j}(h,v),
\Theta^{-,j}(h,\ti{v}))$.  In what follows, we assume that the
$\Theta^j$ are of product type.

Since $\{0\}\times TQ\times TQ\subset \R\times TQ\times TQ$ is a
closed submanifold,
$\Theta^j|_{\{0\}\times TQ\times TQ} = id_{\{0\}\times TQ\times TQ}$
is a diffeomorphism and $\Theta^j$ are local diffeomorphisms on
$\{0\}\times TQ\times TQ$, by Theorem 1
in~\cite{ar:cuell_patrick-skew_critical_problems}, there are open
neighborhoods $\widehat{U^j},\conj{U^j}\subset \R\times TQ\times TQ$
containing $\{0\}\times TQ\times TQ$ such that
$\Theta^j|_{\widehat{U^j}}:\widehat{U^j}\rightarrow \conj{U^j}$ are
diffeomorphisms. In particular, if $\mathcal{C}^j$ is defined
by~\eqref{eq:submanifold_of_matching_arrows_after_blow_up} for each
$j=1,2$, it is easy to verify that
$\Theta^j(\mathcal{C}^j\cap \widehat{U^j})$ is an open subset of
$(\R\times (TQ\oplus TQ))\cap \conj{U^j}$. Let
$\conj{\mathcal{C}} := \cap_{j=1}^2 \Theta^j(\mathcal{C}^j\cap
\widehat{U^j})$; by construction
$\conj{\mathcal{C}}\subset \R\times (TQ\oplus TQ)$ is open and
contains $\{0\}\times \Delta_{TQ}$. Furthermore
$(\Theta^j|_{\widehat{U^j}})^{-1}(\conj{\mathcal{C}}) \subset
\mathcal{C}^j\cap \widehat{U^j}$ is an open subset containing
$\{0\}\times \Delta_{TQ}$.


Next we consider the variational problem on $\mathcal{C}^j$ determined
by $\widehat{\dAFtp_j}$ and $\widehat{\varphi^j}$, whose critical
points $(h,v,\ti{v})$ determine the map $\widehat{\gamma^j}$. After
shrinking some of the open sets, we have $\widehat{\gamma^j}$ defined
on some open subset $V_{\widehat{\gamma^j}}\subset\R\times E$
containing $\{0\}\times Z_E$ and with values in
$\mathcal{C}^j\cap\widehat{U^j}$ so that we can define maps
\begin{equation}\label{eq:phi_bar_and_gamma_hat-def}
  \conj{\varphi^j}:=\widehat{\varphi^j}\circ
  (\Theta^j|_{\widehat{U^j}})^{-1}:\conj{\mathcal{C}}\rightarrow
  \R\times E \stext{ and }
  \conj{\gamma^j} := \Theta^j\circ \widehat{\gamma^j}.
\end{equation}
Last, we define
\begin{equation}\label{eq:sigma_bar-def}
  \conj{\dAFtp_j} := ((\Theta^j|_{\widehat{U^j}})^{-1})^{*_2}(\widehat{\dAFtp_j})
  = i_{2}^*\circ ((\Theta^j|_{\widehat{U^j}})^{-1})^*(\widehat{\dAFtp_j}),
\end{equation}
where $*_2$ and $i_2^*$ were defined in
Section~\ref{sec:remarks_on_cartesian_products} and we use, now,
associated to $X=X_1\times X_2=\R\times(TQ\times TQ)$. By
construction,
$\conj{\dAFtp_j}\in \Gamma(\conj{U^j}, (p_{23}^*T^*(TQ\times
TQ))|_{\conj{U^j}})$.

All together, we have the following commutative diagram
\begin{equation*}
  \xymatrix{
    {\R\times TQ\times TQ} \ar[r]^{\Theta^j} & {\R\times TQ\times TQ} &\\
    {\mathcal{C}^j \cap \widehat{U^j}} \ar[r]^{\Theta^j|_{\mathcal{C}^j \cap \widehat{U^j}}} \ar[dr]^{\widehat{\varphi^j}} 
    \ar@{^{(}->}[u] & {\conj{\mathcal{C}}} \ar[d]_{\conj{\varphi^j}} \ar@{^{(}->}[r] & 
    {\R\times TQ\oplus TQ} \ar@{_{(}->}[lu]\\
    {} & {\R\times E} \ar@/^1pc/[lu]^{\widehat{\gamma^j}} \ar@/_1pc/[u]_{\conj{\gamma^j}}& &
  }
\end{equation*}

\begin{lemma}\label{le:crit_problems_bar_vs_hat_are_equivalent}
  With the previous definitions,
  $\conj{\mathcal{K}}:=(\conj{\dAFtp_j}, \conj{\varphi}^j,
  \ker(T\conj{\varphi}^j))$ is a critical problem on
  $\conj{\mathcal{C}}$. Furthermore,
  $(\Theta^j|_{\mathcal{C}^j\cap \widehat{U^j}})^{-1}$ is a bijection
  between the set of critical points of this problem and those of
  $\widehat{\mathcal{K}}:=(\widehat{\dAFtp_j}, \widehat{\varphi}^j,
  \ker(T\widehat{\varphi}^j))$ on $\mathcal{C}^j\cap \widehat{U^j}$.
  Also, for each
  $(h,z_q,-z_q) \in V_{\widehat{\gamma^j}}\subset \R\times E$,
  $\conj{\gamma^j}(h,z_q,-z_q)$ is the unique critical point of
  $\conj{\mathcal{K}}$ over $(h,z_q,-z_q)$.
\end{lemma}

\begin{proof}
  As $\widehat{\varphi^j}$ are submersions
  (part~\eqref{it:comparison_crit_problems-is_critical_problem} in
  Lemma~\ref{le:comparison_crit_problems}) and $\Theta^j|_{\hat{U_j}}$
  are diffeomorphisms by construction, $\conj{\varphi^j}$ are
  submersions, hence, the first assertion in the statement is
  true. The same argument shows that
  $\ti{\mathcal{K}}:=(((\Theta^j|_{\widehat{U^j}})^{-1})^*(\widehat{\dAFtp_j}),
  \conj{\varphi}^j, \ker(T\conj{\varphi}^j))$ is a critical problem on
  $\conj{\mathcal{C}}$. As
  $\Theta_j|_{\mathcal{C}^j\cap \widehat{U^j}}$ is a diffeomorphism
  from $\mathcal{C}^j\cap \widehat{U^j}$ onto $\conj{\mathcal{C}}$
  that intertwines the corresponding data of the problems
  $\widehat{\mathcal{K}}$ and $\ti{\mathcal{K}}$, by
  Lemma~\ref{le:critical_problems_and_diffeomorphisms} (with
  $f=id_{\R\times E}$), $\Theta_j|_{\mathcal{C}^j\cap \widehat{U^j}}$
  is a bijection between the corresponding critical points.

  As $\ker(T\conj{\varphi}^j) \subset \ker(Tp_1)$ and
  $((\Theta^j|_{\widehat{U^j}})^{-1})^*(\widehat{\dAFtp_j})|_{\ker(Tp_1)}
  = \conj{\dAFtp_j}|_{\ker(Tp_1)}$, the criticality condition is the
  same for the problems $\ti{\mathcal{K}}$ and
  $\conj{\mathcal{K}}$. Hence, the problems $\ti{\mathcal{K}}$ and
  $\conj{\mathcal{K}}$ have the same critical points and the second
  assertion in the statement follows from the analysis of the previous
  paragraph. The last assertion follows immediately from the
  definition of $\widehat{\gamma^j}$ in terms of critical points of
  $\widehat{\mathcal{K}}$ and the definition of $\conj{\gamma^j}$ in
  terms of $\widehat{\gamma^j}$.
\end{proof}

Another tool that will be useful later on is the natural (left)
$\Z_2$-action on $\R\times TQ\times TQ$ defined by
\begin{equation}\label{eq:Z_2-action_RxTQxTQ-def}
  l^{\R\times TQ\times TQ}_{[1]}(h,v,\ti{v}) := (h,\ti{v},v).
\end{equation}
This action lifts to produce other $\Z_2$-actions on various bundles
over $\R\times TQ\times TQ$. More precisely, recall that when the
group $G$ acts on $X$, there are the \jdef{tangent} and
\jdef{cotangent} (left) actions defined by
\begin{equation}\label{eq:G_action_tangent_and_cotangent-def}
  l^{TX}_g(v_x) := T_xl^X_g(v_x) \stext{ and }
  l^{T^*X}_g(\alpha_x) := T^*_{l^X_g(x)}l^X_{g^{-1}}(\alpha_x) =
  \alpha_x\circ T_{l^X_g(x)} l^X_{g^{-1}},
\end{equation}
and, also, a (left) $G$-action on $\mathcal{A}^1(X)$ defined by
\begin{equation}\label{eq:G_action_on_1_forms-def}
  l^{\mathcal{A}^1(X)}_g(\gF) := (l^X_{g^{-1}})^*(\gF) =
  l^{T^*X}_g \circ \gF \circ l^X_{g^{-1}}.
\end{equation}
Notice that $\gF$ being $l^{\mathcal{A}^1(X)}$-invariant is equivalent
to saying that $\gF:X\rightarrow T^*X$ is $G$-equivariant for $l^X$
and $l^{T^*X}$.

The next goal is to find the order of contact of the critical problems
on $\conj{\mathcal{C}}$ given by
$(\conj{\dAFtp_j}, \conj{\varphi}^j, \ker(T\conj{\varphi}^j))$ for
$j=1,2$. This requires comparing $\conj{\dAFtp_2}$ to
$\conj{\dAFtp_1}$, $\ker(T\conj{\varphi}^2)$ to
$\ker(T\conj{\varphi}^1)$ and $\conj{\varphi}^2$ to $\conj{\varphi}^1$
in order to apply Theorem~\ref{th:thm_3_in_cp07} to deduce the contact
order of $\conj{\gamma^2}$ and $\conj{\gamma^1}$ and, so, that of the
corresponding flows. A first step will show us that if the contact
order of the original discretizations is $r$, then
$\conj{\gamma^2} = \conj{\gamma^1} +\mathcal{O}(\grade_{\R\times
  E}^r)$. A second step using the $\Z_2$-actions introduced above will
allow us to refine that result to obtain
$\conj{\gamma^2} = \conj{\gamma^1} +\mathcal{O}(\grade_{\R\times
  E}^{r+1})$.


\subsection{Analysis of $\conj{\dAFtp_j}$}
\label{sec:analysis_of_sigma_bar}

As before, let $\discCP^j:=(Q,\psi^j,L_\CPHS^j,f_\CPHS^j)$ be two
discretizations of the \FMS $(Q,L,f)$ such that their contact order is
$r$. Then
$L_\CPHS^2 = L_\CPHS^1 + \mathcal{O}(\grade_{\R\times
  TQ}^{r+1})$. Hence, by Proposition 2
in~\cite{ar:cuell_patrick-skew_critical_problems}\footnote{It should
  be noted that, even though the definition of contact order used
  in~\cite{ar:cuell_patrick-skew_critical_problems} is slightly
  different from the one we use here (see
  Remark~\ref{rem:no_same_contact_order_defs}), the proof of
  Proposition 2 remains valid for our definition. },
$\widehat{L_\CPS^2} = \widehat{L_\CPS^1} +
\mathcal{O}(\grade_{\R\times TQ}^{r})$, that is, there is a continuous
function $\delta \hat{L}$ such that
\begin{equation}
  \label{eq:order_condition_hat_L_cp}
  \widehat{L_\CPS^2}(h,v) - \widehat{L_\CPS^1}(h,v) =
  \grade_{\R\times TQ}^r(h,v) (\delta \hat{L})(h,v) = h^r (\delta
  \hat{L})(h,v).
\end{equation}
As we are assuming that $L_\CPHS^j$ are smooth, by Proposition 1
in~\cite{ar:cuell_patrick-skew_critical_problems},
$\widehat{L_\CPS^j}$ are smooth\footnote{If we assume that
  $L_\CPHS^j$ is $C^k$, then $\widehat{L_\CPS^j}$ are $C^{k-1}$.}.
Similarly, as
$f_\CPHS^2 = f_\CPHS^1 +\mathcal{O}(\grade_{\R\times TQ}^{r+1})$, by
Proposition 1 in~\cite{ar:cuell_patrick-skew_critical_problems}, we
have
$\widehat{f_\CPS^2} = \widehat{f_\CPS^1} +\mathcal{O}(\grade_{\R\times
  TQ}^{r})$.  Notice how in this process, one order of contact is lost
due to the division by $h$ involved in the ``hat construction''.

\begin{lemma}\label{le:order_of_sigma_2_hat}
  Let $\discCP^j$ be as above. Then the forms
  $\widehat{\dAFtp_j} \in \Gamma(\R\times TQ\times TQ,
  p_{23}^*T^*(TQ\times TQ))$ defined by\footnote{The forms
    $\widehat{\dAFop_j}$ had already been defined
    in~\eqref{eq:mu_hat_and_sigma_hat-def}. The sections
    $\widehat{\dAFtp_j}$ are slightly different from those introduced
    in~\eqref{eq:mu_hat_and_sigma_hat-def}: the latter are the
    pullback of the former to $W^{\mathcal{C}_j}$. In our current
    analysis it is important that the sections $\widehat{\dAFtp_j}$ be
    defined in an open set in $\R\times TQ\times TQ$ that is common to
    $j=1,2$. }
  \begin{equation}\label{eq:order_of_sigma_2_hat-sigma_hat_2-def}
    \begin{split}
      \widehat{\dAFtp_j} := p_{12}^* \widehat{\dAFop_j} +
      p_{13}^*\widehat{\dAFop_j} \stext{ with }
      \widehat{\dAFop_j}:=T_2 \widehat{L_\CPS^j} +
      \widehat{f^j_\CPS}\in \mathcal{A}^1(\R\times TQ)
    \end{split}
  \end{equation}
  have order $r-1$ contact, that is,
  $\widehat{\dAFtp_2} = \widehat{\dAFtp_1} +\mathcal{O}(\grade_{\R\times
    TQ\times TQ}^r)$.
\end{lemma}

\begin{proof}
  By part~\eqref{it:0_residual_implies_order+1-smooth} in
  Lemma~\ref{le:0_residual_implies_order+1}, $\delta \hat{L}$
  in~\eqref{eq:order_condition_hat_L_cp} is smooth, so that we can
  differentiate~\eqref{eq:order_condition_hat_L_cp} with respect to
  $v$ and obtain
  $T_2\widehat{L_\CPS^2}(h,v) - T_2\widehat{L_\CPS^1}(h,v) = h^r
  T_2(\delta \widehat{L})(h,v)$, with $T_2(\delta \widehat{L})$
  smooth, proving that
  $T_2\widehat{L_\CPS^2} = T_2\widehat{L_\CPS^1} +
  \mathcal{O}(\grade_{\R\times TQ}^r)$.

  As
  $\widehat{f_\CPS^2} = \widehat{f_\CPS^1}
  +\mathcal{O}(\grade_{\R\times TQ}^{r})$, by
  Lemma~\ref{le:order_of_sections_of_VB}, we have
  $\widehat{f_\CPS^2}(h,v) - \widehat{f_\CPS^1}(h,v) = h^r
  \widehat{\delta f}(h,v)$ for some smooth section
  $\widehat{\delta f}$ of $T^*(\R\times TQ\times TQ)$. Using the
  previous computations on $\widehat{\dAFtp_j}$ we have
  \begin{equation*}
    \widehat{\dAFtp_2}(h,v,\ti{v}) -
    \widehat{\dAFtp_1}(h,v,\ti{v}) = h^r (T_2(\delta \hat{L})(h,v) +
    T_2(\delta \hat{L})(h,\ti{v}) + \widehat{\delta f}(h,v) + \widehat{\delta
      f}(h,\ti{v})),
  \end{equation*}
  which, by Lemma~\ref{le:order_of_sections_of_VB}, proves that
  $\widehat{\dAFtp_2} = \widehat{\dAFtp_1}
  +\mathcal{O}(\grade_{\R\times TQ}^r)$.
\end{proof}

\begin{lemma}\label{le:order_of_sigma_2_bar}
  In the context of Lemma~\ref{le:order_of_sigma_2_hat}, if
  $\conj{\dAFtp_j}\in \Gamma(\R\times TQ\times TQ,
  p_{23}^*T^*(TQ\times TQ)))$ are defined by~\eqref{eq:sigma_bar-def},
  we have that
  $\conj{\dAFtp_2} = \conj{\dAFtp_1} + \mathcal{O}(\grade_{\R\times
    TQ\times TQ}^r)$. Furthermore, for all $v,\ti{v}\in TQ$,
  \begin{equation}\label{eq:order_of_sigma_2_bar-equal_residuals}
    \cpres^r(\conj{\dAFtp_2}, \conj{\dAFtp_1})(0,v,\ti{v}) = 
    \cpres^r(\widehat{\dAFtp_2}, \widehat{\dAFtp_1})(0,v,\ti{v})
  \end{equation}
  in
  $T_{\widehat{\dAFtp_2}(0,v,\ti{v})} (p_{23}^*T^*(TQ\times
  TQ))\big|_{h=0}$.
\end{lemma}

\begin{proof}
  By Lemma~\ref{le:order_of_sigma_2_hat},
  $\widehat{\dAFtp_2} = \widehat{\dAFtp_1} + \mathcal{O}(\grade_{\R\times
    TQ\times TQ}^r)$. As, by construction,
  $\Theta^2=\Theta^1+\mathcal{O}(\grade_{\R\times TQ\times TQ}^{r+1})$,
  Proposition 4 from~\cite{ar:cuell_patrick-skew_critical_problems}
  can be applied and we see that
  $(\Theta^2)^{-1} = (\Theta^1)^{-1} + \mathcal{O}(\grade_{\R\times
    TQ\times TQ}^{r+1})$. Then, from
  Lemma~\ref{le:order_of_pullback_map_2} (with $f^j=(\Theta^j)^{-1}$
  and $\gF^k=\widehat{\dAFtp_k}$\footnote{Indeed, for this Lemma, we
    are considering $\widehat{\dAFtp_k}$ and $\conj{\dAFtp_k}$ as
    sections over $\R\times TQ\times TQ$. In the same spirit, we
    consider $\Theta^j$ as a diffeomorphism over (an open subset of)
    $\R\times TQ\times TQ$ rather that on the $\mathcal{C}^j$
    manifolds. The reason for doing so is that by working on the same
    manifold ($\R\times TQ\times TQ$ rather than on the differing
    $\mathcal{C}^j$) we can compare orders and use residuals.}), we
  obtain that
  \begin{equation*}
    \begin{split}
      ((\Theta^2)^{-1})^{*_2}(\widehat{\dAFtp_2}) =
      ((\Theta^1)^{-1})^{*_2}(\widehat{\dAFtp_1})
      +\mathcal{O}(\grade_{\R\times TQ\times TQ}^r)
    \end{split}
  \end{equation*}
  Using~\eqref{eq:sigma_bar-def} now, we conclude that
  $\conj{\dAFtp_2} = \conj{\dAFtp_1} + \mathcal{O}(\grade_{\R\times
    TQ\times TQ}^r)$.

  Observe that
  $(\Theta^2)^{-1} = (\Theta^1)^{-1} + \mathcal{O}(\grade_{\R\times
    TQ\times TQ}^{r+1})$ implies that
  $\cpres^r((\Theta^2)^{-1},(\Theta^1)^{-1}) = 0$ and, by
  Corollary~\ref{cor:order_of_cotangent_map_2}, that
  $\cpres^r(T_2^*(\Theta^2)^{-1}, T_2^*(\Theta^1)^{-1}) = 0$. Then,
  using~\eqref{eq:order_of_pullback_map_2-residual}, we have
  \begin{equation}\label{eq:order_of_sigma_2_bar-residual-1}
    \begin{split}
      \cpres^r(\conj{\dAFtp_2}, \conj{\dAFtp_1})(0,v,\ti{v}) =&
      T(T_2^*(\Theta^1)^{-1})(\widehat{\dAFtp_1}
      ((\Theta^2)^{-1}(0,v,\ti{v})))(\cpres^r(\widehat{\dAFtp_2},
      \widehat{\dAFtp_1})
      (\underbrace{(\Theta^2)^{-1}(0,v,\ti{v})}_{=(0,v,\ti{v})})) \\=&
      T(T_2^*(\Theta^1)^{-1})(\widehat{\dAFtp_1}
      ((0,v,\ti{v})))(\cpres^r(\widehat{\dAFtp_2}, \widehat{\dAFtp_1})
      (0,v,\ti{v})).
    \end{split}
  \end{equation}
  Recall that, by construction,
  $\grade_{p_{23}^*T^*(TQ\times TQ)}\circ \widehat{\dAFtp_j} = \grade_{\R\times
    TQ\times TQ}$, so that, applying
  Lemma~\ref{le:residuals_of_maps_of_manifolds_with_a_grade}, we have
  that
  \begin{equation*}
    \cpres^r(\widehat{\dAFtp_2},
    \widehat{\dAFtp_1})(0,v,\ti{v}) \in
    T_{\widehat{\dAFtp_2}(0,v,\ti{v})} \grade_{p_{23}^*T^*(TQ\times
      TQ)}^{-1}(0).
  \end{equation*}
  Also, as
  $\widehat{\dAFtp_j}:\R\times TQ\times TQ \rightarrow
  p_{23}^*T^*(TQ\times TQ))$, we know that, by definition,
  \begin{equation*}
    \cpres^r( \widehat{\dAFtp_2},
    \widehat{\dAFtp_1})(0,v,\ti{v}) \in
    T_{\widehat{\dAFtp_2}(0,v,\ti{v})}p_{23}^*T^*(TQ\times TQ))
  \end{equation*}
  for all $(0,v,\ti{v})\in \grade_{\R\times TQ\times TQ}^{-1}(0)$. Together
  with the previous result, we have that
  \begin{equation*}
    \cpres^r( \widehat{\dAFtp_2},
    \widehat{\dAFtp_1})(0,v,\ti{v}) \in
    T_{\widehat{\dAFtp_2}(0,v,\ti{v})}(p_{23}^*T^*(TQ\times TQ))|_{h=0}.
  \end{equation*}
  
  Then, back in~\eqref{eq:order_of_sigma_2_bar-residual-1} and
  recalling that
  $\Theta^j|_{\{0\}\times TQ\times TQ} = id|_{\{0\}\times TQ\times
    TQ}$,
  \begin{equation*}
    \begin{split}
      \cpres^r(\conj{\dAFtp_2}, \conj{\dAFtp_1})(0,v,\ti{v}) =&
      T(T_2^*(\Theta^1)^{-1}|_{(p_{23}^*T^*(TQ\times
        TQ))|_{h=0}})(\widehat{\dAFtp_1}
      ((0,v,\ti{v})))(\cpres^r(\widehat{\dAFtp_2},
      \widehat{\dAFtp_1}) (0,v,\ti{v})) \\=&
      T(T_2^*(\underbrace{(\Theta^1)^{-1}|_{\{0\}\times TQ\times
          TQ}}_{=id_{\{0\}\times TQ\times
          TQ}}))(\widehat{\dAFtp_1}
      ((0,v,\ti{v})))(\cpres^r(\widehat{\dAFtp_2},
      \widehat{\dAFtp_1}) (0,v,\ti{v})) \\=&
      \cpres^r(\widehat{\dAFtp_2},
      \widehat{\dAFtp_1}) (0,v,\ti{v}).
    \end{split}
  \end{equation*}
\end{proof}

The $\Z_2$-action on $\R\times TQ\times TQ$ introduced
in~\eqref{eq:Z_2-action_RxTQxTQ-def} induces an action
$l^{T^*(\R\times TQ\times TQ)}$ on $T^*(\R\times TQ\times TQ)$
by~\eqref{eq:G_action_tangent_and_cotangent-def}; the subbundle
$p_{23}^*T^*(TQ\times TQ)\subset T^*(\R\times TQ\times TQ)$ is
$l^{T^*(\R\times TQ\times TQ)}$-invariant, so we have a $\Z_2$-action
on $p_{23}^*T^*(TQ\times TQ)$. Also, it is easy to check that the
actions induced by $T^*l^{\R\times TQ\times TQ}_{\tau^{-1}}$ and
$T_2^*l^{\R\times TQ\times TQ}_{\tau^{-1}}$ are the same on
$p_{23}^*T^*(TQ\times TQ)$; we denote any of these actions by
$l^{p_{23}^*T^*(TQ\times TQ)}$. In addition, applying the tangent lift
to this $\Z_2$-action leads to the $\Z_2$-action
$l^{T(p_{23}^*T^*(TQ\times TQ))}$ defined by
$l^{T(p_{23}^*T^*(TQ\times TQ))}_\tau = T(l^{p_{23}^*T^*(TQ\times
  TQ)}_\tau)$ on $T(p_{23}^*T^*(TQ\times TQ))$.  The next result
explores the variance properties of the different objects associated
to the ``action forms''.

Below, if $\gF\in \mathcal{A}^1(X)$ and $Z\subset X$ is a submanifold,
we write $\gF|_Z$ to denote the restriction of $\gF:X\rightarrow T^*X$
to $Z$, so that $\gF|_Z\in \Gamma(Z,T^*X)$.

\begin{lemma}\label{le:Z2_variance_of_sigma_2_bar}
  In the context of Lemmas~\ref{le:order_of_sigma_2_hat}
  and~\ref{le:order_of_sigma_2_bar}, we have that
  \begin{gather}
    l^{\mathcal{A}^1(\R\times TQ\times TQ)}_\tau(\widehat{\dAFtp_j}) =
    \widehat{\dAFtp_j} \stext{ and } \label{eq:Z2_variance_of_sigma_2_bar-hat}\\
    (l^{\mathcal{A}^1(\R\times TQ\times TQ)}_\tau
    (\conj{\dAFtp_j}))|_{\grade_{\R\times TQ\times TQ}^{-1}(0)} =
    \conj{\dAFtp_j}|_{\grade_{\R\times TQ\times
        TQ}^{-1}(0)} \label{eq:Z2_variance_of_sigma_2_bar-bar}
  \end{gather}
  for $\tau\in\Z_2$ and the residuals
  $\cpres^r(\widehat{\dAFtp_2}, \widehat{\dAFtp_1})$ and
  $\cpres^r(\conj{\dAFtp_2}, \conj{\dAFtp_1})$ are $\Z_2$-equivariant.
\end{lemma}

\begin{proof}
  In what follows, we denote $l^{\R\times TQ\times TQ}_\tau$ by
  $l_\tau$, for $\tau\in\Z_2$. It is immediate that
  $p_{13}\circ l_{[1]} = p_{12}$ and $p_{12}\circ l_{[1]} = p_{13}$ on
  $\R\times TQ\times TQ$. Then,
  from~\eqref{eq:order_of_sigma_2_hat-sigma_hat_2-def},
  \begin{equation*}
    \begin{split}
      l^{\mathcal{A}^1(\R\times TQ\times
        TQ)}_{[1]}(\widehat{\dAFtp_j}) =&
      (l_{[1]^{-1}})^*(\widehat{\dAFtp_j}) = (l_{[1]})^*(p_{12}^*
      \widehat{\dAFop_j} + p_{13}^*\widehat{\dAFop_j}) \\=&
      (l_{[1]})^*(p_{12}^* \widehat{\dAFop_j}) +
      (l_{[1]})^*(p_{13}^*\widehat{\dAFop_j}) = (p_{12}\circ
      l_{[1]})^*\widehat{\dAFop_j} + (p_{13}\circ
      l_{[1]})^*\widehat{\dAFop_j} \\=& p_{13}^* \widehat{\dAFop_j} +
      p_{12}^*\widehat{\dAFop_j} = \widehat{\dAFtp_j},
    \end{split}
  \end{equation*}
  proving~\eqref{eq:Z2_variance_of_sigma_2_bar-hat}.

  Using~\eqref{eq:Z2_variance_of_sigma_2_bar-hat}, the fact that
  $\grade_{\R\times TQ\times TQ}\circ l_\tau = \grade_{\R\times
    TQ\times TQ}$ and the fact that $l_\tau$ is a diffeomorphism, we
  apply Lemma~\ref{le:order_and_residuals_of_1_forms} to obtain
  \begin{equation*}
    \begin{split}
      \cpres^r(\widehat{\dAFtp_2},
      \widehat{\dAFtp_1})(l_{[1]}(0,v,\ti{v})) =&
      \cpres^r(\widehat{\dAFtp_2},
      \widehat{\dAFtp_1})(0,\ti{v},v) =
      \cpres^r(l_{[1]}^*(\widehat{\dAFtp_2}),
      l_{[1]}^*(\widehat{\dAFtp_1}))(0,\ti{v},v) \\=&
      T_{\widehat{\dAFtp_1}(0,v,\ti{v}))}(T^*l_{[1]})
      (\cpres^r(\widehat{\dAFtp_2},
      \widehat{\dAFtp_1})(0,v,\ti{v})) \\=&
      l^{T(p_{23}^*T^*(TQ\times
        TQ))}_{[1]}(\cpres^r(\widehat{\dAFtp_2},
      \widehat{\dAFtp_1})(0,v,\ti{v})),
    \end{split}
  \end{equation*}
  so that
  $\cpres^r(\widehat{\dAFtp_2},
  \widehat{\dAFtp_1})$ is $\Z_2$-equivariant for the
  corresponding actions.

  Next we check~\eqref{eq:Z2_variance_of_sigma_2_bar-bar}. Let
  $Z:=\{0\}\times TQ\times TQ = \grade_{\R\times TQ\times TQ}^{-1}(0)$
  and $i_Z:Z\rightarrow \R\times TQ\times TQ$ be the natural
  inclusion. As
  $TZ\subset \ker(Tp_1)|_Z\subset T(\R\times TQ\times TQ)|_Z$ and
  $\conj{\dAFtp_j}|_{\ker(Tp_1)}
  =((\Theta^j)^{-1})^*(\widehat{\dAFtp_j})|_{\ker(Tp_1)}$, we conclude
  that
  $i_Z^*(\conj{\dAFtp_j}) =
  i_Z^*(((\Theta^j)^{-1})^*(\widehat{\dAFtp_j}))$.  As
  $\Theta^j \circ i_Z = i_Z$, we have that
  $(\Theta^j)^{-1}\circ i_Z = i_Z$. Then,
  \begin{equation*}
    \begin{split}
      i_Z^*(\conj{\dAFtp_j}) =&
      i_Z^*(((\Theta^j)^{-1})^*(\widehat{\dAFtp_j})) =
      (\underbrace{((\Theta^j)^{-1})\circ
        i_Z}_{=i_Z})^*(\widehat{\dAFtp_j}) =
      i_Z^*(\widehat{\dAFtp_j}).
    \end{split}
  \end{equation*}
  As $Z$ is $l^{\R\times TQ\times TQ}$-invariant,
  $l^{\R\times TQ\times TQ}$ induces an action $l^Z$ on $Z$ such that
  $l^{\R\times TQ\times TQ} \circ i_Z = i_Z\circ l^Z$. Then
  \begin{equation*}
    \begin{split}
      l^{\mathcal{A}^1(Z)}_\tau(i_Z^*(\conj{\dAFtp_j})) =&
      (l^Z_{\tau^{-1}})^*(i_Z^*(\conj{\dAFtp_j})) =
      (l^Z_{\tau^{-1}})^*(i_Z^*(\widehat{\dAFtp_j})) \\=&
      i_Z^*((l^{\R\times TQ\times
        TQ}_{\tau^{-1}})^*(\widehat{\dAFtp_j})) =
      i_Z^*(\widehat{\dAFtp_j}) = i_Z^*(\conj{\dAFtp_j}).
    \end{split}
  \end{equation*}
  As $\ker(Tp_{23})\simeq p_1^*T\R$ is $\Z_2$-invariant and
  $\conj{\dAFtp_j}$ vanishes on $p_1^*T\R$, it follows that
  $l^{\mathcal{A}^1(Z)}_\tau(\conj{\dAFtp_j})$ also vanishes on
  $p_1^*T\R$. In other words,
  $l^{\mathcal{A}^1(Z)}_\tau(\conj{\dAFtp_j})(h,v,\ti{v})(\delta
  h,\delta v, \delta \ti{v}) =
  l^{\mathcal{A}^1(Z)}_\tau(\conj{\dAFtp_j})(h,v,\ti{v})(0,\delta v,
  \delta \ti{v})$ and, also,
  $\conj{\dAFtp_j}(h,v,\ti{v})(\delta h,\delta v, \delta \ti{v}) =
  \conj{\dAFtp_j}(h,v,\ti{v})(0,\delta v, \delta \ti{v})$. Then,
  \begin{equation*}
    \begin{split}
      l^{\mathcal{A}^1(\R\times TQ\times TQ)}_\tau(\conj{\dAFtp_j})|_Z
      =& i_Z^*(l^{\mathcal{A}^1(\R\times TQ\times
        TQ)}_\tau(\conj{\dAFtp_j})) \circ Tp_{23}|_Z \\=&
      i_Z^*((l_{\tau^{-1}})^*(\conj{\dAFtp_j})) \circ Tp_{23}|_Z =
      (l_{\tau^{-1}} \circ i_Z)^*(\conj{\dAFtp_j})) \circ Tp_{23}|_Z
      \\=& (i_Z \circ l^Z_{\tau^{-1}})^*(\conj{\dAFtp_j})) \circ
      Tp_{23}|_Z = (l^Z_{\tau^{-1}})^*(i_Z^*(\conj{\dAFtp_j}))\circ
      Tp_{23}|_Z \\=&
      l^{\mathcal{A}^1(Z)}_\tau(i_Z^*(\conj{\dAFtp_j}))\circ
      Tp_{23}|_Z = i_Z^*(\conj{\dAFtp_j})\circ Tp_{23}|_Z =
      \conj{\dAFtp_j}|_Z,
    \end{split}
  \end{equation*}
  which proves~\eqref{eq:Z2_variance_of_sigma_2_bar-bar}.
  
  Last, the $\Z_2$-equivariance of
  $\cpres^r(\conj{\dAFtp_2},\conj{\dAFtp_1})$ follows from that of
  $\cpres^r(\widehat{\dAFtp_2}, \widehat{\dAFtp_1})$
  and~\eqref{eq:order_of_sigma_2_bar-equal_residuals}, completing the
  proof.
\end{proof}


\subsection{Analysis of $\conj{\varphi}^j$}
\label{sec:analysis_of_phi_bar}

\begin{lemma}\label{le:order_of_phi_bar}
  Let $(Q,\psi^j,L_\CPHS^j,f_\CPHS^j)$ be two discretizations of the
  \FMS $(Q,L,f)$ that have order $r$ contact. Then $\conj{\varphi^2}$
  and $\conj{\varphi^1}$ defined
  in~\eqref{eq:phi_bar_and_gamma_hat-def} have order $r-1$ contact,
  that is,
  $\conj{\varphi^2} =\conj{\varphi^1} +
  \mathcal{O}(\grade_{\conj{\mathcal{C}}}^r)$. In addition, if we consider
  the $\Z_2$-action on $\R \times TQ\times TQ$ restricted to
  $\grade_{\conj{\mathcal{C}}}^{-1}(0)$,
  $\conj{\varphi^j}|_{\grade_{\conj{\mathcal{C}}}^{-1}(0)}$ and
  $\cpres^{r}(\conj{\varphi^2},
  \conj{\varphi^1}):\grade_{\conj{\mathcal{C}}}^{-1}(0)\rightarrow
  T(\R\times E)$ are $\Z_2$-invariant.
\end{lemma}

\begin{proof}
  Let $\chi^j:\conj{\mathcal{C}}\rightarrow E$ be defined by
  $\zeta^{-1}\circ p_{23}\circ g^j \circ (\Theta^j)^{-1}$, where
  $\zeta$ is the tubular map mentioned in
  Section~\ref{sec:flows_of_discretizations_of_FMS}. Explicitly,
  \begin{equation*}
    \chi^j(h,w,\ti{w}) = \zeta^{-1}
    (\del^{-,j}_h(p_2((\Theta^j)^{-1}(h,w,\ti{w}))),
    \del^{+,j}_h(p_3((\Theta^j)^{-1}(h,w,\ti{w}))).
  \end{equation*}
  Being a composition of smooth maps, it is a smooth map. For
  $(0,w,\ti{w})\in \conj{\mathcal{C}}$ we have
  \begin{equation*}
    \chi^j(0,w,\ti{w}) = \zeta^{-1}(\tau_Q(w),\tau_Q(\ti{w})) = 
    (0_{\tau_Q(w)},0_{\tau_Q(w)}), 
  \end{equation*}
  Therefore, by Proposition 1
  in~\cite{ar:cuell_patrick-skew_critical_problems}, there is
  $e:\grade_{\conj{\mathcal{C}}}^{-1}(0)\rightarrow E$ such that
  $\widehat{\chi^j}:\conj{\mathcal{C}}\rightarrow E$ defined by
  \begin{equation*}
    \widehat{\chi^j}(h,w,\ti{w}) :=
    \begin{cases}
      \frac{\chi^j(h,w,\ti{w})}{h}, \stext{ if } h\neq 0,\\
      e(0,w,\ti{w}), \stext{ if } h=0
    \end{cases}
  \end{equation*}
  is a smooth map. Notice that, by definition,
  \begin{equation*}
    \conj{\varphi^j}(h,w,\ti{w}) = 
    \widehat{\varphi^j}((\Theta^j)^{-1}(h,w,\ti{w})) =
    \begin{cases}
      (h,\frac{\chi^j(h,w,\ti{w})}{h}), \stext{ if } h\neq 0,\\
      (0,\frac{-1}{2}(w+\ti{w}),\frac{1}{2}(w+\ti{w})), \stext{ if }
      h=0.
    \end{cases}
  \end{equation*}
  As both $\widehat{\chi^j}$ and $\conj{\varphi^j}$ are smooth even at
  $h=0$ we conclude that
  $e(0,w,\ti{w}) = \frac{1}{2}(-(w+\ti{w}),(w+\ti{w}))$ and, also, that
  \begin{equation}\label{eq:order_of_phi_bar-conj_phi_and_chi_hat}
    \conj{\varphi^j}(h,w,\ti{w}) = (h,\widehat{\chi^j}(h,w,\ti{w})).
  \end{equation}

  Being $\psi^2 = \psi^1 + \mathcal{O}(t^{r+1})$, from
  Lemma~\ref{le:order_r_contact_for_borders_from_order_of_psis}, we
  have that
  $\del^{+,2} = \del^{+,1} + \mathcal{O}(\grade_{\R\times TQ\times
    TQ}^{r+1})$ and that
  $\del^{-,2} = \del^{-,1} + \mathcal{O}(\grade_{\R\times TQ\times
    TQ}^{r+1})$ and it follows immediately that
  $\del^{\mp,2} = \del^{\mp,1} + \mathcal{O}(\grade_{\R\times
    TQ\times TQ}^{r+1})$. Being $\zeta$ a diffeomorphism, it follows
  from Proposition 3
  in~\cite{ar:cuell_patrick-skew_critical_problems}, that
  $\zeta^{-1}\circ \del^{\mp,2} = \zeta^{-1}
  \circ \del^{\mp,1} + \mathcal{O}(\grade_{\R\times TQ\times
    TQ}^{r+1})$. Being
  $\Theta^2=\Theta^1 +\mathcal{O}(\grade_{\R\times TQ\times TQ}^{r+1})$, by
  Proposition 4 in~\cite{ar:cuell_patrick-skew_critical_problems}, we
  have that
  $(\Theta^2)^{-1} = (\Theta^1)^{-1} +\mathcal{O}(\grade_{\R\times TQ\times
    TQ}^{r+1})$ and, once again by Proposition 3, we have that
  $\chi^2 = \chi^1 +\mathcal{O}(\grade_{\R\times TQ\times TQ}^{r+1})$. By
  Proposition 2 in~\cite{ar:cuell_patrick-skew_critical_problems} we
  have that
  $\widehat{\chi^2} = \widehat{\chi^1} +
  \mathcal{O}(\grade_{\conj{\mathcal{C}}}^r)$. An immediate computation
  using~\eqref{eq:order_of_phi_bar-conj_phi_and_chi_hat} now leads to
  $ \conj{\varphi^2} = \conj{\varphi^1} +
  \mathcal{O}(\grade_{\conj{\mathcal{C}}}^r)$, as wanted.

  Notice that for $(0,w,\ti{w}) \in \grade_{\conj{\mathcal{C}}}^{-1}(0)$,
  \begin{equation}\label{eq:phi_bar_when_h=0}
    \conj{\varphi^j}(0,w,\ti{w}) =
    \widehat{\varphi^j}((\Theta^j)^{-1}(0,w,\ti{w})) =
    \widehat{\varphi^j}(0,w,\ti{w}) =
    (0,-\frac{1}{2}(w+\ti{w}),\frac{1}{2}(w+\ti{w}))
  \end{equation}
  that is invariant under interchanging $w$ with $\ti{w}$, so that
  $\conj{\varphi^j}|_{\grade_{\conj{\mathcal{C}}}^{-1}(0)}$ is
  $\Z_2$-invariant.  The last assertion of the statement is contained
  in the last part of the proof of Theorem 4.7
  in~\cite{ar:patrick_cuell-error_analysis_of_variational_integrators_of_unconstrained_lagrangian_systems}.
\end{proof}
  

\subsection{Analysis of $\ker(T\conj{\varphi}^j)$}
\label{sec:analysis_of_D_bar}

We first state a few general results and, then, specialize the
considerations to the context of discretizations.

It is common for a critical problem $(\gAF,g,\mathcal{D})$ that
$\mathcal{D}=\ker(Tg)$, which is a regular distribution because $g$ is
a submersion. The following result explores some of the properties of
systems of this particular form.
\begin{lemma}\label{le:skew_problem_with_D_from_g}
  Let $g:M\rightarrow N$ be a submersion such that
  $g:(M,\grade_M)\rightarrow (N,\grade_N)$ is a map of manifolds with
  a grade. Then,
  \begin{enumerate}
  \item \label{it:skew_problem_with_D_from_g-contained}
    $\mathcal{D}|_{\grade_M^{-1}(0)}\subset T \grade_M^{-1}(0)$.
  \item \label{it:skew_problem_with_D_from_g-invariant} If the Lie
    group $G$ acts smoothly on $M$ and $N$ in such a way that both
    $\grade_M^{-1}(0)$ is $G$-invariant and $g|_{\grade_M^{-1}(0)}$ is
    $G$-equivariant, then $\mathcal{D}|_{\grade_M^{-1}(0)}$ is
    $G$-invariant. In addition, if $M_0\subset \grade_M^{-1}(0)$ is a
    $G$-invariant submanifold, $\mathcal{D}|_{M_0}$ is also
    $G$-invariant.
  \end{enumerate}
\end{lemma}

\begin{proof}
  Let $\delta m\in \mathcal{D}_{m}$ for some $m\in \grade_M^{-1}(0)$;
  then, if $n:=g(m)$, as $g$ is a submersion, $g^{-1}(n)\subset M$ is
  a regular submanifold, $m \in g^{-1}(n)$ and
  $T_m g^{-1}(n) = \ker(Tg(m)) = \mathcal{D}_m$. On the other hand, as
  $\grade_M = \grade_N \circ g$ and $\grade_M(m)=0$, we have that
  $\grade_N(n)=0$ and $g^{-1}(n)\subset \grade_M^{-1}(0)$. All
  together,
  $\delta m \in \mathcal{D}_m = T_m g^{-1}(n) \subset T
  \grade_M^{-1}(0)$, proving
  point~\eqref{it:skew_problem_with_D_from_g-contained}.

  In the context of
  point~\eqref{it:skew_problem_with_D_from_g-invariant} by the
  $G$-invariance of $\grade_M^{-1}(0)$ and the $G$-equivariance of
  $g|_{\grade_M^{-1}(0)}$ we have that
  $g\circ l^M_\tau|_{\grade_M^{-1}(0)} = l^N_\tau \circ
  g|_{\grade_M^{-1}(0)}$. Taking derivatives at $m\in \grade_M^{-1}(0)$ we
  obtain
  \begin{equation}\label{eq:skew_problem_with_D_from_g-derivative}
    T_{l^M_\tau(m)}g \circ T_ml^M_\tau|_{T_m \grade_M^{-1}(0)} = T_{g(m)}
    l^N_\tau \circ T_mg|_{T_m \grade_M^{-1}(0)}.
  \end{equation}
  Assume that $\delta m \in \mathcal{D}_m = \ker(T_mg)$ for
  $m\in \grade_M^{-1}(0)$ so that, by
  point~\eqref{it:skew_problem_with_D_from_g-contained}
  $\delta m \in
  T_m\grade_M^{-1}(0)$. Applying~\eqref{eq:skew_problem_with_D_from_g-derivative}
  to $\delta m$ we see that
  $l^{TM}_\tau(\delta m) = T_ml^M_\tau(\delta m) \in \ker(T_{l^M_\tau(m)}
  g) = \mathcal{D}_{l^M_\tau(m)}\subset
  \mathcal{D}|_{\grade_M^{-1}(0)}$, proving the $G$-invariance of
  $\mathcal{D}|_{\grade_M^{-1}(0)}$. Last, if $M_0\subset \grade_M^{-1}(0)$ is
  $G$-invariant and $\delta m \in \mathcal{D}_{m_0}$ for some
  $m_0\in M_0$, we have that
  $l^{TM}_\tau(\delta m) \in \mathcal{D}_{l^M_\tau(m_0)} \subset
  \mathcal{D}|_{M_0}$, completing the proof of
  point~\eqref{it:skew_problem_with_D_from_g-invariant}.
\end{proof}

\begin{lemma}\label{le:skew_problem_with_D_from_tig}
  Let $\ti{M},\ti{N}$ be manifolds and define $M:=\R\times \ti{M}$ and
  $N:=\R\times \ti{N}$ with the submersions $\grade_M:=p_1$ and
  $\grade_N:=p_1$. Given a smooth map $\ti{g}:M\rightarrow \ti{N}$ we
  define $g:M\rightarrow N$ by
  $g(h,\ti{m}):=(h,\ti{g}(h,\ti{m}))$. Assume that $g$ is a submersion
  (equivalently, $\ti{g}|_{\grade_M^{-1}(h_0)}$ is a submersion for each
  $h_0\in\R$). Then,
  \begin{enumerate}
  \item \label{it:skew_problem_with_D_from_tig-distribution}
    $\mathcal{D}:=\ker(Tg)$ is a smooth regular distribution on $M$
    and
    \begin{equation}\label{eq:skew_problem_with_D_from_tig-D_def}
      \mathcal{D}_{(h,\ti{m})} = 
      \{(0,\delta{\ti{m}}) \in T_{(h,\ti{m})} M: 
      \delta\ti{m} \in \ker(T_2\ti{g}(h,\ti{m}))\} \subset T_{(h,\ti{m})}M.
    \end{equation}
  \item \label{it:skew_problem_with_D_from_tig-contained}
    $\mathcal{D}|_{\grade_M^{-1}(0)}\subset T \grade_M^{-1}(0)$.
  \item \label{it:skew_problem_with_D_from_tig-invariant} Still in the
    context of
    point~\eqref{it:skew_problem_with_D_from_tig-contained}, if the
    Lie group $G$ acts smoothly on $\ti{M}$ and $\ti{N}$, define a
    $G$-action on $M$ by
    $l^M_\tau(h,\ti{m}) := (h,l^{\ti{M}}_\tau(\ti{m}))$. Assume that
    $\ti{g}|_{\grade_M^{-1}(0)}:\ti{M}\rightarrow \ti{N}$ is
    $G$-equivariant; then $\mathcal{D}|_{\grade_M^{-1}(0)}$ is
    $G$-invariant (for the $G$-action $l^{TM}$). In addition, if
    $\ti{M}_0\subset \ti{M}$ is a $G$-invariant submanifold, then
    $M_0:=\{0\}\times \ti{M}_0 \subset M$ and $\mathcal{D}|_{M_0}$ are
    $G$-invariant.
  \end{enumerate}
\end{lemma}

\begin{proof}
  By definition, we have that $\grade_N\circ g = \grade_M$, so that
  $g:(M,\grade_M)\rightarrow (N,\grade_N)$ is a map of manifolds with a grade.
  Again by definition, $\grade_M^{-1}(0)=\{0\}\times \ti{M}$ is
  $G$-invariant and, by hypothesis,
  \begin{equation*}
    \begin{split}
      g|_{\grade_M^{-1}(0)}(l^M_\tau(0,\ti{m})) =&
      g(0,l^{\ti{M}}_\tau(\ti{m})) =
      (0,\ti{g}(0,l^{\ti{M}}_\tau(\ti{m}))) =
      (0,l^{\ti{N}}_\tau(\ti{g}(0,\ti{m}))) \\=&
      l^N_\tau(0,(\ti{g}(0,\ti{m}))) = l^N_\tau(g(0,\ti{m})) =
      l^N_\tau(g|_{\grade_M^{-1}(0)}(0,\ti{m})),
    \end{split}
  \end{equation*}
  so that $g|_{\grade_M^{-1}(0)}$ is $G$-equivariant. Last, if
  $M_0:=\{0\}\times \ti{M}_0$ with $\ti{M}_0\subset \ti{M}$ a
  $G$-invariant submanifold, it is clear that $M_0$ is
  $G$-invariant. Then, the results follow directly from
  Lemma~\ref{le:skew_problem_with_D_from_g}. The
  expression~\eqref{eq:skew_problem_with_D_from_tig-D_def} is obtained
  from the Cartesian product structure of $M$ and $N$ and the
  corresponding splitting of the tangent spaces:
  \begin{equation*}
    \begin{split}
      \ker(Tg(h,\ti{m})) =& \{ (\delta h,\delta \ti{m}) \in
      T_{(h,\ti{m})}M : Tg(h,\ti{m})(\delta h,\delta \ti{m}) = (0,0)\}
      \\=& \{ (0,\delta \ti{m}) \in T_{(h,\ti{m})}M :
      T_2\ti{g}(h,\ti{m})(\delta \ti{m}) = 0\}.
    \end{split}
  \end{equation*}
\end{proof}

\begin{remark}\label{rem:skew_problem_with_D_from_tig-open}
  The results of Lemma~\ref{le:skew_problem_with_D_from_tig} remain
  valid when $M\subset \R\times \ti{M}$ is a $G$-invariant open subset
  and $\grade_M:=p_1|_M$. Also, $M_0:=\{0\}\times \ti{M}_0\subset M$ must
  be required.
\end{remark}

\begin{lemma}\label{le:contact_order_of_D_from_tig}
  Let $\ti{M},\ti{N}$ be manifolds and $M\subset \R\times \ti{M}$ an
  open subset; let $N:=\R\times \ti{N}$. Consider the grades
  $\grade_M:=p_1|_M$ and $\grade_N:=p_1$. For $j=1,2$ let $g_j:M\rightarrow N$
  be submersions such that such that $\grade_N\circ g_j = \grade_M$ and so that
  $g_2=g_1+\mathcal{O}(\grade_M^r)$. Then the distributions
  $\mathcal{D}^j := \ker(T g_j)$ over $M$ satisfy
  $\mathcal{D}^2=\mathcal{D}^1 + \mathcal{O}(\grade_M^r)$.
\end{lemma}

We prove Lemma~\ref{le:contact_order_of_D_from_tig} by a local
computation in Section~\ref{sec:computation_of_residual_of_i_D_bar}.

We can specialize the previous result to the case of discretizations,
where $\ti{M}:=TQ\oplus TQ$,
$M:=\conj{\mathcal{C}}\subset\R\times (TQ\oplus TQ)$, $\ti{N}:=E$ and
$g_j := \conj{\varphi^j}$. By Lemma~\ref{le:order_of_phi_bar}
$\conj{\varphi}^2 = \conj{\varphi}^1 +
\mathcal{O}(\grade_{\conj{\mathcal{C}}}^r)$.

\begin{lemma}\label{le:contact_order_of_ker_T_bar_phi}
  The distributions $\conj{\mathcal{D}}^j:=\ker(T\conj{\varphi}^j)$
  over $\conj{\mathcal{C}}$ satisfy
  $\conj{\mathcal{D}}^2 = \conj{\mathcal{D}}^1 +
  \mathcal{O}(\grade_{\conj{\mathcal{C}}}^r)$.
\end{lemma}

\begin{prop}\label{prop:residual_i_D_bar_is_Z2_equivariant}
  Considering the $\Z_2$-actions $l^{\conj{\mathcal{C}}}$ and
  $l^{T\GrB} := Tl^{\GrB}$,
  $\cpres^r(i_{\conj{\mathcal{D}^2}},i_{\conj{\mathcal{D}^1}}):
  \grade_{\conj{\mathcal{C}}}^{-1}(0)\rightarrow T\GrB$ is
  $\Z_2$-equivariant.
\end{prop}

The proof is given in
Section~\ref{sec:equivariance_of_residual_kernel_distributions}.


\subsection{Analysis of $\conj{\gamma^j}$}
\label{sec:analysis_of_bar_gamma}

The previous sections provide us with contact estimates of
$\conj{\dAFtp_j}$, $\conj{\varphi^j}$ and $\ker(T\conj{\varphi^j})$. These
are the input for the contact estimates for $\conj{\gamma^j}$ that we
obtain in this section.

In the proof of the existence of a flow for a discretization of a
\FDMSCP we first considered the ``degenerate case'' where $h=0$ and
assigned a unique nondegenerate critical point to each ``boundary
value'', defining a function $\gamma_0$. Then, we extended $\gamma_0$
to a function defined for arbitrary (small) values of $h$ using
Theorem~\ref{th:thm_2_in_cp07}.

Let $(M,\grade_M)$ and $(N,\grade_N)$ be manifolds with grades. Say
that $(\gAF_j,g_j,\mathcal{D}_j)$ for $j=1,2$ are critical problems
over $M$ and $g_j$ have values in $N$ and we assume that
$\gAF_2=\gAF_1+\mathcal{O}(\grade_M^r)$,
$g_2=g_1+\mathcal{O}(\grade_M^r)$ and
$\mathcal{D}_2=\mathcal{D}_1+\mathcal{O}(\grade_M^r)$. In addition,
say that the corresponding $\gamma_j:V_j\rightarrow U_j$ have been
constructed as described above using Theorem~\ref{th:thm_2_in_cp07} as
extensions of the same $\gamma_0:N_0\rightarrow M_0$ for both problems
and where $M_0\subset \grade_M^{-1}(0)$ and
$N_0\subset \grade_N^{-1}(0)$. We will, furthermore, assume that
$N_0 = \grade_N^{-1}(0)$. By Theorem~\ref{th:thm_3_in_cp07} we have
that $\gamma_2 = \gamma_1 + \mathcal{O}(\grade_N^r)$.  Let
$V_\gamma := V_1\cap V_2$, that is an open subset in $N$ and
$V_\gamma^0:=V_\gamma\cap \grade_N^{-1}(0)$, that is an open
submanifold of $\grade_N^{-1}(0)$. In what follows and in order to
simplify the notation, we assume that $N=V_\gamma$, so that
$V_\gamma^0 = \grade_N^{-1}(0) = N_0$ (that is, the domain of both
$\gamma_j$ is all of $N$, which can be achieved by shrinking $N$ as
needed). The following result is a variation of Proposition 7
in~\cite{ar:cuell_patrick-skew_critical_problems}.

\begin{prop}\label{prop:invariance_of_res_gamma}
  In the context of the previous paragraph, let the Lie group $G$ act
  on $M$ preserving the grade, that is, $\grade_M \circ l^M_\tau = \grade_M$ for
  all $\tau \in G$. Assume that
  \begin{enumerate}[label=(\alph*)]
  \item \label{it:invariance_of_res_gamma-g_j_and_grades}
    $\grade_N\circ g_j = \grade_M$ for $j=1,2$,
  \item \label{it:invariance_of_res_gamma-D_j_invariance} for $j=1,2$,
    $\mathcal{D}_j|_{M_0}\subset T \grade_M^{-1}(0)$ and that $M_0$
    and $\mathcal{D}_j|_{\grade_M^{-1}(0)}$ are $G$-invariant subsets,
    while, for all $\tau\in G$,
    \begin{equation}
      \label{eq:invariance_of_res_gamma-invariance}
      l^{\mathcal{A}^1(M)}_\tau(\gAF_j)|_{\grade_M^{-1}(0)} =
      \gAF_j|_{\grade_M^{-1}(0)} \stext{ and }
      g_j\circ l^M_\tau |_{\grade_M^{-1}(0)} = g_j|_{\grade_M^{-1}(0)}.
    \end{equation}
  \item \label{it:invariance_of_res_gamma-alphas_fixed_at_critical_pts}
    for each $m_0\in M_0$, $\tau\in G$ and $j=1,2$,
    $((l^M_\tau)^*(\gAF_j))(m_0) = \gAF_j(m_0)$ in $T^*_{m_0}M$,
  \item \label{it:invariance_of_res_gamma-residuals_variance}
    $\cpres^r(\gAF_2,\gAF_1)$ and
    $\cpres^r(\mathcal{D}_2,\mathcal{D}_1)$ are $G$-equivariant while
    $\cpres^r(g_2,g_1)$ is $G$-invariant.
  \end{enumerate}
  Then, 
  \begin{enumerate}
  \item   
    \begin{equation}
      \label{eq:gamma_preserves_grades}
      \grade_M \circ \gamma_j = \grade_N \stext{ for } j=1,2,
    \end{equation}
  \item \label{it:invariance_of_res_gamma-cpres_gamma_contained}
    $\cpres^r(\gamma_2,\gamma_1)(\grade_N^{-1}(0)) \subset T(\grade_M^{-1}(0))$,
    and
  \item \label{it:invariance_of_res_gamma-cpres_gamma_invariant}
    $\cpres^r(\gamma_2,\gamma_1)(n_0)\in
    T_{\gamma_2(n_0)}(\grade_M^{-1}(0))$ is a $G$-invariant vector for all
    $n_0\in \grade_N^{-1}(0)$.
  \end{enumerate}
\end{prop}

\begin{proof}
  \begin{enumerate}[label=(\roman*)]
  \item As $\grade_M\circ l^M_\tau = \grade_M$ for each $\tau\in G$, the
    embedded submanifold $\grade_M^{-1}(0)\subset M$ is $G$-invariant and,
    for $\tau\in G$ and $m\in M$, $\dot{l^M_\tau}(m)=1$ (see paragraph
    immediately before Proposition 3
    in~\cite{ar:cuell_patrick-skew_critical_problems} for the
    definitions).

  \item As $g_j\circ \gamma_j= id_N$, we have
    $\grade_N = \grade_N\circ g_j\circ \gamma_j = \grade_M\circ \gamma_j$, which
    proves~\eqref{eq:gamma_preserves_grades}.  Since, by
    Theorem~\ref{th:thm_3_in_cp07}, we also have that
    $\gamma_2 = \gamma_1 +\mathcal{O}(\grade_N^r)$, by
    Lemma~\ref{le:residuals_of_maps_of_manifolds_with_a_grade}, we
    conclude that
    $\cpres^r(\gamma_2,\gamma_1)(n) \in T_{\gamma_2(n)} \grade_M^{-1}(0)$
    for all $n\in \grade_N^{-1}(0)$, so that
    point~\eqref{it:invariance_of_res_gamma-cpres_gamma_contained} in
    the statement is true.

  \item In what follows, we fix $\tau\in G$ and define, for $j=1,2$,
    \begin{equation*}
      \ti{\gAF}_j := (l^M_{\tau^{-1}})^*(\gAF_j), \quad 
      \ti{g}_j:= g_j \circ l^M_{\tau^{-1}}, \quad 
      \ti{\mathcal{D}}_j := l^{TM}_\tau(\mathcal{D}_j).
    \end{equation*}

    Recall that, for $n\in N$, $m_j=\gamma_j(n)$ if $m_j\in M$ is the
    unique critical point of the problem
    $(\gAF_j, g_j, \mathcal{D}_j)$ over $n\in N$ contained in $U_j$
    (see paragraph before the Proposition). Then, if we define
    $\ti{m}_j:= l^M_{\tau}(m_j) \in l^M_\tau(U_j)$, we see that
    \begin{equation*}
      \ti{g}_j(\ti{m}_j) = 
      g_j(l^M_{\tau^{-1}}(l^M_{\tau}(m_j))) = g_j(m_j) = n
    \end{equation*}
    and
    \begin{equation*}
      \begin{split}
        \ti{\gAF}_j(\ti{\mathcal{D}}_j|_{\ti{m}_j}) =
        (l^M_{\tau^{-1}})^*(\gAF_j)(l^{TM}_\tau(\mathcal{D}_j|_{m_j}))
        = \gAF_j(\underbrace{Tl^M_{\tau^{-1}}}_{=l^{TM}_{\tau^{-1}}}
        (l^{TM}_\tau(\mathcal{D}_j|_{m_j}))) =
        \gAF_j(\mathcal{D}_j|_{m_j}).
      \end{split}
    \end{equation*}
    Then, if $m_j$ is a critical point of $(\gAF_j,g_j,\mathcal{D}_j)$
    over $n$, we see that
    $\ti{\gAF}_j(\ti{\mathcal{D}}_j|_{\ti{m}_j}) =
    \gAF_j(\mathcal{D}_j|_{m_j})=0$, so that
    $\ti{m}_j\in l^M_\tau(U_j)$ is a critical point of
    $(\ti{\gAF}_j,\ti{g}_j,\ti{\mathcal{D}}_j)$ over $n$. Thus,
    \begin{equation}\label{eq:equivariance_residuals-gammas_and_ti_gammas}
      \ti{\gamma}_j := l^M_\tau \circ \gamma_j.
    \end{equation}
    assigns the (unique) critical points of
    $(\ti{\gAF}_j,\ti{g}_j,\ti{\mathcal{D}}_j)$ contained in
    $l^M_\tau(U_j)$.

  \item \label{it:equivariance_residuals-gammas_restricted_to_h=0}
    Next, we consider $\gamma_j|_{\grade_N^{-1}(0)}$. If
    $n\in \grade_N^{-1}(0)$, let $m := \gamma_j(n) \in M_0$ (here we use
    that $N_0=\grade_N^{-1}(0)$ and that $M_0 = \gamma_j(N_0)$) and
    $\ti{m} := \ti{\gamma}_j(n) = l^M_\tau(m) \in M_0$ (by the
    $G$-invariance of $M_0$); notice that $\gamma_2(n)=\gamma_1(n)$
    since, by Theorem~\ref{th:thm_3_in_cp07},
    $\gamma_2=\gamma_1+\mathcal{O}(\grade_N^r)$; also,
    applying~\eqref{eq:equivariance_residuals-gammas_and_ti_gammas},
    it follows that $\ti{\gamma}_2(n) = \ti{\gamma}_1(n)$.

    As $\gamma_j(\grade_N^{-1}(0))\subset \grade_M^{-1}(0)$ and
    $n\in \grade_N^{-1}(0)$,
    using~\eqref{eq:invariance_of_res_gamma-invariance}, we have
    \begin{equation*}
      g_j(\ti{m}) = g_j(\ti{\gamma}_j(n)) = 
      g_j(\underbrace{l^M_\tau(\gamma_j(n))}_{\in \grade_M^{-1}(0)}) =
      g_j(\gamma_j(n)) = n.
    \end{equation*}
    
    Notice that, as $\grade_M^{-1}(0)$ and $\mathcal{D}_j|_{\grade_M^{-1}(0)}$
    are $G$-invariant,
    \begin{equation*}
      (l^{TM}_\tau(\mathcal{D}_j))|_{\grade_M^{-1}(0)} =
      l^{TM}_\tau(\mathcal{D}_j|_{\grade_M^{-1}(0)}) \subset
      \mathcal{D}_j|_{\grade_M^{-1}(0)},
    \end{equation*}
    hence, considering ranks,
    \begin{equation}\label{eq:equality_of_distributions_on_h=0}
      \ti{\mathcal{D}}_j|_{\grade_M^{-1}(0)} = (l^{TM}_\tau(\mathcal{D}_j))|_{\grade_M^{-1}(0)} = 
      \mathcal{D}_j|_{\grade_M^{-1}(0)}.
    \end{equation}
    Using this last result
    and~\eqref{eq:invariance_of_res_gamma-invariance} we compute
    \begin{equation*}
      \begin{split}
        \gAF_j(\ti{m})(\mathcal{D}_j|_{\ti{m}}) =&
        (l^M_{\tau^{-1}})^*(\gAF_j)(\ti{m})(\mathcal{D}_j|_{\ti{m}}) =
        (l^M_{\tau^{-1}})^*(\gAF_j)(\ti{m})(\ti{\mathcal{D}}_j|_{\ti{m}})
        = \ti{\gAF}_j(\ti{m})(\ti{\mathcal{D}}_j|_{\ti{m}}) = 0
      \end{split}
    \end{equation*}
    because $\ti{m}$ is a critical point of
    $(\ti{\gAF}_j,\ti{g}_j,\ti{\mathcal{D}}_j)$.

    Putting together the last two analyses, we conclude that
    $\ti{m}\in M_0$ is also a critical point of
    $(\gAF_j,g_j,\mathcal{D}_j)$ over $n$. As critical points of
    $(\gAF_j,g_j,\mathcal{D}_j)$ over $n$ are unique in $M_0$ (they
    are unique in the open set $U_j\supset M_0$ where the image of
    $\gamma_j$ is contained, see Theorem~\ref{th:thm_2_in_cp07}), we
    have that $\ti{m}=m$ and, so,
    $\ti{\gamma}_j(n) = \ti{m} = \gamma_j(n)$ and, consequently,
    \begin{equation}
      \label{eq:equivariance_residuals-gamma_and_ti_gamma_coincide}
      \ti{\gamma}_j|_{\grade_N^{-1}(0)} = \gamma_j|_{\grade_N^{-1}(0)}.
    \end{equation}
    It also follows that $l^M_\tau(m) = m$ for all $m\in M_0$.
    
 \item Next we do a first comparison of residuals of $\gamma_j$ and
    $\ti{\gamma}_j$. If $n\in \grade_N^{-1}(0)$, using (12)
    from~\cite{ar:cuell_patrick-skew_critical_problems},
    \begin{equation*}
      \begin{split}
        \cpres^r(\ti{\gamma}_2,\ti{\gamma}_1)(n) =& \cpres^r(l^M_\tau
        \circ \gamma_2, l^M_\tau \circ \gamma_1)(n) \\=&
        T_{\gamma_2(n)}l^M_\tau(\cpres^r (\gamma_2, \gamma_1)(n)) =
        l^{TM}_\tau(\cpres^r (\gamma_2, \gamma_1)(n)).
      \end{split}
    \end{equation*}

  \item \label{it:equivariance_residuals-Gs_and_Fs} The next step is
    to recall the implicit characterization of
    $\cpres^r(\gamma_2,\gamma_1)$ at a critical point given by (13)
    and (14) in~\cite{ar:cuell_patrick-skew_critical_problems}. When
    $n_0\in \grade_N^{-1}(0)$ and $m_0:=\gamma_2(n_0)\in M_0$,
    $u_0:=\cpres^r(\gamma_2, \gamma_1)(n_0)\in T_{m_0}M$ is the unique
    solution of the equation $F(u)=0$ constrained by $G(u)=0$ for
    $F:T_{m_0}M\rightarrow (\mathcal{D}_1)_{m_0}^*$ and
    $G:T_{m_0}M\rightarrow T_{n_0}N$ given by\footnote{In the original
      definition of (13)
      in~\cite{ar:cuell_patrick-skew_critical_problems} the last term
      appears without a bar on top of the residual but, in order to
      make sense of the composition,
      $\cpres^r(\mathcal{D}_2,\mathcal{D}_1)(m_0)$ must be taken as an
      element of
      $\hom((\mathcal{D}_2)_{m_0},T_{m_0}M/(\mathcal{D}_2)_{m_0})$,
      which we denote by putting a bar on top (see
      Section~\ref{sec:grassmannian-tangent_spaces}). In addition,
      by~\eqref{eq:equality_of_distributions_on_h=0},
      $(\mathcal{D}_1)_{m_0} = (\mathcal{D}_2)_{m_0}$ if $m_0\in M_0$
      (so that $\conj{\gAF_1(m_0)}$ is well defined), so that the
      composition is well defined.}
    \begin{gather*}
      F(u) := \dot{\gamma_2}(n_0) (d_{\mathcal{D}_1}\gAF_1)(m_0)(u)+
      \cpres^r(\gAF_2,\gAF_1)(m_0) +
      \conj{\gAF_1(m_0)}\circ \conj{\cpres^r(\mathcal{D}_2,\mathcal{D}_1)(m_0)}\\
      G(u) := T_{m_0} g_1(u) + \cpres^r(g_2,g_1)(m_0).
    \end{gather*}
    We want to compare the conditions $F(u)=0$ and $G(u)=0$ with
    the conditions $\ti{F}(u)=0$ and $\ti{G}(u)=0$ corresponding to
    the characterization of
    $\ti{u}_0:=\cpres^r(\ti{\gamma}_2,\ti{\gamma}_1)(n_0)$. Notice
    that, as $n_0\in N_0=\grade_N^{-1}(0)$, $m_0,\ti{m}_0 \in M_0$, and
    by~\eqref{eq:equivariance_residuals-gamma_and_ti_gamma_coincide},
    $m_0 = \gamma_j(n_0) = \ti{\gamma}_j(n_0) = \ti{m}_0$, so that
    $u_0, \ti{u}_0 \in T_{m_0} \grade_M^{-1}(0)\subset T_{m_0}M$.

    \begin{enumerate}[label=(\Alph*)]
    \item \label{it:equivariance_residuals-Gs_and_Fs-F} Analysis of
      $\ti{F}$. Fix $m_0\in M_0\subset \grade_M^{-1}(0)$ as above,
      $u\in T_{m_0}\grade_M^{-1}(0)$ and $v\in (\mathcal{D}_1)_{m_0}$.

      \begin{enumerate}[label=(\Roman*)]
      \item First term in $\ti{F}$. For $m_0\in M$ and
        $u,v \in \ker(T_{m_0}g_1)$ we compute
        $d_{\mathcal{D}_1}\gAF_1(m_0)(u)(v)$
        using~\eqref{eq:skew_hessian-def}. In our case,
        $m_0\in \grade_M^{-1}(0)$ and, if
        $u\in T_{m_0}\grade_M^{-1}(0)$ (as is the case when
        $u:=\cpres(\gamma_2,\gamma_1)(g_1(m_0))$), we have
          \begin{equation*}
            \begin{split}
              d_{\mathcal{D}_1}\gAF_1(m_0)(u)(v) =&
              d(\gAF_1(V))(m_0)(u) =
              i_{\grade_M^{-1}(0),M}^*(d(\gAF_1(V)))(m_0)(u) \\=&
              d(i_{\grade_M^{-1}(0),M}^*(\gAF_1(V)))(m_0)(u),
            \end{split}
          \end{equation*}
          where $V\in \VF(M)$ is such that
          $V(m') \in (\mathcal{D}_1)_{m'}$ for all $m'\in M$ and
          $V(m_0)=v$. In addition, as
          $V|_{\grade_M^{-1}(0)} \in
          \mathcal{D}_1|_{\grade_M^{-1}(0)}\subset T\grade_M^{-1}(0)$,
          we have that
          $i_{\grade_M^{-1}(0),M}^*(\gAF_1(V)) =
          i_{\grade_M^{-1}(0),M}^*(\gAF_1)(V|_{\grade_M^{-1}(0)})$.
          As $T\grade_M^{-1}(0)\subset TM|_{\grade_M^{-1}(0)}$, from
          the first identity
          in~\eqref{eq:invariance_of_res_gamma-invariance},
          \begin{equation*}
            i_{\grade_M^{-1}(0),M}^*((l^M_{\tau^{-1}})^*(\gAF_1)) =
            i_{\grade_M^{-1}(0),M}^*(\gAF_1), 
          \end{equation*}
          so that
          $i_{\grade_M^{-1}(0),M}^*(\ti{\gAF}_1) =
          i_{\grade_M^{-1}(0),M}^*(\gAF_1)$ and, then,
          \begin{equation*}
            \begin{split}
              i_{\grade_M^{-1}(0),M}^*(\ti{\gAF}_1)(V|_{\grade_M^{-1}(0)})
              =
              i_{\grade_M^{-1}(0),M}^*(\gAF_1)(V|_{\grade_M^{-1}(0)}).
            \end{split}
          \end{equation*}
          As the computation of $d_{\mathcal{D}_1}\gAF_1(m_0)(u)(v)$
          is a local matter near $m_0\in \grade_M^{-1}(0)$ we only
          need a local extension $V$ of $v$ with values in
          $\mathcal{D}_1$ defined near $m_0$. But then, as
          $\mathcal{D}_1|_{\grade_M^{-1}(0)} =
          \ti{\mathcal{D}}_1|_{\grade_M^{-1}(0)}$
          by~\eqref{eq:equality_of_distributions_on_h=0}, we can find
          a local extension $\ti{V}$ of $v$ defined near $m_0$ and
          with values in $\ti{\mathcal{D}}_1$ such that
          $V|_{\grade_M^{-1}(0)} = \ti{V}|_{\grade_M^{-1}(0)}$. All
          together, for $m_0\in \grade_M^{-1}(0)$ and
          $u\in T\grade_M^{-1}(0)$, we have
          \begin{equation*}
            \begin{split}
              d_{\mathcal{D}_1}\gAF_1(m_0)(u)(v) =&
              d(i_{\grade_M^{-1}(0),M}^*(\gAF_1)(V|_{\grade_M^{-1}(0)}))(m_0)(u)
              \\=&
              d(i_{\grade_M^{-1}(0),M}^*(\ti{\gAF}_1)(V|_{\grade_M^{-1}(0)}))(m_0)(u)
              \\=&
              d(i_{\grade_M^{-1}(0),M}^*(\ti{\gAF}_1)(\ti{V}|_{\grade_M^{-1}(0)}))(m_0)(u)
              = d_{\ti{\mathcal{D}}_1}\ti{\gAF}_1(m_0)(u)(v).
            \end{split}
          \end{equation*}
          Notice also that, as $\grade_M\circ \gamma_j = \grade_N$, we
          have $\dot{\gamma_j}(m)=1$ for all $m\in M$. For the same
          reason, $\dot{\ti{\gamma}_j}(m)=1$. We conclude that the
          first terms of $F(u)$ and $\ti{F}(u)$ coincide.
        
      \item Second term in $\ti{F}$. First, we have to understand how,
        when $m'\in \grade_M^{-1}(0) \subset M$,
        $\cpres^r(\gAF_2,\gAF_1)(m')$ is an element of
        $\mathcal{D}_{m'}^*\subset T_{m'}^* M$: as noted in
        Remark~\ref{rem:vertical_vectors_and_base_vectors_in_tangent_bundle},
        $\cpres^r(\gAF_2,\gAF_1)(m') \in T_{\gAF_2(m')}T^*M$ is a
        vertical element of the bundle and, as such, it is naturally
        identified with an element of $T_{m'}^*M$ and, as
        $\mathcal{D}_{m'}\subset T_{m'}M$,
        $\cpres^r(\gAF_2,\gAF_1)(m') \in \mathcal{D}_{m'}^*$.

        For $m_0\in M_0$ as above, we proved
        in~\eqref{it:equivariance_residuals-gammas_restricted_to_h=0}
        that $l^M_\tau(m_0) = m_0$. Then, by the $G$-equivariance of
        $\cpres^r(\gAF_2,\gAF_1)$, we have that
        \begin{equation*}
          \begin{split}
            \cpres^r(\gAF_2,\gAF_1)(m_0) =&
            \cpres^r(\gAF_2,\gAF_1)(l^M_{\tau^{-1}}(m_0)) =
            l^{TT^*M}_{\tau^{-1}}(\cpres^r(\gAF_2,\gAF_1)(m_0)) \\=&
            T_{\gAF_2(m_0)}
            l^{T^*M}_{\tau^{-1}}(\cpres^r(\gAF_2,\gAF_1)(m_0)) \\=&
            \cpres^r(l^{T^*M}_{\tau^{-1}} \circ \gAF_2,
            l^{T^*M}_{\tau^{-1}} \circ \gAF_1)(m_0) \\=&
            \cpres^r(l^{T^*M}_{\tau^{-1}} \circ \gAF_2\circ
            l^M_{\tau}, l^{T^*M}_{\tau^{-1}} \circ \gAF_1\circ
            l^M_{\tau})(l^M_{\tau^{-1}}(m_0)) \\=&
            \cpres^r((l^M_{\tau^{-1}})^*\gAF_2,
            (l^M_{\tau^{-1}})^*\gAF_1)(m_0) = \cpres^r(\ti{\gAF}_2,
            \ti{\gAF}_1) (m_0).
          \end{split}
        \end{equation*}
        Hence the second terms of $F(u)$ and $\ti{F}(u)$ coincide.

      \item Third term in $\ti{F}$. Recall that the Grassmann bundle
        is $\GrB(TM,k) := \cup_{m\in M} (\{m\}\times \Gr(T_mM,k))$ and
        that the action $l^M_\tau$ induces an action $l^{TM}_\tau$ on
        $TM$. As
        $l^{TM}_\tau|_{T_m M}:T_mM\rightarrow T_{l^M_\tau(m)}M$ is an
        isomorphism, it induces a diffeomorphism
        $l^{\GrB}_\tau:\GrB(TM,k)\rightarrow \GrB(TM,k)$ by
        $l^{\GrB}_\tau(m,\mathcal{D}) = (l^M_\tau(m),
        l^{TM}_\tau(\mathcal{D}))$. As
        $i_\mathcal{D}:M\rightarrow \GrB(TM,k)$ is defined by
        $i_\mathcal{D}(m):=(m,\mathcal{D}_m)$, we have that
        \begin{equation*}
          \begin{split}
            i_{\ti{\mathcal{D}}}(m) =& (m, \ti{\mathcal{D}}_m) =
            (l^M_\tau(l^M_{\tau^{-1}}(m)),
            l^{TM}_\tau(\mathcal{D}_{l^M_{\tau^{-1}}(m)})) \\=&
            l^{\GrB}_\tau(l^M_{\tau^{-1}}(m),
            \mathcal{D}_{l^M_{\tau^{-1}}(m)}) = (l^{\GrB}_\tau\circ
            i_\mathcal{D} \circ l^M_{\tau^{-1}})(m).
          \end{split}
        \end{equation*}
        For $\ti{m}\in \grade_M^{-1}(0)$ we have
        \begin{equation*}
          \begin{split}
            \cpres^r(\ti{\mathcal{D}}_2,\ti{\mathcal{D}}_1)(\ti{m}) =&
            \cpres^r(i_{\ti{\mathcal{D}}_2},
            i_{\ti{\mathcal{D}}_1})(\ti{m}) = \\=&
            \cpres^r(l^{\GrB}_\tau \circ i_{\mathcal{D}_2} \circ
            l^M_{\tau^{-1}}, l^{\GrB}_\tau \circ i_{\mathcal{D}_1}
            \circ l^M_{\tau^{-1}})(\ti{m}) \\=&
            T_{(\mathcal{D}_2)_{\ti{m}'}}l^{\GrB}_\tau
            (\cpres^r(i_{\mathcal{D}_2},
            i_{\mathcal{D}_1})(\underbrace{l^M_{\tau^{-1}}(\ti{m})}_{=:\ti{m}'})).
          \end{split}
        \end{equation*}
        Notice that $Tl^{\GrB}_\tau = l^{T\GrB}_\tau$ is the action on
        $T\GrB(TM,k)$ that makes
        $\cpres^r(i_{\mathcal{D}_2},i_{\mathcal{D}_1}):\grade_M^{-1}(0)\rightarrow
        T\GrB(TM,k)$ $G$-equivariant in the hypothesis of the current
        result. Thus,
        \begin{equation*}
          \begin{split}
            T_{(\mathcal{D}_2)_{\ti{m}'}}
            l^{\GrB}_\tau(\cpres^r(i_{\mathcal{D}_2},
            i_{\mathcal{D}_1})(\ti{m}')) =&
            l^{T\GrB}_\tau(\cpres^r(i_{\mathcal{D}_2},
            i_{\mathcal{D}_1})(\ti{m}')) = \cpres^r(i_{\mathcal{D}_2},
            i_{\mathcal{D}_1})(l^M_\tau(\ti{m}')) \\=&
            \cpres^r(i_{\mathcal{D}_2}, i_{\mathcal{D}_1})(\ti{m}).
          \end{split}
        \end{equation*}       
        Putting together the two computations we conclude that, for any 
        $\ti{m}\in \grade_M^{-1}(0)$,
        \begin{equation*}
          \cpres^r(\ti{\mathcal{D}}_2,\ti{\mathcal{D}}_1)(\ti{m}) = 
          \cpres^r(i_{\mathcal{D}_2},i_{\mathcal{D}_1})(\ti{m}) =
          \cpres^r(\mathcal{D}_2, \mathcal{D}_1)(\ti{m}).
        \end{equation*}
        
        On the other hand, using
        hypothesis~\ref{it:invariance_of_res_gamma-alphas_fixed_at_critical_pts},
        if $m_0\in M_0$, we have
        $\gAF_j(m_0) = ((l^M_{\tau^{-1}})^*(\gAF_j))(m_0) =
        \ti{\gAF}_j(m_0)$, so that so that
        \begin{equation*}
          \conj{\gAF_j(m_0)} = \conj{\ti{\gAF}_j(m_0)},
        \end{equation*}
        where the bar denotes the linear map
        $(T_{m_0}M)/(\mathcal{D}_j)_{m_0}\rightarrow \R$ induced by
        the corresponding covector. We conclude that
        \begin{equation*}
          \conj{\ti{\gAF}_j(m_0)} \circ \conj{\cpres^r(\ti{\mathcal{D}}_2, 
            \ti{\mathcal{D}}_1)(m_0)} = 
          \conj{\gAF_j(m_0)} \circ 
          \conj{\cpres^r(\mathcal{D}_2, \mathcal{D}_1)(m_0)},
        \end{equation*}
        proving that the third terms of $F(u)$ and $\ti{F}(u)$
        coincide.
      \end{enumerate}

    \item \label{it:equivariance_residuals-Gs_and_Fs-G} Analysis of
      $\ti{G}$. Let $\ti{m}\in \grade_M^{-1}(0)$ and
      $u\in T_{\ti{m}}h^{-1}_M(0)$. Choose
      $\lambda:\R\rightarrow h^{-1}_M(0)$ such that
      $\lambda(0)=\ti{m}$ and $\lambda'(0)=u$. Then, by the
      $G$-invariance of $g_1|_{\grade_M^{-1}(0)}$,
      \begin{equation*}
        T_{\ti{m}}\ti{g}_1(u) = \frac{d}{dt}\bigg|_{t=0} \ti{g}_1(\lambda(t)) =
        \frac{d}{dt}\bigg|_{t=0} g_1(l^M_{\tau^{-1}}(\lambda(t))) = 
        \frac{d}{dt}\bigg|_{t=0} g_1(\lambda(t)) = T_{\ti{m}} g_1(u).
      \end{equation*}
      On the other hand, as $l^N\circ g_j=l^M$, we have
      $\dot{l^M_{\tau^{-1}}}(\ti{m})=1$ and, then,
      \begin{equation*}
        \begin{split}
          \cpres^r(\ti{g}_2,\ti{g}_1)(\ti{m}) =& \cpres^r(g_2 \circ
          l^M_{\tau^{-1}}, g_1 \circ l^M_{\tau^{-1}})(\ti{m}) =
          \cpres^r(g_2, g_1)(l^M_{\tau^{-1}}(\ti{m})).
        \end{split}
      \end{equation*}
      All together, for $m_0\in M_0$, letting $\ti{m} := m_0$, so that
      $l^M_{\tau^{-1}}(\ti{m})=m_0$, and any
      $u\in T_{m_0}\grade_M^{-1}(0)$, we conclude that,
      \begin{equation}\label{eq:tiG_vs_G}
        \begin{split}
          \ti{G}(u) =& T_{m_0}\ti{g}_1(u) +
          \cpres^r(\ti{g}_2,\ti{g}_1)(m_0) = T_{m_0} g_1(u) +
          \cpres^r(g_2,g_1)(m_0) = G(u).
        \end{split}
      \end{equation}
    \end{enumerate}

    As a consequence of the analysis carried out in
    parts~\ref{it:equivariance_residuals-Gs_and_Fs-F}
    and~\ref{it:equivariance_residuals-Gs_and_Fs-G}, we conclude that,
    for $m_0\in M_0$ and $u\in T_{m_0}\grade_M^{-1}(0)$,
    \begin{equation}\label{eq:Fs_and_Gs-conclusion}
      \ti{F}(u) = F(u) \stext{ and } \ti{G}(u) = G(u).
    \end{equation}

    Continuing in the setting of the beginning of
    part~\ref{it:equivariance_residuals-Gs_and_Fs}, we have that
    \begin{equation*}
      F(u_0) = 0,\quad G(u_0)=0 \stext{ and } 
      \ti{F}(\ti{u}_0)=0,\quad \ti{G}(\ti{u}_0) = 0, 
    \end{equation*}
    and, by~\eqref{eq:Fs_and_Gs-conclusion},
    \begin{equation*}
      F(\ti{u}_0) = \ti{F}(\ti{u}_0) = 0 \stext{ and } 
      G(\ti{u}_0) = \ti{G}(\ti{u}_0) = 0
    \end{equation*}
    we see that both $u_0$ and $\ti{u}_0$ are solutions of $F(u)=0$
    and $G(u)=0$, which has a unique solution. Hence,
    \begin{equation}
      \label{eq:res_gamma_equals_res_gamma_tilde}
      \cpres^r(\gamma_2, \gamma_1)(n_0) = u_0 =\ti{u}_0 = 
      \cpres^r(\ti{\gamma}_2,\ti{\gamma}_1)(n_0) 
      \stext{ for all } n_0\in \grade_N^{-1}(0)
    \end{equation}

  \item Last, for $n_0\in \grade_N^{-1}(0)$,
    using~\eqref{eq:res_gamma_equals_res_gamma_tilde},
    \begin{equation*}
      \begin{split}
        l^{TM}_{\tau^{-1}}(\cpres^r(\gamma_2,\gamma_1)(n_0)) =&
        T_{\gamma_2(n_0)}l^M_{\tau^{-1}}(\cpres^r(\gamma_2,\gamma_1)(n_0))
        = \cpres^r(l^M_{\tau^{-1}}\circ\gamma_2,
        l^M_{\tau^{-1}}\circ\gamma_1)(n_0)) \\=& \cpres^r(\ti{\gamma}_2,
        \ti{\gamma}_1)(n_0) = \cpres^r(\gamma_2,\gamma_1)(n_0),
      \end{split}
    \end{equation*}
    proving the $G$-invariance of the vector
    $\cpres^r(\gamma_2,\gamma_1)(n_0)$ for all $n_0\in \grade_N^{-1}(0)$ as
    we wanted, completing the proof of the Proposition.
    
  \end{enumerate}
\end{proof}

Proposition~\ref{prop:invariance_of_res_gamma} is a general result for
critical problems. The next result is an application of that result to
the case of discretizations of a \FMS.

Let $\discCP^j:=(Q,\psi^j,L_\CPDS^j,f_\CPDS^j)$ with $j=1,2$
be two discretizations of the \FMS $(Q,L,f)$ with contact order $r$
among themselves. We have the associated critical problems
$(\conj{\dAFtp_j},\conj{\varphi^j}, \ker(T\conj{\varphi^j}))$ over the
open subset $\conj{\mathcal{C}}\subset \R\times(TQ\oplus TQ)$ and with
boundary values in $\R\times E$. We have the grades
$\grade_{\conj{\mathcal{C}}}:=p_1|_{\conj{\mathcal{C}}}$ and
$\grade_{\R\times E}:=p_1$ and, as we saw in
Lemmas~\ref{le:critical_points_with_h_0}
and~\ref{le:nondegeneray_of_critical_points}, together with
Lemma~\ref{le:crit_problems_bar_vs_hat_are_equivalent}
$\gamma_0:\grade_{\R\times E}^{-1}(0)\rightarrow \{0\}\times
\Delta_{TQ}\subset \grade_{\conj{\mathcal{C}}}^{-1}(0)$ defined by
$\gamma_0(0,-z_q,z_q):=(0,z_q,z_q)$ assigns to each boundary value a
nondegenerate critical point of both critical problems. Observe that
$\gamma_0$ is a diffeomorphism between closed submanifolds. Also by
Lemma~\ref{le:crit_problems_bar_vs_hat_are_equivalent}, both functions
$\conj{\gamma}^j$ are the extensions of $\gamma_0$ provided by
Theorem~\ref{th:thm_2_in_cp07} for each of the critical problems. The
next result essentially verifies that the hypotheses of
Proposition~\ref{prop:invariance_of_res_gamma} hold in this ``bar
context'' and, then, derives properties of the $\conj{\gamma}^j$
functions.

\begin{prop}\label{prop:order_of_gamma_bar_and_variance_of_residual}
  With the notation as above,
  \begin{enumerate}
  \item \label{it:order_of_gamma_bar_and_variance_of_residual-order}
    $\conj{\gamma}_2 = \conj{\gamma}_1 + \mathcal{O}(\grade_{\R\times
      E}^r)$ and
  \item \label{it:order_of_gamma_bar_and_variance_of_residual-residual}
    $\cpres^r(\conj{\gamma}_2,\conj{\gamma}_1)(0,-z_q,z_q)\in
    T_{\conj{\gamma}_2(0,-z_q,z_q)} (\grade_{\conj{\mathcal{C}}}^{-1}(0))
    \subset T_{\conj{\gamma}_2(0,-z_q,z_q)} \conj{\mathcal{C}}$ is
    invariant by the $\Z_2$-action $l^{T\conj{\mathcal{C}}}$ for all
    $z_q\in TQ$.
  \end{enumerate}
\end{prop}

\begin{proof}
  By Lemma~\ref{le:order_of_sigma_2_bar} we have
  $\conj{\dAFtp_2} = \conj{\dAFtp_1} + \mathcal{O}(\grade_{\R\times
    TQ\times TQ}^r)$, by Lemma~\ref{le:order_of_phi_bar} we have
  $\conj{\varphi^2} =\conj{\varphi^1} +
  \mathcal{O}(\grade_{\conj{\mathcal{C}}}^r)$ and, by
  Lemma~\ref{le:contact_order_of_ker_T_bar_phi}, we have that
  $\conj{\mathcal{D}}^2 = \conj{\mathcal{D}}^1 +
  \mathcal{O}(\grade_{\conj{\mathcal{C}}}^r)$. Then, the first statement
  follows from Theorem~\ref{th:thm_3_in_cp07}.

  The second statement is an application of
  Proposition~\ref{prop:invariance_of_res_gamma}, whose hypotheses are
  satisfied in this case, as we see below.

  Direct evaluation shows that
  $\grade_{\conj{\mathcal{C}}} \circ l_\tau =
  \grade_{\conj{\mathcal{C}}}$ and
  $ \grade_{\R\times E}\circ\conj{\varphi}^j =
  \grade_{\conj{\mathcal{C}}}$, proving that
  hypothesis~\ref{it:invariance_of_res_gamma-g_j_and_grades} in
  Proposition~\ref{prop:invariance_of_res_gamma} is satisfied.

  If we apply Lemma~\ref{le:skew_problem_with_D_from_tig} taking
  Remark~\ref{rem:skew_problem_with_D_from_tig-open} into account with
  $\ti{M}:=TQ\oplus TQ$,
  $M:=\conj{\mathcal{C}}\subset\R\times (TQ\oplus TQ)$,
  $\ti{M}_0 := \Delta_{TQ}$, $\ti{N}:=E$ and
  $\ti{g}:=p_2\circ \conj{\varphi^j}$, we see that
  $\conj{\mathcal{D}}_j:= \ker(T\conj{\varphi}^j)$ satisfies
  $\conj{\mathcal{D}}_j|_{\grade_M^{-1}}\subset T
  \grade_{M}^{-1}(0)$. Also, as
  $M_0:=\{0\}\times \ti{M}_0 = \{0\}\times \Delta_{TQ} \subset
  \conj{\mathcal{C}} = M$, if we consider the $\Z_2$-action
  $l^{\ti{M}}_{[1]}(w,\ti{w}):= (\ti{w},w)$ the subsets $M_0$ and
  $\conj{\mathcal{D}}_j|_{\grade_M^{-1}(0)}$ are
  $\Z_2$-invariant. Still in the same setting, for the trivial
  $\Z_2$-action on $E$, we have that
  $\ti{g}(0,w,\ti{w}) = p_2(\conj{\varphi^j}(0,w,\ti{w})) =
  (-\frac{1}{2}(w+\ti{w}), \frac{1}{2}(w+\ti{w}))$, so that
  $\ti{g}(0,\cdot)$ is $\Z_2$-equivariant (that means, in fact, that
  it is $\Z_2$-invariant) and it is immediate that the same thing
  applies to $g(0,w,\ti{w}):=(0,\ti{g}(0,w,\ti{w}))$; hence, again by
  the same result, $\conj{\mathcal{D}}_j|_{\grade_M^{-1}(0)}$ and
  $\conj{\mathcal{D}}_j|_{M_0}$ are $\Z_2$-invariant. We also observe
  that in the same context, by
  Lemma~\ref{le:Z2_variance_of_sigma_2_bar}, we know that
  $\conj{\dAFtp_j}|_{T\grade_{M}^{-1}(0)}$ is $\Z_2$-invariant. All
  together,
  hypothesis~\ref{it:invariance_of_res_gamma-D_j_invariance} in
  Proposition~\ref{prop:invariance_of_res_gamma} is satisfied.

  Also by Lemma~\ref{le:Z2_variance_of_sigma_2_bar} and still in the
  context of the previous paragraph, we have that
  $\conj{\dAFtp_j}|_{\grade_{\R\times TQ\times TQ}^{-1}(0)}$ is
  $\Z_2$-invariant which, in turn, implies that
  $\conj{\dAFtp_j}|_{\{0\}\times \Delta_{TQ}}$ is $\Z_2$-invariant so
  that
  hypothesis~\ref{it:invariance_of_res_gamma-alphas_fixed_at_critical_pts}
  in Proposition~\ref{prop:invariance_of_res_gamma} is satisfied.

  Last, still in the same context as above,
  $\cpres^r(\conj{\dAFtp_2}, \conj{\dAFtp_1})$ is $\Z_2$-equivariant
  by Lemma~\ref{le:Z2_variance_of_sigma_2_bar},
  $\cpres^r(i_{\conj{\mathcal{D}^2}},i_{\conj{\mathcal{D}^1}})$ is
  $\Z_2$-equivariant by
  Proposition~\ref{prop:residual_i_D_bar_is_Z2_equivariant} and
  $\cpres^{r}(\conj{\varphi^2}, \conj{\varphi^1})$ are
  $\Z_2$-invariant by Lemma~\ref{le:order_of_phi_bar}, thus satisfying
  hypothesis~\ref{it:invariance_of_res_gamma-residuals_variance} in
  Proposition~\ref{prop:invariance_of_res_gamma}.

  All together, we apply
  Proposition~\ref{prop:invariance_of_res_gamma} to see that
  $\cpres^r(\conj{\gamma}_2,\conj{\gamma}_1)(0,-z_q,z_q)$ is a
  $\Z_2$-invariant vector of $T(\grade_M^{-1}(0))$ for all
  $z_q\in TQ$.
\end{proof}

\begin{proof}[Proof of Theorem~\ref{th:contact_order_discrete_FMS}]
  As we saw in the proof of
  Theorem~\ref{th:flow_of_discretizations_of_FMS}, the flow $F^j$ of
  the \FDMSCP obtained from the discretization $\discCP^j$ is
  constructed using the maps $\widehat{\gamma}^j$ that assign critical
  points of a critical problem ---hence trajectories of the \FDMSCP---
  to boundary values. Proposition 6
  in~\cite{ar:cuell_patrick-skew_critical_problems} can be used to
  prove that if
  $\widehat{\gamma_2} = \widehat{\gamma_1}
  +\mathcal{O}(\grade_{\R\times E}^r)$, then
  $F^2 = F^1 + \mathcal{O}(\grade_{\R\times TQ\times TQ}^r)$ and that,
  furthermore, if
  $\cpres^r(\widehat{\gamma_2},\widehat{\gamma_1})(0,-z_q,z_q)$ is
  $\Z_2$-invariant by the $\Z_2$-action $l^{T(\R\times TQ\times TQ)}$
  for all $z_q\in TQ$ ---what is called ``being symmetric'' in
  Proposition 6--- then
  $F^2 = F^1 + \mathcal{O}(\grade_{\R\times TQ\times TQ}^{r+1})$, that
  is, the contact order of the flows is $r$.

  By~\eqref{eq:phi_bar_and_gamma_hat-def}, we know that
  \begin{equation}
    \label{eq:gamma_hat_in_terms_of_gamma_bar}
    \widehat{\gamma^j} = (\Theta^j)^{-1} \circ \conj{\gamma^j}.
  \end{equation}
  As,
  $\Theta^2 = \Theta^1 + \mathcal{O}(\grade_{\R\times TQ\times
    TQ}^{r+1})$ and $\Theta^j$ is a diffeomorphism preserving the
  grades, by Proposition 4
  in~\cite{ar:cuell_patrick-skew_critical_problems} we have that
  $(\Theta^2)^{-1} = (\Theta^1)^{-1} + \mathcal{O}(\grade_{\R\times
    TQ\times TQ}^{r+1})$. By
  point~\ref{it:order_of_gamma_bar_and_variance_of_residual-order} in
  Proposition~\ref{prop:order_of_gamma_bar_and_variance_of_residual}
  we know that
  $\conj{\gamma}_2 = \conj{\gamma}_1 + \mathcal{O}(\grade_{\R\times
    E}^r)$. Then, applying Proposition 3
  in~\cite{ar:cuell_patrick-skew_critical_problems}
  to~\eqref{eq:gamma_hat_in_terms_of_gamma_bar} (recalling that both
  $(\Theta^j)^{-1}$ and $\conj{\gamma^j}$ preserve the grades), we
  have that
  $\widehat{\gamma}_2 = \widehat{\gamma}_1 +
  \mathcal{O}(\grade_{\R\times E}^r)$. As we mentioned above, this
  last fact proves that
  $F^2=F^1+\mathcal{O}(\grade_{\R\times TQ\times TQ}^r)$, that is, we
  proved that the flows have contact order one less than the claim in
  the statement; next we use the $\Z_2$-equivariance properties to add
  an extra power to the contact order.

  Still using Proposition 3
  of~\cite{ar:cuell_patrick-skew_critical_problems} in the same
  setting as before we see that, as
  $(\Theta^2)^{-1} = (\Theta^1)^{-1} + \mathcal{O}(\grade_{\R\times
    TQ\times TQ}^{r+1})$,
  \begin{equation}\label{eq:contact_order_discrete_FMS-comparison_1}
    \begin{split}
      \cpres^r(\widehat{\gamma^2},\widehat{\gamma^1})(0,-z_q,z_q) =&
      \underbrace{\cpres^r((\Theta^2)^{-1},
        (\Theta^2)^{-1})(0,z_q,z_q)}_{=0} \\&+ T_{(0,z_q,z_q)}
      (\Theta^1)^{-1}(\cpres^r(\conj{\gamma^2},\conj{\gamma^1})(0,-z_q,z_q)).
    \end{split}
  \end{equation}
  In addition, as
  $\cpres^r(\conj{\gamma}_2,\conj{\gamma}_1)(0,-z_q,z_q)\in
  T_{\conj{\gamma}_2(0,-z_q,z_q)} (\grade_{\conj{\mathcal{C}}}^{-1}(0))$ by
  point~\eqref{it:order_of_gamma_bar_and_variance_of_residual-residual}
  in
  Proposition~\ref{prop:order_of_gamma_bar_and_variance_of_residual},
  and
  $\Theta^j|_{\{0\}\times TQ\times TQ} = id_{\{0\}\times TQ\times
    TQ}$,
  \begin{equation*}
    \begin{split}
      T_{(0,z_q,z_q)}
      (\Theta^1)^{-1}(\cpres^r(\conj{\gamma^2},\conj{\gamma^1})&(0,-z_q,z_q))
      \\=& T_{(0,z_q,z_q)} ((\Theta^j)^{-1}|_{\{0\}\times TQ\times TQ})
      (\cpres^r(\conj{\gamma^2},\conj{\gamma^1})(0,-z_q,z_q)) \\=&
      \cpres^r(\conj{\gamma^2},\conj{\gamma^1})(0,-z_q,z_q).
    \end{split}
  \end{equation*}
  This last computation together
  with~\eqref{eq:contact_order_discrete_FMS-comparison_1}, proves that
  \begin{equation*}
    \cpres^r(\widehat{\gamma^2},\widehat{\gamma^1})(0,-z_q,z_q) =
    \cpres^r(\conj{\gamma^2},\conj{\gamma^1})(0,-z_q,z_q)
  \end{equation*}
  for all $z_q\in TQ$. Finally, by
  point~\eqref{it:order_of_gamma_bar_and_variance_of_residual-residual}
  in
  Proposition~\ref{prop:order_of_gamma_bar_and_variance_of_residual}
  we know that the right hand side of the last identity is a
  $\Z_2$-invariant vector for $l^{T(\R\times TQ\times TQ)}$, we
  conclude that the same thing happens to
  $\cpres^r(\widehat{\gamma^2},\widehat{\gamma^1})(0,-z_q,z_q)$ and,
  as we stated above,
  $F^2=F^1+\mathcal{O}(\grade_{\R\times TQ\times TQ}^{r+1})$ by Proposition
  6 in~\cite{ar:cuell_patrick-skew_critical_problems}.
\end{proof}


\section{Forced discrete mechanical systems on $Q\times Q$}
\label{sec:forced_discrete_mechanical_systems_on_QxQ}

So far in this paper we have studied a type of discrete dynamical
system, the \FDMSCP introduced in
Definition~\ref{def:forced_discrete_mechanicalCP_system}, that can be
used to approximate the behavior of a \FMS. There is a more well known
type of discrete dynamical system that we call \jdef{forced discrete
  mechanical system on $Q\times Q$} and that is also used to model a
\FMS. In this section we review such systems and derive some error
analysis results for them from our results of
Sections~\ref{sec:discretizations_a_la_cuell_and_patrick}
and~\ref{sec:error_analysis}. In particular, we introduce and study
the properties of the exact \discretizationMW in
Example~\ref{ex:exact_discretization_MW-simple_bd} and
Theorem~\ref{th:flow_of_exact_discrete_systems-FDMS} and, in
Theorem~\ref{th:order_of_flow_MW_vs_order_discMW}, we relate the
contact order of the flow of a \discretizationMW to that of the
underlying continuous system.

\begin{definition}
  A \jdef{forced discrete mechanical system on $Q\times Q$} (\FDMSMW)
  is a $4$-tuple $(Q,W,L_d,f_d)$ where $Q$ is a differentiable
  manifold, the \jdef{configuration space}, $W\subset Q\times Q$ is an
  open neighborhood of $\Delta_Q$, $L_d:W\rightarrow \R$ is a smooth
  function, the \jdef{discrete Lagrangian}, and
  $f_d\in \mathcal{A}^1(W)$, the \jdef{discrete force}.
\end{definition}

As $T^*(Q\times Q) \simeq p_1^*T^*Q \oplus p_2^*T^*Q$ it is standard
practice to write $f_d(q_0,q_1) = (f_d^-(q_0,q_1),f_d^+(q_0,q_1))$,
where $f_d^+$ is a section of $p_2^*T^*Q$ and $f_d^-$ is a section of
$p_1^*T^*Q$.

\FDMSMWs become dynamical systems by the \jdef{discrete
  Lagrange--D'Alembert principle} as follows. A \jdef{discrete path}
is an element $q_\cdot = (q_0,\ldots,q_N)\in Q^{N+1}$. An
\jdef{infinitesimal variation} $\delta q_\cdot$ over $q_\cdot$
satisfies $\delta q_k \in T_{q_k}Q$ for $k=0,\ldots,N$ and has fixed
enpoints when $\delta q_0=0$ and $\delta q_N=0$. A discrete path
$q_\cdot$ is a trajectory of $(Q,L_d,f_d)$ if $q_\cdot$ is a critical
point of
$\mwAFnp(q_\cdot):=\sum_{k=1}^N (dL_d(q_{k-1},q_k) +
f_d(q_{k-1},q_k))$ for all infinitesimal fixed endpoint variations
$\delta q_\cdot$ over $q_\cdot$.  Standard computations lead to the
corresponding equations of motion
\begin{equation*}
  T^2L_d(q_{k-1},q_k) + T^1L_d(q_k,q_{k+1}) + f_d^+(q_{k-1},q_k) +
  f_d^-(q_k,q_{k+1}) = 0,
\end{equation*}
for $k=1,\ldots,N-1$, known as the \jdef{discrete Lagrange--D'Alembert
  equations}.

\begin{remark}\label{rem:trajectoriies_FDMS_as_critical_points}
  According to the previous description, the trajectories of a \FDMSMW
  $(Q,L_d,f_d)$ can be defined as critical points of a critical
  problem. Indeed, let $\mathcal{C}_d:=Q^{N+1}$ and define
  $g_d:\mathcal{C}_d\rightarrow Q\times Q$ by
  $g_d(q_0,\ldots,q_N):=(q_0,q_N)$. Also, let
  $\mwAFnp:=\sum_{j=1}^N p_{j-1,j}^*\mwAFop \in
  \mathcal{A}^1(\mathcal{C}_d)$, for
  $\mwAFop:=dL_d+f_d\in \mathcal{A}^1(Q\times Q)$. Then
  $q_\cdot \in \mathcal{C}_d$ is a trajectory of $(Q,L_d,f_d)$ if and
  only if it is a critical point of $(\mwAFnp,g_d,\ker(Tg_d))$. If
  $L_d$ and $f_d$ are only defined in the open subset
  $W\subset Q\times Q$, then we replace $\mathcal{C}_d$ with
  \begin{equation}\label{eq:D_d^W-def}
    \mathcal{C}_d^W := p_{01}^{-1}(W)\cap \cdots\cap p_{(N-1),
      N}^{-1}(W) \subset Q^{N+1}.
  \end{equation}
\end{remark}

\FDMSMWs have been studied for many years and a good reference is Part
III
of~\cite{ar:marsden_west-discrete_mechanics_and_variational_integrators}.


Just as in the case of \FDMSCP, a notion of discretization of a \FMS
will be essential to the error analysis of \FDMSMW. A crucial difference
between the two cases is that while for a discretization the value
$h=0$ is not special, \FDMSMWs are usually not well behaved at $h=0$
since this represents ``instantaneous displacement'' between the fixed
points $q_0$ and $q_1$. For that reason, the notion of discretization
in the context of \FDMSMW will be slightly more subtle than before.

First we recall that when $(Q,L,f)$ is a regular \FMS, we have its
exact discretizations
$\discCPEexp := (Q,(\psi^E,\dBTp,\dBTm),L_\CPDS^E,f_\CPDS^E)$
introduced in Example~\ref{ex:exact_discretization_forced-system},
where $\dBTp$ and $\dBTm$ are arbitrary, but satisfying the conditions
given in Definition~\ref{def:discretization_of_TQ}. Applying
Lemma~\ref{le:discretizations_yield_discrete_tangent_bundles-U_and_W}
to $\discCPEexp$ (with $\ti{\del}^{E\mp}$), we have two sets
$\mathcal{U}^E\subset U^{X_{L,f}} \subset \R\times TQ$ and
$\mathcal{W}^E\subset \R\times Q\times Q$ such that
$\{0\}\times TQ\subset \mathcal{U}^E$,
$\{0\}\times \Delta_Q\subset \mathcal{W}^E$ and the map
$\ti{\del}^{E\mp}(h,v):=(h,\del^{E\mp}_h(v))$ is a diffeomorphism
between the open subsets
${\mathcal{U}^E}^* := \mathcal{U}^E\SM (\{0\}\times TQ)$ and
${\mathcal{W}^E}^* := \mathcal{W}^E\SM (\{0\}\times Q\times Q)$.

\begin{definition}\label{def:discretization_MW}
  Let $(Q,L,f)$ be a regular \FMS and $\mathcal{W}^E$ be as in the
  previous paragraph. A \jdef{\discretizationMW} of $(Q,L,f)$ is a
  $4$-tuple $(Q,W,L_{d,\cdot},f_{d,\cdot})$ where
  $\mathcal{W}^E\subset W\subset \R\times Q\times Q$ and
  $W^*:=W\SM (\{0\}\times Q\times Q)$ is open; a function
  $L_{d,\cdot}:W^*\rightarrow\R$ and a section $f_{d,\cdot}$ over
  $W^*$ of $p_{23}^*T^*(Q\times Q)$ such that
  $L_{d,\cdot}\circ \ti{\del}^{E\mp}|_{{\mathcal{U}^E}^*}$ and
  $(\ti{\del}^{E\mp}|_{{\mathcal{U}^E}^*})^{*_2}(f_{d,\cdot})$
  (defined over ${\mathcal{U}^E}^*$) extend to $\mathcal{U}^E$ as a
  smooth function $L_{ext,\cdot}$ and a smooth section $f_{ext,\cdot}$
  that satisfy
  \begin{equation}\label{eq:discretization_MW-order_conditions}
    \begin{split}
      L_{ext,\cdot} =& h L
      +\mathcal{O}(\grade_{\R\times TQ}^2), \\
      f_{ext,\cdot} =& h f + \mathcal{O}(\grade_{\R\times TQ}^2).
    \end{split}
  \end{equation}
\end{definition}

In order to emphasize the distinction between discretizations in the
sense of Definition~\ref{def:discretization_CP} and
\discretizationsMW, in this section we will call
\jdef{\discretizationCPExp} to a discretization in the sense of
Definition~\ref{def:discretization_CP}.

Using the following result it is possible to transform \discretizationsCPExp of a regular \FMS into \discretizationsMW of the same system,
and conversely.

\begin{prop}\label{prop:correspondence_discretizations_CP_and_MW}
  Let $(Q,L,f)$ be a regular \FMS and $(\psi^E,\alpha^+,\alpha^-)$ be
  the exact discretization of $TQ$ constructed in
  Example~\ref{ex:exact_discretization_forced}. For any \discretizationCPExp
  $\discCPexp=(Q,\psi^E,L_\CPDS,f_\CPDS)$ of $(Q,L,f)$ so that
  $L_\CPDS$ and $f_\CPDS$ are defined over $\mathcal{U}^E$
  we define
  \begin{equation}
    \label{eq:correspondence_discretizations_CP_and_MW-MW_from_CP}
    L_{d,\cdot} := L_\CPDS \circ (\ti{\del}^{E\mp}|_{{\mathcal{U}^E}^*})^{-1}
    \stext{ and }
    f_{d,\cdot} := ((\ti{\del}^{E\mp}|_{{\mathcal{U}^E}^*})^{-1})^{*_2}(f_\CPDS).
  \end{equation}
  Then $\discMW=(Q,\mathcal{W}^E,L_{d,\cdot},f_{d,\cdot})$ is a
  \discretizationMW of $(Q,L,f)$. Furthermore, this assignment is a
  bijection between the two types of discretization of $(Q,L,f)$.
\end{prop}

\begin{proof}
  Given $\discCPexp$ as in the statement, defining $L_{d,\cdot}$ and
  $f_{d,\cdot}$
  using~\eqref{eq:correspondence_discretizations_CP_and_MW-MW_from_CP}
  ensures that $L_{d,\cdot}:{\mathcal{W}^E}^*\rightarrow\R$ is a
  smooth function and $f_{d,\cdot}$ is a smooth section of
  $p_{23}^*T^*(Q\times Q)$ over ${\mathcal{W}^E}^*$. Then,
  \begin{equation*}
    L_{d,\cdot}\circ (\ti{\del}^{E\mp}|_{{\mathcal{U}^E}^*}) =
    L_\CPDS \circ (\ti{\del}^{E\mp}|_{{\mathcal{U}^E}^*})^{-1}
    \circ (\ti{\del}^{E\mp}|_{{\mathcal{U}^E}^*}) = L_\CPDS|_{{\mathcal{U}^E}^*}
  \end{equation*}
  that extends smoothly over $\mathcal{U}^E$ to $L_\CPDS$, so
  that $L_{ext,\cdot} = L_\CPDS$ and, then,
  \begin{equation*}
    L_{ext,\cdot} = L_\CPDS = h L +\mathcal{O}(\grade_{\R\times TQ}^{2}),
  \end{equation*}
  where the last identity comes from $\discCPexp$ being a
  \discretizationCPExp. Hence the first condition
  in~\eqref{eq:discretization_MW-order_conditions} is
  valid. Similarly, for $(h,v) \in {\mathcal{U}^E}^*$ and
  $(0, \delta v) \in T_{(h,v)}(\R\times TQ)$,
  \begin{equation*}
    \begin{split}
      (\ti{\del}^{E^\mp}|_{{\mathcal{U}^E}^*})^{*_2}(f_{d,\cdot})&(h,v)(0,
      \delta v) =
      (\ti{\del}^{E^\mp}|_{{\mathcal{U}^E}^*})^*(f_{d,\cdot})(h,v)(0,
      \delta v) \\=&
      f_{d,\cdot}(\ti{\del}^{E^\mp}|_{{\mathcal{U}^E}^*}(h,v))(T_{(h,v)}
      \ti{\del}^{E^\mp}|_{{\mathcal{U}^E}^*} (0,\delta v)) \\=&
      (((\ti{\del}^{E\mp}|_{{\mathcal{U}^E}^*})^{-1})^{*_2}(f_\CPDS))
      (\ti{\del}^{E^\mp}|_{{\mathcal{U}^E}^*}(h,v))(\underbrace{T_{(h,v)}
        \ti{\del}^{E^\mp}|_{{\mathcal{U}^E}^*} (0,\delta
        v)}_{=:(0,\delta q,\delta q')}) \\=&
      (((\ti{\del}^{E\mp}|_{{\mathcal{U}^E}^*})^{-1})^*(f_\CPDS))
      (\ti{\del}^{E^\mp}|_{{\mathcal{U}^E}^*}(h,v))(0,\delta q,\delta
      q') \\=& f_\CPDS(h,v)
      (T_{\ti{\del}^{E^\mp}|_{{\mathcal{U}^E}^*}(h,v)}
      (\ti{\del}^{E\mp}|_{{\mathcal{U}^E}^*})^{-1}(T_{(h,v)}
      \ti{\del}^{E^\mp}|_{{\mathcal{U}^E}^*} (0,\delta v))) \\=&
      f_\CPDS(h,v)(0,\delta v).
    \end{split}
  \end{equation*}
  Then,
  \begin{equation*}
    (\ti{\del}^{E^\mp}|_{{\mathcal{U}^E}^*})^{*_2}(f_{d,\cdot}) =
    f_\CPDS|_{{\mathcal{U}^E}^*}
  \end{equation*}
  and, as $f_\CPDS$ is smooth over $\mathcal{U}^E$, we see that
  $f_{ext,\cdot} = f_\CPDS$ and, then,
  \begin{equation*}
    f_{ext,\cdot} = f_\CPDS = h f +\mathcal{O}(\grade_{\R\times TQ}^2)
  \end{equation*}
  because, once again, $\discCPexp$ is a \discretizationCPExp of
  $(Q,L,f)$. Hence, we see that the second condition
  in~\eqref{eq:discretization_MW-order_conditions} is also valid. All
  together, $(Q,\mathcal{W}^E,L_{d,\cdot},f_{d,\cdot})$ is a
  \discretizationMW of $(Q,L,f)$.

  Conversely, let $\discMW=(Q,\mathcal{W}^E,L_{d,\cdot},f_{d,\cdot})$
  be as in the statement and define $L_\CPDS := L_{ext,\cdot}$
  and $f_\CPDS := f_{ext,\cdot}$, so that both are smooth over
  $\mathcal{U}^E$ and
  \begin{equation*}
    L_\CPDS = L_{ext,\cdot} = h L + \mathcal{O}(\grade_{\R\times TQ}^2),
    \stext{ and }
    f_\CPDS = f_{ext,\cdot} = h f +
    \mathcal{O}(\grade_{\R\times TQ}^2).
  \end{equation*}
  Thus, $(Q,\psi^E,L_\CPDS, f_\CPDS)$ is a \discretizationCPExp of
  $(Q,L,f)$.

  It is immediate that the two procedures described above are mutually
  inverse.
\end{proof}

\begin{example}\label{ex:exact_discretization_MW-simple_bd}
  Let $\discCPEexp=(Q,\psi^E,L_\CPDS^E,f_\CPDS^E)$ be the
  exact \discretizationCPExp associated to the regular \FMS $(Q,L,f)$
  constructed in Example~\ref{ex:exact_discretization_forced-system}
  (corresponding to arbitrary $\dBTp$ and $\dBTm$). As both
  $L_\CPDS^E$ and $f_\CPDS^E$ are defined over
  $\mathcal{U}^E$, by
  Proposition~\ref{prop:correspondence_discretizations_CP_and_MW},
  $\discMWE := (Q,\mathcal{W}^E,L_{d,\cdot}^E,f_{d,\cdot}^E)$ defined
  by
  \begin{equation}
    \label{eq:exact_discretization_MW-def}
    L_{d,\cdot}^E := L_\CPDS^E \circ (\ti{\del}^{E\mp}|_{{\mathcal{U}^E}^*})^{-1}
    \stext{ and }
    f_{d,\cdot}^E :=
    ((\ti{\del}^{E\mp}|_{{\mathcal{U}^E}^*})^{-1})^{*_2}(f_\CPDS^E)
  \end{equation}
  is a \discretizationMW of $(Q,L,f)$ that we call the \jdef{exact
    \discretizationMW} of $(Q,L,f)$.

  We can specialize the construction of $\discMWE$ to the case where
  \begin{equation}
    \label{eq:simple_dBDp_dBDm-def}
    \dBTp(h):=h \stext{ and } \dBTm(h)=0 \stext{ for all } h\in\R.
  \end{equation}
  We have $\del^{E\mp}_t(v) = (\tau_Q(v),\tau_Q(F^{X_{L,f}}_t(v)))$
  or, if $q(t) := \tau_Q(F^{X_{L,f}}_t(v))$,
  $\del^{E\mp}_t(v) = (q(0),q(t))$. Thus, for
  $(h,q_0,q_1)\in {\mathcal{W}^E}^*$ there is a unique
  $(h,v_{01})\in \mathcal{U}^{E^*}$ so that
  $(h,q_0,q_1) = \ti{\del}^{E\mp}_h(v_{01}) =
  (h,q_{01}(0),q_{01}(h))$, where $q_{01}(t)$ is the trajectory of
  $(Q,L,f)$ starting at $t=0$ with velocity
  $v_{01}$. Then
  \begin{equation}\label{eq:exact_L_d-MW}
    \begin{split}
      L_{d,h}^E(q_0,q_1) =&
      L_\CPDS^E((\ti{\del}^{E\mp}|_{{\mathcal{U}^E}^*})^{-1}(h,q_0,q_1))
      = L_\CPHS^E(v_{01}) \\=& \int_0^h L(F_t^{X_{L,f}}(v_{01})) dt =
      \int_0^h L(q_{01}'(t)) dt.
    \end{split}
  \end{equation}
  Similarly, for $(h,q_0,q_1)\in {\mathcal{W}^E}^*$ and
  $(0,\delta q_0,\delta q_1)\in p_{23}^*T(Q\times Q)\subset
  T_{(h,q_0,q_1)}{\mathcal{W}^E}^*$, we have
  \begin{equation*}
    \begin{split}
      f_{d,h}^E&(q_0,q_1)(0, \delta q_0,\delta q_1) =
      ((\ti{\del}^{E\mp}|_{{\mathcal{U}^E}^*})^{-1})^{*_2}(f_\CPDS^E)(h,q_0,q_1)
      (0, \delta q_0,\delta q_1) \\=& f_\CPDS^E(\underbrace{
        (\ti{\del}^{E\mp}|_{{\mathcal{U}^E}^*})^{-1}(h,q_0,q_1)}_{=(h,v_{01})})
      (T_{(h,q_0,q_1)}((\ti{\del}^{E\mp}|_{{\mathcal{U}^E}^*})^{-1})(0,
      \delta q_0,\delta q_1)) \\=& f_\CPHS^E(v_{01})
      (T^{23}_{(h,q_0,q_1)}((\del^{E\mp}_h|_{{\mathcal{U}^E}^*\cap
        (\{h\}\times TQ)})^{-1})( \delta q_0,\delta q_1)) \\=&
      \int_0^h \check{f}(q_{01}'(t))(T_{v_{01}}(\tau_Q\circ
      F^{X_{L,f}}_t)(T^{23}_{(h,q_0,q_1)}((\del^{E\mp}_h|_{{\mathcal{U}^E}^*\cap
        (\{h\}\times TQ)}))^{-1})(\delta q_0,\delta q_1)) dt \\=&
      \int_0^h \check{f}(q_{01}'(t))(T^{23}_{(h,q_0,q_1)}(\tau_Q\circ
      F^{X_{L,f}}_t \circ (\del^{E\mp}_h|_{{\mathcal{U}^E}^*\cap
        (\{h\}\times TQ)})^{-1})(\delta q_0,\delta q_1)) dt.
    \end{split}
  \end{equation*}
  Notice that
  $\tau_Q\circ F^{X_{L,f}}_t \circ
  (\del^{E\mp}_h|_{{\mathcal{U}^E}^*\cap (\{h\}\times TQ)})^{-1}$ maps
  $(h,q_0,q_1)$ to $q_{01}(t)$, the trajectory of $(Q,L,f)$ satisfying
  $q_{01}(0)=q_0$ and $q_{01}(h)=q_1$ (with
  $(h,q_{01}'(0)) \in {\mathcal{U}^E}^*\cap (\{h\}\times TQ)$) . Then
  \begin{equation*}
    T^{23}_{(h,q_0,q_1)}(\tau_Q\circ
    F^{X_{L,f}}_t \circ
    (\del^{E\mp}_h|_{{\mathcal{U}^E}^*\cap
      (\{h\}\times TQ)})^{-1})(\delta q_0,\delta q_1) =
    \pd{q_{01}(t)}{q_0} \delta q_0 + \pd{q_{01}(t)}{q_1} \delta q_1
  \end{equation*}
  and, so,
  \begin{equation}\label{eq:exact_f_d-MW}
    f_{d,h}^E(q_0,q_1)(0,\delta q_0,\delta q_1) = \int_0^h
    \check{f}(q_{01}'(t))(\pd{q_{01}(t)}{q_0} \delta q_0 +
    \pd{q_{01}(t)}{q_1} \delta q_1)dt.
  \end{equation}
\end{example}

The expressions~\eqref{eq:exact_L_d-MW} and~\eqref{eq:exact_f_d-MW}
had been proposed in Section 3.2.4
of~\cite{ar:marsden_west-discrete_mechanics_and_variational_integrators}
as definitions for the exact discrete Lagrangian and force associated
to the regular \FMS $(Q,L,f)$. There, it is claimed that, for $h$
sufficiently small, $L_{d,h}^E$ and $f_{d,h}^E$ are well defined,
locally, near $(q_0,q_1)$. This claim is not verified in the paper:
even the $f=0$ case is not trivial and it has only been properly
solved
in~\cite{ar:marrero_martindediego_martinez-on_the_exact_discrete_lagrangian_function_for_variational_integrators_theory_and_applications}. Example~\ref{ex:exact_discretization_MW-simple_bd}
proves that, indeed, $L_{d,h}^E$ and $f_{d,h}^E$ are well
defined. Below, in
Theorem~\ref{th:flow_of_exact_discrete_systems-FDMS}, we study an
additional property of the exact discretization $\discMWE$ that
explains the ``exact'' part of its name.

\begin{definition}\label{def:contact_order_discMW}
  Let $(Q,L,f)$ be a regular \FMS and
  $\discCPEexp = (Q, \psi^E, L_\CPDS^E, f_\CPDS^E)$ its exact
  \discretizationCPExp introduced in
  Example~\ref{ex:exact_discretization_forced-system}. A
  \discretizationMW $\discMW = (Q,W,L_{d,\cdot},f_{d,\cdot})$ has
  \jdef{order $r$ contact} if
  \begin{equation*}
    L_{ext,\cdot} = L_\CPDS^E + \mathcal{O}(\grade_{\R\times TQ}^{r+1})
    \stext{ and }
    f_{ext,\cdot} = f_\CPDS^E + \mathcal{O}(\grade_{\R\times TQ}^{r+1}),
  \end{equation*}
  where $L_{ext,\cdot}$ and $f_{ext,\cdot}$ are the extensions of
  $L_{d,\cdot}\circ \ti{\del}^{E\mp}|_{{\mathcal{U}^E}^*}$ and
  $(\ti{\del}^{E\mp}|_{{\mathcal{U}^E}^*})^{*_2}(f_{d,\cdot})$ to
  $\mathcal{U}^E$.
\end{definition}

\begin{remark}
  Definition~\ref{def:contact_order_discMW} extends to the (smooth)
  forced case the contact order notion for $L_d$ introduced in Section
  2.3.1
  of~\cite{ar:marsden_west-discrete_mechanics_and_variational_integrators},
  where the expression
  \begin{equation}
    \label{eq:L_d_contact_order_MW-def}
    \abs{L_d(q(0),q(h),h)-L_d^E(q(0),q(h),h)}\leq C h^{r+1}
  \end{equation}
  is used, and where $q(t)$ is a trajectory of the continuous system
  with initial condition $(q,\dot{q})$ in a certain open subset of
  $TQ$. Indeed,
  $L_d(q(0),q(h),h) = (L_{d,\cdot}\circ
  \ti{\del}^{E\mp}|_{{\mathcal{U}^E}^*})(h,q,\dot{q})$ and
  $L_d^E(q(0),q(h),h) = L^E_\CPS(h,q,\dot{q})$, so that
  condition~\eqref{eq:L_d_contact_order_MW-def} (for smooth functions)
  is equivalent to our
  $L_{ext,\cdot} = L_\CPDS^E + \mathcal{O}(\grade_{\R\times
    TQ}^{r+1})$. Notice that the introduction of $L_{ext}$ is needed
  because $L_d$ is usually not well defined when $h=0$.
\end{remark}

\begin{corollary}\label{cor:correspondence_CP_vs_MW_preserves_order}
  The correspondence between \discretizationsCPExp and \discretizationsMW
  of a regular \FMS defined in
  Proposition~\ref{prop:correspondence_discretizations_CP_and_MW}
  preserves the contact order.
\end{corollary}

\begin{proof}
  Let $\discCPexp=(Q,\psi^E,L_\CPDS,f_\CPDS)$ be a
  \discretizationCPExp and
  $\discMW=(Q,\mathcal{W}^E,L_{d,\cdot},f_{d,\cdot})$ be
  \discretizationMW of the regular \FMS $(Q,L,f)$ related under the
  bijection defined in
  Proposition~\ref{prop:correspondence_discretizations_CP_and_MW}.
  According to the proof of
  Proposition~\ref{prop:correspondence_discretizations_CP_and_MW} the
  extensions $L_{ext,\cdot}$ and $f_{ext,\cdot}$ of $\discMW$ are
  $L_\CPDS$ and $f_\CPDS$ respectively.  Then $\discCPexp$
  has order $r$ contact if
  $L_\CPDS = L_\CPDS^E + \mathcal{O}(\grade_{\R\times
    TQ}^{r+1})$ and
  $f_\CPDS = f_\CPDS^E + \mathcal{O}(\grade_{\R\times
    TQ}^{r+1})$. These conditions are then, the same as
  $L_{ext,\cdot} = L_\CPDS^E + \mathcal{O}(\grade_{\R\times
    TQ}^{r+1})$ and
  $f_{ext,\cdot} = f_\CPDS^E + \mathcal{O}(\grade_{\R\times
    TQ}^{r+1})$, that are the conditions for $\discMW$ to have order
  $r$ contact.
\end{proof}

Let $\discMW = (Q,\mathcal{W}^E,L_{d,\cdot},f_{d,\cdot})$ be a
\discretizationMW of the regular \FMS $(Q,L,f)$. Consider the space
of (length $2$) paths
\begin{equation*}
  \mathcal{C}_{\MWS}^* := p_{123}^{-1}({\mathcal{W}^E}^*) \cap
  p_{134}^{-1}({\mathcal{W}^E}^*),
\end{equation*}
that is an open subset of $(\R\SM\{0\}) \times Q\times Q\times Q$, the
function $g_{\MWS}:\mathcal{C}_{\MWS}^*\rightarrow \R\times Q\times Q$
defined by $g_{\MWS}(h,q_0,q_1,q_2):=(h,q_0,q_2)$ and the section of
$p_{234}^*(T^*(Q\times Q\times Q))$ over
$\mathcal{C}_{\MWS}^*$\footnote{The sub-index $\MWS$ used in some
  places, as in $\mathcal{C}_{\MWS}^*$, comes from Marsden and West,
  and is used for objects that relate to the description of Discrete
  Mechanics as reviewed in their
  paper~\cite{ar:marsden_west-discrete_mechanics_and_variational_integrators}.}
defined by
\begin{equation*}
  \mwAFDtp := p_{123}^*(\mwAFDop) + p_{134}^*(\mwAFDop) \stext{ where }
  \mwAFDop := d_{23}L_{d,\cdot} + f_{d,\cdot}.
\end{equation*}
Let $\discCPexp = (Q,\psi^E,L_\CPDS,f_\CPDS)$ be the
\discretizationCPExp of $(Q,L,f)$ associated to $\discMW$ by
Proposition~\ref{prop:correspondence_discretizations_CP_and_MW}.

\begin{prop}\label{prop:crit_points_CP_vs_MW}
  With the notation as above,
  \begin{enumerate}
  \item \label{it:crit_points_CP_vs_MW-problem}
    $(\mwAFDtp,g_{\MWS},\ker(Tg_{\MWS}))$ is a critical problem over
    $\mathcal{C}_{\MWS}^*$.
  \item \label{it:crit_points_CP_vs_MW-bijection} Let
    $F:\mathcal{C}^*\rightarrow \mathcal{C}_{\MWS}^*$ be defined by
    $F(h,v,\ti{v}):=(h, \del^{E-}_h(v), \del^{E+}_h(v),
    \del^{E+}_h(\ti{v}))$. Then $F$ establishes a bijection between
    the critical points of the critical problem
    $(\dAFtpst,g|_{\mathcal{C}^*},\ker(Tg|_{\mathcal{C}^*}))$ over
    $\mathcal{C}^*$ (associated to $\discCPexp$,
    see~\eqref{eq:submanifold_of_matching_arrows_after_blow_up},~\eqref{eq:g_cp-def}
    and~\eqref{eq:alpha_cp-def}) and those of the problem
    $(\mwAFDtp,g_{\MWS},\ker(Tg_{\MWS}))$.
  \end{enumerate}
\end{prop}

\begin{proof}
  Being
  $\mathcal{C}_{\MWS}^*\subset (\R\SM\{0\}) \times Q\times Q\times Q$ an
  open subset and $g_{\MWS}$ a projection, it is a submersion. Thus,
  point~\eqref{it:crit_points_CP_vs_MW-problem} follows from the definitions.

  As
  $\ti{\del}^{E\mp}: {\mathcal{U}^E}^*\rightarrow {\mathcal{W}^E}^*$
  are diffeomorphisms, so is
  $\ti{\del}^{E\mp}\times
  \ti{\del}^{E\mp}:{\mathcal{U}^E}^*\times{\mathcal{U}^E}^*
  \rightarrow {\mathcal{W}^E}^*\times {\mathcal{W}^E}^*$, and it is
  straightforward to derive from this result that $F$ is a
  diffeomorphism. Observe that
  \begin{equation*}
    g_{\MWS}(F(h,v,\ti{v})) = (h,\del^{E-}_h(v),\del^{E+}_h(\ti{v})) =
    g|_{\mathcal{C}^*}(h,v,\ti{v}).
  \end{equation*}
  By
  construction,~\eqref{eq:correspondence_discretizations_CP_and_MW-MW_from_CP}
  is valid, and we deduce from it that
  $((\ti{\del}^{E\mp}|_{{\mathcal{U}^E}^*})^{-1})^{*_2}(\dAFop) =
  \mwAFDop$ and, eventually, that for $(h,v)\in {\mathcal{U}^E}^*$ and
  $(0,\delta v) \in T_{(h,v)}(\R\times TQ)$,
  $\dAFop(h,v)(0,\delta v) =
  (\ti{\del}^{E\mp}|_{{\mathcal{U}^E}^*})^*(\mwAFDop)(h,v)(0,\delta
  v)$. Then, for $(h,v,\ti{v})\in \mathcal{C}^*$ and
  $(0,\delta v, \delta \ti{v}) \in
  \ker(Tg|_{\mathcal{C}^*})_{(h,v,\ti{v})}$,
  \begin{equation*}
    \begin{split}
      F^*(\mwAFDtp)&(h,v,\ti{v})(0,\delta v, \delta \ti{v}) =
      F^*(p_{123}^*(\mwAFDop)+p_{134}^*(\mwAFDop))(h,v,\ti{v})(0,\delta
      v, \delta \ti{v}) \\=& ((\underbrace{p_{123}\circ
        F}_{=\ti{\del}^{E\mp}|_{{\mathcal{U}^E}^*}\circ
        p_{12}})^*(\mwAFDop)+ (\underbrace{p_{134}\circ
        F}_{=\ti{\del}^{E\mp}|_{{\mathcal{U}^E}^*}\circ
        p_{13}})^*(\mwAFDop))(h,v,\ti{v})(0,\delta v, \delta \ti{v})
      \\=&(p_{12}^*((\ti{\del}^{E\mp}|_{{\mathcal{U}^E}^*})^*(\mwAFDop))
      + p_{13}^*((\ti{\del}^{E\mp}|_{{\mathcal{U}^E}^*})^*(\mwAFDop)))
      (h,v,\ti{v})(0,\delta v, \delta \ti{v}) \\=&
      ((\ti{\del}^{E\mp}|_{{\mathcal{U}^E}^*})^*(\mwAFDop)(h,v)(0,\delta
      v) +
      ((\ti{\del}^{E\mp}|_{{\mathcal{U}^E}^*})^*(\mwAFDop)(h,\ti{v})(0,\delta
      \ti{v}) \\=& \dAFop(h,v)(0,\delta v) + \dAFop(h,\ti{v})(0,\delta
      \ti{v}) = \dAFtpst(h,v,\ti{v})(0,\delta v,\delta \ti{v}).
    \end{split}
  \end{equation*}
  Then, point~\eqref{it:crit_points_CP_vs_MW-bijection} in the
  statement follows from
  Lemma~\ref{le:critical_problems_and_diffeomorphisms} with
  $f:=id_{Q\times Q}$.
\end{proof}

\begin{prop}\label{prop:FDMS_from_discretization_MW}
  Let $\discMW := (\mathcal{W}^E,f_{d,\cdot},L_{d,\cdot})$ be a
  \discretizationMW of the regular \FMS $(Q,L,f)$ and $Q_0\subset Q$ a
  relatively compact open subset. Then,
  \begin{enumerate}
  \item \label{it:FDMS_from_discretization_MW-system} there is $a>0$
    such that whenever $0\neq \abs{h}<a$, the tuple defined by
    ${\aFDMS}_h:=(Q_0,\mathcal{W}^E_h, L_{d,h}|_{\mathcal{W}^E_h},
    f_{d,h}|_{\mathcal{W}^E_h})$ with
    \begin{equation*}
      \mathcal{W}^E_h :=
      p_{23}(p_1^{-1}(\{h\})\cap \mathcal{W}^E\cap p_{23}^{-1}(Q_0\times Q_0))
      \subset Q_0\times Q_0
    \end{equation*}
    is a \FDMSMW.
  \item \label{it:FDMS_from_discretization_MW-equivalence} Let
    $(h,q_0,q_1,q_2)\in \{h\}\times
    \mathcal{C}_d^{\mathcal{W}^E_h}\subset \mathcal{C}_{\MWS}^*$ where
    $0<\abs{h}<a$. The following assertions are equivalent.
    \begin{enumerate}
    \item \label{it:FDMS_from_discretization_MW-equivalence-cp_MW}
      $(h,q_0,q_1,q_2)$ is a critical point of the problem
      $(\mwAFDtp,g_{\MWS},\ker(Tg_{\MWS}))$ associated to $\discMW$.
    \item
      \label{it:FDMS_from_discretization_MW-equivalence-cp_h_fixed}
      $(q_0,q_1,q_2)$ is a critical point of the problem
      $(\mwAFtp,g_d,\ker(T g_d))$ associated to ${\aFDMS}_h$.
    \item \label{it:FDMS_from_discretization_MW-equivalence-trajectory}
      $(q_0,q_1,q_2)$ is a trajectory of ${\aFDMS}_h$.
    \end{enumerate}
  \end{enumerate}
\end{prop}

\begin{proof}
  From Theorem~\ref{th:discretizations_yield_discrete_tangent_bundles}
  applied to the exact \discretizationCPExp $\discCPEexp$ associated with
  $(Q,L,f)$ with the given $Q_0$, there is $a>0$ so that, for
  $0\neq \abs{h}<a$,
  $\del^{E\mp}_h(\mathcal{V}^E_h)\subset Q_0\times Q_0$ is an open
  neighborhood of $\Delta_{Q_0}$. Furthermore, as seen in the proof of
  Theorem~\ref{th:discretizations_yield_discrete_tangent_bundles},
  $\mathcal{W}^E_h:=\del^{E\mp}_h(\mathcal{V}^E_h)$ is also given by
  the formula of the statement. Last, as $L_{d,\cdot}$ and
  $f_{d,\cdot}$ are smooth over $\mathcal{W}^E$, it follows that
  $L_{d,h}|_{\mathcal{W}^E_h}$ and $f_{d,h}|_{\mathcal{W}^E_h}$ are
  smooth over $\mathcal{W}^E_h$, concluding the proof of
  point~\eqref{it:FDMS_from_discretization_MW-system}.

  Let $\mathcal{C}_h:=\{h\}\times \mathcal{C}_d^{\mathcal{W}^E_h}$,
  that is an embedded submanifold of $\mathcal{C}_{\MWS}^*$. Define
  $i_h:\mathcal{C}_h\rightarrow \mathcal{C}_{\MWS}^*$ as the inclusion
  and $I_h:\mathcal{C}_d^{\mathcal{W}^E_h}\rightarrow \mathcal{C}_h$
  by $I_h(q_0,q_1,q_2):=(h,q_0,q_1,q_2)$; clearly, $I_h$ is a
  diffeomorphism. Then, $(h,q_0,q_1,q_2) \in \mathcal{C}_h$ is a
  critical point of the problem $(\mwAFDtp,g_{\MWS},\ker(T g_{\MWS}))$
  over $\mathcal{C}_{\MWS}^*$ if and only if it is a critical point of
  the problem
  $(i_h^*(\mwAFDtp),g_{\MWS}\circ i_h, \ker(T (g_{\MWS}\circ i_h))$ over
  $\mathcal{C}_h$. In turn, this last condition is equivalent under
  $I_h$ to $(q_0,q_1,q_2)$ being a critical point of the problem
  $(\mwAFtp,g_d,\ker(T g_d))$ over
  $\mathcal{C}_d^{\mathcal{W}^E_h}$. This proves the equivalence of
  points~\eqref{it:FDMS_from_discretization_MW-equivalence-cp_MW}
  and~\eqref{it:FDMS_from_discretization_MW-equivalence-cp_h_fixed}. The
  equivalence of
  points~\eqref{it:FDMS_from_discretization_MW-equivalence-cp_h_fixed}
  and~\eqref{it:FDMS_from_discretization_MW-equivalence-trajectory}
  follows immediately from
  Remark~\ref{rem:trajectoriies_FDMS_as_critical_points}.
\end{proof}

\begin{example}\label{ex:exact_FDMS}
  Given a regular \FMS $(Q,L,f)$ and $Q_0\subset Q$ a relatively
  compact open subset, by
  Proposition~\ref{prop:FDMS_from_discretization_MW} applied to $Q_0$
  and the exact \discretizationMW $\discMWE$ of $(Q,L,f)$, there is
  $a>0$ so that for any $h$ with $0\neq \abs{h}<a$,
  $\aFDMSE:=(\mathcal{W}^E_h, L_{d,h}^E|_{\mathcal{W}^E_h},
  f_{d,h}^E|_{\mathcal{W}^E_h})$ is a \FDMSMW that we call the
  \jdef{exact forced discrete mechanical system}.
\end{example}

The following result summarizes the relationship between the discrete
systems obtained from discretizations as well as their critical points
and trajectories.

\begin{theorem}
  \label{th:correspondence_between_systems_derived_from_discretizations}
  Let $(Q,L,f)$ be a regular \FMS, $Q_0\subset Q$ be a relatively
  compact open subset, and $\discCPexp$ and $\discMW$ be a \discretizationCPExp
  and a \discretizationMW of $(Q,L,f)$ related via the equivalence
  given in
  Proposition~\ref{prop:correspondence_discretizations_CP_and_MW}. Let
  $a>0$ be given by
  Theorem~\ref{th:discretizations_yield_discrete_tangent_bundles}
  applied to the exact \discretizationCPExp $\discCPEexp$ associated to
  $(Q,L,f)$ and the given $Q_0$. Then, for each $0\neq \abs{h}<a$,
  \begin{enumerate}
  \item \label{it:correspondence_between_systems_derived_from_discs-discs_and_sys}
    $\del^{E\mp}_h|_{\mathcal{V}^E_h}$ defines a correspondence
    between the \FDMSCPs
    $\aFDLSh = (Q_0,\mathcal{V}^E_h, L_\CPHS|_{\mathcal{V}^E_h},
    L_\CPHS|_{\mathcal{V}^E_h})$ constructed by
    Proposition~\ref{prop:FDLS_from_discretization-general_disc_TQ}
    (applied to the \discretizationCPExp
    $(Q,\psi^E,L_\CPDS,f_\CPDS)$) and the \FDMSMWs
    $\aFDMSh := (Q_0,\mathcal{W}^E_h, L_{d,h}|_{\mathcal{W}^E_h},
    L_{d,h}|_{\mathcal{W}^E_h})$ constructed by
    Proposition~\ref{prop:FDMS_from_discretization_MW}. Furthermore,
    the following diagram is commutative.
    \begin{equation*}
      \xymatrix{
        {\text{\discretizationsCPExp}} \ar@{<->}[r]^-{\ti{\del}^{E\mp}} \ar[d] &
        {\text{\discretizationsMW}} \ar[d] \\
        {\text{\FDMSCPsExp on $Q_0$}} \ar@{<->}[r]_-{\del^{E\mp}_h|_{\mathcal{V}^E_h}} &
        {\text{\FDMSMWs on $Q_0$}}
      }
    \end{equation*}
  \item For fixed \discretizationCPExp $\discCPexp$, \discretizationMW
    $\discMW$, \FDMSCPExp $\aFDLSh$ and \FDMSMW $\aFDMSh$ satisfying
    the relations stated above, we have
    \begin{enumerate}

    \item \label{it:correspondence_between_systems_derived_from_discs-crit_pt-disc_sys-CP}
      The map $p_{23}|_{\mathcal{C}^*}$ defines a bijection between
      the critical points in
      $\{h\}\times \mathcal{C}_d^{\mathcal{V}^E_h} \subset
      \mathcal{C}^*$ of the critical problem associated to $\discCPexp$
      and those of the problem associated to $\aFDLSh$.
    \item \label{it:correspondence_between_systems_derived_from_discs-crit_pt-disc_sys-MW}
      The map $p_{234}|_{\mathcal{C}_{\MWS}^*}$ defines a bijection
      between the critical points in
      $\{h\}\times \mathcal{C}_d^{\mathcal{W}^E_h} \subset
      \mathcal{C}_{\MWS}^*$ of the critical problem associated to
      $\discMW$ and those of the problem associated to $\aFDMSh$.
    \item \label{it:correspondence_between_systems_derived_from_discs-crit_pt-disc_disc}
      The map
      $F(h,v,\ti{v}):=(h,\del^{E-}_h(v),\del^{E+}_h(v),\del^{E+}_h(\ti{v}))$
      is a bijection between the critical points of $\discCPexp$ and
      those of $\discMW$. Furthermore, those critical points of
      $\discCPexp$ in
      $p_{12}^{-1}(\mathcal{V}^E_h)\cap p_{13}^{-1}(\mathcal{V}^E_h)
      \subset \{h\}\times TQ\times TQ$ correspond to those points of
      $\discMW$ in $\{h\}\times \mathcal{C}_d^{\mathcal{W}^E_h}$.
    \item \label{it:correspondence_between_systems_derived_from_discs-crit_pt-sys}
      The map
      $F(v,\ti{v}):=(\del^{E-}_h(v),\del^{E+}_h(v),\del^{E+}_h(\ti{v}))$
      is a bijection between the critical points of the problem
      $(\dAFnp,g_{\mathcal{C}_N}, \ker(Tg_{\mathcal{C}_N}))$
      associated to $\aFDLSh$ and the critical points of the problem
      $(\mwAFtp,g_d,\ker(Tg_d))$ associated to $\aFDMSh$.
    \item \label{it:correspondence_between_systems_derived_from_discs-crit_pt-sys_CP_MW}
      The map $F$ introduced in
      point~\eqref{it:correspondence_between_systems_derived_from_discs-crit_pt-sys}
      is a bijection between the trajectories of $\aFDLSh$ and those of
      $\aFDMSh$.
      \end{enumerate}
  \end{enumerate}
\end{theorem}

\begin{proof}
  Part~\eqref{it:correspondence_between_systems_derived_from_discs-discs_and_sys}
  follows comparing the constructions described in
  Propositions~\ref{prop:FDLS_from_discretization-general_disc_TQ}
  and~\ref{prop:FDMS_from_discretization_MW} and recalling that
  $\ti{\del}^{E\mp}|_{{\mathcal{U}^E}^*}$ is a
  diffeomorphism. Point~\eqref{it:correspondence_between_systems_derived_from_discs-crit_pt-disc_sys-CP}
  follows from Lemma~\ref{le:trajectories_and_critical_points-h_not_0}
  and Remark~\ref{rem:forced_variational_pple_as_skew_critical}.
  Point~\eqref{it:correspondence_between_systems_derived_from_discs-crit_pt-disc_sys-MW}
  follows from point~\ref{it:FDMS_from_discretization_MW-equivalence}
  in
  Proposition~\ref{prop:FDMS_from_discretization_MW}. Point~\eqref{it:correspondence_between_systems_derived_from_discs-crit_pt-disc_disc}
  follows from Proposition~\ref{prop:crit_points_CP_vs_MW}.
  Point~\eqref{it:correspondence_between_systems_derived_from_discs-crit_pt-sys}
  is a consequence of
  points~\eqref{it:correspondence_between_systems_derived_from_discs-crit_pt-disc_sys-CP},~\eqref{it:correspondence_between_systems_derived_from_discs-crit_pt-disc_sys-MW}
  and~\eqref{it:correspondence_between_systems_derived_from_discs-crit_pt-disc_disc}. Last,
  point~\eqref{it:correspondence_between_systems_derived_from_discs-crit_pt-sys_CP_MW}
  follows from
  point~\eqref{it:correspondence_between_systems_derived_from_discs-crit_pt-sys}
  using Remarks~\ref{rem:forced_variational_pple_as_skew_critical}
  and~\ref{rem:trajectoriies_FDMS_as_critical_points}.
\end{proof}

The following result proves that the trajectories of the exact
discrete \FDMSMW $\aFDMSE$ associated to a regular \FMS $(Q,L,f)$ are,
precisely, the trajectories of $(Q,L,f)$ evaluated at discrete time
steps $t=0,h,2h,\ldots$.

\begin{theorem}\label{th:flow_of_exact_discrete_systems-FDMS}
  Let $(Q,L,f)$ be a regular \FMS and $Q_0\subset Q$ be a relatively
  compact open subset. Let $a>0$ be the constant provided by
  Proposition~\ref{prop:FDMS_from_discretization_MW} when applied to
  $Q_0$ together with the exact \discretizationMW $\discMWE$
  associated to $(Q,L,f)$
  (Example~\ref{ex:exact_discretization_MW-simple_bd}); consider the
  special case where $\dBTp(h)=h$ and $\dBTm(h)=0$ for all $h$. Then,
  for any $\conj{q}\in Q_0$ and $h$ such that $0\neq \abs{h}<a$, there
  is an open subset $\hat{U} \subset Q\times Q\times Q$ such that
  $\conj{q})\in p_1(\hat{U})$ and, for all $(q_0,q_1,q_2)\in \hat{U}$,
  the following statements are equivalent.
  \begin{enumerate}
  \item \label{it:flow_of_exact_discrete_systems-FDMS-traj}
    $(q_0,q_1,q_2)$ is a trajectory of the exact \FDMSMW
    $\aFDMSEh = (Q_0,\mathcal{W}^E_h,L_{d,h}^E,f_{d,h}^E)$ (see
    Example~\ref{ex:exact_FDMS}).
  \item \label{it:flow_of_exact_discrete_systems-FDMS-flow} There is a
    trajectory $q:[0,2h]\rightarrow Q$ of $(Q,L,f)$ with
    $q'(0),q'(h)\in \mathcal{V}^E_h$ ---the domain of the exact \FDMSCPExp
    for the same $Q_0$ and $h$ as above--- such that $q_0 = q(0)$,
    $q_1 = q(h)$ and $q_2 = q(2h)$.
  \end{enumerate}
\end{theorem}

\begin{proof}
  Let $\mathcal{U}^E\subset \R\times TQ$ be the open neighborhood of
  $\{0\}\times TQ$ obtained when
  Lemma~\ref{le:discretizations_yield_discrete_tangent_bundles-U_and_W}
  is applied to the exact \discretizationCPExp
  $\discCPEexp=(Q,\psi^E,L_\CPDS^E,f_\CPDS^E)$ associated to
  the \FMS $(Q,L,f)$. Also, let $U^{X_{L,f}}\subset \R\times TQ$ be
  the open subset introduced in
  Example~\ref{ex:exact_discretization_forced}, where the flow of the
  Lagrangian vector field $X_{L,f}$ is defined; recall from
  Example~\ref{ex:exact_discretization_forced} that, for the exact
  systems, the domains of $\del^{E+}$ and $\del^{E-}$ are
  $U^{X_{L,f}}$ and that, as observed in
  Lemma~\ref{le:discretizations_yield_discrete_tangent_bundles-U_and_W},
  $\mathcal{U}^E\subset U^{X_{L,f}}$. Last, let $\ti{U}$ be the open
  neighborhood of $\{0\}\times \Delta_{TQ_0}$ that appears in
  Theorem~\ref{thm:flow_of_exact_discrete_systems}.

  Let $\gamma_1:\R\times Q\rightarrow \R\times TQ$ be
  $\gamma_1(t,q):=(t,0_q)$. As $\gamma_1$ is continuous and
  $\mathcal{U}^E\subset\R\times TQ$ is open with
  $\gamma_1(0,\conj{q})=(0,0_{\conj{q}})\in \mathcal{U}^E$, there are
  $a_1>0$ and $V^1_{\conj{q}}\subset Q_0\subset Q$ an open
  neighborhood of $\conj{q}$ such that if $q\in V^1_{\conj{q}}$ and
  $\abs{t}<a_1$, then
  $\gamma_1(t,q) \subset \mathcal{U}^E \subset U^{X_{L,f}}$. In
  particular, $(t,0_q)\in U^{X_{L,f}}$.

  The same argument as in the previous paragraph with the functions
  $\gamma_2:(-a_1,a_1)\times V^1_{\conj{q}}\rightarrow \R\times TQ$
  and
  $\gamma_3:(-a_1,a_1)\times V^1_{\conj{q}}\rightarrow \R\times
  TQ\times TQ$ defined by $\gamma_2(t,q):=(t,F^{X_{L,f}}_t(0_q))$ and
  $\gamma_3(t,q):=(t,0_q,F^{X_{L,f}}_t(0_q))$ using the open subsets
  $\mathcal{U}^E$ and $\ti{U}$ respectively, produces a constant
  $a_{23}\in (0,a_1]$ and an open neighborhood
  $V^{23}_{\conj{q}}\subset V^1_{\conj{q}}$ of $\conj{q}$ such that if
  $\abs{t}<a_{23}$ and $q\in V^{23}_{\conj{q}}$ we have
  $(t,F^{X_{L,f}}_t(0_q))\in \mathcal{U}^E$ and
  $(t,0_q,F^{X_{L,f}}_t(0_q))\in \ti{U}$.
  
  Observe that if $\abs{t}<a_{23}$ and $q\in V^{23}_{\conj{q}}$, then
  $(t,F^{X_{L,f}}_t(0_q))\in \mathcal{U}^E\subset U^{X_{L,f}}$, so
  that $F^{X_{L,f}}_{2t}(0_q) = F^{X_{L,f}}_{t}(F^{X_{L,f}}_{t}(0_q))$
  is defined. Now let
  $\gamma_4:(-a_{23},a_{23})\times V^{23}_{\conj{q}} \rightarrow
  Q\times Q$ be
  $\gamma_4(t,q):=(\tau_Q(F^{X_{L,f}}_t(0_q)),
  \tau_Q(F^{X_{L,f}}_{2t}(0_q))$. As $\gamma_4$ is continuous and
  $\gamma_4(0,\conj{q}) = (\conj{q},\conj{q})\in Q_0\times Q_0$ that
  is an open subset of $Q\times Q$, there are $a\in (0,a_{23}]$ and an
  open neighborhood $V_{\conj{q}}\subset V^{23}_{\conj{q}}$ of
  $\conj{q}$ in $Q_0$ such that if $\abs{t}<a$ and
  $q\in V_{\conj{q}}$, then
  $(\tau_Q(F^{X_{L,f}}_t(0_q)), \tau_Q(F^{X_{L,f}}_{2t}(0_q)))\in
  Q_0\times Q_0$.

  Collecting the previous results, we see that, if $q\in V_{\conj{q}}$
  and $\abs{t}<a$, then
  \begin{equation}
    \label{eq:flow_of_exact_discrete_systems-FDMS-v_01_and_v12-def}
    v_{01}:=0_q \stext{ and } v_{12}:=F^{X_{L,f}}_t(0_q)
  \end{equation}
  have the following properties.
  $(t,v_{01}),(t,v_{12})\in \mathcal{U}^E$ so, in particular, in the
  domain of $\del^{E+}$ and $\del^{E-}$. Then,
  \begin{equation*}
    \begin{split}
      \del^{E-}(t,v_{01}) =& \del^{E-}_t(v_{01}) = \tau_Q(0_q)
      = q\in Q_0, \\ \del^{E+}(t,v_{01}) =&
      \del^{E+}_t(v_{01}) =
      \tau_Q(F^{X_{L,f}}_t(v_{01})) = \tau_Q(F^{X_{L,f}}_t(0_q)) \in Q_0,\\
      \del^{E-}(t,v_{12}) =& \del^{E-}_t(v_{12}) =
      \tau_Q(v_{12}) =
      \tau_Q(F^{X_{L,f}}_t(0_q)) \in Q_0,\\
      \del^{E+}(t,v_{12}) =& \del^{E+}_t(v_{12}) =
      \tau_Q(F^{X_{L,f}}_t(F^{X_{L,f}}_t(0_q))) =
      \tau_Q(F^{X_{L,f}}_{2t}(0_q)) \in Q_0.
    \end{split}
  \end{equation*}
  The exact \FDMSCPExp
  $\aFDLSE_t = (Q_0,\mathcal{V}^E_t,L_\CPTS^E,f_\CPTS^E)$ constructed in
  Proposition~\ref{prop:FDLS_from_discretization-general_disc_TQ}
  (applied to the exact \discretizationCPExp $\discCPEexp$) has
  \begin{equation*}
    \mathcal{V}^E_t:=\{v\in TQ:(t,v)\in \mathcal{U}^E\}\cap
    (\del^{E\mp}_t)^{-1}(Q_0\times Q_0),
  \end{equation*}
  and, then, the previous computations show that
  $v_{01},v_{12} \in \mathcal{V}^E_t$ and that
  $\del^{E+}_t(v_{01}) = \del^{E-}_t(v_{12})$. Finally, observe that
  $(t,v_{01},v_{12}) \in \ti{U}$.

  With $h$ such that $0\neq \abs{h}<a$, by
  point~\eqref{it:correspondence_between_systems_derived_from_discs-discs_and_sys}
  in
  Theorem~\ref{th:correspondence_between_systems_derived_from_discretizations}
  applied to $\aFDLSE_h$, we have the exact \FDMSMW
  $\aFDMSEh = (Q_0,\mathcal{W}^E_h,L_{d,h}^E,f_{d,h}^E)$ where
  $\mathcal{W}^E_h := \del^{E\mp}_h(\mathcal{V}^E_h)$. Define
  $\hat{U}_1:=p_{12}^{-1}(\mathcal{W}^E_h)\cap
  p_{23}^{-1}(\mathcal{W}^E_h)\subset Q\times Q\times Q$ and
  $R_h:\hat{U}_1\rightarrow \R\times TQ\times TQ$ by
  \begin{equation*}
    R_h(q_0,q_1,q_2):=(h,\del^{E\mp}_h|_{\mathcal{V}^E_h}^{-1}(q_0,q_1),
    \del^{E\mp}_h|_{\mathcal{V}^E_h}^{-1}(q_1,q_2)).
  \end{equation*}
  Finally, define $\hat{U}:=R_h^{-1}(\ti{U})$. Observe that, by our
  previous computations (with $t=h$), for any $q\in V_{\conj{q}}$, if
  we define $v_{01},v_{12}\in \mathcal{V}^E_h$
  by~\eqref{eq:flow_of_exact_discrete_systems-FDMS-v_01_and_v12-def}
  we have that $\del^{E+}_h(v_{01})=\del^{E-}_h(v_{12})$, so that
  $\del^{E\mp}_h(v_{01}) = (q_0,q_1) \in \mathcal{W}^E_h$ and
  $\del^{E\mp}_h(v_{12}) = (q_1,q_2) \in \mathcal{W}^E_h$. In
  addition, $R_h(q_0,q_1,q_2) = (h, v_{01},v_{12}) \in \ti{U}$, so
  that $(q_0,q_1,q_2)\in \hat{U}$; observe that taking
  $q_0=q=\conj{q}$, $\conj{q}\in p_1(\hat{U})$.

  In what follows, let $(q_0,q_1,q_2)\in \hat{U}$.

  \eqref{it:flow_of_exact_discrete_systems-FDMS-traj} \imp
  \eqref{it:flow_of_exact_discrete_systems-FDMS-flow}. If
  $(q_0,q_1,q_2)$ is a trajectory of the \FDMSMW $\aFDMSEh$, let
  $v_{01}:=\del^{E\mp}_h|_{\mathcal{V}^E_h}^{-1}(q_0,q_1)$ and
  $v_{12}:=\del^{E\mp}_h|_{\mathcal{V}^E_h}^{-1}(q_1,q_2)$, so that
  $v_{01},v_{12}\in \mathcal{V}^E_h$. By
  point~\eqref{it:correspondence_between_systems_derived_from_discs-crit_pt-sys_CP_MW}
  in
  Theorem~\ref{th:correspondence_between_systems_derived_from_discretizations}
  $(v_{01},v_{12})$ is a trajectory of the exact \FDMSCPExp $\aFDLSE_h$.
  Also, $(h,v_{01},v_{12}) = R_h(q_0,q_1,q_2) \in \ti{U}$ and, by
  Theorem~\ref{thm:flow_of_exact_discrete_systems}, we have that
  $v_{12}=F^{X_{L,f}}_h(v_{01})$.

  Let $q(t):=\tau_Q(F^{X_{L,f}}_t(v_{01}))$ so that
  $q'(t) = F^{X_{L,f}}_t(v_{01})$ and $q(t)$ is a trajectory of
  $(Q,L,f)$. Then,
  \begin{equation*}
    \begin{split}
      q(0) =& \tau_Q(F^{X_{L,f}}_0(v_{01})) =\tau_Q(v_{01}) =
      \del^{E-}_h(v_{01}) = q_0,\\
      q(h) =& \tau_Q(F^{X_{L,f}}_h(v_{01})) = \del^{E+}_h(v_{01}) = q_1,\\
      q(2h) =& \tau_Q(F^{X_{L,f}}_{2h}(v_{01})) =
      \tau_Q(F^{X_{L,f}}_h(F^{X_{L,f}}_h(v_{01}))) =
      \del^{E+}_h(v_{12}) = q_2.
    \end{split}
  \end{equation*}
  As we also have
  \begin{equation*}
    q'(0) = F^{X_{L,f}}_0(v_{01}) = v_{01} \in \mathcal{V}^E_h \stext{ and }
    q'(h) = F^{X_{L,f}}_h(v_{01}) = v_{12} \in \mathcal{V}^E_h,
  \end{equation*}
  we conclude that
  statement~\eqref{it:flow_of_exact_discrete_systems-FDMS-flow} is
  valid.
  
  \eqref{it:flow_of_exact_discrete_systems-FDMS-flow} \imp
  \eqref{it:flow_of_exact_discrete_systems-FDMS-traj}. Let $q(t)$ be a
  trajectory of the \FMS $(Q,L,f)$ such that $q(0)=q_0$, $q(h)=q_1$,
  $q(2h)=q_2$, and $q'(0),q'(h)\in \mathcal{V}^E_h$. Let
  $v_{01}:=q'(0)$ and $v_{12}:=q'(h)$. As $q(t)$ is a trajectory of
  $(Q,L,f)$ we have that $q'(t) = F^{X_{L,f}}_t(v)$ for some $v\in TQ$
  but, as $v_{01} = q'(0) = F^{X_{L,f}}_0(v) = v$, we conclude that
  $q'(t) = F^{X_{L,f}}_t(v_{01})$. Notice that
  $v_{12} = q'(h) = F^{X_{L,f}}_h(v_{01})$. Since
  \begin{equation*}
    \begin{split}
      \del^{E\mp}_h(v_{01}) =& (\del^{E-}_h(v_{01}),
      \del^{E+}_h(v_{01})) = (\tau_Q(v_{01}),
      \tau_Q(F^{X_{L,f}}_h(v_{01}))) \\=& (\tau_Q(q'(0)),
      \tau_Q(q'(h)))
      = (q(0),q(h)) = (q_0,q_1), \stext{ and }\\
      \del^{E\mp}_h(v_{12}) =& (\del^{E-}_h(v_{12}),
      \del^{E+}_h(v_{12})) = (\tau_Q(v_{12}),
      \tau_Q(F^{X_{L,f}}_h(v_{12}))) \\=& (\tau_Q(q'(h)),
      \tau_Q(F^{X_{L,f}}_{2h}(v_{01})) =(\tau_Q(q'(h)),
      \tau_Q(q'(2h))) \\=& (q(h),q(2h)) = (q_1,q_2),
    \end{split}
  \end{equation*}
  we have that $(h,v_{01},v_{12}) = R_h(q_0,q_1,q_2) \in
  \ti{U}$. Hence, by Theorem~\ref{thm:flow_of_exact_discrete_systems},
  we know that $(v_{01},v_{12})$ is a trajectory of $\aFDLSE_h$. Then,
  by
  point~\eqref{it:correspondence_between_systems_derived_from_discs-crit_pt-sys_CP_MW}
  in
  Theorem~\ref{th:correspondence_between_systems_derived_from_discretizations}
  we have that, for $F$ defined in
  point~\eqref{it:correspondence_between_systems_derived_from_discs-crit_pt-sys}
  of the same Theorem, $F(v_{01},v_{12}) = (q_0,q_1,q_2)$ is a
  trajectory of $\aFDMSEh$.
\end{proof}

\begin{definition}\label{def:flow_discretization_MW}
  Let $\discMW = (Q,W,L_{d,\cdot},f_{d,\cdot})$ be a \discretizationMW
  of the regular \FMS $(Q,L,f)$. A \jdef{flow} of $\discMW$ is a
  smooth function $F:{W'}^*\rightarrow Q\times Q$ where
  $W'\subset \mathcal{W}^E$ and
  ${W'}^*:=W'\SM (\{0\}\times Q\times Q)$ such that, for any
  $(h,q_0,q_1)\in {W'}^*$, $F(h,q_0,q_1) = (q_1,q_2)$ for some
  $q_2\in Q$ such that $(h,q_0,q_1,q_2) \in \mathcal{C}_h$ is a
  critical point of the problem $(\mwAFDtp,g_{\MWS},\ker(Tg_{\MWS}))$ over
  $\mathcal{C}_{\MWS}^*$. An \jdef{extended flow} of $\discMW$ is a
  function $\ti{F}:{W'}^*\rightarrow \R\times Q\times Q$ such that
  $\ti{F}(h,q_0,q_1) = (h, F(h,q_0,q_1))$ for a flow $F$ of $\discMW$.
\end{definition}

\begin{lemma}\label{le:flows_discretizations_CP_and_MW-critical}
  Let $\discCPexp=(Q,\psi^E,L_\CPDS,f_\CPDS)$ and
  $\discMW=(Q,\mathcal{W}^E,L_{d,\cdot}, f_{d,\cdot})$ be a
  \discretizationCPExp and a \discretizationMW of the regular \FMS
  $(Q,L,f)$ that are related by the bijection given in
  Proposition~\ref{prop:correspondence_discretizations_CP_and_MW}. Let
  $U'\subset \R\times TQ$ be an open neighborhood of $\{0\}\times TQ$
  and $\ti{F}_\CPS:U'\rightarrow \R\times TQ$ an extended flow of
  $\discCPexp$. Define
  $U'':=U'\cap \mathcal{U}^E \cap \ti{F}_\CPS^{-1}(\mathcal{U}^E)$,
  ${U''}^*:=U''\SM (\{0\}\times TQ)$, $W'':=\ti{\del}^{E\mp}(U'')$ and
  ${W''}^*:=\ti{\del}^{E\mp}({U''}^*)$. Then
  $\ti{F}_{\MWS}:{W''}^*\rightarrow \R\times Q\times Q$ defined by
  \begin{equation}
    \label{eq:flow_MW_from_flow_CP-def}
    \ti{F}_{\MWS} :=  \ti{\del}^{E\mp}|_{{\mathcal{U}^E}^*} \circ \ti{F}_\CPS
    \circ (\ti{\del}^{E\mp}|_{{U''}^*})^{-1}
  \end{equation}
  is an extended flow of $\discMW$, that is, for any
  $(h,(q_0,q_1))\in {W''}^*$, if
  $(h,q_1,q_2) := \ti{F}_{\MWS}(h,q_0,q_1)$, then
  $(h,q_0,q_1,q_2) \in \mathcal{C}_h$ is a critical point of the
  critical problem $(\mwAFDtp,g_{\MWS},\ker(Tg_{\MWS}))$ over $\mathcal{C}_{\MWS}^*$.
\end{lemma}

\begin{proof}
  As $U'$, $\mathcal{U}^E$ and $\ti{F}_\CPS^{-1}(\mathcal{U}^E)$ are
  open neighborhoods of $\{0\}\times TQ$ in $\R\times TQ$, so is
  $U''$. Then, ${U''}^*\subset {\mathcal{U}^E}^*$ is open and, as
  $\ti{\del}^{E\mp}|_{{\mathcal{U}^E}^*}:{\mathcal{U}^E}^*\rightarrow
  {\mathcal{W}^E}^*$ is a diffeomorphism,
  ${W''}^* = \ti{\del}^{E\mp}({U''}^*)$ is open in
  ${\mathcal{W}^E}^*$. Then, $\ti{F}_{\MWS}$ is a well defined smooth
  map.

  For $(h,q_0,q_1)\in {W''}^*$ we have that
  $(h,q_0,q_1) = \ti{\del}^{E\mp}(h,v_{01})$ for some
  $(h,v_{01})\in {U''}^*$. Then, as $\ti{F}_\CPS$ is an extended flow
  of $\discCPexp$, $(h,v_{12}):=\ti{F}_\CPS(h,v_{01})$ must satisfy that
  $(h,v_{01},v_{12})\in \mathcal{C}^*$ is a critical point of
  $\discCPexp$; let $(h,q_1',q_2) = \ti{\del}^{E\mp}(h,v_{12})$ and
  notice that, as $(h,v_{01},v_{12})\in \mathcal{C}^*$ it must be
  $q_1'=q_1$. Then, by point~\eqref{it:crit_points_CP_vs_MW-bijection}
  of Proposition~\ref {prop:crit_points_CP_vs_MW},
  $(h,\del^{E-}_h(v_{01}),\del^{E+}_h(v_{01}),\del^{E+}_h(v_{12}))
  = (h,q_0,q_1,q_2)$, is a critical point of $\discMW$. Hence
  $\ti{F}_{\MWS}$ is an extended flow of $\discMW$, as stated.
\end{proof}

\begin{remark}\label{rem:extended_flow_MW_does_not_extend_but}
  Notice that the extended flow function $\ti{F}_{\MWS}$ constructed for
  $\discMW$ in Lemma~\ref{le:flows_discretizations_CP_and_MW-critical}
  is not defined when $h=0$. Still, we have that, because
  of~\eqref{eq:flow_MW_from_flow_CP-def},
  \begin{equation*}
    (\ti{\del}^{E\mp}|_{{\mathcal{U}^E}^*})^{-1} \circ \ti{F}_{\MWS} \circ
    \ti{\del}^{E\mp}|_{{U''}^*} = \ti{F}_\CPS|_{{U''}^*},
  \end{equation*}
  and $\ti{F}_\CPS|_{{U''}^*}$ extends smoothly to $U''$ that is an
  open neighborhood of $\{0\}\times TQ$ in $\R\times TQ$. Even more,
  $\ti{F}_\CPS|_{U''}(0,v) = (0,v)$ for all $v\in TQ$.
\end{remark}

As we noticed elsewhere, \discretizationsMW are not well behaved at
$h=0$. Motivated by
Remark~\ref{rem:extended_flow_MW_does_not_extend_but} we introduce a
notion of contact order for the flow of such discretizations.

\begin{definition}\label{def:contact_order_of_flows_MW}
  Let $\discMW = (\mathcal{W}^E,L_{d,\cdot},f_{d,\cdot})$ be a
  \discretizationMW of the regular \FMS $(Q,L,f)$ and
  $\ti{F}:W^*\rightarrow \R\times Q\times Q$ be an extended flow of
  $\discMW$. It is said that $\ti{F}$ has \jdef{contact order $r$} if
  there exists an open neighborhood $U$ of $\{0\}\times TQ$ in
  $\R\times TQ$ and a smooth extension $\ti{F}_{ext}$ of the function
  $(\ti{\del}^{E\mp}|_{{\mathcal{U}^E}^*})^{-1} \circ \ti{F}
  \circ \ti{\del}^{E\mp}|_{U^*}$ with $U^*:=U\SM(\{0\}\times TQ)$ such
  that
  \begin{equation*}
    \ti{F}_{ext} = \ti{F}_\CPS^E + \mathcal{O}(\grade_{\R\times TQ}^{r+1}).
  \end{equation*}
  where $\ti{F}_\CPS^E$ is an extended flow of the exact
  \discretizationCPExp $\discCPEexp$ associated to $(Q,L,f)$.
\end{definition}

\begin{theorem}\label{th:order_of_flow_MW_vs_order_discMW}
  Let $\discMW = (\mathcal{W}^E,L_{d,\cdot},f_{d,\cdot})$ be a
  \discretizationMW of the regular \FMS $(Q,L,f)$ that has order $r$
  contact. Then there exists an extended flow
  $\ti{F}:W^*\rightarrow \R\times Q\times Q$ that has contact order
  $r$.
\end{theorem}

\begin{proof}
  Let $\discCPexp = (Q,\psi^E,L_\CPDS,f_\CPDS)$ be the
  \discretizationCPExp associated to $\discMW$ by
  Proposition~\ref{prop:correspondence_discretizations_CP_and_MW}. By
  Theorem~\ref{th:flow_of_discretizations_of_FMS}, there is an open
  neighborhood $U'$ of $\{0\}\times TQ$ in $\R\times TQ$ and an extended flow
  function $\ti{F}_\CPS:U'\rightarrow \R\times TQ$ for
  $\discCPexp$. Then, by
  Lemma~\ref{le:flows_discretizations_CP_and_MW-critical} and
  Remark~\ref{rem:extended_flow_MW_does_not_extend_but}, the function
  $\ti{F}_{\MWS}$ defined by~\eqref{eq:flow_MW_from_flow_CP-def} is an
  extended flow of $\discMW$ with the required properties. Indeed, the
  correspondence $\discMW \leftrightarrow \discCPexp$ preserves the
  contact order by
  Corollary~\ref{cor:correspondence_CP_vs_MW_preserves_order}, so that
  $\discCPexp$ has contact order $r$. Then, by
  Theorem~\ref{th:contact_order_discrete_FMS}, $\ti{F}_\CPS$ also has
  contact order $r$, meaning that
  \begin{equation*}
    \ti{F}_\CPS = \ti{F}_\CPS^E + \mathcal{O}(\grade_{\R\times TQ}^{r+1})
  \end{equation*}
  and, as $\ti{F}_\CPS=\ti{F}_{ext}$ (see
  Remark~\ref{rem:extended_flow_MW_does_not_extend_but}), we conclude
  that the statement is true.
\end{proof}

Theorem~\ref{th:order_of_flow_MW_vs_order_discMW} proves that \FDMSMWs
coming from \discretizationsMW that have contact order $r$ (in the
sense of Definition~\ref{def:contact_order_discMW}) have flows that
have, also, contact order $r$ with the exact flow (in the sense of
Definition~\ref{def:contact_order_of_flows_MW}), a result that is
mentioned, but not proved, in Section 3.2.4
of~\cite{ar:marsden_west-discrete_mechanics_and_variational_integrators}.


\appendix


\section{Additional contact order properties}
\label{sec:additional_contact_order_properties}

In this appendix we prove some useful properties having to do with how
the contact order of maps relates to other operations.

\begin{lemma}\label{le:order_r_contact_for_borders_from_order_of_psis}
  If two discretizations of $TQ$, $\psi_1$ and $\psi_2$ have order $r$
  contact, then the corresponding border maps $\del^{+,1}$ and
  $\del^{+,2}$ as well as $\del^{-,1}$ and $\del^{-,2}$ have order $r$
  contact with respect to $h$, that is,
  $\del^{+,2} = \del^{+,1} + \mathcal{O}(\grade_{\R\times TQ}^{r+1})$
  and
  $\del^{-,2} = \del^{-,1} + \mathcal{O}(\grade_{\R\times TQ}^{r+1})$
  (see~\eqref{eq:gades_RxTQ_and_R^2xTQ}).
\end{lemma}

\begin{proof}
  We only deal with $\del^{+,j}$ (with the $-$ case being completely
  similar). We assume that $\psi_1$ and $\psi_2$ have a common domain
  $U\supset \{0\}\times\{0\}\times TQ$ (otherwise, take the
  intersection of both domains). As
  $\psi_2(h,t,v) = \psi_1(h,t,v) + \mathcal{O}(\grade_{\R^2\times
    TQ}^{r+1})$, for all $(h_0,0,v_0)\in U$, there is an open set
  $U'\subset U$ containing $(h_0,0,v_0)$ and a chart $(V,\phi)$ of $Q$
  containing $\psi_1(h_0,0,v_0) = \tau_Q(v_0) = \psi_2(h_0,0,v_0)$
  with $\psi_j(U')\subset V$ (for $j=1,2$) and a continuous function
  $(\delta \psi)_\phi:U'\rightarrow \R^{\dim(Q)}$ such that
  \begin{equation}\label{eq:order_r_condition_for_psi_j}
    \phi(\psi_2(h,t,v))-\phi(\psi_1(h,t,v)) = 
    t^{r+1} (\delta \psi)_\phi(h,t,v)
  \end{equation}
  for all $(h,t,v)\in U'$. 

  Let $A^+:(-a,a)\times TQ\rightarrow \R^2\times TQ$ be as in
  Definition~\ref{def:smooth_discretization} and define
  $U'':=(A^+)^{-1}(U')$ that is open because $A^+$ is continuous; also
  $(h_0,v_0)\in U''$. In addition,
  $\del^{+,j}(h,v) = \psi_j(A^+(h,v))$, so that
  $\del^{+,j}(U'')\subset V$. Taking $t=\dBTp(h)$
  in~\eqref{eq:order_r_condition_for_psi_j} for $(h,v)\in U''$ we
  obtain
  \begin{equation}\label{eq:order_r_condition_for_partial_j}
    \phi(\del^{+,2}(h,v))-\phi(\del^{+,1}(h,v)) = 
    \dBTp(h)^{r+1} (\delta \psi)_\phi(h,\dBTp(h),v).
  \end{equation}
  Recalling that $\dBTp$ is $C^1$ (at least), and that
  $\dBTp(0)=0$, we can write
  \begin{equation*}
    \dBTp(h) = (\dBTp)'(0) h + \epsilon_1(h) h
  \end{equation*}
  for $\epsilon_1(h):=\frac{\dBTp(h) - (\dBTp)'(0) h}{h}$. By
  differentiability of $\dBTp$ at $h=0$, we have that
  $\lim_{h\rightarrow 0}\epsilon_1(h)=0$; hence, defining
  $\epsilon_1(0)=0$, we see that $\epsilon_1$ is continuous at $h=0$
  (and $C^1$ for $h\neq 0$ because so is $\dBTp$). Going back
  to~\eqref{eq:order_r_condition_for_partial_j}, we have
  \begin{equation*}
    \begin{split}
      \phi(\del^{+,2}(h,v))-\phi(\del^{+,1}(h,v)) =&
      ((\dBTp)'(0) h + \epsilon_1(h) h)^{r+1} (\delta
      \psi)_\phi(h,\dBTp(h),v) \\=& h^{r+1}
      \underbrace{((\dBTp)'(0) + \epsilon_1(h))^{r+1}(\delta
        \psi)_\phi(h,\dBTp(h),v)}_{=:(\delta \del^+)_\phi(h,v)},
    \end{split}
  \end{equation*}
  where, by construction, $(\delta \del^+)_\phi(h,v)$ is continuous
  over $U''$. Thus,
  $\del^{+,2} = \del^{+,1} + \mathcal{O}(\grade_{\R\times TQ}^{r+1})$.
\end{proof}

\begin{lemma}\label{le:order_and_residuals_of_1_forms}
  Let $f:(X,\grade_X)\rightarrow (Y,\grade_Y)$ be a diffeomorphism and
  $\gF^j\in \mathcal{A}^1(Y)$ such that
  $\gF^2 = \gF^1 +\mathcal{O}(\grade_Y^r)$. Then
  $f^*(\gF^j)\in \mathcal{A}^1(X)$ satisfy
  $f^*(\gF^2) = f^*(\gF^1) +\mathcal{O}(\grade_X^r)$. Furthermore, if
  $x_0\in \grade_X^{-1}(0)$,
  \begin{equation}
    \label{eq:order_and_residuals_of_1_forms-residuals-1}
    \cpres^r(f^*\gF^2, f^*\gF^1)(x_0) =
    T_{\gF^1(f(x_0))}(T^*f)(\cpres^r(\gF^2,\gF^1)(f(x_0))).
  \end{equation}
\end{lemma}

\begin{proof}
  Recall that $f^*(\gF^j) = T^*f\circ \gF^j \circ f$ and that $f$
  has contact order $r$ with itself for any $r\in\N$ as well as
  $\cpres^r(f,f)=0$. Also, as $\gF^2 = \gF^1 +\mathcal{O}(\grade_Y^r)$,
  applying Proposition 3
  in~\cite{ar:cuell_patrick-skew_critical_problems} twice, we have
  that $f^*(\gF^2) = f^*(\gF^1) +\mathcal{O}(\grade_X^r)$ and, for
  $x_0\in \grade_X^{-1}(0)$,
  \begin{equation*}
    \begin{split}
      \cpres^r(f^*\gF^2, f^*\gF^1)(x_0) =& \cpres^r(T^*f \circ
      \gF^2\circ f, T^*f \circ \gF^1\circ f)(x_0) \\=&
      T_{\gF^1(f(x_0))}(T^*f)(\cpres^r(\gF^2\circ f, \gF^1\circ
      f)(x_0)) \\=&
      T_{\gF^1(f(x_0))}(T^*f)(\underbrace{\dot{f}^r(x_0)}_{=1}
      \cpres^r(\gF^2, \gF^1)(f(x_0)))
    \end{split}
  \end{equation*}
  which proves~\eqref{eq:order_and_residuals_of_1_forms-residuals-1}.
\end{proof}

\begin{remark}\label{rem:vertical_vectors_and_base_vectors_in_tangent_bundle}
  In the context of Lemma~\ref{le:order_and_residuals_of_1_forms}, we
  have that, if $\tau_{T^*Y,Y}T^*Y\rightarrow Y$ is the cotangent
  bundle, and $y_0\in \grade_Y^{-1}(0)$, then
  $\cpres^r(\gF^2,\gF^1)(y_0)\in T_{\gF^2(y_0)}T^*Y$ and
  \begin{equation*}
    \begin{split}
      T_{\gF^2(y_0)}\tau_{T^*Y,Y}(\cpres^r(\gF^2,\gF^1)(y_0))
      =& \cpres^r(\tau_{T^*Y,Y}\circ \gF^2, \tau_{T^*Y,Y}\circ
      \gF^1)(y_0) \\=& \cpres^r(id_Y,id_Y)(y_0) = 0,
    \end{split}
  \end{equation*}
  so that $\cpres^r(\gF^2,\gF^1)(y_0)$ is a vertical vector in
  $T_{\gF^2(y_0)} T^*Y$. It is well known that mapping $T^*_{y_0}Y$
  into $T_{\gF^2(y_0)} T^*Y$ by
  $\beta_{y_0}\mapsto \frac{d}{dt}\big|_{t=0}(\gF^2(y_0) + t
  \beta_{y_0})$ is an isomorphism onto the subspace of vertical
  vectors.
\end{remark}

In the next result we use the following notation for maps
$f:\R^m\rightarrow\R^n$ written as $f(x)=(f^1(x),\ldots,f^n(x))$. The
\jdef{usual derivative} is the map
$Df:\R^m\rightarrow \hom(\R^m,\R^n)$ given by
$Df(x)(\sum a_j e_j) = \sum_{jk} \pd{f^k}{x_j}(x) a_j e_k$ and the
\jdef{tangent map} is the map $Tf:T\R^m\rightarrow T\R^n$ such that
$Tf(x)(\delta x) := (f(x),Df(x)(\delta x))$.

\begin{lemma}\label{le:order_of_tangent_map}
  Let $f^j:X\rightarrow Y$ ($j=1,2$) be two smooth maps such that
  $f^2 = f^1+\mathcal{O}(\grade_X^r)$. Then $Tf^j:TX\rightarrow TY$
  satisfy $Tf^2 = Tf^1 +\mathcal{O}(\grade_{TX}^{r-1})$ for
  $\grade_{TX}:= \grade_X\circ \tau_{TX,X}$.
\end{lemma}

\begin{proof}
  For $x_0\in \grade_X^{-1}(0)$, using
  Lemma~\ref{le:order_of_maps_in_terms_of_local_expression}, there are
  local charts $(V,\phi)$ and $(W,\psi)$ of $X$ and $Y$ such that
  $x_0\in V$, $f_j(V)\subset W$ and an open subset $U'\subset \phi(V)$
  containing $\phi(x_0)$ as well as a continuous function
  $\delta\check{f}:U'\rightarrow \R^{\dim(Y)}$ such that
  \begin{equation}\label{eq:order_of_tangent_map-order_f_local}
    \check{f}_2(u')-\check{f}_1(u') = 
    \grade_{\R^{\dim(X)}}(u')^r \delta\check{f}(u') \stext{for all} u'\in U',
  \end{equation}
  where $\grade_{\R^{\dim(X)}} := \grade_X\circ \phi^{-1}$ and
  $\check{f}_j := \psi\circ f_j\circ \phi^{-1}$.

  Differentiating~\eqref{eq:order_of_tangent_map-order_f_local} we obtain
  \begin{equation*}
    T\check{f}_2(u')(\delta u')-T\check{f}_1(u')(\delta u') = 
    T(\grade_{\R^{\dim(X)}}^r \delta\check{f})(u')(\delta u')
    \stext{ for all } u'\in U',\ \delta u'\in T_{u'}U'
  \end{equation*}
  so that
  \begin{equation*}
    (\check{f}_2(u'),D\check{f}_2(u')(\delta u')) - 
    (\check{f}_1(u'),D\check{f}_1(u')(\delta u')) = 
    (\grade_{\R^{\dim(X)}}(u')^r \delta\check{f}(u'), 
    D(\grade_{\R^{\dim(X)}}^r \delta\check{f})(u')(\delta u')).
  \end{equation*}
  Then, as
  \begin{equation*}
    \begin{split}
      D(\grade_{\R^{\dim(X)}}^r \delta\check{f})(u')(\delta u')) =&
      \sum_{jk} \pd{(\grade_{\R^{\dim(X)}}^r
        (\delta\check{f})^k)}{u'_j}(u')(\delta u')_j e_k \\=& \sum_{j}
      r \grade_{\R^{\dim(X)}}(u')^{r-1}
      \pd{\grade_{\R^{\dim(X)}}}{u'_j}(u') (\delta u')_j
      \delta\check{f}(u') \\ &+ \sum_{jk} \grade_{\R^{\dim(X)}}(u')^r
      \pd{(\delta\check{f})^k}{u'_j}(\delta u')_j e_k \\=&
      \grade_{\R^{\dim(X)}}(u')^{r-1} \delta D\check{f}(u',\delta u')
    \end{split}
  \end{equation*}
  for
  $\delta D\check{f}(u',\delta u') := \sum_k (\delta
  D\check{f}(u',\delta u'))_k e_k$ and
  \begin{equation}\label{eq:order_of_tangent_map-delta_D_f}
    (\delta
    D\check{f}(u',\delta u'))_k := \sum_j \left( r \pd{\grade_{\R^{\dim(X)}}}{u'_j}(u')
      (\delta u')_j (\delta\check{f}(u'))_k + 
      \grade_{\R^{\dim(X)}}(u') \pd{(\delta\check{f})^k}{u'_j}(\delta
      u')_j\right),
  \end{equation}
  we see that
  \begin{equation*}
    \begin{split}
      \widecheck{Tf_2}(u',\delta u') - \widecheck{Tf_1}(u',\delta u')
      =& (\check{f}_2(u'),D\check{f}_2(u')(\delta u')) -
      (\check{f}_1(u'),D\check{f}_1(u')(\delta u')) \\=&
      (\grade_{\R^{\dim(X)}}(u')^r \delta\check{f}(u'),
      \grade_{\R^{\dim(X)}}^{r-1}(u') \delta D\check{f}(u',\delta
      u'))\\=&
      \underbrace{\grade_{\R^{\dim(X)}}^{r-1}(u')}_{=\grade_{T\R^{\dim(X)}}^{r-1}(u',\delta
        u')} \underbrace{(\grade_{\R^{\dim(X)}}(u')
        \delta\check{f}(u'), \delta D\check{f}(u',\delta
        u'))}_{=:\delta T\check{f}(u',\delta u')}
    \end{split}
  \end{equation*}
  with $\delta T\check{f}$ continuous on $TU'$. Hence, by
  Lemma~\ref{le:order_of_maps_in_terms_of_local_expression}, we
  conclude that $Tf^2 = Tf^1 +\mathcal{O}(\grade_{TX}^{r-1})$, as wanted.
\end{proof}

\begin{remark}\label{rem:order_of_tangent_map-on_subbundle}
  According to Lemma~\ref{le:order_of_tangent_map} the contact order
  of tangent maps may drop by one with respect to the contact order of
  the original maps. Notice that if $\delta u'$ is in
  $\ker(T_{u'}\grade_{\R^{\dim(X)}})$
  in~\eqref{eq:order_of_tangent_map-delta_D_f}, then one additional
  factor $\grade_{\R^{\dim(X)}}(u')$ is obtained and the contact order
  doesn't drop. In global terms, if one considers the restriction of
  $Tf^j$ to a subbundle $\mathcal{V}\subset TX$ such that
  $\mathcal{V}\subset \ker(T\grade_X)$, then
  $Tf^2|_\mathcal{V} = Tf^1|_\mathcal{V}
  +\mathcal{O}(\grade_\mathcal{V}^r)$ for
  $\grade_\mathcal{V}:=\grade_{TX}|_\mathcal{V}$.
\end{remark}

Let $M$ and $N$ be manifolds; then $(\R\times M,\grade_{\R\times M})$
with $\grade_{\R\times M}(h,m):=h$ is a manifold with a grade. Notice
that $\ker(T\grade_{\R\times M}) \simeq p_2^*TM$ and recall that, for
a map $f:\R\times M\rightarrow N$ we have the map
$T_2f:=T_{p_2^*TM}f : p_2^*TM \rightarrow TN$ introduced in
Section~\ref{sec:remarks_on_cartesian_products} as the fiberwise
restriction of $Tf$ to $p_2^*TM$.

\begin{corollary}\label{cor:order_of_tangent_map_2}
  Let $f^j:\R\times M\rightarrow N$ ($j=1,2$) be two smooth maps such
  that $f^2 = f^1+\mathcal{O}(\grade_M^r)$ for
  $\grade_M(h,m):=h$. Then,
  $T_2f^2 = T_2f^1 +\mathcal{O}(\grade_{p_2^*TM}^r)$ for
  $\grade_{p_2^*TM}(h,\delta m):= h$.
\end{corollary}

\begin{proof}
  It follows from Lemma~\ref{le:order_of_tangent_map} using
  Remark~\ref{rem:order_of_tangent_map-on_subbundle} for
  $\mathcal{V}:=p_2^*TM$.
\end{proof}

When $f:X\rightarrow Y$ is a diffeomorphism, one can define the
\jdef{cotangent map} $T^*f:T^*Y\rightarrow T^*X$ by
$T^*f(\gF_y) := \gF_y \circ Tf(f^{-1}(y))$. If $(U,\phi)$
and $(V,\psi)$ are coordinates in $X$ and $Y$ such that
$f(U)\subset V$ we have the commutative diagrams
\begin{equation*}
  \xymatrix{ {U}\ar[r]^{f|_U} \ar[d]_{\phi} & {V} \ar[d]^{\psi} \\
    {\R^{\dim(X)}} \ar[r]_{\check{f}} & {\R^{\dim(Y)}}
  }
  \stext{ and }
  \xymatrix{
    {T^*U} \ar[d]_{T^*\psi} & {T^*V} \ar[l]_{T^*f} \ar[d]^{T^*\phi} \\
    {T^*\R^{\dim(X)}} \ar[d]_{\Delta_{\dim(X)}} & {T^*\R^{\dim(Y)}} 
    \ar[d]^{\Delta_{\dim(Y)}} \ar[l]^{T^*\check{f}}\\
    {\R^{\dim(X)}\times \R^{\dim(X)}} & {\R^{\dim(Y)}\times \R^{\dim(Y)}} 
    \ar[l]^{\widecheck{T^*f}}
  }
\end{equation*}
where $\widecheck{T^*f}$, the coordinate expression of $T^*f$, is
$\widecheck{T^*f}(y,a) :=
(\check{f}^{-1}(y),\transpose{[D\check{f}(\check{f}^{-1}(y))]} a)$
(here we are identifying $(\R^n)^*$ with $\R^n$ using the canonical
inner product via $\Delta_n$).

\begin{lemma}\label{le:order_of_cotangent_map}
  Let $(X,\grade_X)$ and $(Y,\grade_Y)$ be manifolds with a grade and
  $f_j:X\rightarrow Y$ ($j=1,2$) be two diffeomorphisms such that
  $f_j(\grade_X^{-1}(0))\subset \grade_Y^{-1}(0)$ and
  $f_2 = f_1+\mathcal{O}(\grade_X^r)$. Then
  $T^*f_j:T^*Y\rightarrow T^*X$ satisfy
  $T^*f_2 = T^*f_1 +\mathcal{O}(\grade_{T^*Y}^{r-1})$ for
  $\grade_{T^*Y}:= \grade_Y\circ \tau_{T^*Y,Y}$.
\end{lemma}

\begin{proof}
  Fix $\gF_{y_0}\in \grade_{T^*Y}^{-1}(0)$ where
  $\grade_{T^*Y}:=\grade_Y\circ \tau_{T^*Y,Y}$ and
  $\tau_{T^*Y,Y}(\gF_{y_0}) = y_0$. Let
  $x_0:=f_j^{-1}(y_0) \in \grade_X^{-1}(0)$. Using
  Lemma~\ref{le:order_of_maps_in_terms_of_local_expression}, there are
  local charts $(V,\phi)$ and $(W,\psi)$ of $X$ and $Y$ such that
  $x_0\in V$, $f_j(V)\subset W$ and an open subset $U'\subset \phi(V)$
  containing $\phi(x_0)$ as well as a continuous function
  $\delta\check{f}:U'\rightarrow \R^{\dim(Y)}$ such that
  \begin{equation}\label{eq:order_of_cotangent_map-order_f_local}
    \check{f}_2(u')-\check{f}_1(u') = 
    \grade_{\R^{\dim(X)}}(u')^r \delta\check{f}(u') \stext{for all} u'\in U',
  \end{equation}
  where $\grade_{\R^{\dim(X)}} := \grade_X\circ \phi^{-1}$ and
  $\check{f}_j := \psi\circ f_j\circ \phi^{-1}$. It was shown in
  Lemma~\ref{le:order_of_tangent_map} that
  \begin{equation*}
    \widecheck{Tf_2}(u',\delta u') - \widecheck{Tf_1}(u',\delta u') = 
    \grade_{\R^{\dim(X)}}^{r-1}(u')\delta T\check{f}(u',\delta u'),
  \end{equation*}
  for all $u'\in U'$ and $\delta u'\in T_{u'}U'$, and where
  $\delta T\check{f}$ is continuous in $TU'$. In particular, looking
  at the second factor,
  \begin{equation*}
    D\check{f}_2(u',\delta u') - D\check{f}_1(u',\delta u') = 
    \grade_{\R^{\dim(X)}}^{r-1}(u')\delta D\check{f}(u',\delta u') 
    \stext{ in } \R^{\dim(Y)},
  \end{equation*}
  for all $u'\in U'$ and $\delta u'\in T_{u'}U'$ or, equivalently,
  \begin{equation}\label{eq:order_of_cotangent_map-order_Df_local}
    D\check{f}_2(u') - D\check{f}_1(u') = 
    \grade_{\R^{\dim(X)}}^{r-1}(u')\delta D\check{f}(u') 
    \stext{ in } \hom(\R^{\dim(X)},\R^{\dim(Y)}),
  \end{equation}
  for all $u'\in U'$.

  The local expressions $\widecheck{T^*f_j}$ of $T^*f_j$ with respect
  to the charts $(T^*W, \Delta_{\R^{\dim(Y)}}\circ T^*\psi^{-1})$
  ---which contains $\gF_{y_0}$--- and
  $(T^*V, \Delta_{\R^{\dim(X)}}\circ T^*\phi^{-1})$ are, for
  $u'\in U'$ and $a\in\R^{\dim(Y)}$,
  \begin{equation*}
    \widecheck{T^*f_j}(u',a) = 
    (\check{f}_j^{-1}(u'), \transpose{[D\check{f}_j(\check{f}_j^{-1}(u'))]} a).
  \end{equation*}

  Then,
  \begin{equation}\label{eq:order_of_cotangent_map-local_cotangent-1}
    \begin{split}
      \widecheck{T^*f_2}(u',a) - \widecheck{T^*f_1}(u',a) =&
      (\check{f}_2^{-1}(u'),
      \transpose{[D\check{f}_2(\check{f}_2^{-1}(u'))]} a) -
      (\check{f}_1^{-1}(u'),
      \transpose{[D\check{f}_1(\check{f}_1^{-1}(u'))]} a) \\=&
      (\underbrace{\check{f}_2^{-1}(u')-\check{f}_1^{-1}(u')}_{=:I_1},
      \underbrace{(\transpose{[D\check{f}_2(\check{f}_2^{-1}(u'))]} -
        \transpose{[D\check{f}_1(\check{f}_1^{-1}(u'))]}) a}_{=:I_2}).
    \end{split}
  \end{equation}
  
  As $f_2=f_1+\mathcal{O}(\grade_X^r)$, by Proposition 4
  in~\cite{ar:cuell_patrick-skew_critical_problems}, and shrinking
  $U'$ if needed, we have that
  \begin{equation}\label{eq:order_of_cotangent_map-inverse_bounds}
    (\check{f_2})^{-1}(u')-(\check{f_1})^{-1}(u') = \grade_{\R^{\dim(Y)}}(u')^r \delta
    (\check{f})^{-1}(u') \stext{ for all } u'\in U',
  \end{equation}
  showing that
  \begin{equation}\label{eq:order_of_cotangent_map-inverse_bounds-2}
    I_1 = \grade_{\R^{\dim(Y)}}(u')^r \delta
    (\check{f})^{-1}(u').
  \end{equation}

  Also, we have
  \begin{equation}\label{eq:order_of_cotangent_map-order_Df_local-partial}
    D\check{f}_2(\underbrace{\check{f}_2^{-1}(u')}_{=:u'_2}) - 
    D\check{f}_1(\underbrace{\check{f}_1^{-1}(u')}_{=:u'_1}) = 
    D\check{f}_2(u'_2) - D\check{f}_2(u'_1) + D\check{f}_2(u'_1) - 
    D\check{f}_1(u'_1).
  \end{equation}
  Using that, from~\eqref{eq:order_of_cotangent_map-inverse_bounds},
  $u'_2-u'_1 = \grade_{\R^{\dim(Y)}}(u')^r \delta (\check{f})^{-1}(u')$, we
  have that
  \begin{equation*}
    \begin{split}
      D\check{f}_2(u'_2) - D\check{f}_2(u'_1) =& \int_0^1 \frac{d}{dt}
      D\check{f}_2(u'_1+t(u'_2-u'_1)) dt \\=& \int_0^1
      D(D\check{f}_2)(u'_1+t(u'_2-u'_1)) \grade_{\R^{\dim(Y)}}(u')^r \delta
      (\check{f})^{-1}(u') dt \\=& \grade_{\R^{\dim(Y)}}(u')^r
      \underbrace{\int_0^1 D(D\check{f}_2)(u'_1+t(u'_2-u'_1))\delta
        (\check{f})^{-1}(u')dt}_{=:\delta I_3(u')}.
    \end{split}
  \end{equation*}
  On the other hand,
  by~\eqref{eq:order_of_cotangent_map-order_Df_local},
  \begin{equation*}
    D\check{f}_2(u'_1) - D\check{f}_1(u'_1) = 
    \grade_{\R^{\dim(X)}}^{r-1}(\underbrace{u'_1}_{=\check{f}_1^{-1}(u')}) 
    \delta D\check{f}(u'_1) = 
    \grade_{\R^{\dim(Y)}}^{r-1}(u') \delta D\check{f}(\check{f}_1^{-1}(u')).
  \end{equation*}
  Back to~\eqref{eq:order_of_cotangent_map-order_Df_local-partial}
  applying the results of the last two computations we have that
  \begin{equation*}
    \begin{split}
      D\check{f}_2(\check{f}_2^{-1}(u')) -
      D\check{f}_1(\check{f}_1^{-1}(u')) = \grade_{\R^{\dim(Y)}}(u')^{r-1}
      \underbrace{(\grade_{\R^{\dim(X)}}(u') \delta I_3(u') + \delta
        D\check{f}(\check{f}_1^{-1}(u')))}_{\delta I(u')}.
    \end{split}
  \end{equation*}
  This last expression, together
  with~\eqref{eq:order_of_cotangent_map-inverse_bounds-2},
  turn~\eqref{eq:order_of_cotangent_map-local_cotangent-1} into
  \begin{equation}\label{eq:order_of_cotangent_map-local_error_estimate}
    \begin{split}
      \widecheck{T^*f_2}(u',a) - \widecheck{T^*f_1}(u',a) =&
      (\grade_{\R^{\dim(Y)}}(u')^r \delta (\check{f})^{-1}(u'),
      \grade_{\R^{\dim(Y)}}(u')^{r-1} \transpose{[\delta I(u')]}a) \\=&
      \grade_{\R^{\dim(Y)}}(u')^{r-1} \underbrace{(\grade_{\R^{\dim(Y)}}(u') \delta
      (\check{f})^{-1}(u'), \transpose{[\delta I(u')]}a)}_{=:\delta T^*f(u',a)},
    \end{split}
  \end{equation}
  with $\delta T^*f(u',a)$ continuous in $u'$ and $a$.  By
  Lemma~\ref{le:order_of_maps_in_terms_of_local_expression}, we
  conclude that $T^*f_2 = T^*f_1+\mathcal{O}(\grade_{T^*Y}^{r-1})$.
\end{proof}

\begin{remark}\label{rem:order_of_cotangent_map-on_subbundle}
  According to Lemma~\ref{le:order_of_cotangent_map} the contact order
  of cotangent maps may drop by one with respect to the contact order
  of the original maps. This is a direct consequence of
  Lemma~\ref{le:order_of_tangent_map}, as can be seen in the
  expression~\eqref{eq:order_of_cotangent_map-order_Df_local} that is,
  eventually, used to find the local
  expression~\eqref{eq:order_of_cotangent_map-local_error_estimate}. Notice
  that if one were to restrict $Tf_j$ to any subbundle
  $\mathcal{V}\subset\ker(T\grade_M)$, as observed in
  Remark~\ref{rem:order_of_tangent_map-on_subbundle}, the contact
  order of the restricted $Tf_j$ would be the same as the one among
  the $f_j$ and, in the local
  description~\eqref{eq:order_of_cotangent_map-order_Df_local}, the
  power $r-1$ would change to $r$ (when evaluated in the local
  coordinates of the elements of $\mathcal{V}$). This change,
  eventually, leads to the corresponding change of the contact order
  of the (restricted) cotangent maps.
\end{remark}

\begin{remark}\label{rem:restrictions_of_tangent_and_cotangent_maps}
  Let $\mathcal{V}\subset TX$ be a subbundle such that
  $\mathcal{V}\subset\ker(T \grade_X)$; denote the inclusion of
  $\mathcal{V}$ into $TX$ by $i_\mathcal{V}$. For any
  $f:X\rightarrow Y$ we define $T_\mathcal{V}f:=Tf|_{\mathcal{V}}$. If
  $f_1,f_2:X\rightarrow Y$ are such that $f_2=f_1+\mathcal{O}(\grade_X^r)$,
  from Remark~\ref{le:order_of_tangent_map} we have that
  $T_\mathcal{V} f_2 = T_\mathcal{V}
  f_1+\mathcal{O}(\grade_X^r)$. Similarly, when $f:X\rightarrow Y$ is a
  diffeomorphism and $\mathcal{V}$ is as above, we define
  $T_\mathcal{V}^*f:=i_\mathcal{V}^*\circ T^*f:T^*Y\rightarrow
  \mathcal{V}^*$. If $f_1,f_2:X\rightarrow Y$ are diffeomorphisms such
  that $f_2=f_1+\mathcal{O}(\grade_X^r)$, from
  Remark~\ref{rem:order_of_cotangent_map-on_subbundle} we have that
  $T_\mathcal{V}^* f_2 = T_\mathcal{V}^* f_1+\mathcal{O}(\grade_X^r)$.
\end{remark}

Let $M$ be a manifold; then $(\R\times M,\grade_{\R\times M})$ with
$\grade_{\R\times M}(h,m):=h$ is a manifold with a grade. Let
$(N,\grade_N)$ be another manifold with a grade and
$f:\R\times M\rightarrow N$ be a diffeomorphism. Consider the
subbundle $\ker(T\grade_{\R\times M}) \simeq p_2^*TM$ of
$T(\R\times M)$ and the map
$T_2^*f:=T_{p_2^*TM}^*f : T^*N \rightarrow p_2^*T^*M$ introduced in
Section~\ref{sec:remarks_on_cartesian_products}.

\begin{corollary}\label{cor:order_of_cotangent_map_2}
  Let $f^j:\R\times M\rightarrow N$ ($j=1,2$) be two diffeomorphisms
  such that $f^j(\grade_{\R\times M}^{-1}(0))\subset \grade_N^{-1}(0)$
  for $\grade_{\R\times M}=p_1$. Assume that
  $f^2 = f^1+\mathcal{O}(\grade_{\R\times M}^r)$. Then
  $T^*_2f^j:T^*N\rightarrow p_2^*T^*M$ satisfy
  $T^*_2f^2 = T^*_2f^1 +\mathcal{O}(\grade_{T^*N}^r)$ for
  $\grade_{T^*N}:=\grade_N\circ \tau_{T^*N,N}$.
\end{corollary}

\begin{proof}
  It follows from Lemma~\ref{le:order_of_cotangent_map} using
  Remark~\ref{rem:restrictions_of_tangent_and_cotangent_maps} for
  $\mathcal{V}:=p_2^*TM$.
\end{proof}

\begin{lemma}\label{le:order_of_pullback_map_2}
  Let $f^j:\R\times M\rightarrow \R\times N$ ($j=1,2$) be two
  diffeomorphisms such that
  $\grade_{\R\times N} \circ f^j = \grade_{\R\times M}$ for
  $\grade_{\R\times M}=p_1$ and $\grade_{\R\times N}=p_1$ and let
  $\gF^k\in \Gamma(\R\times N,p_2^* T^*N)$. Assume that
  $f^2 = f^1+\mathcal{O}(\grade_{\R\times M}^r)$ and that
  $\gF_2 = \gF_1 + \mathcal{O}(\grade_{\R\times N}^r)$. Then
  $(f^2)^{*_2}(\gF_2) = (f^1)^{*_2}(\gF_1) +
  \mathcal{O}(\grade_{\R\times M}^r)$, for $*_2$ defined
  in~\eqref{eq:*_2-def}. Furthermore, if we consider the map
  $\grade_{p_2^*T^*N} := \grade_{\R\times N}\circ \tau_{T^*(\R\times
    N)}$, we have
  \begin{equation}\label{eq:order_of_pullback_map_2-residual}
    \begin{split}
      \cpres^r((f^2)^{*_2}(\gF_2),(f^1)^{*_2}(&\gF_1))(0,m) =
      \cpres^r(T_2^*f^2,T_2^*f^1)(\gF_2(f^2(0,m))) \\& + T(T_2^*
      f^1)(\gF_1(f^2(0,m)))(\cpres^r(\gF_2,\gF_1)(f^2(0,m))) \\& +
      T(T_2^*f^1 \circ \gF_1)(f^2(0,m))(\cpres^r(f^2,f^1)(0,m))
    \end{split}
  \end{equation}
\end{lemma}

\begin{proof}
  From~\eqref{eq:*_2-def},
  $(f^j)^{*_2}(\gF_k) = T_2^*f^j \circ \gF_k \circ f^j$. We have that
  $\grade_{\R\times N} \circ f^j = \grade_{\R\times M}$ so that
  $f^j(\grade_{\R\times M}^{-1}(0)) \subset \grade_{\R\times
    N}^{-1}(0)$ and $\dot{f^j}(0,m)=1$. Also, as
  $\gF_j\in \Gamma(\R\times N, p_2^*T^*N)$, we have that
  $\grade_{p_2^*T^*N}(\gF_k(h,n)) =\grade_{\R\times
    N}(\tau_{T^*(\R\times N)}(\gF_k(h,n))) = \grade_{\R\times
    N}(h,n)$, so that
  $\grade_{p_2^*T^*N} \circ \gF_k = \grade_{\R\times N}$ and,
  consequently,
  $\gF_k(\grade_{\R\times N}^{-1}(0)) \subset
  \grade_{p_2^*T^*N}^{-1}(0)$ and $\dot{\gF_k}(0,n)=1$. As, by
  hypothesis, $f^2 = f^1+\mathcal{O}(\grade_{\R\times M}^r)$, by
  Corollary~\ref{cor:order_of_cotangent_map_2},
  $T_2^*f^2 = T_2^*f^1+\mathcal{O}(\grade_{p_2^*T^*N}^r)$ and, as
  $\gF_2 = \gF_1 + \mathcal{O}(\grade_{\R\times N}^r)$, applying
  Proposition 3 from~\cite{ar:cuell_patrick-skew_critical_problems}
  twice, we have that
  \begin{equation*}
    (f^2)^{*_2}(\gF_2) = T_2^*f^2 \circ \gF_2 \circ f^2 = T_2^*f^1
    \circ \gF_1 \circ f^1 + \mathcal{O}(\grade_{\R\times M}^r) =
    (f^1)^{*_2}(\gF_1) + \mathcal{O}(\grade_{\R\times M}^r).
  \end{equation*}
  As for the corresponding residuals, using the same result, we have
  \begin{equation*}
    \begin{split}
      \cpres^r((f^2)^{*_2}(\gF_2),(f^1)^{*_2}(\gF_1))(0,m) =&
      \cpres^r(T_2^*f^2 \circ \gF_2 \circ f^2, T_2^*f^1 \circ \gF_1
      \circ f^1)(0,m) \\=& \underbrace{\dot{f^2}(0,m)^r}_{=1}
      \cpres^r(T_2^*f^2 \circ \gF_2,T_2^*f^1 \circ \gF_1)(f^2(0,m))
      \\& + T(T_2^*f^1 \circ \gF_1)(f^2(0,m))(\cpres^r(f^2,f^1)(0,m))
    \end{split}
  \end{equation*}
  Notice that, since $\grade_{\R\times N} \circ f^2 = \grade_{\R\times M}$, we
  have $\dot{f^2}(0,m)=1$. Similarly
  \begin{equation*}
    \begin{split}
      \cpres^r(T_2^*f^2 \circ \gF_2,T_2^*f^1 \circ \gF_1)(f^2(0,m)) =&
      \underbrace{\dot{\gF_2}(f^2(0,m))^r}_{=1}
      \cpres^r(T_2^*f^2,T_2^*f^1)(\gF_2(f^2(0,m))) \\&+ T(T_2^*
      f^1)(\gF_1(f^2(0,m)))(\cpres^r(\gF_2,\gF_1)(f^2(0,m))).
    \end{split}
  \end{equation*}
  All together, we see that~\eqref{eq:order_of_pullback_map_2-residual}
  is valid.
\end{proof}


\section{Things about Grassmannians}
\label{sec:things_about_grassmannians}

In this appendix we review a few basic facts about the real
Grassmannian manifolds.


\subsection{Definition and coordinates}
\label{sec:grassmannian-definition_and_coordinates}

Let $V$ be a finite dimensional $\R$-vector space and $k\in\NZ$. The
Grassmannian $\Gr(V,k)$ is the set of $k$-dimensional subspaces of
$V$. Fixing $D\in \Gr(V,k)$, a coordinate chart of $\Gr(V,k)$
containing $D$ can be constructed as follows. Choose $F\subset V$ such
that $V=D\oplus F$ and let $p_D:V\rightarrow D$ and
$p_F:V\rightarrow F$ be the associated projections. Then,
$U_{D,F}:=\{D'\in \Gr(V,k) : p_D|_{D'}:D'\rightarrow D \text{ is an
  isomorphism }\} \subset \Gr(V,k) $ is an open subset and the map
$\phi_{D,F} : U_{D,F} \rightarrow \hom(D,F) \simeq
\R^{k\times(\dim(V)-k)}$ defined by
$\phi_{D,F}(D') := p_F\circ (p_D|_{D'})^{-1}$ is a coordinate
function\footnote{In fact, $\phi_{D,F}(D')$ is not a coordinate
  function because its codomain is $\hom(D,F)$ rather than
  $\R^{\dim(D)\dim(F)}$. This last step is achieved by fixing bases in
  $D$ and $F$.}. It is easy to check that,
$\phi_{D,F}^{-1}:\hom(D,F)\rightarrow U_{D,F}$ is
$\phi_{D,F}^{-1}(\Delta') = \im(i_{D,V} + i_{F,V}\circ \Delta')$,
where $i_{D,V} \in \hom(D,V)$ and $i_{F,V} \in \hom(F,V)$ are the
natural inclusions.


\subsection{Tangent spaces}
\label{sec:grassmannian-tangent_spaces}

Fix $D\in \Gr(V,k)$. In what follows we establish a (canonical)
isomorphism between $T_D\Gr(V,k)$ and $\hom(D,V/D)$ ---choosing a
complement $F$ so that $V=D\oplus F$, $V/D\simeq F$, the description
matches the one given in the previous paragraph.

Given $\zeta \in T_D\Gr(V,k)$, in what follows we define
$\conj{\zeta}\in \hom(D,V/D)$. First, let $F(t)$ be a smooth curve in
$\Gr(V,k)$ such that $F(0)=D$ and $F'(0)=\zeta$. Now, for any
$v\in D$, choose a smooth curve $v(t) \in F(t)\subset V$ for all $t$
and such that $v(0)=v$. Finally, define
\begin{equation*}
  \conj{\zeta}(v) := \pi^{V,V/D}(v'(0)).
\end{equation*}
It can be checked that $\zeta \mapsto \conj{\zeta}$ is well defined
and, indeed, it is an isomorphism.

Now, let $A\in\hom(V,\ti{V})$, where $\ti{V}$ is another $\R$-vector
space. Assume that $A$ is an isomorphism, so that it defines a map
$A^{\Gr}:\Gr(V,k)\rightarrow \Gr(\ti{V},k)$ by
$A^{\Gr}(\langle v_1,\ldots,v_k\rangle) := \langle
A(v_1),\ldots,A(v_k)\rangle$. Then, for $D\in \Gr(V,k)$, we can
compute the representative of
$T_DA^{\Gr}\in\hom(T_D\Gr(V,k), T_{\ti{D}}\Gr(\ti{V},k))$, where
$\ti{D}:=A^{\Gr}(D)$, under the isomorphism considered in the previous
paragraph. Explicitly, we have
\begin{equation*}
  \xymatrix{
    {T_D\Gr(V,k)} \ar[d]_{\sim} \ar[r]^{T_D A^{\Gr}} &
    {T_{\ti{D}}\Gr(\ti{V},k)} \ar[d]^{\sim}\\
    {\hom(D,V/D)} \ar[r]_{\conj{T_D A^{\Gr}}} & {\hom(\ti{D},\ti{V}/\ti{D})}
  }
\end{equation*}
Let $\zeta \in T_D\Gr(V,k)$ and choose $F(t)$ a smooth curve in
$\Gr(V,k)$ such that $F(0)=D$ and $F'(0)=\zeta$. Then
$\ti{F}(t) := A^{\Gr}(F(t))$ is a smooth curve in $\Gr(\ti{V},k)$
satisfying $\ti{F}(0) = \ti{D}$ and
$\ti{F}'(0) = T_D A^{\Gr}(\zeta) =: \ti{\zeta} \in
T_{\ti{D}}\Gr(\ti{V},k)$. Given $\ti{v}\in \ti{D}$, choose a smooth
curve $\ti{v}(t) \in \ti{F}(t) \in \Gr(\ti{V},k)$ so that
$\conj{\ti{\zeta}}(\ti{v}) =
\pi^{\ti{V},\ti{V}/\ti{D}}(\ti{v}'(0))$. Then we can define
$v(t) := A^{-1}(\ti{v}(t)) \in F(t)$ and $v:=v(0)\in D$ and we have
that $\conj{\zeta}(v) = \pi^{V,V/D}(v'(0))$. Now we combine the
constructions
\begin{equation*}
  \begin{split}
    \conj{\ti{\zeta}}(A(v)) =& \conj{\ti{\zeta}}(\ti{v}) =
    \pi^{\ti{V},\ti{V}/\ti{D}}(\ti{v}'(0)) =
    \pi^{\ti{V},\ti{V}/\ti{D}}((A\circ v)'(0)) \\=&
    \pi^{\ti{V},\ti{V}/\ti{D}}(A(v'(0))) =
    \conj{A}(\pi^{V,V/D}(v'(0))) = \conj{A}(\conj{\zeta}(v))
  \end{split}
\end{equation*}
where $\conj{A}\in \hom(V/D,\ti{V}/\ti{D})$ is the map induced by $A$
on the quotient spaces, that is
$\conj{A}\circ \pi^{V,V/D} = \pi^{\ti{V},\ti{V}/\ti{D}}\circ A$. As the
previous identity holds for all $v\in V$ (and $A$ is an isomorphism),
we have that
$\conj{\ti{\zeta}} \circ A|_D = \conj{A} \circ \conj{\zeta}$, hence
\begin{equation}\label{eq:grassmannian-tangent-tangent_application}
  \conj{T_D A^{\Gr}(\zeta)} \circ A|_D = \conj{A} \circ \conj{\zeta}.
\end{equation}


\section{Local computations in the Grassmann bundle}
\label{sec:local_computations_in_the_grassmann_bundle}

In this appendix we review the basic constructions of Grassmann
bundles and use them to perform a very detailed analysis leading to
the proof of some properties of the residual maps
$\cpres^r(\mathcal{D}_1,\mathcal{D}_2)$, where
$\mathcal{D}_j:=\ker(T\varphi_j)$ is a distribution associated to some
submersion.

Let $\discCP^j=(Q,\psi^j,L_\CPHS^j,f_\CPHS^j)$ be two discretizations
of the \FMS $(Q,L,f)$ that have order $r$ contact. If
$\conj{\varphi^j}:\conj{\mathcal{C}}\rightarrow \R\times TQ$ are
defined by~\eqref{eq:phi_bar_and_gamma_hat-def} (notice that
$E\simeq TQ$ and, in what follows, we prefer the more explicit $TQ$),
we have defined the distribution $\conj{\mathcal{D}^j}$ on
$\conj{\mathcal{C}}$ by
$\conj{\mathcal{D}^j}:= \ker(T\conj{\varphi^j})$. Explicitly,
using~\eqref{eq:skew_problem_with_D_from_tig-D_def}, if $n:=\dim(Q)$,
\begin{equation*}
  \conj{\mathcal{D}^j}(h,w,\ti{w}) = \ker(T_{(h,w,\ti{w})}\conj{\varphi^j}) 
  \in \Gr(T_{(h,w,\ti{w})}\conj{\mathcal{C}},n).
\end{equation*}
We can also define
$i_{\conj{\mathcal{D}^j}}:\conj{\mathcal{C}}\rightarrow \GrB$, where
$\GrB:=\GrB(T\conj{\mathcal{C}},n)$, by
\begin{equation*}
  i_{\conj{\mathcal{D}^j}}(h,w,\ti{w}) := \conj{\mathcal{D}^j}(h,w,\ti{w}).
\end{equation*}

Recall that any $G$-action $l^{\conj{\mathcal{C}}}$ defined in
$\conj{\mathcal{C}}$, induces $G$-actions on $\GrB$ and
$T\GrB$ by
\begin{equation*}
  l^{\GrB}_\tau(m,\mathcal{D}) := (l^{\conj{\mathcal{C}}}_\tau(m),
  T_ml^{\conj{\mathcal{C}}}_\tau(\mathcal{D})) \stext{ and }
  l^{T\GrB}_\tau((m,\mathcal{D}),\xi) := (l^{\GrB}_\tau(m,\mathcal{D}),
  T_{(m,\mathcal{D})}l^{\GrB}_\tau(\xi)).
\end{equation*}


\subsection{Coordinate charts}
\label{sec:coordinate_charts}

In what follows, we want to prove that
$i_{\conj{\mathcal{D}^2}} = i_{\conj{\mathcal{D}^1}} +
\mathcal{O}(\grade_{\conj{\mathcal{C}}}^r)$ and that the function
$\cpres^r(i_{\conj{\mathcal{D}^2}},i_{\conj{\mathcal{D}^1}}) :
\grade_{\conj{\mathcal{C}}}^{-1}(0) \rightarrow T\GrB$ is
$\Z_2$-equivariant. An important step will be the description of the
function $\delta i_{\conj{\mathcal{D}}}$ on $\conj{\mathcal{C}}$
---that would appear in~\eqref{eq:contact_order_condition-def} for
$f_j=i_{\conj{\mathcal{D}^j}}$ ---, at least for $h$ close to $0$. In
order to do this, we pick
$(0,w_0,\ti{w}_0)\in \grade_{\conj{\mathcal{C}}}^{-1}(0)$ and construct
coordinate charts of $\conj{\mathcal{C}}$ and $\GrB$ containing
$(0,w_0,\ti{w}_0)$ and $\conj{\mathcal{D}^j}(0,w_0,\ti{w}_0)$
respectively. Let $q_0:=\tau_Q(w_0)=\tau_Q(\ti{w}_0)$ and choose a
chart $(U,\psi)$ of $Q$ containing $q_0$. Then, let
$\hat{U}:= (\R\times (TU\oplus TU))\cap \conj{\mathcal{C}} \subset
\conj{\mathcal{C}}$; notice that, if $w_0$ and $\ti{w_0}$ are
sufficiently small, $\conj{\mathcal{C}}$ can be shrunk so that it is
$\Z_2$-invariant and, consequently, so is $\hat{U}$ (for the action
$l^{\R\times TQ\oplus TQ}_{[1]}(h,w,\ti{w}):=(h,\ti{w},w)$). Define
$\hat{\psi}:\hat{U}\rightarrow \R^{m}$ for $m:=1+3n$ by
$\hat{\psi}(h,w,\ti{w}):=(h,\psi(q),T_q\psi(w), T_q\psi(\ti{w}))$ for
$q:=\tau_Q(w)$. If we define $l^{\R^m}_{[1]}(h,x,a,b):=(h,x,b,a)$, we
have a $\Z_2$-action on $\R^m$ and $\hat{\psi}$ is $\Z_2$-equivariant.

On the other hand, we write
$\R^{m} = \R^1\oplus \R^n\oplus \R^n\oplus\R^n$ and label $\{e^h\}$,
$\{e^x_j:j=1,\ldots,n\}$, $\{e^a_j:j=1,\ldots,n\}$,
$\{e^b_j:j=1,\ldots,n\}$ the canonical bases of each of the direct
summands. Let 
\begin{equation}
  \label{eq:subspaces_for_definition_of_grassmanian_coords-def}
  \DD:=\vspan{\{e^b_j: j=1,\ldots, n\}} \stext{ and }
  \F:=\vspan{\{e^h,e^x_j,e^a_j:j=1,\ldots,n\}},
\end{equation}
so that $\DD\in \Gr(\R^m,n)$ and $\R^m = \DD\oplus\F$; let
$p_\DD\in \hom(\R^m,\DD)$ and $p_\F\in\hom(\R^m,\F)$ be the
projections associated to that direct sum decomposition. Define
\begin{equation}\label{eq:U_DF-def}
  U_{\DD,\F} := 
  \{D\in \Gr(\R^m,n):p_\DD|_D\in\hom(D,\DD) \text{ is an isomorphism}\}
\end{equation}
and
\begin{equation}\label{eq:hat_U_DF-def}
  \widehat{U_{\DD,\F}}:=\{((h,w,\ti{w}),\mathcal{D})\in \GrB: 
  (h,w,\ti{w})\in \hat{U} \text{ and } 
  T_{(h,w,\ti{w})}\hat{\psi}(\mathcal{D})\in U_{\DD,\F}\} 
\end{equation}
and let
$\Delta_{\DD,\F}:U_{\DD,\F}\rightarrow \R^{(1+2n)\times n}$ be
defined by
\begin{equation}\label{eq:Delta_DF-def}
  \Delta_{\DD,\F}(D):=p_\F \circ (p_\DD|_D)^{-1}.
\end{equation}
It is easy to check that $\Delta_{\DD,\F}$ is, indeed, a bijection
with inverse
\begin{equation*}
  \Delta_{\DD,\F}^{-1}(K):=\{(K v^b,v^b)\in \R^m: v^b\in\R^n\}.
\end{equation*}
Last, we define
$\Psi_{\DD,\F}:\widehat{U_{\DD,\F}}\rightarrow \R^m\times
\R^{(1+2n)\times n}$ by
\begin{equation}\label{eq:Psi_coordinates_on_grassmann_bundle}
  \Psi_{\DD,\F}((h,w,\ti{w}),\mathcal{D}) := (\hat{\psi}(h,w,\ti{w}), 
  \Delta_{\DD,\F}(T_{(h,w,\ti{w})}\hat{\psi}(\mathcal{D})));
\end{equation}
then, $(\widehat{U_{\DD,\F}},\Psi_{\DD,\F})$ is a coordinate chart
in $\GrB$.

In what follows, we also choose a coordinate chart $\R\times TU$ of
$\R\times TQ$ with coordinate function
$\eta(h,w):=(h,\psi(q),T_q\psi(w))$, for $q:=\tau_Q(w)$.


\subsection{Local description of $\conj{\varphi^j}$}
\label{sec:local_description_of_phi_bar}

With the notation as above, notice that, as
$(\{0\}\times (TU\oplus TU)) \subset \hat{U}\cap
\conj{\varphi^j}^{-1}(\R\times TU)$ (here we are using that
$\conj{\varphi^j}(0,w,\ti{w}) = (0,\frac{-1}{2}(w+\ti{w}))$), we can
replace $\hat{U}$ with
$\hat{U}\cap \conj{\varphi^j}^{-1}(\R\times TU)$, so that $\hat{U}$ is
open, $(\{0\}\times (TU\oplus TU)) \subset \hat{U}$ and
$\conj{\varphi^j}(\hat{U}) \subset \R\times TU$. With these choices,
we can describe $\conj{\varphi^j}$ locally in the chosen coordinates
of $\conj{\mathcal{C}}$ and $\R\times TQ$:
$\check{\conj{\varphi^j}}:\check{U}\rightarrow \R\times \psi(U) \times
\R^n$ where $\check{U}:= \hat{\psi}(\hat{U})$, that is, a smooth map
between open sets of $\R^m$ and $\R^{1+2n}$, whose domain contains
$(0,x_0,a_0,b_0):=\hat{\psi}(0,w_0,\ti{w}_0)$. We have
\begin{equation*}
  \check{\conj{\varphi^j}}(h,x,a,b) := (h, \chi^j(h,x,a,b), 
  \varsigma^j(h,x,a,b))
\end{equation*}
with $\chi^j$ and $\varsigma^j$ smooth and, because
of~\eqref{eq:phi_bar_when_h=0},
\begin{equation*}
  \chi^j(0,x,a,b)=x \stext{ and } \varsigma^j(0,x,a,b) = -\frac{1}{2}(a+b).
\end{equation*}
As a consequence,
\begin{equation}\label{eq:D234_phi_bar_check}
  \begin{split}
    D_{234}\check{\conj{\varphi^j}}(h,x,a,b) =& \left(
      \begin{array}{ccc}
         0_{1\times n} & 0_{1\times n} & 0_{1\times n}\\
        D_2\chi^j(h,x,a,b) & D_3\chi^j(h,x,a,b) & D_4\chi^j(h,x,a,b)\\
        D_2\varsigma^j(h,x,a,b) & D_3\varsigma^j(h,x,a,b) & D_4\varsigma^j(h,x,a,b)\\
      \end{array}
    \right),
  \end{split}
\end{equation}
and, in particular,
\begin{equation}\label{eq:D234_phi_bar_check-h=0}
  \begin{split}
    D_{234}\check{\conj{\varphi^j}}(0,x,a,b) =& \left(
      \begin{array}{cccc}
        0_{1\times n} & 0_{1\times n} & 0_{1\times n}\\
        1_{n\times n} & 0_{n\times n} & 0_{n\times n}\\
        0_{n\times n} & \frac{-1}{2} 1_{n\times n} & \frac{-1}{2} 1_{n\times n}\\
      \end{array}
    \right).
  \end{split}
\end{equation}


\subsection{Distribution associated to $\check{\conj{\varphi^j}}$}
\label{sec:distribution_locally}

Recalling that 
\begin{equation*}
  T\check{\conj{\varphi^j}}(h,x,a,b) = (\check{\conj{\varphi^j}}(h,x,a,b), 
  D\check{\conj{\varphi^j}}(h,x,a,b)), 
\end{equation*}
we have that
\begin{equation}\label{eq:ker_T234-def}
  \begin{split}
    \ker(T\check{\conj{\varphi^j}}(h,x,a,b)) =& \{ (\delta y_h, \delta
    y_x, \delta y_a, \delta y_b) \in \R^{1+3n} :\\&\phantom{\{ (\delta y_h,}
    D\check{\conj{\varphi^j}}(h,x,a,b)\transpose{(\delta y_h, \delta
      y_x, \delta y_a, \delta y_b)} = \transpose{(0,0,0,0)}\}
    \\\subset& T\R^4
  \end{split}
\end{equation}
and the kernel condition can be rewritten as
\begin{equation*}
  \left(
    \begin{array}{c}
      0\\0\\0
    \end{array}
  \right) = \Xi
  \left(
    \begin{array}{c}
      \delta y_h \\\delta y_x\\\delta y_a\\\delta y_b
    \end{array}
  \right)
\end{equation*}
for
\begin{equation*}
  \Xi =
  \left(
    \begin{array}{cccc}
      1& 0_{1\times n} & 0_{1\times n} & 0_{1\times n}\\
      D_1\chi^j(h,x,a,b) & D_2\chi^j(h,x,a,b) & D_3\chi^j(h,x,a,b) & D_4\chi^j(h,x,a,b)\\
      D_1\varsigma^j(h,x,a,b) & D_2\varsigma^j(h,x,a,b) & D_3\varsigma^j(h,x,a,b) & D_4\varsigma^j(h,x,a,b)\\
    \end{array}
  \right),
\end{equation*}
or $\delta y_h = 0$ and
\begin{equation*}
    \left(
    \begin{array}{c}
      0\\0
    \end{array}
  \right) = B^j(h,x,a,b) \left(
    \begin{array}{c}
      \delta y_x\\\delta y_a
    \end{array}
  \right) + A^j(h,x,a,b) \left(\delta y_b\right)
\end{equation*}
for
\begin{equation}\label{eq:A^j_and_B^j-def}
  \begin{split}
    B^j(h,x,a,b) :=&   \left(
      \begin{array}{cc}
        D_2\chi^j(h,x,a,b) & D_3\chi^j(h,x,a,b)\\
        D_2\varsigma^j(h,x,a,b) & D_3\varsigma^j(h,x,a,b)\\
      \end{array}
    \right),\\
    A^j(h,x,a,b) :=&   \left(
      \begin{array}{ccc}
        D_4\chi^j(h,x,a,b)\\
        D_4\varsigma^j(h,x,a,b)\\
      \end{array}
    \right).
  \end{split}
\end{equation}
Notice that, because of~\eqref{eq:D234_phi_bar_check-h=0}, we have
\begin{equation}\label{eq:A_and_B-h=0}
  B^j(0,x,a,b) = \left(
    \begin{array}{cc}
      1_{n\times n} & 0_{n\times n} \\
      0_{n\times n} & \frac{-1}{2} 1_{n\times n} 
    \end{array}
  \right) \stext{ and }
  A^j(0,x,a,b) = \left(
    \begin{array}{c}
      0_{n\times n} \\
      \frac{-1}{2} 1_{n\times n} 
    \end{array}
  \right).
\end{equation}
By~\eqref{eq:A_and_B-h=0} we see that $B^j(0,x,a,b)$ is invertible
and, so, $B^j(h,x,a,b)$ is invertible for sufficiently small $h$ (in a
neighborhood of $(0,x_0,a_0,b_0)$). Hence, we can
rewrite~\eqref{eq:ker_T234-def} as
\begin{equation*}
  \begin{split}
    \ker(T\check{\conj{\varphi^j}}(h,x,a,b)) = \{ (0,\delta &y_x,
    \delta y_a, \delta y_b) \in \R^{1+3n}: \\&\svec{\delta y_x}{\delta y_a}
    = -B^j(h,x,a,b)^{-1} A^j(h,x,a,b) \delta y_b\}.
  \end{split}
\end{equation*}
If we rewrite this last expression as
\begin{equation*}
  \begin{split}
    \ker(T\check{\conj{\varphi^j}}(h,x,a,b)) =& \bigg\{ (\delta
    y_h,\delta y_x, \delta y_a, \delta y_b) \in \R^{1+3n}: \\&\left(
      \begin{array}{c}
        \delta y_h\\\delta y_x\\\delta y_a
      \end{array}
    \right) = \left(
      \begin{array}{c}
        0_{1\times n}\\-B^j(h,x,a,b)^{-1} A^j(h,x,a,b)
      \end{array}
    \right) \left(\delta y_b\right)\bigg\},
  \end{split}
\end{equation*}
looking at~\eqref{eq:Delta_DF-def}, we have
\begin{equation}\label{eq:Delta_DF-ker_check}
  \Delta_{\DD,\F}(\ker(T\check{\conj{\varphi^j}}(h,x,a,b))) = 
          \left(
            \begin{array}{c}
              0_{1\times n}\\-B^j(h,x,a,b)^{-1} A^j(h,x,a,b)
            \end{array}
          \right)
\end{equation}
and, in particular,
\begin{equation}\label{eq:Delta_DF-ker_check-h=0}
  \Delta_{\DD,\F}(\ker(T\check{\conj{\varphi^j}}(0,x,a,b))) = 
  \left(
    \begin{array}{c}
      0_{1\times n}\\-B^j(0,x,a,b)^{-1} A^j(0,x,a,b)
    \end{array}
  \right) = \left(
    \begin{array}{c}
      0_{1\times n}\\0_{n\times n}\\-1_{n\times n}
    \end{array}
  \right).
\end{equation}


\subsection{Local description of $i_{\conj{\mathcal{D}^j}}$}
\label{sec:local_description_of_i_D_bar}

Recall that 
\begin{equation*}
  \begin{split}
    T_{(h,w,\ti{w})}\hat{\psi}(\conj{\mathcal{D}^j}(h,w,\ti{w})) =&
    T_{(h,w,\ti{w})}\hat{\psi}(\ker(T\conj{\varphi^j}(h,w,\ti{w}))) =
    \ker(T\check{\conj{\varphi^j}}(\hat{\psi}(h,w,\ti{w}))
  \end{split}
\end{equation*}
so that, from~\eqref{eq:Delta_DF-ker_check}, we have
\begin{equation*}
  \Delta_{\DD,\F}(T_{(h,w,\ti{w})}\hat{\psi}(\conj{\mathcal{D}^j}(h,w,\ti{w}))) =
  \left(
    \begin{array}{c}
      0_{1\times n}\\-B^j(\hat{\psi}(h,w,\ti{w}))^{-1} A^j(\hat{\psi}(h,w,\ti{w}))
    \end{array}
  \right)
\end{equation*}
Using the coordinate function $\Psi_{\DD,\F}$ on $\GrB$
defined in~\eqref{eq:Psi_coordinates_on_grassmann_bundle}, we have
\begin{equation*}
  \begin{split}
    \Psi_{\DD,\F}(i_{\conj{\mathcal{D}^j}}(h,w,\ti{w})) =&
    (\hat{\psi}(h,w,\ti{w}),
    \Delta_{\DD,\F}(T_{(h,w,\ti{w})}\hat{\psi}(\conj{\mathcal{D}^j}(h,w,\ti{w}))))
    \\=& \bigg(\hat{\psi}(h,w,\ti{w}),  \left(
      \begin{array}{c}
        0_{1\times n}\\-B^j(\hat{\psi}(h,w,\ti{w}))^{-1} A^j(\hat{\psi}(h,w,\ti{w}))
      \end{array}
    \right)\bigg),
  \end{split}
\end{equation*}
with $A^j$ and $B^j$ defined in~\eqref{eq:A^j_and_B^j-def}. In particular, 
\begin{equation}
  \label{eq:Psi_i_D_bar_explicit-h=0}
  \Psi_{\DD,\F}(i_{\conj{\mathcal{D}^j}}(0,w,\ti{w})) = 
  \bigg(\hat{\psi}(0,w,\ti{w}),  \left(
    \begin{array}{c}
      0_{1\times n}\\0_{n\times n}\\-1_{n\times n}
    \end{array}
  \right)\bigg).
\end{equation}

If we now take $(h,w,\ti{w})$ with $h$ small enough that
$B^j(\hat{\psi}(h,w,\ti{w}))$ are invertible for $j=1,2$, we can
compute
\begin{equation}\label{eq:diff_Psi_i_D_bar}
  \begin{split}
    \Psi_{\DD,\F}(&i_{\conj{\mathcal{D}^2}}(h,w,\ti{w})) -
    \Psi_{\DD,\F}(i_{\conj{\mathcal{D}^1}}(h,w,\ti{w})) \\=&
    \bigg(\hat{\psi}(h,w,\ti{w}), \left(
      \begin{array}{c}
        0_{1\times n}\\-B^2(\hat{\psi}(h,w,\ti{w}))^{-1} A^2(\hat{\psi}(h,w,\ti{w}))
      \end{array}
    \right)\bigg)
    \\&-
    \bigg(\hat{\psi}(h,w,\ti{w}), \left(
      \begin{array}{c}
        0_{1\times n}\\-B^1(\hat{\psi}(h,w,\ti{w}))^{-1} A^1(\hat{\psi}(h,w,\ti{w}))
      \end{array}
    \right)\bigg) \\=&
    \bigg(0, \left(
      \begin{array}{c}
        0_{1\times n} \\ B^1(\hat{\psi}(h,w,\ti{w}))^{-1} A^1(\hat{\psi}(h,w,\ti{w})) - 
        B^2(\hat{\psi}(h,w,\ti{w}))^{-1} A^2(\hat{\psi}(h,w,\ti{w}))
      \end{array}
    \right)\bigg).
  \end{split}
\end{equation}
Omitting the evaluation vector, we see that
\begin{equation}\label{eq:Bs_and_As-expanded}
  \begin{split}
    (B^1)^{-1}A^1 - (B^2)^{-1} A^2 =& (B^1)^{-1}A^1 -(B^1)^{-1} A^2 +
    (B^1)^{-1} A^2 - (B^2)^{-1} A^2 \\=& (B^1)^{-1}(A^1-A^2) +
    ((B^1)^{-1}-(B^2)^{-1}) A^2
  \end{split}
\end{equation}


\subsection{Computation of
  $\cpres^r(i_{\conj{\mathcal{D}^2}},i_{\conj{\mathcal{D}^1}})$}
\label{sec:computation_of_residual_of_i_D_bar}

Under the current hypotheses, by Lemma~\ref{le:order_of_phi_bar},
$\conj{\varphi^2} = \conj{\varphi^1} +
\mathcal{O}(\grade_{\conj{\mathcal{C}}}^r)$, that easily leads to
$\check{\conj{\varphi^2}} = \check{\conj{\varphi^1}} +
\mathcal{O}(\grade_{\R^m}^r)$ (see
Lemma~\ref{le:order_of_maps_in_terms_of_local_expression}). Explicitly,
this last expression means that
\begin{equation}\label{eq:delta_check_phi_bar-explicit}
  \check{\conj{\varphi^2}}(h,x,a,b) - \check{\conj{\varphi^1}}(h,x,a,b) =
  h^r \delta \check{\conj{\varphi}}(h,x,a,b),
\end{equation}
for a continuous (hence smooth) function
\begin{equation*}
  \delta \check{\conj{\varphi}}(h,x,a,b) = 
  (0,\delta \chi(h,x,a,b), \delta \varsigma(h,x,a,b))\in\R^{1+2n}.
\end{equation*}
Notice that, in particular,
\begin{equation}\label{eq:residue_local-delta_check_phi_bar-explicit}
  \cpres^r(\check{\conj{\varphi^2}}, \check{\conj{\varphi^1}})(0,x,a,b) = 
  \delta \check{\conj{\varphi}}(0,x,a,b) = 
  (0,\delta \chi(0,x,a,b), \delta \varsigma(0,x,a,b))
\end{equation}
and, by~\eqref{eq:residual_vs_delta_f},
\begin{equation}\label{eq:res_phi_bar_vs_res_local_phi_bar}
  \cpres^r(\check{\conj{\varphi^2}}, \check{\conj{\varphi^1}})(0,x,a,b) = 
  T_{\conj{\varphi^2}(\hat{\psi}^{-1}(0,x,a,b))} \eta 
  (\cpres^r(\conj{\varphi^2}, \conj{\varphi^1})(\hat{\psi}^{-1}(0,x,a,b)))
\end{equation}
Applying now $D_{234}$ to both sides
of~\eqref{eq:delta_check_phi_bar-explicit}, we obtain
  \begin{equation*}
  \begin{split}
    D_{234}\check{\conj{\varphi^2}}(h,x,a,b) -&
    D_{234}\check{\conj{\varphi^1}}(h,x,a,b) = h^r D_{234}\delta
    \check{\conj{\varphi}}(h,x,a,b)\\=& h^r \left(
      \begin{array}{ccc}
        0_{1\times n} & 0_{1\times n} & 0_{1\times n}\\
        D_2\delta \chi(h,x,a,b) & D_3\delta \chi(h,x,a,b) & D_4\delta \chi(h,x,a,b) \\
        D_2\delta \varsigma(h,x,a,b) & D_3\delta \varsigma(h,x,a,b) & D_4\delta \varsigma_(h,x,a,b) 
      \end{array}
    \right).
  \end{split}
\end{equation*}
Taking into account~\eqref{eq:D234_phi_bar_check}
and~\eqref{eq:A^j_and_B^j-def}, we have that
\begin{equation*}
  B^2(h,x,a,b) - B^1(h,x,a,b) = \underbrace{h^r \left(
      \begin{array}{cc}
        D_2\delta \chi(h,x,a,b) & D_3\delta \chi(h,x,a,b) \\
        D_2\delta \varsigma(h,x,a,b) & D_3\delta \varsigma(h,x,a,b)  
      \end{array}
    \right)}_{=:\delta B}
\end{equation*}
and
\begin{equation*}
  A^2(h,x,a,b) - A^1(h,x,a,b) = \underbrace{h^r \left(
    \begin{array}{cc}
      D_4\delta \chi(h,x,a,b) \\
      D_4\delta \varsigma(h,x,a,b) 
    \end{array}
  \right)}_{=:\delta A}.
\end{equation*}

Notice that, if $B^2-B^1=\delta B$ and both $B^j$ are invertible, then
$(B^1)^{-1} - (B^2)^{-1} = (B^2)^{-1}\delta B (B^1)^{-1}$, so that
$((B^1)^{-1} - (B^2)^{-1})A^2 = (B^2)^{-1}\delta B
(B^1)^{-1}A^2$. Considering the previous computations we have
\begin{equation*}
  \begin{split}
    (B^1&(h,x,a,b)^{-1} - B^2(h,x,a,b)^{-1})A^2(h,x,a,b)\\
    &=B^2(h,x,a,b)^{-1}\delta B(h,x,a,b) B^1(h,x,a,b)^{-1}A^2(h,x,a,b)
    \\ &=h^r B^2(h,x,a,b)^{-1} \left(
      \begin{array}{cc}
        D_2\delta \chi(h,x,a,b) & D_3\delta \chi(h,x,a,b) \\
        D_2\delta \varsigma(h,x,a,b) & D_3\delta \varsigma(h,x,a,b)  
      \end{array}
    \right)B^1(h,x,a,b)^{-1}A^2(h,x,a,b)
  \end{split}
\end{equation*}
and
\begin{equation*}
  \begin{split}
    B^1(h,x,a,b)^{-1}(A^1(h,x,a,b)-A^2(h,x,a,b)) =& -h^r B^1(h,x,a,b)\left(
      \begin{array}{cc}
        D_4\delta \chi(h,x,a,b) \\
        D_4\delta \varsigma(h,x,a,b) 
      \end{array}
    \right).
  \end{split}
\end{equation*}

Back in~\eqref{eq:diff_Psi_i_D_bar},
using~\eqref{eq:Bs_and_As-expanded} and the last two computations, we
have
\begin{equation}\label{eq:diff_Psi_i_D_bar_and_delta_i}
  \begin{split}
    \Psi_{\DD,\F}(i_{\conj{\mathcal{D}^2}}(h,w,\ti{w})) -
    \Psi_{\DD,\F}(i_{\conj{\mathcal{D}^1}}(h,w,\ti{w})) =& h^r \delta
    i_{\conj{\mathcal{D}}}(h,w,\ti{w}) \\=& h^r \bigg(0, \left(
      \begin{array}{c}
        0_{1\times n} \\ C(\hat{\psi}(h,w,\ti{w}))
      \end{array}
    \right)\bigg)
  \end{split}
\end{equation}
where, omitting the corresponding evaluation points $(h,x,a,b)$,
\begin{equation}\label{eq:C_def}
  \begin{split}
    C:=& -(B^1)^{-1}\left(
      \begin{array}{cc}
        D_4\delta \chi \\
        D_4\delta \varsigma 
      \end{array}
    \right) + (B^2)^{-1} \left(
      \begin{array}{cc}
        D_2\delta \chi & D_3\delta \chi \\
        D_2\delta \varsigma & D_3\delta \varsigma  
      \end{array}
    \right)(B^1)^{-1}A^2.
  \end{split}
\end{equation}

Now we remark that, because of~\eqref{eq:A_and_B-h=0}, for $h$
sufficiently small, $B^j(h,x,a,b)$ are invertible, so that the
previous computations are justified.

\begin{proof}[Proof of Lemma~\ref{le:contact_order_of_D_from_tig}]
  As $C$ defined in~\eqref{eq:C_def} is continuous in a neighborhood
  of $(0,x_0,a_0,b_0)$,
  expression~\eqref{eq:diff_Psi_i_D_bar_and_delta_i} says that
  $i_{\conj{\mathcal{D}^2}} = i_{\conj{\mathcal{D}^1}} +
  \mathcal{O}(\grade_{\conj{\mathcal{C}}}^r)$. In other words,
  $\conj{\mathcal{D}^2} =\conj{\mathcal{D}^1} +
  \mathcal{O}(\grade_{\conj{\mathcal{C}}}^r)$.
\end{proof}

Also, omitting the corresponding evaluation points $(0,x,a,b)$, we
have
\begin{equation*}
  \begin{split}
    C =& \left(
      \begin{array}{cc}
        -D_4\delta \chi \\
        2D_4\delta \varsigma 
      \end{array}
    \right) + \left(
      \begin{array}{cc}
        D_2\delta \chi & D_3\delta \chi \\
        -2D_2\delta \varsigma & -2D_3\delta \varsigma  
      \end{array}
    \right)\left(
      \begin{array}{c}
        0_{n\times n}\\1_{n\times n}
      \end{array}
    \right) =
    \left(
      \begin{array}{cc}
        -D_4\delta \chi + D_3\delta \chi\\
        2(D_4\delta \varsigma - D_3\delta \varsigma)
      \end{array}
    \right).
  \end{split}
\end{equation*}
Comparing~\eqref{eq:contact_order_condition-def}
with~\eqref{eq:diff_Psi_i_D_bar_and_delta_i} and taking into account
the last computation, from~\eqref{eq:residual_vs_delta_f} we conclude
that
\begin{equation}
  \label{eq:res^r(i_D_bar)}
  \begin{split}
    \cpres^r(i_{\conj{\mathcal{D}^2}},i_{\conj{\mathcal{D}^1}})&(0,w,\ti{w})
    =
    T_{\underbrace{\Psi_{\DD,\F}(\conj{\mathcal{D}^2}(0,w,\ti{w}))}_{=((0,x,a,b),\transpose{(0,0,-1_{n\times
            n})})}}\Psi_{\DD,\F}^{-1} \left(0_m, \left(
      \begin{array}{cc}
        0_{1\times n}\\ C(0,x,a,b)
      \end{array}
    \right)\right) \\=&
  T_{\Psi_{\DD,\F}(\conj{\mathcal{D}^2}(0,w,\ti{w}))}\Psi_{\DD,\F}^{-1}
  \left(0_m, \left(
      \begin{array}{cc}
        0_{1\times n}\\
        -D_4\delta \chi(0,x,a,b) + D_3\delta \chi(0,x,a,b)\\
        2(D_4\delta \varsigma(0,x,a,b) - D_3\delta \varsigma(0,x,a,b))
      \end{array}
    \right)\right)
  \end{split}
\end{equation}
for $(0,x,a,b):=\hat{\psi}(0,w,\ti{w})$.


\subsection{$\cpres^r(i_{\conj{\mathcal{D}^2}},i_{\conj{\mathcal{D}^1}})$
  at different points}
\label{sec:res^r_at_different_points}

Recall that, by Lemma~\ref{le:order_of_phi_bar},
$\conj{\varphi^j}|_{\grade_{\conj{\mathcal{C}}}^{-1}(0)}$ and
$\cpres^{r}(\conj{\varphi^2},
\conj{\varphi^1}):\grade_{\conj{\mathcal{C}}}^{-1}(0)\rightarrow T(\R\times
TQ)$ are $\Z_2$-invariant. Then, as $\hat{\psi}$ is
$\Z_2$-equivariant,~\eqref{eq:res_phi_bar_vs_res_local_phi_bar}
implies that
$\cpres^{r}(\check{\conj{\varphi^2}}, \check{\conj{\varphi^1}})$ is
$\Z_2$-invariant. Then,
by~\eqref{eq:residue_local-delta_check_phi_bar-explicit},
\begin{equation*}
  \delta \chi(0,x,a,b) = \delta \chi(0,x,b,a) \stext{ and }
  \delta \varsigma(0,x,a,b) = \delta \varsigma(0,x,b,a).
\end{equation*}
Consequently,
\begin{equation*}
  D_3\delta \chi(0,x,a,b) = D_4\delta \chi(0,x,b,a) \stext{ and }
  D_3\delta \varsigma(0,x,a,b) = D_4\delta \varsigma(0,x,b,a).
\end{equation*}
Then, from~\eqref{eq:res^r(i_D_bar)} and the last identities, we have
\begin{equation}\label{eq:res^r(i_D_bar)-switched}
  \begin{split}
    \cpres^r(i_{\conj{\mathcal{D}^2}},&i_{\conj{\mathcal{D}^1}})(0,\ti{w},w)
    \\=&
    T_{\Psi_{\DD,\F}(\conj{\mathcal{D}^2}(0,\ti{w},w))}\Psi_{\DD,\F}^{-1}
    \left(0_m,\left(
      \begin{array}{cc}
        0_{1\times n}\\
        -D_4\delta \chi(0,x,b,a) + D_3\delta \chi(0,x,b,a)\\
        2(D_4\delta \varsigma(0,x,b,a) - D_3\delta \varsigma(0,x,b,a))
      \end{array}
    \right)\right) \\=&
    T_{\Psi_{\DD,\F}(\conj{\mathcal{D}^2}(0,\ti{w},w))}\Psi_{\DD,\F}^{-1} \left(0_m,\left(
      \begin{array}{cc}
        0_{1\times n}\\
        -D_3\delta \chi(0,x,a,b) + D_4\delta \chi(0,x,a,b)\\
        2(D_3\delta \varsigma(0,x,a,b) - D_4\delta \varsigma(0,x,a,b))
      \end{array}
    \right)\right),
  \end{split}
\end{equation}
for $(0,x,a,b):=\hat{\psi}(0,w,\ti{w})$.


\subsection{Local expression of the $\Z_2$-action on $\GrB$}
\label{sec:local_expression_of_Z2_action_on_G}

As we know $\Z_2$ acts on $\GrB$ by
\begin{equation*}
  l^{\GrB}_{[1]}((h,w,\ti{w}),\mathcal{D}):= 
  ((h,\ti{w},w),T_{(h,w,\ti{w})} l^{\conj{\mathcal{C}}}_{[1]}(\mathcal{D})).
\end{equation*}
Below we want to give a description of this action in local
terms. Recall that $\DD$ and $\F$ were defined
by~\eqref{eq:subspaces_for_definition_of_grassmanian_coords-def}; define
\begin{equation*}
  \ti{\DD}:=\vspan{\{e^a_j:=1,\ldots, n\}} \stext{ and } 
  \ti{\F}:=\vspan{\{e^h,e^x_j,e^b_j:j=1,\ldots,n\}},
\end{equation*}
also, define $U_{\ti{\DD},\ti{\F}}$ and $\widehat{U_{\ti{\DD},\ti{\F}}}$
according to~\eqref{eq:U_DF-def} and~\eqref{eq:hat_U_DF-def}
respectively. Finally, let
\begin{equation}
  \label{eq:U_hat_sigma-def}
  \hat{U}_\sigma:= \widehat{U_{\DD,\F}}\cap
  \widehat{U_{\ti{\DD},\ti{\F}}};
\end{equation}
by definition, $\hat{U}_\sigma\subset \GrB$ is open. We
characterize the elements of $\hat{U}_\sigma$ as follows. If
$((h,w,\ti{w}),\mathcal{D})\in\widehat{U_{\DD,\F}}$, then
$(h,w,\ti{w})\in \hat{U}$ and
$T_{(h,w,\ti{w})}\hat{\psi}(\mathcal{D}) \in U_{\DD,\F}$, so that
\begin{equation*}
  T_{(h,w,\ti{w})}\hat{\psi}(\mathcal{D}) = \left\{\left(
    \begin{array}{c}
      K^h\\K^x\\K^a\\1_{n\times n}
    \end{array}
  \right) (\delta y^b) : \delta y^b\in\DD \right\}, 
  \stext{ where }
  \left(
    \begin{array}{c}
      K^h\\K^x\\K^a
    \end{array}
  \right) =
  \Delta_{\DD,\F}(T_{(h,w,\ti{w})}\hat{\psi}(\mathcal{D})).
\end{equation*}
Analogously, if $((h,w,\ti{w}),\mathcal{D})$ is in
$\widehat{U_{\ti{\DD},\ti{\F}}}$, then $(h,w,\ti{w})\in \hat{U}$ and
\begin{equation*}
  T_{(h,w,\ti{w})}\hat{\psi}(\mathcal{D}) = \left\{\left(
    \begin{array}{c}
      \ti{K}^h\\\ti{K}^x\\1_{n\times n}\\\ti{K}^b
    \end{array}
  \right) (\delta y^a) : \delta y^a\in\ti{\DD} \right\}, 
  \stext{ where }
  \left(
    \begin{array}{c}
      \ti{K}^h\\\ti{K}^x\\\ti{K}^b
    \end{array}
  \right) =
  \Delta_{\ti{\DD},\ti{\F}}(T_{(h,w,\ti{w})}\hat{\psi}(\mathcal{D})).
\end{equation*}
Comparing the two descriptions of
$T_{(h,w,\ti{w})}\hat{\psi}(\mathcal{D})$ we see that
$((h,w,\ti{w}),\mathcal{D})\in\hat{U}_\sigma$ if and only if $K^a$ and
$\ti{K}^b$ are invertible matrices.

Next we verify that $\hat{U}_\sigma$ is
$l^{\GrB}_{[1]}$-invariant. Let
$((h,w,\ti{w}),\mathcal{D})\in\hat{U}_\sigma\subset\widehat{U_{\DD,\F}}$. Then, 
\begin{equation*}
  l^{\GrB}_{[1]}((h,w,\ti{w}),\mathcal{D}) = 
  ((h,\ti{w},w), T_{(h,w,\ti{w})} l^{\conj{\mathcal{C}}}_{[1]} 
  (\mathcal{D})).
\end{equation*}
As
\begin{equation*}
  T_{(h,\ti{w},w)}\hat{\psi}(T_{(h,w,\ti{w})} l^{\conj{\mathcal{C}}}_{[1]}(\mathcal{D}))
  = T_{} l^{\R^m}_{[1]}(T_{(h,w,\ti{w})}\hat{\psi}(\mathcal{D})) = 
  l^{\R^m}_{[1]}(T_{(h,w,\ti{w})}\hat{\psi}(\mathcal{D})) =: \ti{D},
\end{equation*}
using the characterizations given above, we have
\begin{equation}\label{eq:local_description_Z2_action_on_grassmann_bundle}
  \begin{split}
    \ti{D} =& l^{\R^m}_{[1]}\left(\left\{\left(
    \begin{array}{c}
      K^h\\K^x\\K^a\\1_{n\times n}
    \end{array}
  \right) (\delta y^b) : \delta y^b\in\DD \right\}\right) = \left\{\left(
    \begin{array}{c}
      K^h\\K^x\\1_{n\times n}\\K^a
    \end{array}
  \right) (\delta y^b) : \delta y^b\in\DD \right\} \\=& \left\{\left(
    \begin{array}{c}
      K^h(K^a)^{-1}\\K^x(K^a)^{-1}\\(K^a)^{-1}\\1_{n\times n}
    \end{array}
  \right) (\underbrace{K^a\delta y^b}_{=:\delta z^b}) : \delta
  y^b\in\DD \right\} = \left\{\left(
    \begin{array}{c}
      K^h(K^a)^{-1}\\K^x(K^a)^{-1}\\(K^a)^{-1}\\1_{n\times n}
    \end{array}
  \right) (\delta z^b) : \delta
  z^b\in\DD \right\},
  \end{split}
\end{equation}
and we see that $\ti{D}\in U_{\DD,\F}$, so that, as $\hat{U}$ is
$\Z_2$-invariant,
$l^{\GrB}_{[1]}((h,w,\ti{w}),\mathcal{D}) \in
\widehat{U_{\DD,\F}}$. A completely similar computation shows that
$l^{\GrB}_{[1]}((h,w,\ti{w}),\mathcal{D}) \in
\widehat{U_{\ti{\DD},\ti{\F}}}$ and, all together, we have that
$l^{\GrB}_{[1]}((h,w,\ti{w}),\mathcal{D}) \in\hat{U}_\sigma$, proving
that $\hat{U}_\sigma$ is $\Z_2$-invariant.

Last, we want to give a local description of $l^{\GrB}$ restricted to
$\hat{U}_\sigma$, using the coordinate function
$\Psi_{\DD,\F}|_{\hat{U}_\sigma}$. We have the diagram
\begin{equation*}
  \xymatrix{
    {\hat{U}_\sigma} \ar[r]^{l^{\GrB}_{[1]}} \ar[d]_{\Psi_{\DD,\F}} 
    & {\hat{U}_\sigma} \ar[d]^{\Psi_{\DD,\F}}\\
    {\R^m\times \R^{(1+2n)\times n}} \ar[r]_{l^{\GrB(\R^m)}_{[1]}} 
    & {\R^m\times \R^{(1+2n)\times
        n}},
  }
\end{equation*}
where $l^{\GrB(\R^m)}_{[1]}$ is defined so that the diagram is
commutative (the other arrows being diffeomorphisms). Then, for $K=\left(
  \begin{array}{c}
    K^h\\K^x\\K^a
  \end{array}
\right)$, with $K^a$ invertible, we have
\begin{equation}\label{eq:Z_2_action_on_grassmann_bundle-local}
  \begin{split}
    l^{\GrB(\R^m)}_{[1]}((h,x,a,b),K)=&
    \Psi_{\DD,\F}(l^{\GrB}_{[1]}(\Psi_{\DD,\F}^{-1}(((h,x,a,b),K))))
    \\=&
    \Psi_{\DD,\F}(l^{\GrB}_{[1]}(\hat{\psi}^{-1}(h,x,a,b),T_{(h,x,a,b)}\hat{\psi}^{-1}(\im(\svec{K}{1_{n\times
        n}})))) \\=& \Psi_{\DD,\F}(\hat{\psi}^{-1}(h,x,b,a),
    T_{\hat{\psi}^{-1}(h,x,a,b)}
    l^{\conj{\mathcal{C}}}_{[1]}(T_{(h,x,a,b)}\hat{\psi}^{-1}(\im(\svec{K}{1_{n\times
        n}})))) \\=& ((h,x,b,a),\\&
    \Delta_{\DD,\F}(\underbrace{T_{\hat{\psi}^{-1}(h,x,b,a)}\hat{\psi}(T_{\hat{\psi}^{-1}(h,x,a,b)}
      l^{\conj{\mathcal{C}}}_{[1]}(T_{(h,x,a,b)}\hat{\psi}^{-1}(\im(\svec{K}{1_{n\times
          n}}))))}_{=l^{\R^m}_{[1]}(\im(\svec{K}{1_{n\times n}}))}))
    \\=& \left((h,x,b,a), \left(
    \begin{array}{c}
      K^h(K^a)^{-1}\\K^x(K^a)^{-1}\\(K^a)^{-1}
    \end{array}
  \right)\right)
  \end{split}
\end{equation}
where we
used~\eqref{eq:local_description_Z2_action_on_grassmann_bundle} in the
last identity.


\subsection{Local description of the lifted $\Z_2$-action on
  $T\GrB$}
\label{sec:local_description_of_lifted_Z2_action_on_tangent_grasmann_bundle}

Here we describe the $\Z_2$-action on $T\GrB$ obtained by
lifting $l^{\GrB}$. In particular, we want to give an explicit
description of this action on $T\hat{U}_\sigma$, where
$\hat{U}_\sigma$ is the $\Z_2$-invariant open subset defined
in~\eqref{eq:U_hat_sigma-def}. 

We start by finding an expression for
$T_{((h,x,a,b),K)}l^{\GrB(\R^m)}_{[1]}((\delta h, \delta x, \delta a,
\delta b), (\delta K))$, for $K = \left(
  \begin{array}{c}
    K^h\\K^x\\K^a
  \end{array}
\right)\in\R^{(1+2n)\times n}$ with $K^a$ invertible (hence
$((h,x,a,b),K)\in \Psi_{\DD,\F}(\hat{U}_\sigma)$).
\begin{equation}
  \label{eq:local_description_Z2_action_on_grasmann_bundle-partial}
  \begin{split}
    T_{((h,x,a,b),K)}&l^{\GrB(\R^m)}_{[1]}((\delta h, \delta x, \delta
    a, \delta b), (\delta K)) =\\=& \frac{d}{dt}\bigg|_{t=0}
    l^{\GrB(\R^m)}_{[1]}((h+t\delta h, x+t\delta x, a+t\delta a,
    b+t\delta b), (K+t\delta K))\\=& \frac{d}{dt}\bigg|_{t=0}
    \bigg((h+t\delta h, x+t\delta x, b+t\delta b, a+t \delta a), K(t)
    \bigg)
  \end{split}
\end{equation}
where $K(t)$ is the curve in $\R^{(1+2n)\times n}$ defined by
\begin{equation*}
  K(t) = \left(
    \begin{array}{c}
      (K^h+t(\delta K)^h)(K^a+t(\delta K)^a)^{-1}\\
      (K^x+t(\delta K)^x)(K^a+t(\delta K)^a)^{-1}\\
      (K^a+t(\delta K)^a)^{-1}
    \end{array}
  \right).
\end{equation*}
As $K^a$ is invertible, for $t$ small, $K^a+t(\delta K)^a$ is also
invertible and we have
\begin{equation*}
  1_{n\times n} = (K^a+t(\delta K)^a)^{-1}(K^a+t(\delta K)^a) ,
\end{equation*}
so that applying $\frac{d}{dt}\big|_{t=0}$ on both sides, we obtain
\begin{equation*}
  0_{n\times n}=\frac{d}{dt}\bigg|_{t=0}((K^a+t(\delta K)^a)^{-1}) K^a + (K^a)^{-1} (\delta K)^a,
\end{equation*}
and, then,
\begin{equation*}
  \frac{d}{dt}\bigg|_{t=0}((K^a+t(\delta K)^a)^{-1}) = 
  -(K^a)^{-1} (\delta K)^a (K^a)^{-1}.
\end{equation*}
Thus,
\begin{equation*}
  \begin{split}
    \frac{d}{dt}\bigg|_{t=0} K(t) =& \frac{d}{dt}\bigg|_{t=0} \left(
      \begin{array}{c}
        (K^h+t(\delta K)^h)(K^a+t(\delta K)^a)^{-1}\\
        (K^x+t(\delta K)^x)(K^a+t(\delta K)^a)^{-1}\\
        (K^a+t(\delta K)^a)^{-1}
      \end{array}
    \right) \\=&
    \left(
      \begin{array}{c}
        (\delta K)^h(K^a)^{-1} - K^h (K^a)^{-1} (\delta K)^a (K^a)^{-1}\\
        (\delta K)^x(K^a)^{-1} - K^x (K^a)^{-1} (\delta K)^a (K^a)^{-1}\\
        -(K^a)^{-1} (\delta K)^a (K^a)^{-1}
      \end{array}
    \right)
  \end{split}
\end{equation*}

Back
in~\eqref{eq:local_description_Z2_action_on_grasmann_bundle-partial},
we have
\begin{equation*}
  \begin{split}
    T_{((h,x,a,b),K)}&l^{\GrB(\R^m)}_{[1]}((\delta h, \delta x,
    \delta a, \delta b), (\delta K)) = \bigg( \left((h,x,b,a),\left(
    \begin{array}{c}
      K^h(K^a)^{-1}\\K^x(K^a)^{-1}\\(K^a)^{-1}
    \end{array}
    \right)\right), \\&\left((\delta h,\delta x, \delta b, \delta a),
    \left(
    \begin{array}{c}
      (\delta K)^h(K^a)^{-1} - K^h (K^a)^{-1} (\delta K)^a (K^a)^{-1}\\
      (\delta K)^x(K^a)^{-1} - K^x (K^a)^{-1} (\delta K)^a (K^a)^{-1}\\
      -(K^a)^{-1} (\delta K)^a (K^a)^{-1}
    \end{array}
  \right)\right) \bigg).
\end{split}
\end{equation*}

Finally, the action $l^{T\GrB}$ on $T\hat{U}_\sigma$ satisfies
\begin{equation}
  \label{eq:local_description_Z2_action_on_grasmann_bundle-final}
  \begin{split}
    l^{T\GrB}_{[1]}((&T_{((h,x,a,b),K)}\Psi_{\DD,\F}^{-1})((\delta h,
    \delta x, \delta a, \delta b), \delta K)) =\\=&
    T_{\Psi_{\DD,\F}^{-1}((h,x,a,b),K)}l^{\GrB}_{[1]}(T_{((h,x,a,b),K)}\Psi_{\DD,\F}^{-1}((\delta
    h, \delta x, \delta a, \delta b), \delta K)) \\=&
    T_{((h,x,a,b),K)}(l^{\GrB}_{[1]}\circ \Psi_{\DD,\F}^{-1})((\delta
    h, \delta x, \delta a, \delta b), \delta K)) \\=&
    T_{((h,x,a,b),K)}(\Psi_{\DD,\F}^{-1}\circ
    l^{\GrB(\R^m)}_{[1]})((\delta h, \delta x, \delta a, \delta b),
    \delta K)) \\=&
    T_{l^{\GrB(\R^m)}_{[1]}((h,x,a,b),K)}\Psi_{\DD,\F}^{-1}
    (T_{((h,x,a,b),K)}l^{\GrB(\R^m)}_{[1]}((\delta h, \delta x, \delta
    a, \delta b), \delta K)) \\=&
    T_{l^{\GrB(\R^m)}_{[1]}((h,x,a,b),K)}\Psi_{\DD,\F}^{-1}\bigg(
    (\delta h,\delta x, \delta b, \delta a), \\&
    \phantom{T_{l^{\R^m}_{[1]}((h,x,a,b),K)}\Psi_{\DD,\F}^{-1}\bigg(}
    \left(
      \begin{array}{c}
        (\delta K)^h(K^a)^{-1} - K^h (K^a)^{-1} (\delta K)^a (K^a)^{-1}\\
        (\delta K)^x(K^a)^{-1} - K^x (K^a)^{-1} (\delta K)^a (K^a)^{-1}\\
        -(K^a)^{-1} (\delta K)^a (K^a)^{-1}
      \end{array}
    \right)\bigg)
  \end{split}
\end{equation}


\subsection{Equivariance of
  $\cpres^r(i_{\conj{\mathcal{D}^2}},i_{\conj{\mathcal{D}^1}})$}
\label{sec:equivariance_of_residual_kernel_distributions}

In this section we conclude the local analysis of the maps
$i_{\conj{\mathcal{D}^j}}$ and prove
Proposition~\ref{prop:residual_i_D_bar_is_Z2_equivariant}.

\begin{proof}[Proof of
  Proposition~\ref{prop:residual_i_D_bar_is_Z2_equivariant}]
  With the previous choices of coordinates, for
  $(0,w,\ti{w})\in \hat{U}$, we
  recall~\eqref{eq:Psi_i_D_bar_explicit-h=0}:
  \begin{equation}\label{eq:Psi_i_D_bar_explicit-h=0-K_defined}
    \Psi_{\DD,\F}(i_{\conj{\mathcal{D}^j}}(0,w,\ti{w})) = 
    \bigg(\hat{\psi}(0,w,\ti{w}),  \underbrace{\left(
      \begin{array}{c}
        0\\0\\-1_{n\times n}
      \end{array}
    \right)}_{=:K}\bigg).
\end{equation}
  We see that the local representation of
  $i_{\conj{\mathcal{D}^j}}(0,w,\ti{w})$ has $K^a=-1_{n\times n}$ that
  is invertible so that, as seen in
  Section~\ref{sec:local_description_of_lifted_Z2_action_on_tangent_grasmann_bundle},
  $i_{\conj{\mathcal{D}^j}}(0,w,\ti{w}) \in \hat{U}_\sigma$. Then,
  using~\eqref{eq:local_description_Z2_action_on_grasmann_bundle-final}
  with $h=0$ and recalling~\eqref{eq:Psi_i_D_bar_explicit-h=0}, we
  obtain
  \begin{equation*}
    \begin{split}
      l^{T\GrB}_{[1]}(T_{((0,x,a,b),K)}\Psi_{\DD,\F}^{-1}(0_m, \delta K)) =&
      T_{\Psi_{\DD,\F}^{-1}((0,x,a,b),K)}l^{\GrB}_{[1]}
      (T_{((0,x,a,b),K)}\Psi_{\DD,\F}^{-1}(0_m, \delta K)) \\=&
      T_{l^{\GrB(\R^m)}_{[1]}((0,x,a,b),K)}\Psi_{\DD,\F}^{-1}\bigg(
      0_m, \left(
      \begin{array}{c}
        -(\delta K)^h \\
        -(\delta K)^x \\
        -(\delta K)^a
      \end{array}
    \right)\bigg),
    \end{split}
  \end{equation*}
  where $(0,x,a,b):=\hat{\psi}(0,w,\ti{w})$. Notice that, for $K$ as
  in~\eqref{eq:Psi_i_D_bar_explicit-h=0-K_defined},
  $l^{\GrB(\R^m)}_{[1]}((0,x,a,b),K) = ((0,x,b,a),K)$, due
  to~\eqref{eq:Z_2_action_on_grassmann_bundle-local}. Taking
  $\cpres^r(i_{\conj{\mathcal{D}^2}},i_{\conj{\mathcal{D}^1}})(0,w,\ti{w})
  \in T_{\conj{\mathcal{D}^2}(0,w,\ti{w})}\GrB$ as the tangent vector
  in the last computation and using~\eqref{eq:res^r(i_D_bar)}, we
  obtain
  \begin{equation*}
    \begin{split}
      l^{T\GrB}_{[1]}(&\cpres^r(i_{\conj{\mathcal{D}^2}},i_{\conj{\mathcal{D}^1}})
      (0,w,\ti{w})) =
      \\=& l^{T\GrB}_{[1]}\bigg(T_{((0,x,a,b),K)}\Psi_{\DD,\F}^{-1} \bigg(0_m,\left(
        \begin{array}{cc}
          0\\
          -D_4\delta \chi(0,x,a,b) + D_3\delta \chi(0,x,a,b)\\
          2(D_4\delta \varsigma(0,x,a,b) - D_3\delta \varsigma(0,x,a,b))
        \end{array}
      \right)\bigg)\bigg) \\=& T_{((0,x,b,a),K)}\Psi_{\DD,\F}^{-1}\bigg(
      0_m, \left(
        \begin{array}{c}
          0\\
          D_4\delta \chi(0,x,a,b) - D_3\delta \chi(0,x,a,b))\\
          -2(D_4\delta \varsigma(0,x,a,b) - D_3\delta \varsigma(0,x,a,b))
        \end{array}
      \right)\bigg). 
    \end{split}
  \end{equation*}
  Notice that, using~\eqref{eq:Psi_i_D_bar_explicit-h=0-K_defined} and
  the $\Z_2$-equivariance of $\hat{\psi}$,
  \begin{equation*}
    \begin{split}
      \Psi_{\DD,\F}(\conj{\mathcal{D}^2}(0,\ti{w},w)) =&
      (\hat{\psi}(0,\ti{w},w), K) =
      (\hat{\psi}(l^{\conj{\mathcal{C}}}_{[1]}(0,w,\ti{w}), K) \\=&
      (l^{\R^m}_{[1]}(\hat{\psi}(0,w,\ti{w})),K) = ((0,x,b,a),K),
    \end{split}
  \end{equation*}
  and the previous expression becomes
  \begin{equation*}
    \begin{split}
      l^{T\GrB}_{[1]}(\cpres^r(i_{\conj{\mathcal{D}^2}},&i_{\conj{\mathcal{D}^1}})
      (0,w,\ti{w}))
      =\\=&T_{\Psi_{\DD,\F}(\conj{\mathcal{D}^2}(0,\ti{w},w))}\Psi_{\DD,\F}^{-1}\bigg(
      0_m, \left(
        \begin{array}{c}
          0\\
          D_4\delta \chi(0,x,a,b) - D_3\delta \chi(0,x,a,b))\\
          -2(D_4\delta \varsigma(0,x,a,b) - D_3\delta \varsigma(0,x,a,b))
        \end{array}
      \right)\bigg).
    \end{split}
  \end{equation*}

  Comparison of this last expression
  with~\eqref{eq:res^r(i_D_bar)-switched} shows that
  \begin{equation*}
    l^{T\GrB}_{[1]}
    (\cpres^r(i_{\conj{\mathcal{D}^2}},i_{\conj{\mathcal{D}^1}})(0,w,\ti{w})) = 
    \cpres^r(i_{\conj{\mathcal{D}^2}},i_{\conj{\mathcal{D}^1}})
    (l^{\conj{\mathcal{C}}}_{[1]}(0,w,\ti{w})),
  \end{equation*}
  concluding the proof.
\end{proof}




\def\cprime{$'$} \def\polhk#1{\setbox0=\hbox{#1}{\ooalign{\hidewidth
  \lower1.5ex\hbox{`}\hidewidth\crcr\unhbox0}}} \def\cprime{$'$}
  \def\cprime{$'$}
\providecommand{\bysame}{\leavevmode\hbox to3em{\hrulefill}\thinspace}
\providecommand{\MR}{\relax\ifhmode\unskip\space\fi MR }
\providecommand{\MRhref}[2]{%
  \href{http://www.ams.org/mathscinet-getitem?mr=#1}{#2}
}
\providecommand{\href}[2]{#2}


\end{document}